\documentclass{article}[11pt]

\usepackage{fullpage}

\usepackage[margin=1in]{geometry}
\usepackage{amsmath}
\usepackage{amsfonts}
\usepackage{amssymb}
\usepackage{amsthm,mathtools,mathabx,accents,}
\usepackage{newlfont}
\usepackage{graphicx}
\usepackage{xcolor,enumerate,mathrsfs, scalerel}
\usepackage{esint}
\usepackage{hyperref}

\newtheorem{thm}{Theorem}[section]
\newtheorem{cor}[thm]{Corollary}
\newtheorem{lem}[thm]{Lemma}
\newtheorem{prop}[thm]{Proposition}
\newtheorem{conj}[thm]{Conjecture}

\theoremstyle{definition}
\newtheorem{defn}[thm]{Definition}
\newtheorem{rem}[thm]{Remark}
\numberwithin{equation}{section}

\newcommand{\norm}[1]{\Vert#1\Vert}

\newcommand{\na}{\nabla}
\newcommand{\pa}{\partial}
\newcommand{\lec}{\lesssim}

\renewcommand{\div}{\operatorname{div}}

\newcommand\de{\delta}

\newcommand\Ga{\Gamma}

\newcommand{\la}{\lambda}

\renewcommand{\th}{\theta}

\newcommand{\T}{\mathbb{T}}
\newcommand{\R}{\mathbb{R}}
\newcommand{\Z}{\mathbb{Z}}
\newcommand{\N}{\mathbb{N}}

\newcommand{\supp}{\operatorname{supp}}

\newcommand{\I}{\operatorname{Id}}

\newcommand{\MM}[1]{\ensuremath{\mathcal{M}}\left(#1\right)}
\newcommand{\MMSS}{\MM{N,L_v,\lambda_q,\ell_q^{-1}}}
\newcommand{\MMSa}{\MM{N,L_{v}-1,\lambda_q,\ell_q^{-1}}}

\newcommand{\MMSt}{\MM{N,L_t,\lambda_q,\ell_q^{-1}}}
\newcommand{\MMSta}{\MM{N,L_{t}-1,\lambda_q,\ell_q^{-1}}}
\newcommand{\MMStb}{\MM{N,L_{t}-2,\lambda_q,\ell_q^{-1}}}
\newcommand{\MMStc}{\MM{N,L_{t}-3,\lambda_q,\ell_q^{-1}}}
\newcommand{\MMStd}{\MM{N,L_{t}-4,\lambda_q,\ell_q^{-1}}}

\newcommand{\MMSA}{\MM{r,1,\mu_{q + 1},\ell_{t, q}^{-1}}}
\newcommand{\MMR}{\MM{N,L_{[r]},\lambda_q,\ell_q^{-1}}}

\DeclarePairedDelimiter{\ceil}{\lceil}{\rceil}

\title{\bf An Onsager-type Theorem for General 2D Active Scalar Equations }

\begin{document}
\author{
{ \small {\bf Xuanxuan ZHAO}}\thanks{\footnotesize The Institute of Mathematical Sciences, The Chinese University of Hong Kong, \href{xxzhao@math.cuhk.edu.hk}{xxzhao@math.cuhk.edu.hk}.}}
\date{}
	\maketitle
	\begin{abstract}
	This paper concerns the  Onsager-type problem for general 2-dimensional active scalar equations of the form: $\pa_t \theta+u\cdot\nabla \theta= 0$, with $u=T[\theta]$ being a divergence-free velocity field and  $T$ being a Fourier multiplier operator with  symbol $m$. It is shown that if $m$ is a  odd and  homogeneous symbol of order $\delta$: $m(\lambda\xi)=\lambda^{\delta} m(\xi)$, where $\lambda>0, -1\le\delta\le0$, then there exists a nontrivial temporally compact-supported weak solution $\theta\in C_t^0 C_x^{\frac{2\delta}{3}-}$, which fails to conserve Hamiltonian. This result is sharp since all weak solutions  of class $C_t^0C_x^{\frac{2\delta}{3}+}$ will necessarily conserve the Hamiltonian (which is proved by P. Isett and A. Ma in arXiv:2403.08279, 2024.) and thus resolves the flexible part of the generalized Onsager conjecture for general 2D odd active scalar equations.   Also, in the appendix,  analogous results have been obtained for general 2D and 3D even active scalar equations. The proof is achieved by using  convex integration scheme at the level $v=-\nabla^{\perp}\cdot\theta$ together with a Newton scheme recently introduced by V. Giri and R. O. Radu (2D Onsager conjecture: a Newton-Nash iteration. Invent. math. (2024).). Moreover, a novel algebraic lemma and sharp estimates for some complicated trilinear Fourier multipliers are established to overcome the difficulties caused by the generality of the equations.
\end{abstract}
\noindent\textbf{Keywords:} Onsager-type conjecture, active scalar equations, convex integration, Multilinear Fourier multipliers estimate
\tableofcontents
\section{Introduction}
\label{sec-intro}
\subsection{The problem and main results}\label{prob and results}
Consider  active scalar equations (ASE) on a periodic spatial domain $\mathbb{T}^2 = \mathbb{R}^2/(2\pi\mathbb{Z})^2$ of the form:
\begin{equation}
	\begin{cases}
		\partial_t \theta + u \cdot \nabla \theta = 0, \\
		u = T[\theta],
	\end{cases}
	\label{ASE}
\end{equation}
where the operator $T$  is represented in frequency space by the multiplier:
\[
\hat{u}(\xi) = m(\xi)\hat{\theta}(\xi),
\]
where $m(\xi)$ is a tempered distribution and homogeneous of degree $\delta$ with $-1 \leq \delta \leq 0$, so that $m(\lambda \xi) = \lambda^\delta m(\xi)$ for $\lambda > 0$. We assume that $m(\xi)$ is smooth away from the origin, odd, and perpendicular to the frequency $\xi$, i.e., $m(-\xi) = -m(\xi)$ and $m(\xi) \cdot \xi = 0$ for $\xi \neq 0$. In the appendix, the case with even $m(\xi)$  is also considered.

There are many physically motivated examples of active scalar equations, such as those describing atmospheric and oceanic fluid flows, stratification, rapid rotation, hydrostatic, and geostrophic balance. Below, we list some  specific examples:

(1) Active scalar equations with  odd multipliers:

(1.1) The modified surface quasi-geostrophic (mSQG) equation with symbol given by \cite{CMT, HPGS, Ma22}:
\[
m(\xi) = i \left\langle -\xi_2, \xi_1 \right\rangle |\xi|^{\delta-1},
\]
where $-1 \leq \delta \leq 0$. It is clear that $m(\xi)$ is an odd symbol, smooth on the unit sphere, and homogeneous of degree $\delta$. The case $\delta=-1$ corresponds to the 2D Euler equations in vorticity form, while the case $\delta=0$ corresponds to the SQG equation.

(1.2)There are numerous additional examples of odd ASE beyond the mSQG equations. A special class consists of those with symbol of the form:

\[
m(\xi) = \left(c_2 \xi_1^2 - c_3 \xi_1 \xi_2 + c_1 \xi_2^2\right)^{\frac{\delta-1}{2}} i \left\langle -\xi_2, \xi_1 \right\rangle,
\]
where $c_1, c_2, c_3 \in \mathbb{R}$, $-1 \leq \delta \leq 0$, and $c_3^2 < 4c_1c_2$. Clearly, the mSQG family is a special case of this by setting $c_1 = c_2$ and $c_3 = 0$. This family of ASEs are closely related to this paper and will be thoroughly analyzed in Section \ref{sec-class}.

(2) Active scalar equations with even multipliers:

(2.1) The incompressible porous media (IPM) equation, with velocity given by Darcy's law \cite{B, CG, IV}. The symbol in this case is given by
\[
m(\xi) = \left\langle \xi_1 \xi_2, -\xi_1^2 \right\rangle |\xi|^{-2},
\]
which is an even symbol, bounded and smooth on the unit sphere. The IPM equation also has a three-dimensional analogue, with the symbol
\[
m(\xi) = \left\langle \xi_1 \xi_3, \xi_2 \xi_3, -\xi_1^2 - \xi_2^2 \right\rangle |\xi|^{-2}.
\]

(2.2) The magneto-geostrophic (MG) equation \cite{FGSV, M08, ML,IV} is a three-dimensional active scalar equation with symbol given by
\[
m(\xi) = \left\langle \xi_2 \xi_3 |\xi|^2 + \xi_1 \xi_2^2 \xi_3, -\xi_1 \xi_3 |\xi|^2 + \xi_2^3 \xi_3, -\xi_2^2 (\xi_1^2 + \xi_2^2) \right\rangle \left( \xi_3^2 |\xi|^2 + \xi_2^4 \right)^{-1},
\]
for all $\xi \in \mathbb{Z}^3$ with $\xi_3 \neq 0$, and $m(\xi_1, \xi_2, 0) = 0$. The symbol of the MG equation is even, homogeneous of degree zero, and unbounded, since $m(\zeta^2, \zeta, 1) \to \infty$ as $|\zeta| \to \infty$.

The equation \eqref{ASE} admits several conservation laws, one of which is the Hamiltonian. For sufficiently smooth solutions $\theta$ of \eqref{ASE}, the Hamiltonian is defined as
\[
\mathcal{H}(t) = \frac{1}{2} \int_{\mathbb{T}^2} \theta T_0[\theta](x,t) \, dx,
\]
where $T_0$ is a Fourier multiplier with
\[
\widehat{T_0[\theta]}(\xi) = \bar{m}(\xi)\hat{\theta}(\xi) = \frac{-m(\xi) \cdot i\xi^\perp}{|\xi|^2} \hat{\theta}(\xi).
\]
Here, one has used the fact that $m(\xi)$ can be written as $m(\xi) = \bar{m}(\xi) i \xi^\perp$ due to the incompressibility condition $m(\xi) \cdot \xi = 0$. It follows that the Hamiltonian $\mathcal{H}(t)$ is conserved:
\[
\frac{d}{dt} \mathcal{H}(t) = \int_{\mathbb{T}^2} T_0[\theta] \partial_t \theta \, dx
= -\int_{\mathbb{T}^2} T_0[\theta] \nabla \cdot (\theta \nabla^\perp T_0[\theta]) \, dx
= \int_{\mathbb{T}^2} \theta \nabla T_0[\theta] \cdot \nabla^\perp T_0[\theta] \, dx = 0.
\]

Then a natural question that arises for active scalar equations is: what is the minimal regularity required for the conservation of the Hamiltonian? This question is exactly the Onsager-type conjecture for the ASE equations which is analogous to the famous one raised by Onsager for the Euler equation (\cite{Onsager49}) :

\begin{conj}\label{conj1}
\begin{enumerate}
\item Positive/Rigid part: All weak solutions $\theta$ satisfying $\Lambda^{-1}\theta \in C_{t,x}^{1+\frac{2\delta}{3}}$ of \eqref{ASE} conserve the Hamiltonian. Here  $\Lambda=(-\Delta)^\frac12$
\item Negative/flexible part: However, for any $\frac{1+\delta}{2} \leq \gamma < 1+\frac{2\delta}{3}$, there exist weak solutions with $\Lambda^{-1}\theta \in C_t^0C^\gamma_x$ that do not conserve the Hamiltonian.
\end{enumerate}
\end{conj}

The condition $\gamma \geq \frac{1+\delta}{2}$ ensures that the Hamiltonian is well-defined. The main purpose of this paper is to address the flexible part of Conjecture \ref{conj1}.

\begin{rem}
For the rigid part of Conjecture \ref{conj1}, in the special case of the SQG equation (i.e., $m(\xi) = \frac{i\xi^\perp}{|\xi|}$), the Hamiltonian conservation for solutions with $\theta \in L^3_{x,t}$ has been proved in \cite{IV}. For the vorticity formulation of the Euler equation (i.e., $m(\xi) = \frac{i\xi^\perp}{|\xi|^2}$),  energy conservation ($E(t) = \frac{1}{2} \int_{\mathbb{T}^2} |u(x,t)|^2 \, dx$, where $u = \nabla^\perp(-\Delta)^{-1}\theta$) for solutions in $u \in C^{\frac{1}{3}+}_{x,t}$ is established in \cite{CET}. More recently, for the more general mSQG equations (i.e., $m(\xi) = i\xi^\perp|\xi|^{\delta-1}$, with $-1 \leq \delta \leq 0$), the Hamiltonian conservation for solutions in $\theta \in L^3_t \dot{H}^{\frac{\delta-1}{2}}$ and $\Lambda^{-1}\theta \in L^3_t\dot{B}^{1+\frac{2\delta}{3}}_{3,c(\mathbb{N})}$ when $-1 \leq \delta < 0$, and $\theta \in L^3_{x,t} \cap L^3_t \dot{H}^{-\frac{1}{2}}$, $\Lambda^{-1}\theta \in L^3_t \dot{B}^{1}_{3,c(\mathbb{N})}$ when $\delta = 0$ has been obtained in \cite{IMnew}, which unifies the rigid parts of the Onsager-type conjecture for  both the 2D Euler and SQG equations. In fact, the method in \cite{IMnew} can be generalized to establish the Hamiltonian conservation for the broader ASE class considered in this paper with a slightly weaker regularity class: $\theta \in L^3_{x,t} \cap L^3_t \dot{H}^{\frac{\delta-1}{2}}$ and $\Lambda^{-1}\theta \in L^3_t C^{(1 + \frac{2\delta}{3})+}_x$, which remains sharp in view of Conjecture \ref{conj1}.

\end{rem}

Concerning the flexible part of the conjecture, which is considerably more difficult than its rigid part, in the case of the Euler equation, the conjecture has been proved by the  the so-called convex integration method or Nash iteration technique. This method was originally introduced by Nash to study the isometric embedding problem (\cite{Nash}) and was first adapted by De Lellis and Sz\'ekelyhidi in their seminal works (\cite{DLS1, DLS2}) to construct flexible solutions to the Euler equations. Later, in the breakthrough \cite{Is}, Isett fully resolved the flexible part of the Onsager conjecture for the 3D Euler equation using the glue technique and Mikado flow (see also \cite{BDLSV} for an extension). Recently, in \cite{GR23}, Giri and Radu solved the Onsager conjecture for the 2D Euler equation using the Newton-Nash scheme to address the problem of intersecting 2D Mikado flows in different directions. Additionally, Novack and Vicol \cite{NV22}, building on \cite{BMNV21} where energy-non-conservative $H^{\frac{1}{2}-}$ weak solutions of 3D incompressible Euler equation are constructed, provided a new proof of the intermittent Onsager conjecture for the 3D Euler equation using $L^2$-based spaces. In recent works \cite{GKN23} and \cite{GKN23'}, based on \cite{BMNV21} and \cite{NV22}, Giri, Kwon, and Novack solved the $L^3$-based strong Onsager conjecture using a wavelet-inspired $L^3$-based convex integration framework for the 3D Euler equation.

For the case of the SQG equation, Buckmaster, Shkoller, and Vicol in \cite{BSV} obtained the first partial result towards the flexible part of Conjecture \ref{conj1} by constructing weak solutions satisfying $\Lambda^{-1}\theta \in C_t^0C_x^\gamma$ for any $\frac{1}{2}<\gamma<\frac{4}{5}$ with a non-conservative Hamiltonian. Their approach builds on a Nash iteration scheme at the level of the potential field $v = \Lambda^{-1} u$ (where $-\nabla^\perp \cdot v = \theta$). Later, Isett and Ma (\cite{IM}) obtained the same result by working directly at the level of the scalar $\theta$, using a novel double-divergence convex integration scheme to tackle the problem of the oddness of the multiplier for the SQG equation. In \cite{CKL21}, Cheng, X., Kwon, H. and Li, D also achieved the same result by using  convex integration scheme to construct steady-state weak solutions to the SQG equation.  Recently, Dai, Giri, and Radu \cite{DGR24}, as well as Isett and S. Looi \cite{IL}, independently fully solved the flexible part of Conjecture \ref{conj1} by constructing weak solutions satisfying $\Lambda^{-1}\theta \in C_t^0C_x^\gamma$ for any $\frac{1}{2}\le\gamma < 1$, with a non-conservative Hamiltonian.

Then a natural question arises: what about the modified SQG equation ($m(\xi)=i\xi^{\perp}|\xi|^{\delta-1}, -1\le\delta\le0$), or more generally, the odd active scalar equation (ASE)? For the first question, \cite{Ma22} obtained a partial result by constructing weak solutions satisfying $\Lambda^{-1}\theta \in C_t^0C_x^\gamma$ for any $\gamma < \frac{4+3\delta}{5}$, with a non-conservative Hamiltonian. To the author's knowledge, there are no sharp results for the flexible part of Conjecture \ref{conj1} for either of these cases. Therefore, the main goal of this paper is to fully solve the flexible part of Conjecture \ref{conj1}, which, in particular, yields unified treatments for the 2D Euler equation and the SQG equation.

The main result of this paper is the following theorem, which  fully resolves  the flexible part of  Onsager-type conjecture for general ASE:

\begin{thm}\label{Main-thm1}
	For any $\frac{1+\delta}{2}\leq\gamma<1+\frac{2\delta}{3}$, there exist non-trivial weak solutions $\theta$ to \ref{ASE}, with compact support in time, such that $\Lambda^{-1}\theta\in C^{\gamma}\left(\mathbb{R}_t\times\mathbb{T}_x^2\right)$.
\end{thm}

\begin{rem}
	To make sense of the Hamiltonian, we define weak solutions in spaces of sufficiently low regularity so that Conjecture \ref{conj1} and Theorem \ref{Main-thm1} make sense. Following the notation of weak solutions in \cite{IMnew}, for a smooth, compactly-supported vector field $\psi(x)$ on $\mathbb{T}^2\times \mathbb{R}$, the quadratic form 
	$$Q[\psi,\theta]=\int_{0}^{1}\int_{\mathbb{T}^2}\psi T[\theta]\cdot\nabla\theta\, dx\, dt$$
	satisfies the estimate:
	$$|Q[\psi,\theta]|\lesssim_{m} \|\nabla^2\psi\|_{L^{\infty}}\|\theta\|^2_{\dot{H}^{\frac{\delta-1}{2}}}, \delta>-1$$
	$$|Q[\psi,\theta]|\lesssim_{m,\varepsilon} \|\nabla^{2+\varepsilon}\psi\|_{L^{\infty}}\|\theta\|^{2}_{\dot{H}^{\frac{\delta-1}{2}}}, \delta=-1$$
	for smooth $\theta$ with zero-mean. Thus the quadratic form $Q[\theta,\psi]$ has a unique bounded extension to $\theta\in\dot{H}^{\frac{\delta-1}{2}}$ and our notation of weak solution is well-defined.
\end{rem}

\subsection{Comments on previous related works, Difficulties, and New ideas}
We remark that, compared to previous works addressing the Onsager conjecture in special cases such as the 2D Euler equation in vorticity form [\cite{GR23}] and the SQG equation [\cite{IL}, \cite{DGR24}], which actually use some special structures of these specific equations, our results seem more general. This generality comes from the fact that our method not only unifies the treatments for the 2D Euler equation and the SQG equation, but also applies to a wide class of active scalar equations, provided that the associated multiplier \( m \) satisfies the following natural conditions:

\begin{enumerate}
	\item \textbf{Oddness}: \( m(-\xi) = -m(\xi) \),
	\item \textbf{Incompressibility}: \( m(\xi) \cdot \xi = 0 \) for all \( \xi \neq 0 \),
	\item \textbf{Smoothness}: \( m \) is smooth away from the origin,
	\item \textbf{Homogeneity}: \( m(\lambda \xi) = \lambda^{\delta} m(\xi) \) for all \( \lambda > 0 \) and \( \xi \neq 0 \), where \( -1 \leq \delta \leq 0 \).
\end{enumerate}

These conditions impose minimal constraints on \( m \). The key requirement is the homogeneity condition \( -1 \leq \delta \leq 0 \), which is reasonable since physical models corresponding to \( \delta < -1 \) and \( \delta > 0 \) are uncommon.
 Compared with the  Euler equation (which has special nonlinear structure $\div(u\otimes u)$) and the SQG equation (which has higher order of homogeneity $\delta=0$),
     the main difficulty in achieving the sharp regularity of Conjecture \ref{conj1} for general ASE arises from three key challenges:
\begin{enumerate}
	\item The first difficulty stems from the negative homogeneity of the multiplier $m(\lambda\xi)=\lambda^{\delta}m(\xi),  -1\leq\delta<0$. To be more specific, if one use the convex integration scheme directly at the level of scalar $\theta$,  then a  essential challenge is that sharp estimates for certain bi-linear Fourier multipliers are not attainable. One such multiplier is of the form 
	$$S[f,g]=\Delta^{-1}\div\left(T[\Delta f]g+T[g]\Delta f\right).$$
	 where $\Delta$ is the Laplacian operator. This bi-linear Fourier multiplier arises from specific commutators in the construction of the Newton step. After some manipulation in Fourier analysis, one would expect the estimate
	$$\|S[f,g]\|_0\lesssim \|f\|_1\|g\|_{\delta}+\|f\|_{1+\delta}\|g\|_0.$$
	However, if $\delta<0$ and $g$ has a low-frequency component, one can at best expect a rough estimate $\|g\|_{\delta}\le\|g\|_0$, which prevents achieving sharp regularity. This issue of course does not arise when $\delta=0$, as is the case for the SQG equation. To address this difficulty, we establish the convex integration scheme at the level of the potential field $v=(-\Delta)^{-1}\nabla^{\perp}\theta$, where $\theta=-\nabla^{\perp}\cdot v$ :
	$$\begin{cases}
		\partial_t v_q + u_q \cdot \nabla v_q - (\nabla v_q)^T\cdot u_q +\nabla p_q  =\div R_q,   \\
		\div v_q  = 0,  \\
		u_q  = T_1[v_q], 
	\end{cases}$$
	
	where $T_1$ is an even scalar multiplier with $\hat{u}(\xi)=\tilde{m}(\xi)\hat{v}(\xi)=-m(\xi)\cdot i\xi^{\perp} \hat{v}(\xi)$. Although this momentum formulation of ASE seems to make  the convex integration scheme more comlpex, it indeed offers an essential advantage: $T_1$ is now  a homogeneous multiplier of degree $0\le\delta+1\le1$, which is one degree higher than the original multiplier $T$. This additional degree enables us to derive the desired sharp estimates for some bi-linear Fourier multipliers similar to $S[f,g]$. Moreover, the pressure term $\nabla p$ can be used to  absorb some problematic terms that lack good estimates. Of course, one does not have to use such formulation of ASE  in the special case $\delta=0$. We note here that in  \cite{BSV}, this formulation of ASE for the special case of the SQG equation has also been used for a different purpose of  tackling the problem of oddness of the multiplier $m(\xi)=i\xi^{\perp}/|\xi|$ (with order of  homogeneity $\delta=0$), which now   seems unnecessary since one can use the convex integration scheme directly at the level of scalar $\theta$ together with $\nabla^{\perp}\cdot\div R_q$ or $\div\div R_q$  argument to address the same problem, as used in \cite{IM}, \cite{IL} and \cite{DGR24}.  
	
	\item The second difficulty  arises from establishing sharp  estimates  for some complicated trilinear Fourier multipliers.  When $\delta=0$, one can obtain sharp estimates for the $C^0$ norm of $u_q$ as $\|u_q\|_0\lesssim \delta_q\lambda_{q}$, here some  notations from section \ref{sec-induct} have been used. However, for $-1\le\delta<0$,  only the sharp estimates for the $C^1$ norm of $u_q$ of the form $\|u_q\|_1\lesssim\delta_q\lambda_{q}^{2+\delta}$ are available, which  will lead to restrictions when estimating various error terms involving $u_q$.  Specifically, when estimating the stress error term $R_{q,n}$ in section \ref{n+1-New}, we need to control the material derivative of a bi-linear Fourier multiplier: $(\partial_t+{u}_q\cdot\nabla)S[v_q,\psi_{k,n+1}] =: D_t S[v_q,\psi_{k,n+1}]$. Here, we use the notations from section \ref{n+1-New}, where
	$$S[v_q,\psi_{k,n+1}]=\Lambda^{-1}\big((\nabla u_q)^{T}\cdot\nabla^{\perp}\psi_{k,n+1}-(\nabla v_q)^{T}\cdot T_1[\nabla^{\perp}\psi_{k,n+1}]\big)$$
	For the special case  $\delta=0$,   $\|u_q\|_0$ can be estimated, thus one can write:
	\begin{equation}\notag
		\begin{aligned}
			D_t S[v_q,\psi_{k,n+1}]&={u}_q\cdot\nabla S[v_q,\psi_{k,n+1}] + S[\partial_t v_q,\psi_{k,n+1}] + S[ v_q,\partial_t\psi_{k,n+1}]\\
			&={u}_q\cdot\nabla S[v_q,\psi_{k,n+1}] + S[D_t v_q,\psi_{k,n+1}] + S[ v_q,D_t\psi_{k,n+1}]\\
			&-S[u_q\cdot\nabla v_q,\psi_{k,n+1}] - S[ v_q,u_q\cdot\nabla\psi_{k,n+1}]
		\end{aligned}
	\end{equation}
	and then estimate each term separately. However, such estimates involve $\|u_q\|_0$, which is unavailable for $-1\le\delta<0$. To address this difficulty, we revise the strategy and instead write:
	$$D_t S[v_q,\psi_{k,n+1}]=S[D_t v_q,\psi_{k,n+1}]+S[ v_q,D_t\psi_{k,n+1}]+S'[u_q,v_q,\psi_{k,n+1}]$$
	where $S'[u_q,v_q,\psi_{k,n+1}]={u}_q\cdot\nabla S[v_q,\psi_{k,n+1}]-S[u_q\cdot\nabla v_q,\psi_{k,n+1}]-S[v_q,u_q\cdot\nabla\psi_{k,n+1}]$ is a tri-linear Fourier multiplier. To handle this term effectively, we treat it as a whole to exploit some subtle cancellations, rather than estimating it term by term.  We  will make use of the  following  tri-linear Fourier multiplier estimate (see Lemma \ref{lem.trilin}, which is a key technical result of this paper):
	$$\|S'[u_q,v_q,\psi_{k,n+1}]\|_0\lesssim\|u_q\|_1(\|v_q\|_1\|\psi_{k,n+1}\|_{1+\delta}+\|v_q\|_{2+\delta}\|\psi_{k,n+1}\|_0)+\|u_q\|_{2+\delta}\|v_q\|_1\|\psi_{k,n+1}\|_0.$$
Note that \( 2 + \delta \geq 1 \), and thus this estimate is sharp. Such a tri-linear Fourier multiplier estimate offers an effective and efficient way to estimate various commutator terms, and to our  knowledge, it has not appeared in the existing literature on convex integration schemes. It is important to mention that the actual treatment of \( D_t[v_q, \psi_{k,n+1}] \) differs slightly from this argument because we do not have a sharp estimate for \( \| D_t v_q \|_0 \), though we do have a good estimate for \( \| D_t \nabla v_q \|_0 \). Therefore, the explicit expression of \( S[v_q, \psi_{k,n+1}] \) must be used carefully. However, the primary goal here is to convey the main ideas, so we present this simplified argument. 

It is important to note that such tri-linear Fourier multipliers do not appear in the convex integration scheme for the 2D Euler equation \cite{GR23} (corresponding to the case $\delta = -1$ and $T_1 = \text{identity}$). In that case, the bi-linear Fourier multiplier \( S[v_q, \psi_{k,n+1}] \) vanishes due to the special nonlinear structure \( \operatorname{div}(u_q \otimes u_q) \) of the Euler equations. However, such a special nonlinear structure is unavailable for general ASE.

	\item The third difficulty arises from the generality of the multiplier $m(\xi)$. As in the standard convex integration scheme, we have to use the low-frequency part of the quadratic perturbation term $T_1[w_{q+1}]\cdot w_{q+1}-(\nabla w_{q+1})^{T}T_1[w_{q+1}]$ to cancel the old stress $R_q$. For the Euler or SQG equations, due to their special multiplier structures, the stress $R_q$ can be required to be both symmetric and trace-free. This property enables \cite{GR23} and \cite{DGR24} to decompose the old stress $R_q$ into simple tensors of the form $\xi \otimes \xi$, allowing the application of standard geometric lemmas (Lemma D.1 in \cite{GR23}). 
	
	In contrast, for general ASE, we need to decompose $R_q$ into tensors of the form $\xi \otimes \nabla \bar{m}(\xi)$, which are generally not symmetric. Thus, we cannot utilize the standard geometric lemma. Instead, we must construct a well-defined anti-divergence operator $\mathcal{R}(m)$, depending on $m$, to find a 2-dimensional or 3-dimensional space $\mathcal{A}(m) \subset \mathring{\mathbb{M}}^{2\times2}$  such that for all $\xi \in \mathbb{R}^2$ and smooth vector fields $u$ with zero mean, the trace-free part of $\xi \otimes \nabla \bar{m}(\xi)$ lies in $\mathcal{A}(m)$ and $\mathcal{R}(m) u \subset \mathcal{A}(m)$, with $\operatorname{div} \mathcal{R}(m) u=u$. This is discussed further in section \ref{anti-div and geolem}. To find such an operator $\mathcal{R}(m)$, we need to classify the multiplier $m$, which will allow us to use $\mathcal{R}(m)$ to construct an algebraic lemma for decomposing the stress $R$ into the form $\xi \otimes \nabla \bar{m}(\xi)$.
	\end{enumerate}
	
\begin{rem}
	In the main part of  this paper, we do not consider the case when the multiplier $m$ is even. For a general even multiplier $m$ of zero-order, we can prove an analogous Onsager-type theorem using the standard Newton-Nash scheme directly at the scalar level of $\theta$. The proof is much simpler than for the case of odd multipliers, so we only provide a sketch of the convex integration framework in the appendix for completeness and for the  convenience of readers.
\end{rem}

\subsection{Further remarks}\label{open}
Given that we require the homogeneity of the multiplier $m$ to satisfy $-1 \leq \delta \leq 0$ in our main results, a natural question arises: what happens for  the general case of $\delta \in \mathbb{R}$?  The case $\delta < -1$ might   be treated similarly to the range $-1 \leq \delta < 0$ as follows: we transform the scalar equation $\partial_t \theta + u \cdot \nabla \theta = 0$ with $u = T[\theta]$ into the system 
\[
\partial_t v + u \cdot \nabla v - (\nabla v)^{T}u + \nabla p = 0, \quad \operatorname{div} v = 0, \quad u = T_1[v]
\]
to gain an additional degree of homogeneity, where $\theta = -\nabla^{\perp} v$. For the general case $-k \leq \delta \leq 0$, where $k \in \mathbb{N}$, it seems plausible to consider the following tensor system:
\[
\begin{cases}
	\partial_t v^{(k)} + B_k[u^{(k)}, v^{(k)}] \in \ker(\nabla^{\perp} \cdot)^k, \\
	u^{(k)} = T_k[v^{(k)}], \\
	v^{(k)} \in \ker(\nabla^{\perp} \cdot)^k,
\end{cases}
\]
along with the iterative relations:
\[
\begin{cases}
	v^{(0)} = \theta, \quad v^{(1)} = v, \quad T^{(0)} = T, \quad T^{(1)} = T_1, \\
	B_0[u^{(0)}, v^{(0)}] = u^{(0)} \cdot \nabla v^{(0)}, \\
	B_1[u^{(1)}, v^{(1)}] = u^{(1)} \cdot \nabla v^{(1)} - (\nabla v^{(1)})^{T} u^{(1)}, \\
	v^{(k-1)} = -\nabla^{\perp} \cdot v^{(k)}, \\
	B_{k-1}[u^{(k-1)}, v^{(k-1)}] = -\nabla^{\perp} \cdot B_k[u^{(k)}, v^{(k)}].
\end{cases}
\]

where $B_k[u^{(k)}, v^{(k)}]$ are quadratic forms in $u^{(k)}$ and $v^{(k)}$, and $T_k$ is a Fourier multiplier with homogeneity degree $0 \leq \delta + k \leq k$, which is $k$ degrees higher than the original multiplier $T$. Therefore, if one can find well-defined forms for $B_k$ and $T_k$ and apply a convex integration scheme at the level of $v^{(k)}$, it should be possible to treat the case $\delta < -1$. 

For $\delta > 0$, the main obstacle is that the $C^0$ norm of the Newton perturbation exceeds that of the Nash perturbation (Lemma \ref{le-est-Nash}): 
\[
\|w^{(p)}_{q+1}\|_0 \lesssim \delta_{q+1}^{1/2}, \quad \|w^{(t)}_{q+1}\|_0 \lesssim \delta_{q+1}^{1/2} \left( \frac{\lambda_{q+1}}{\lambda_q} \right)^{\delta/3},
\]
which makes it difficult to close the inductive assumption $\|v_{q+1} - v_q\|_0 \lesssim \delta_{q+1}^{1/2}$. Another difficulty arises from the lower-order terms $\delta B_{\lambda}^{\xi}$ (Lemma \ref{le_est_delta_B}), which admits only a weak improvement $\lambda_q / \lambda_{q+1}$ over the main term:
\[
\|\delta B_{\lambda}^{\xi}\|_0 \lesssim \delta_{q+1} \lambda_{q+1}^{1+\delta} \frac{\lambda_q}{\lambda_{q+1}},
\]
whereas a sharper estimate would require
\[
\|\delta B_{\lambda}^{\xi}\|_0 \lesssim \delta_{q+1} \lambda_{q+1}^{1+\delta} \left( \frac{\lambda_q}{\lambda_{q+1}} \right)^{1 + \delta/3}.
\]
This weak improvement comes from Lemma \ref{le-bilinear-odd}, where some terms are  expanded  only to first order. However, by expanding up to second order using a Taylor expansion, one can obtain a better estimate
\[
\|\delta B_{\lambda}^{\xi}\|_0 \lesssim \delta_{q+1} \lambda_{q+1}^{1+\delta} \left( \frac{\lambda_q}{\lambda_{q+1}} \right)^2.
\]
 A minor technical issue also arises from mollification errors, but this can be resolved by requiring the spatial mollification function $\zeta$ (from Section \ref{sec-Newton}) to have a higher vanishing moment. We remark here that one might employ the techniques from \cite{BMNV21}, \cite{NV22}, \cite{GKN23}, and \cite{GKN23'} to tackle the first main obstacle using certain multi-step Newton-Nash iteration involving parameters $\mu_{q+1,n}$, $\delta_{q+1,n}$, $\lambda_{q,n}$, and higher-order stress $R_{q,n}$. 

On the other hand, although this paper focuses on 2D active scalar equations, similar arguments may be extended to higher-dimensional ASEs to address related Onsager-type conjectures.

\subsection{Notations}
Throughout this paper, we will use the following notations:
\begin{enumerate}
\item
\textbf{Geometric upper bounds with two base}(\cite{BMNV21}). For all $n, N \in\mathbb{N}$, 
\begin{align}
	\MM{n,N,\lambda,\Lambda} := \lambda^{\min\{n,N\}} \Lambda^{\max\{n-N,0\}}    \,.
	\notag
\end{align}
\item \textbf{Definition of various Fourier multipliers}. 	Three Fourier multipliers: $T$, $T_0$ and $T_1$ correspond to multipliers $m(\xi)$, $\bar{m}(\xi)=\frac{-m(\xi) \cdot i \xi^{\perp}}{|\xi|^2}$ and $\tilde{m}(\xi)=-m(\xi) \cdot i \xi^{\perp}$ respectively. 
\item \textbf{Standard vector field notation}. Set $\nabla=(\pa_1,\pa_2)$, $\nabla^{\perp}=(-\pa_2,\pa_1)$ and $x^{\perp}=(-x_2, x_1)$ for  vector $x=(x_1,x_2).$
\item \textbf{Trace-free part of a matrix}: for a given matrix $A=\begin{bmatrix}
a & b\\
c&d
\end{bmatrix}$, let $\mathring{A}=\begin{bmatrix}
\frac12(a-d) & b\\
c&\frac12(d-a)
\end{bmatrix}$ to be its trace-free part. Especially, for given vectors $a=(a_1,a_2)$ and $b=(b_1,b_2)$, $a\mathring{\otimes}b$ denotes the trace-free part of $a\otimes b$.  Also,  a vector $(c_1,c_3,c_2)$ is said to be orthogonal to a given trace-free matrix $B=\begin{bmatrix}
e&f\\
g&-e
\end{bmatrix}$ if $c_1\cdot g+c_3\cdot e+c_2\cdot f=0$.
\item  \textbf{Vinogradov notations}: the notation $X\lesssim Y$ means $|X|\le C Y$ for some positive constant $C$. Also, the notation $X\lesssim_{k} Y $ means $|X|\lesssim C_k Y$ for some positive constant $C_k$ depending on $k$.
\end{enumerate}

	\section{Inductive bounds and the proof of the main theorem}
	\label{sec-induct} 
	
		\subsection{The AS momentum equations}  
	\label{sec:AS:momentum}
	
As mentioned before, we will use the momentum form of the active scalar equation
	\begin{subequations}
		\label{eq:ASM}
		\begin{align}
			\partial_t v + u \cdot \nabla v - ( \nabla v)^T\cdot u +\nabla p  =0&,  \label{eq:ASM-a} \\
			\div v  = 0&, \label{eq:ASM-b} \\
			u  = T_1[v]& \label{eq:ASM-c} \,,
		\end{align}
	\end{subequations}
	where $T_1$ is a scalar even multiplier with $\hat{u}(\xi)=\tilde{m}(\xi) \hat{v}(\xi)=-m(\xi)\cdot i\xi^{\perp} \hat{v}(\xi)$ . This formulation of active scalar equations is also used in \cite{BSV} for the special case of SQG equation. The AS momentum equation  can be equivalently written as
	$$
	\partial_t v + u^\perp \, ( \nabla ^\perp \cdot v)+\nabla p=0 \,.
	$$
	Define a scalar function $\theta$ in terms of  the vorticity of the potential velocity as :
	$$
	\theta = -  \nabla ^\perp \cdot v \,,
	$$
	Then \eqref{eq:ASM-a} becomes
	\begin{align}\label{eq:ASM-p-theta}
		\partial_t v - \theta u^\perp +\nabla p=0.
	\end{align}
	Thus a  direct computation confirms that $\theta$ is  a solution of (\ref{ASE}).

	The proof of theorem \ref{Main-thm1} will be achieved by the iterative construction of smooth solutions $(v_q, u_q, R_q)$ to the ASM-Reynolds system 
	\begin{equation} \label{ASMR-q}
		\begin{cases}
			\partial_t v_q + u_q \cdot \nabla v_q - ( \nabla v_q)^T\cdot u_q +\nabla p_q   = \div R_q, \\ 
			\div v_q = 0,\\
			u_q=T_1[v_q],
		\end{cases}
	\end{equation}
	where the  stress $R_q$ is a trace-free 2-tensor. Here and throughout $q \in \mathbb N$ will denote the stage of the iteration. The goal is to construct this sequence so that $(v_q,u_q)$ converges in the required H\"older space, while $R_q$ converges to zero. The limit  will yield a solution to \ref{eq:ASM} and thus a weak solution $\theta=-\nabla^{\perp}\cdot v$ to the ASE.
	
	\subsection{Inductive bounds}
	
	Let $\lambda_q$ be a frequency parameter defined as
	\begin{equation}\label{def.laq}
		\la_q :=  \ceil{a^{b^q}},
	\end{equation}
	where $a$ is a large real number and $b$ is such that $0<b-1\ll 1$. Define an amplitude parameter $\de_q$ as
	\begin{equation}\label{def.deq}
		\de_q := \la_q^{-2\beta}
	\end{equation}
	where $\beta$ is the coefficient  determining the regularity of the constructed solution.

	Let $L_v, L_R, L_t \in \mathbb N \setminus \{0\}$, $M > 0$ and $0< \alpha < 1$. and for $ r \in \{0, 1\}$,  the following inductive estimates hold:

	\begin{equation}  \label{induct-u}
		\| v_q\|_N \leq M \delta_q^{1/2} \lambda_q^N, \,\,\, \quad 1\leq N \leq L_v
	\end{equation}
	\begin{equation}  \label{induct-u'}
		\| u_q\|_N \leq M \delta_q^{1/2}\lambda_q^{1+\delta} \lambda_q^N,\quad 1\leq N \leq L_v \,\,\, 
	\end{equation}
	\begin{equation} \label{induct-R}
		\|D_t^r R_q\|_{N} \leq \delta_{q+1}\lambda_{q+1}^{1+\delta} (\delta_q^{1/2}\lambda_q^{2+\delta})^r\lambda_q^{N - 2\alpha}, \,\,\, 0\leq N\leq L_{[r]}
	\end{equation}
	Where and in what follows,  $L_{[0]}=L_R, L_{[1]}=L_t$ . 
	And it is also  assumed that the temporal support of $R_q$ satisfies
	\begin{equation}\label{induct-sup}
		\supp_t R_q \subset  [-2 +(\de_q^{\frac12}\la_q^{2+\delta})^{-1}, -1 -(\de_q^{\frac12}\la_q^{2+\delta})^{-1} ]  
		\cup [1 + (\de_q^{\frac12}\la_q^{2+\delta})^{-1} , 2 - (\de_q^{\frac12}\la_q^{2+\delta})^{-1} ].  
	\end{equation}
	
	Then  the key inductive proposition can be stated as:
	\begin{prop}\label{prop.main} Assume that $-1\le\delta\le0$, $L_v=100, L_R=75, L_t=50$, $0<\beta<1+\frac{2\delta}{3}$ and  $b>1$ is sufficiently close to $1$ such that
		\begin{equation}\label{cond-beta}
	\beta<\min\left\{\frac{\delta}{3}+\left(1+\frac{\delta}{3}\right)\frac{b+1}{2b}, \frac{3(1+\delta)b-\delta}{3(2b-1)}\right\}.
		\end{equation}
		 Then there exist $M_0=M_0(\beta, L_v, L_R, L_t) > 0$  and  $\alpha_0=\alpha_0(\beta, b)\in(0,1)$  such that for any $M > M_0$ and $0 < \alpha < \alpha_0$, there exists $a_0=a_0(\beta,b,\alpha, M, L_v, L_R, L_t)> 1$ , with the properties that for any $a > a_0$ the following holds: Given a smooth solution $(v_q, u_q, R_q)$ of \eqref{ASMR-q} and the inductive assumptions \eqref{induct-u}- \eqref{induct-sup}, there exists another smooth solution $(v_{q+1},u_{q+1},R_{q+1})$ again satisfying \eqref{induct-u}- \eqref{induct-sup} with $q$ replaced by $q+1$. Moreover, it holds that
		\begin{equation}\label{eq.prop.main}
			\la_{q+1}^{-1} \| (v_{q+1}-v_q)\|_1 + \|v_{q+1}-v_q\|_0 \leq 2M\delta_{q+1}^\frac12 \,
		\end{equation}
		and
		\begin{equation}\label{eq.prop.main-2}
			\supp_t v_{q+1} \cup \supp_t u_{q+1} \subset (-2,-1)\cup(1,2)\,.
		\end{equation}
	\end{prop}

	\subsection{The proof of the main theorem}
	Assuming Proposition \ref{prop.main} first, one can prove theorem \ref{Main-thm1} as follows. The proof is standard as in most other convex integration schemes  and similar to  that of \cite{GR23} and \cite{DGR24}. And the rest of this paper later is devoted to establishing Proposition \ref{prop.main}.
	\begin{proof}
		
		Let  $\frac{1+\delta}{2}\le\gamma <\beta< 1+\frac{2\delta}{3}$, where $\gamma$ is the H\"older coefficient in the  the theorem \ref{Main-thm1}. Fix $b$ such that \eqref{cond-beta} holds	and let $M_0$ and $\alpha_0$ be the constants given in Proposition \ref{prop.main}. Fix also $M > M_0+|\tilde{m}(1,0)|+ 2$ and $\alpha < \min \{\alpha_0, 1/2\}$. Then, let $a_0$ be given in Proposition \ref{prop.main} in terms of these fixed parameters. We do not fix $a > a_0$ until the end of the proof.
		
		We now  construct the base case for the inductive Proposition \ref{prop.main}. Let $f : \R \to [0,1]$ be a smooth function supported in $[-7/4, 7/4]$, such that $f = 1$ on $[-5/4, 5/4]$. Consider
		\begin{align*}
			&v_0 (x,t) = f(t) \delta_0^{1/2} \sin ( \la_0 x_1)e_2\,, \,\,\quad u_0(x,t) =  f(t) \delta_0^{1/2} \lambda_0^{1+\delta}\tilde{m}(1,0)\sin(\lambda_0x_1)e_2,\\
			&p_0(x,t)=\frac14f^2(t)\delta_0\la_0^{1+\delta}\cos(2\lambda_0x_1),\quad   R_0(x,t) =-f'(t) \frac{\delta_0^{1/2}}{\la_0}  \begin{pmatrix}
				0 & \cos(\lambda_{0}x_1)\\ \cos(\lambda_{0}x_1)& 0
			\end{pmatrix},
		\end{align*}
		where $(x_1, x_2)$ are the standard coordinates on $\T^2$ with $(e_1,e_2)$ being the associated unit vectors. Note that $R_0$ is symmetric .
		It is easy to check  that the tuple $(v_0,p_0,R_0)$ solves the  ASM-Reynolds system~\eqref{ASMR-q} and for  $N \geq 1$,  it holds that 
		\begin{equation*}
			\|v_0\|_N  \leq M\delta_0^{\frac{1}{2}} \lambda_0^N,
		\end{equation*}
		\begin{equation*}
			\|u_0\|_N  \leq M\delta_0^{\frac{1}{2}}\lambda_{0}^{1+\delta} \lambda_0^{N},
		\end{equation*}
		and thus  \eqref{induct-u} and \eqref{induct-u'} are valid. Also, for  $N \geq 0$,
		\begin{equation*}
			\|R_0\|_N \leq 2 \sup_t |f'(t)| \frac{\delta_0^{1/2}}{\lambda_0} \lambda_0^N.
		\end{equation*}
	     Note that  $\beta(2b-1) < \frac{1}{2}+(1+\delta)b$  implies
		\begin{equation*}
			2 \sup_t |f'(t)| < \delta_1 \lambda_0^{\frac{1}{2}}\lambda_1^{1+\delta}\delta_0^{-\frac{1}{2}}
		\end{equation*}
		for sufficiently large $a$. Then, 
		\begin{equation*}
			\|R_0\|_N \leq \delta_1\lambda_1^{1+\delta} \lambda_0^{-2\alpha} \lambda_0^N,
		\end{equation*}
		 which yields  \eqref{induct-R}  for $r=0$ due to  $\alpha < 1/4$. To estimate the terms involving material derivatives, one gets by direct calculations that 
		\begin{equation*}
				\partial_t R_0 + u_0 \cdot \nabla R_0 =
				-f''(t) \frac{\delta_0^{1/2}}{\lambda_0} \begin{pmatrix}
					0 & \cos ( \la_0 x_1)\\ \cos ( \la_0 x_1) & 0
				\end{pmatrix}
		\end{equation*}
		To show \eqref{induct-R} for $r=1$, one can choose $a$ large enough so that 
		\begin{equation*}
			2\sup_t |f''(t)|\frac{\delta_0^{1/2}}{\lambda_0}\leq \delta_1\lambda_1^{1+\delta}\delta_{0}\lambda_0^{2+\delta}\lambda_0^{-2\alpha}.
		\end{equation*}
		Finally, note that $\supp_t R_0 \subset [-7/4, 7/4] \setminus (-5/4, 5/4)$, and, thus, the condition \eqref{induct-sup} is satisfied provided that 
		\begin{equation*}
			(\delta_0^{1/2} \lambda_0^{2+\delta})^{-1} < \frac{1}{4},
		\end{equation*}
		which, once again, can be guaranteed by the choice of $a$.
		
		We finally fix $a$ so that all of the desired inequalities are satisfied, and conclude that the triple $(v_0, p_0, R_0)$ satisfies all the requirements to be the base case for the inductive Proposition \ref{prop.main}. Let $\{(v_q, p_q, R_q)\}$ be the sequence of solutions to the  system \eqref{ASMR-q} given by the iterative application of the proposition. Then  \eqref{eq.prop.main} implies that 
		\begin{equation*}
			\|v_{q+1} - v_q\|_\gamma \lesssim \|v_{q+1} - v_q\|_0^{1-\gamma} \|v_{q+1} - v_q\|_1^{\gamma} \lesssim \delta_{q+1}^{\gamma-\beta}  .
		\end{equation*}
		Since $\gamma< \beta<1+\frac{2\delta}{3}$, $\{v_q\}$ is a Cauchy sequence in $C_tC^{\gamma}_x$ and, thus, it converges in this space to a velocity field $v$. Moreover,
		$$ \|R_q\|_\gamma \lesssim \|R_q\|_0^{1-\gamma} \|R_q\|_1^{\gamma} \lesssim \delta_{q+1}\lambda_{q+1}^{1+\delta}\lambda_q^{\gamma}\lesssim\lambda_{q+1}^{(\frac{1}{b}-2)\beta+1+\delta}$$
		
		Thus $R_q$ converges to zero in $C_{t}C_{x}^\gamma$ since one can take $\frac{(1+\delta)b}{2b-1}<\beta<1+\frac{2\delta}{3}$. 
	Recall  \eqref{eq.prop.main-2}. The constructed $v$ satisfies $\supp_t v \subset [-2, 2]$ and 
		\begin{equation*}
			v(x,t) = \delta_0^{1/2} \sin(\lambda_0 x_1)e_2,
		\end{equation*}
	\begin{equation*}
			u(x,t) = \delta_0^{1/2}\la_0^{1+\delta}\tilde{m}(1,0) \sin(\lambda_0 x_1)e_2,
		\end{equation*}
		whenever $t \in [-1,1]$. Hence
$$\theta=-\nabla^{\perp}\cdot v=-\delta_{0}^{\frac12}\lambda_{0}\cos(\la_0 x_1),\quad t\in[-1,1]$$
$$\pa_t\theta+u\cdot \nabla\theta=0$$
Hence, the proof of Theorem \ref{Main-thm1} is completed.
	\end{proof}
	
	\subsection{Technical outline of the proof}
	\label{sec-heuristics}
	
We now present the main ideas of the proof of Proposition~\ref{prop.main} at a heuristic level for the reader’s convenience. It is important to note that the notations used here may  differ slightly from the actual definitions in this paper. The argument here and the notations of the Newton-Nash framework of this paper are similar to that of \cite{GR23} and \cite{DGR24}.

	\subsubsection{The Newton steps}
		Assume that $R_q$ has temporal support over a time interval of length $\tau_q=\left(\de_q^{\frac{1}{2}} \la_q^{2+\delta}\right)^{-1}$ centered at some time $t_0$.  Let $\Phi$ be the backward flow of $u_q$ with origin at $t_0$,
	\begin{equation}
		\begin{cases}
			\partial_t\Phi+u_q\cdot\nabla \Phi=0,\\
			\Phi_{t=t_0}= x.
		\end{cases}
	\end{equation}
	
	Sets also $X$ to be the Lagrangian flow of $u_q$ starting at $t=t_0$.  Utilizing the algebraic Lemma \ref{le-geo}, we can find  smooth  amplitude functions $a_\xi$ such that
	$$a_\xi=\lambda_{q+1}\delta_{q+1}^{1/2}\mathcal{L}^{(-m),\xi}_{(\nabla\Phi^{T}\xi^{(1)},\nabla\Phi^{T}\xi^{(2)},\nabla\Phi^{T}\xi^{(3)})}\left(D-\frac{R_q}{\delta_{q+1}\la_{q+1}^{1+\delta}}\right)$$
	satisfy
	$$\div\sum_{\xi\in F}\underbrace{\frac14  \lambda_{q+1}^{\delta-1}a_\xi^2 \nabla\Phi^{T}\xi\mathring{\otimes}\nabla \bar{m}(\nabla\Phi^{T}\xi)}_{A_\xi}=-\div R_q$$
	where $ F=\left\{\pm\xi^{(1)}, \pm\xi^{(2)},\widehat{\pm\xi^{(3)}}\right\}\subset \mathbb{Z}^2$  is a finite set of directions. The definitions of $F$, $D$ and the  smooth function $\mathcal{L}$ above are given in section \ref{sec-alge}.
	
	Let $\{g_{\xi}\}_{\xi\in F}$ be a set of 1-periodic functions of time with unit norm in $L^2(0,1)$ satisfying
	\[\supp_tg_{\xi}\cap \supp_t g_{\xi'}=0, \ \ \mbox{for} \ \xi\neq\xi'.\]
	The main purpose of introducing these profiles is to make the temporal supports of the Nash perturbations corresponding to different directions to be disjoint, which is also the main new idea of the Newton step. Then  define
	\begin{equation}\notag
		\begin{split}
			f_{\xi}&=1-g_{\xi}^2,\\
			f_{\xi}^{[1]}&=\int_0^tf_{\xi}(s)\, ds.
		\end{split}
	\end{equation}
	
	Let $\mu_{q+1}\gg \tau_q^{-1}$  be a temporal frequency parameter to be fixed later and define the first Newton perturbation $w_{q+1,1}^{(t)}=\nabla^{\perp}\psi^{(t)}_{q+1,1}$, where $\psi^{(t)}_{q+1,1}$ is the solution to the following transport  equations around $u_q$ starting from the initial time $t_0$:
	\begin{equation}\label{tran}
		\begin{cases}
			\partial_t \psi_{q+1,1}^{(t)}+u_q\cdot\nabla \psi_{q+1,1}^{(t)}=\sum_{\xi\in F}f_{\xi}(\mu_{q+1}t)\Delta^{-1}\nabla^{\perp}\cdot\div A_{\xi},\\
			\psi_{q+1,1}^{(t)}|_{t=t_0}=\frac{1}{\mu_{q+1}}\sum_{\xi\in  F}f_{\xi}^{[1]}(\mu_{q+1}t_0)\Delta^{-1}\nabla^{\perp}\cdot\div A_{\xi}|_{t=t_0}.
		\end{cases}
	\end{equation}
	then  solving the  transport equation (\ref{tran}) yields 
	\begin{equation}\notag
		\begin{split}
			&w_{q+1,1}^{(t)}(X,t)=\nabla^{\perp}\psi^{(t)}_{q+1,1}\\
			&=\frac{1}{\mu_{q+1}}\sum_{\xi\in F}f_{\xi}^{[1]}(\mu_{q+1}t_0)\nabla^{\perp}\Delta^{-1}\nabla^{\perp}\cdot\div A_{\xi}|_{t=t_0}+\int_{t_0}^t \sum_{\xi\in F} f_{\xi}(\mu_{q+1}s)\nabla^{\perp}\Delta^{-1}\nabla^{\perp}\cdot\div A_{\xi}(X(\cdot,s),s)\,ds\\
			&=\frac{1}{\mu_{q+1}}\sum_{\xi\in F}f_{\xi}^{[1]}(\mu_{q+1}t)\nabla^{\perp}\Delta^{-1}\nabla^{\perp}\cdot\div A_{\xi}(X,t)-\int_{t_0}^t \sum_{\xi\in F} f_{\xi}^{[1]}(\mu_{q+1}s)\frac{D_t\nabla^{\perp}\Delta^{-1}\nabla^{\perp}\cdot\div A_{\xi}}{\mu_{q+1}}(X(\cdot,s),s)\,ds.
		\end{split}
	\end{equation}
Note that the second term  above  is much smaller than the first term due to the condition $\mu_{q + 1}\gg\tau_q^{-1}$. Thus it holds that  
	\begin{equation}\label{w-t-size}
		w_{q+1,1}^{(t)}\approx \frac{1}{\mu_{q+1}}\sum_{\xi\in F}f_{\xi}^{[1]}(\mu_{q+1}t)\nabla^{\perp}\Delta^{-1}\nabla^{\perp}\cdot\div A_{\xi}(X,t),
	\end{equation}
which implies the following estimate
	\begin{equation}\notag
		\|w_{q+1,1}^{(t)}\|_0\lesssim \frac{\lambda_q\lambda_{q+1}^{1+\delta} \delta_{q+1}}{\mu_{q+1}}.
	\end{equation}
	Since the operator $T_1$ is homogeneous of order $1+\delta$, it is expected that  
	$$\|T_1[w_{q+1,1}^{(t)}]\|_0\lesssim \frac{\lambda_q^{2+\delta}\lambda_{q+1}^{1+\delta} \delta_{q+1}}{\mu_{q+1}}$$
	The main drawback of the Newton step is the following newly generated  Newton error:
	\[R_{q+1}^{\text{Newton}}=\div^{-1}\left(T_1[w_{q+1,1}^{(t)}]^{\perp}\left(\nabla^{\perp}\cdot w^{(t)}_{q+1,1}\right) \right).\]
	where  the anti-divergence  operator $\div^{-1}=\mathcal{R}(m)$ is given in appendix \ref{sec-class}, which depends essentially on the multiplier $m$. We will  fix the use of such anti-divergence operator throughout this paper. Since the operator $\div^{-1}$ is  $-1$-order, the following estimate is expected:
\begin{equation}\label{est-Newton}
\|R_{q+1}^{\text{Newton}}\|_0\lesssim \lambda_{q}^{-1}\frac{\lambda_q^{4+\delta}\lambda_{q+1}^{2(\delta+1)} \delta_{q+1}^2}{\mu_{q+1}^2}.
\end{equation}
	Note that $\psi^{(t)}_{q+1,1}$ is the solution to the transport equation \eqref{tran} and  $w_{q+1,1}^{(t)}=\nabla^{\perp}\psi^{(t)}_{q+1,1}$, which are defined only locally   in time. Thus we need to use temporal cut-off functions to glue together these perturbations. To be more specific,  let $\widetilde \chi$ be a standard smooth temporal cut-off function, which satisfies $\widetilde \chi=1$ on $\cup_{\xi} \supp A_{\xi}$ and $|\partial_t\widetilde \chi| \lesssim \tau_q^{-1}$. Then the actual Newton perturbation will be the superposition of $\tilde{\chi}w^{(t)}_{q+1,1}$. This procedure will produce the following glue error:
	\begin{equation}\notag
	\begin{aligned}
	R_q^{\text{glue}}&=\div^{-1}\partial_t\widetilde \chi w_{q+1,1}^{(t)}\\
	&\approx \frac{1}{\mu_{q+1}}\pa_t\tilde{\chi}\sum_{\xi\in F}f_{\xi}^{[1]}(\mu_{q+1}t)\div^{-1}\nabla^{\perp}\Delta^{-1}\nabla^{\perp}\cdot\div A_{\xi}
	\end{aligned}
	\end{equation}
	 where the equation $\eqref{w-t-size}$ is used. Since that $\div^{-1}\nabla^{\perp}\Delta^{-1}\nabla^{\perp}\cdot\div$ is a zero-order operator,	 the following  estimate is expected:
	\[\|R_q^{\text{glue}}\|_0\lesssim \frac{1}{\mu_{q+1}}|\partial_t\widetilde \chi|\|A_{\xi}\|_0\lesssim \frac{\tau_q^{-1}\lambda_{q+1}^{1+\delta}\delta_{q+1}}{\mu_{q+1}}.\]

	 However, compared to $R_q$, $R_q^{\text{glue}}$  has only weak improved estimates by a factor of
	 $$\frac{\tau_q^{-1}}{\mu_{q + 1}}=\left(\frac{\la_q}{\la_{q+1}}\right)^{1+\frac{2\delta}{3}-\beta},\quad\text{which comes from the choice of \eqref{mu}}$$
	  This weak improvement is not enough to place $R_q^{\text{glue}}$ into the new error $R_{q+1}$, which is expected to have a better estimate:
	  $$\|R_{q+1}\|_0\lesssim \delta_{q+2}\la_{q+2}^{1+\delta}\approx\delta_{q+1}\la_{q+1}^{1+\delta}\left(\frac{\la_q}{\la_{q+1}}\right)^{1+\frac{\delta}{3}}$$
	Hence it is necessary to repeat the Newton step above up to $\Gamma\approx\left\lceil\frac{1+\frac{\delta}{3}}{1+\frac{2 \delta}{3}-\beta}\right\rceil$ times until the final glue error can be absorbed into $R_{q+1}$.

	\subsubsection{The Nash step}
	 By the argument in the Newton step, we have constructed  a new solution to the ASM-Reynolds system \eqref{ASMR-q} with velocity fields
	\[v_{q,\Gamma}=v_q+w_{q+1}^{(t)}=v_q+\sum_{n=1}^\Gamma w_{q+1,n}^{(t)}\]
	\[u_{q,\Gamma}=T_1[v_q+w_{q+1}^{(t)}]=u_q+\sum_{n=1}^\Gamma T_1[w_{q+1,n}^{(t)}]\]
 and new stress error of the form 
 $$R_{q,\Gamma}=\underbrace{-\sum_{\xi\in F}g^2_{\xi}A_\xi}_{=:R^{\text{rem}}_{q}} +\text{Newton errors}$$

	Let  $\widetilde \Phi$ be the backward flow of $u_{q,\Gamma}$ starting at time $t=t_0$. Then we define the Nash perturbation by
		\[w_{q+1}^{(p)}= \sum_{\xi\in F}\Delta^{-1} \nabla^{\perp} P_{\approx \lambda_{q+1 }}\left(g_\xi \bar a_\xi \cos\left(\lambda_{q+1}\tilde{\Phi}\cdot\xi\right)\right)\]
   where $P_{\approx \lambda_{q+1}}$ localizes in Fourier space of order $\lambda_{q+1}$. Note that $w^{(p)}_{q+1}$ defined in this way is indeed divergence free due to the existence of the factor $\nabla^{\perp}$. The amplitude functions $\bar{a}_{\xi}$ are defined similarly as for $a_\xi$ as follows:
	\[\bar a_\xi=\lambda_{q+1}\delta_{q+1}^{1/2}\mathcal{L}^{(-m),\xi}_{(\nabla\tilde{\Phi}^{T}\xi^{(1)},\nabla\tilde{\Phi}^{T}\xi^{(2)},\nabla\tilde{\Phi}^{T}\xi^{(3)})}\left(D-\frac{R_q}{\delta_{q+1}\la_{q+1}^{1+\delta}}\right),\]

		 The main advantage of involving  the factors $\{g_{\xi}\}$ is that they have disjoint temporal supports, which ensures
	\begin{equation*}
		\begin{aligned}
			&T_1[w^{(p)}_{q+1}]^{\perp}\left(\nabla^{\perp}\cdot w^{(p)}_{q+1}\right)\\ 
			&=-\sum_{\xi\in F}T_1[\nabla\Delta^{-1}P_{\approx \lambda_{q+1 }}\left(g_\xi \bar a_\xi \cos\left(\lambda_{q+1}\Phi\cdot\xi\right)\right)]P_{\approx \lambda_{q+1 }}\left(g_\xi \bar a_\xi \cos\left(\lambda_{q+1}\Phi\cdot\xi\right)\right)\\
		\end{aligned}
	\end{equation*}
Using the bi-linear micro-local Lemma \ref{le-bilinear-odd}, one can get
	$$T[w^{(p)}_{q+1}]^{\perp}\left(\nabla^{\perp}\cdot w^{(p)}_{q+1}\right)= \div\sum_{\xi \in F} g_\xi^2\underbrace{\frac{1}{4}\lambda_{q+1}^{\delta-1}\bar{a}_\xi^2  \nabla\tilde{\Phi}^{T}\xi\mathring{\otimes}\nabla\bar{m}\left(\nabla\tilde{\Phi}^{T}\xi\right)}_{\bar{A}_\xi} +\delta B^\xi_{\lambda} 
	$$
	
	where $\delta B^\xi_\lambda$ are lower order terms with explicit  formula given in Lemma \ref{le-bilinear-odd} and thus we have
	\begin{equation*}
		\div  R_q^{\text{rem}} +  T[w^{(p)}_{q+1}]^{\perp}\left(\nabla^{\perp}\cdot w^{(p)}_{q+1}\right)=\div  \sum_{\xi \in F} g_\xi^2\left( \underbrace{ A_\xi -\bar A_\xi}_{R_{q+1}^{\text{flow}}} +  \delta B^\xi_{\lambda} \right).
	\end{equation*}
	
	The main difference between $A_{\xi}$ and $\bar{A}_{\xi}$ comes from $\Phi$ and $\tilde{\Phi}$, thus the following estimate is expected for the flow error:
	\[\|R_{q+1}^{\text{flow}}\|_0\lesssim \lambda^{1+\delta}_{q+1}\delta_{q+1}\|\nabla\Phi-\nabla\widetilde \Phi\|_0.\]
	Note that
	\begin{equation}\notag
		\begin{cases}
			\partial_t(\Phi-\widetilde \Phi)+u_q\cdot\nabla (\Phi-\widetilde \Phi)=T_1[w_{q+1}^{(t)}]\cdot\nabla\widetilde \Phi,\\
			(\Phi-\widetilde \Phi)|_{t=t_0}=0.
		\end{cases}
	\end{equation}
	thus we have the following stability estimate  on time-scales of size $\tau_q$,
	\[\|\nabla\Phi-\nabla\widetilde \Phi\|_0\lesssim \lambda_q\tau_q\|T_1[w_{q+1}^{(t)}]\|_0,\]
	which implies
	\[\|R_{q+1}^{\text{flow}}\|_0\lesssim \frac{\delta_{q+1}^{2}\lambda_{q+1}^{2\delta+2} \lambda_q^{3+\delta} \tau_q}{\mu_{q+1}}.\]

	 In view of Lemma \ref{le_est_delta_B}, the low order term $\delta B^{\xi}_{\lambda}$  satisfies 
	\begin{equation*}
		\|\delta B^\xi_{\lambda}\|_0 \lesssim \delta_{q+1}\lambda_{q+1}^{1+\delta}\frac{\lambda_q}{\lambda_{q+1}}.
	\end{equation*}
One of the main errors generated in the Nash step is the transport error:
	\[R_{q+1}^{\text{transport}}=\div^{-1}\left((\partial_t+u_{q,\Gamma}\cdot\nabla)w_{q+1}^{(p)}\right),\]
 which is expected to have the estimate
 \begin{equation}\label{est-transport}
 \|R_{q+1}^{\text{transport}}\|_0\lesssim \lambda_{q+1}^{-1}\mu_{q+1}\delta_{q+1}^{\frac12}.
 \end{equation}
Optimizing the estimates \eqref{est-transport} and \eqref{est-Newton} gives
\begin{equation}\label{mu}
\mu_{q+1}=\delta_{q+1}^{\frac12}\lambda_{q+1}^{1+\frac{2\delta}{3}}\lambda_q^{1+\frac{\delta}{3}}.
\end{equation}
Note that $\mu_{q+1} > \tau_q^{-1}$ holds since $\beta<1+\frac{2\delta}{3}$.
	Then the following estimates are achieved:
	\begin{equation*}
		\|R_{q+1}^{\text{transport}}\|_0 + \|R_{q+1}^{\text{Newton}}\|_0 \lesssim \delta_{q+1} \lambda_{q+1}^{\frac{2\delta}{3}}\lambda_q^{1+\frac{\delta}{3}},
	\end{equation*}
	In view of the desired condition $\|R_{q+1}\|_0\le\delta_{q+2}\la_{q+2}^{1+\delta}$, it is required that
	 $$\|R_{q+1}^{\text{transport}}\|_0 + \|R_{q+1}^{\text{Newton}}\|_0 +\|\delta B^\xi_{\lambda}\|_0+\|R_{q+1}^{\text{flow}}\|_0\leq \lambda_{q+2}^{1+\delta}\delta_{q+2},$$
	  which implies
	\begin{eqnarray*}
		\delta_{q+1} \lambda_{q+1}^{\frac{2\delta}{3}}\lambda_q^{1+\frac{\delta}{3}} \leq \lambda_{q+2}^{1+\delta}\delta_{q+2} \iff \beta < \frac{\delta}{3}+\left(1+\frac{\delta}{3}\right)\frac{b+1}{2b},
	\end{eqnarray*}
	
	\begin{equation*}
		\lambda_{q+1}^{1+\delta}\delta_{q+1}\frac{\lambda_{q}}{\lambda_{q+1}} \leq \lambda_{q+2}^{1+\delta}\delta_{q+2} \iff \beta<\frac{1+(1+\delta)b}{2b},
	\end{equation*}
	\begin{equation*}
		\frac{\delta_{q+1}^2\lambda_{q+1}^{2\delta+2} \lambda_q^{3+\delta} \tau_q}{\mu_{q+1}} \leq \lambda_{q+2}^{1+\delta}\delta_{q+2} \iff \beta<\frac{3(1+\delta)b-\delta}{3(2b-1)},
	\end{equation*}
and this is compatible with the $C_t^0C_x^{(1+\frac{2\delta}{3})-}$ regularity.
	
		\section{The anti-divergence operator and an algebraic lemma}\label{anti-div and geolem}
As mentioned before, we need to  use the low-frequency part of the quadratic perturbation term  
\[
T_1[w_{q+1}] \cdot w_{q+1} - (\nabla w_{q+1})^{T} T_1[w_{q+1}]
\]  
to cancel the old stress \(R_q\), which is further reduced to  the problem of decomposing \(R_q\) into tensors of the form \(\xi \otimes \nabla \bar{m}(\xi)\). To achieve this, it is necessary to construct a well-defined anti-divergence operator \(\mathcal{R}(m)\) depending on \(m\) and  an algebraic lemma. This is the main focus of this section.

\subsection{Classification of the Multiplier \texorpdfstring{$m$}{m} and the Corresponding Anti-Divergence Operator}\label{sec-class}

Before introducing our algebraic lemma, we first classify the multiplier \( m \), as both the algebraic lemma and the second-order divergence inverse operator  depend essentially on the structure of \( m \).

First, we claim that there exist two distinct directions \( \xi^{(1)} \) and \( \xi^{(2)} \) such that the matrices
\[
A = \xi^{(1)} \mathring{\otimes} \nabla \bar{m}\left(\xi^{(1)}\right), \quad B = \xi^{(2)} \mathring{\otimes} \nabla \bar{m}\left(\xi^{(2)}\right)
\]
are linearly independent. If otherwise, then for any \( \xi, \eta \in \mathbb{R}^2 \), the matrices \( \xi \mathring{\otimes} \nabla \bar{m}(\xi) \) and \( \eta \mathring{\otimes} \nabla \bar{m}(\eta) \) would be linearly dependent. This would imply that
\begin{equation}\notag
	C_1 \equiv \frac{\xi_1 \partial_1 \bar{m}(\xi)-\xi_2 \partial_2 \bar{m}(\xi)}{\xi_2 \partial_1 \bar{m}(\xi)} = \frac{\eta_1 \partial_1 \bar{m}(\eta)-\eta_2 \partial_2 \bar{m}(\eta)}{\eta_2 \partial_1 \bar{m}(\eta)},\quad C_2\equiv\frac{\xi_1 \partial_2 \bar{m}(\xi)}{\xi_2 \partial_1 \bar{m}(\xi)} = \frac{\eta_1 \partial_2 \bar{m}(\eta)}{\eta_2 \partial_1 \bar{m}(\eta)}
\end{equation}
for some constant \( C_1, C_2 \). Consequently, we would get \( \xi_1^2 = C_1\xi_1 \xi_2+C_2\xi_2^2 \), which is a contradiction.

Secondly, suppose one  can select three directions \( \xi^{(1)}, \xi^{(2)}, \xi^{(3)} \) such that the matrices
\[
A = \xi^{(1)} \mathring{\otimes} \nabla \bar{m}\left(\xi^{(1)}\right), \quad B = \xi^{(2)} \mathring{\otimes} \nabla \bar{m}\left(\xi^{(2)}\right), \quad C = \xi^{(3)} \mathring{\otimes} \nabla \bar{m}\left(\xi^{(3)}\right)
\]
are linearly independent. 	We fix these directions throughout the paper. In this case, we define A second-order symmetric anti-divergence operator as
\[
(\div^{-1}u)^{i,j}=\Delta^{-1}\left(\pa_i u^{j}+\pa_j u_i-\delta_{i,j}\div u\right)
\]

At last, if for all \( \xi, \eta, \zeta \in \mathbb{R}^2 \), the matrices \( \xi \mathring{\otimes} \nabla \bar{m}(\xi), \eta \mathring{\otimes} \nabla \bar{m}(\eta), \zeta \mathring{\otimes} \nabla \bar{m}(\zeta) \) are linearly dependent, then
\[
\begin{vmatrix}
	a(\xi) & a(\eta) & a(\zeta) \\
	b(\xi) & b(\eta) & b(\zeta) \\
	c(\xi) & c(\eta) & c(\zeta)
\end{vmatrix} \equiv 0, \quad \forall \xi, \eta, \zeta,
\]
where
\begin{equation}\label{1}
	a(\xi) = 2\xi_2 \partial_1 \bar{m}(\xi), \quad b(\xi) = \xi_2 \partial_2 \bar{m}(\xi) - \xi_1 \partial_1 \bar{m}(\xi), \quad c(\xi) = -2 \xi_1 \partial_2 \bar{m}(\xi).
\end{equation}

This implies that there exists a non-zero constant vector \( (c_1, c_3, c_2) \) orthogonal to the vector field \( (a(\xi), b(\xi), c(\xi)) \) such that
\begin{equation}\label{2}
	c_1 \cdot a(\xi) + c_3 \cdot b(\xi) + c_2 \cdot c(\xi) = 0, \quad \forall \xi \in \mathbb{R}^2.
\end{equation}
This further implies the existence of a 2-dimensional subspace \( \mathcal{A}(m) \subset \mathring{\mathbb{M}}^{2 \times 2} \) such that \( \xi \mathring{\otimes} \nabla \bar{m}(\xi) \in \mathcal{A}(m) \) for all \( \xi \). Here, \( \mathring{\mathbb{M}}^{2 \times 2} \) is the space of trace-free matrices, and \( \mathcal{A}(m) \) is the orthogonal complement of the constant vector \( (c_1, -c_3, -c_2) \), depending on the multiplier \( m \). We define \( \mathcal{A}(m) = \mathring{\mathbb{M}}^{2 \times 2} \) in the second case discussed previously.

Now, using polar coordinates and noting that \( \bar{m} \) is homogeneous of order \( \delta - 1 \), one can let
\[
\xi_1 = r \cos(\theta), \quad \xi_2 = r \sin(\theta),
\]
which gives \( \bar{m}(\xi) = \bar{m}(r, \theta) = r^{\delta - 1} f(\theta) = r^{a_0} f(\theta), \) where \( a_0 = \delta - 1 \) and \( f(\theta) = \bar{m}(\cos(\theta), \sin(\theta)) \) is a smooth periodic function of \( \theta \). Applying standard chain rules for polar coordinates yield
\begin{equation}\label{3}
	\partial_1 \bar{m}(\xi)=r^{a_0 -1}\left(a_0 f(\theta)\cos(\theta)-f'(\theta)\sin(\theta)\right)
	\partial_2 \bar{m}(\xi)=r^{a_0 -1}\left(a_0 f(\theta)\sin(\theta)+f'(\theta)\cos(\theta)\right)
\end{equation}
Then combine this with  \eqref{1}  to get 
\begin{equation}\label{4}
	\begin{aligned}
		a(\xi)&=2r^{a_0}\left(a_0 f(\theta)\sin(\theta)\cos(\theta)-f'(\theta)\sin^2(\theta)\right),\\
		b(\xi)&=r^{a_0}\left(a_0 f(\theta)\left(\sin^2(\theta)-\cos^2(\theta)\right)+2f'(\theta)\sin(\theta)\cos(\theta)\right),\\
		c(\xi)&=-2r^{a_0}\left(a_0 f(\theta)\sin(\theta)\cos(\theta)+f'(\theta)\cos^2(\theta)\right).
	\end{aligned}
\end{equation}
Substituting  \eqref{4} into \eqref{2} yields
\begin{equation}\label{5}
	d(\theta)f'(\theta)-\frac{a_0}{2}d'(\theta)f(\theta)=0
\end{equation}
where
\begin{equation}\label{6}
	d(\theta)=(c_1-c_2)\cos(2\theta)+c_3\sin(2\theta)-c_1-c_2.
\end{equation}
Solving the ode \eqref{5} gives
\begin{equation}\label{7}
	\begin{aligned}
		f(\theta)&=k|(c_1-c_2)\cos(2\theta)+c_3\sin(2\theta)-c_1-c_2|^{\frac{a_0}{2}},\\
		&=k|\sqrt{(c_1-c_2)^2+c_3^2}\cos(2\theta+\phi)-c_1-c_2|^{\frac{a_0}{2}},\\
	\end{aligned}
\end{equation}

There are three cases to be considered:

(i)$\boldsymbol{c_3^2<4c_1c_2}$, In this case, the quadratic form $c_2\xi_1^2-c_3\xi_1\xi_2+c_1\xi_2^2$ is positive definite or negative definite. Set
\begin{equation}\notag
	p_1(\xi_1,\xi_2)=\frac{-ic_2\xi_1}{c_2\xi_1^2-c_3\xi_1\xi_2+c_1\xi_2^2}\quad
	p_2(\xi_1,\xi_2)=\frac{ic_1\xi_2}{c_2\xi_1^2-c_3\xi_1\xi_2+c_1\xi_2^2}
\end{equation}·
\begin{equation}\notag
	p_3(\xi_1,\xi_2)=\frac{-i(c_1\xi_2-c_3\xi_1)}{c_2\xi_1^2-c_3\xi_1\xi_2+c_1\xi_2^2}\quad
	p_4(\xi_1,\xi_2)=\frac{-ic_1\xi_1}{c_2\xi_1^2-c_3\xi_1\xi_2+c_1\xi_2^2}
\end{equation}
\begin{equation}\notag
	p_5(\xi_1,\xi_2)=\frac{-ic_2\xi_2}{c_2\xi_1^2-c_3\xi_1\xi_2+c_1\xi_2^2}\quad
	p_6(\xi_1,\xi_2)=\frac{-i(c_2\xi_1-c_3\xi_2)}{c_2\xi_1^2-c_3\xi_1\xi_2+c_1\xi_2^2}
\end{equation}
Now we define a trace-free anti-divergence operator  as
\begin{equation}\notag
	\mathcal{R}_2 u=\begin{pmatrix}
		P_1u_1+P_2u_2&P_3u_1+P_4u_2\\
		P_5u_1+P_6u_2&-P_1u_1-P_2u_2
	\end{pmatrix}
\end{equation}
where for $i\in\left\{1,2,3,4,5,6\right\}, $$P_i$ are order $-1$ differential operators  with 
$$\widehat{P_i(f)}(\xi)=p_i(\xi_1,\xi_2)\hat{f}(\xi)$$
for a smooth function $f$ with zero mean. One can easily check that $\div \mathcal{R}_2 u = u$ for any smooth vector field u with zero mean and thus $\mathcal{R}_2$ is indeed an anti-divergence operator. One can also check that $\mathcal{R}_2 f  \subset \mathcal{A}(m)$.

(ii)$\boldsymbol{c_3^2=4c_1c_2}$, in this case we may assume that $c_1, c_2>0$ and thus $c_3=\pm2\sqrt{c_1c_2}$, by \ref{7}. Noting that $\cos(\theta)=\frac{\xi_1}{|\xi|}, \sin(\theta)=\frac{\xi_2}{|\xi|}$,  one can get 
\begin{equation}\label{8}
	\begin{aligned}
		f(\theta)&=k\left|(c_1-c_2)\cos(2\theta)+c_3\sin(2\theta)-c_1-c_2\right|^{\frac{a_0}{2}},\\
		&=k\left|(c_1-c_2)(\cos^2(\theta)-\sin^2(\theta))+2c_3\sin(\theta)\cos(\theta)-c_1-c_2\right|^{\frac{a_0}{2}},\\
		&=k\left|(c_1-c_2)(\cos^2(\theta)-\sin^2(\theta))\pm4\sqrt{c_1c_2}\sin(\theta)\cos(\theta)-c_1-c_2\right|^{\frac{a_0}{2}},\\
		&=2^{\frac{a_0}{2}}k\left|\sqrt{c_2}\cos(\theta)\pm\sqrt{c_1}\sin(\theta)\right|^{a_0},\\
		&=2^{\frac{a_0}{2}}k|\xi|^{-a_0}\left|\sqrt{c_2}\xi_1\pm\sqrt{c_1}\xi_2\right|^{a_0}.\\
	\end{aligned}
\end{equation}
Thus to ensure that $f(\theta)$ is a smooth function of $\theta$, we can choose that $a_0=2k_1$ for some non-negative integer $k_1$. And in this case, it holds that  
\begin{equation}\notag
	\begin{aligned}
		f(\theta)
		&=2^{k_1}k\left(\sqrt{c_2}\cos(\theta)\pm\sqrt{c_1}\sin(\theta)\right)^{2k_1},\\
		&=2^{k_1}k|\xi|^{-2k_1}\left(\sqrt{c_2}\xi_1\pm\sqrt{c_1}\xi_2\right)^{2k_1}.\\
	\end{aligned}
\end{equation}
Hence,
\begin{equation}\notag
	\bar{m}(\xi)=|\xi|^{2k_1}f(\theta)=2^{k_1}k\left(\sqrt{c_2}\xi_1\pm\sqrt{c_1}\xi_2\right)^{2k_1}.
\end{equation}

(iii)$\boldsymbol{c_3^2>4c_1c_2}$, then similar to case (ii), one can get $a_0=4k_2$ for some non-negative integer $k_2$ and thus
$f(\theta)=\frac{2^{{2k_2}}k}{|\xi|^{4k_2}}\left(c_1\xi_2^2-c_3\xi_1\xi_2+c_2\xi_1^2\right)^{2k_2}$ so that
$$\bar{m}(\xi)=|\xi|^{4k_2}f(\theta)=2^{2k_2}k\left(c_1\xi_2^2-c_3\xi_1\xi_2+c_2\xi_1^2\right)^{2k_2}$$

\textbf{In conclusion}, except for the case that $c_3^2\ge4c_1c_2$, $$m(\xi)=2^{k_1}k\left(\sqrt{c_2}\xi_1\pm\sqrt{c_1}\xi_2\right)^{2k_1}i\xi^{\perp}$$ or $$m(\xi)=2^{2k_2}k\left(c_1\xi_2^2-c_3\xi_1\xi_2+c_2\xi_1^2\right)^{2k_2}i\xi^{\perp}$$, we can find space $\mathcal{A}(m)$ and  anti-divergence $\mathcal{R}(m)$ such that for all $\xi \in \mathbb{R}^2$ and smooth vector field $u$ with zero mean, it holds that $\xi\mathring{\otimes}\nabla\bar{m}(\xi)\in \mathcal{A}(m)$ and  $\mathcal{R}(m)u\subset \mathcal{A}(m)$ and $\div\mathcal{R}(m)u=u$.
\begin{rem}
	Since the exceptional multipliers mentioned above are all of order $\geq 1$, and this paper concerns only  cases where the order of $m$ is $-1 \leq \delta \leq 0$, these exceptional cases do not impose restrictions. However, it should be noted that when $k_1 = 0$ or $k_2 = 0$, then $m(\xi) = k i \xi^\perp$ and thus $u=T[\theta]=k\nabla^{\perp}\theta$, the nonlinear term $T[\theta] \cdot \nabla \theta$ in the ASE vanishes, so the equation is reduced to a stationary form. Consequently, the only solutions are stationary and conserve Hamiltonian, which explains why our convex integration scheme fails to handle these cases, since our method generates temporally compact-supported solutions, which  does not apply to  such scenarios. 
	
\end{rem}

\subsection{An algebraic lemma}\label{sec-alge}

As in the case of the Euler equation, we need an algebraic lemma to cancel the old stress. First, following the arguments in  section \ref{sec-class}, we can fix and choose two or three linearly independent matrices:
\[
A = \xi^{(1)} \mathring{\otimes} \nabla \bar{m} \left(\xi^{(1)}\right), \quad B = \xi^{(2)} \mathring{\otimes} \nabla \bar{m} \left(\xi^{(2)}\right),
\]
\[
{C} = {\xi^{(3)}} \mathring{\otimes} \nabla \bar{m} \left({\xi^{(3)}}\right) \quad \text{(if possible)}.
\]
 \( \xi^{(1)}, \xi^{(2)}, {\xi^{(3)}} \in \mathbb{Z}^2 = \widehat{\mathbb{T}}^2 \) are nonzero frequencies. We will adopt the convention that $\widehat{\xi^{(3)}}={\xi^{(3)}}$ and $\widehat{C} = \widehat{\xi^{(3)}} \mathring{\otimes} \nabla \bar{m} \left(\widehat{\xi^{(3)}}\right)=C$ if there exists three linearly independent matrices $A,B,C$ above. Otherwise, $\widehat{\xi^{(3)}}$ and $\widehat{C}$ will be omitted if there are only two linearly independent matrices $A,B$. Set \( F = \left\{ \pm \xi^{(1)}, \pm \xi^{(2)}, \pm \widehat{\xi^{(3)}} \right\} \). The existence of these vectors is guaranteed by the condition that \( m(\xi) \) is odd, and the orthogonality condition \( \xi \cdot m(\xi) = 0 \), as discussed earlier.

Similar to \cite{IM}, for parameters \( (p_1, p_2, \hat{p}_3) \), one can define a linear map
\[
\mathcal{L}^{(m)}_{(p_1, p_2, \hat{p}_3)} : \mathbb{R} \times \mathbb{R} \times \hat{\mathbb{R}} \rightarrow \mathcal{A}(m) \subset \mathring{\mathbb{M}}^{2 \times 2},
\]
where $\mathring{\mathbb{M}}^{2 \times 2}$ is the space of trace-free \( 2 \times 2 \) matrices, and \( \mathcal{A}(m) \) is a subspace of $\mathring{\mathbb{M}}^{2 \times 2}$ as discussed in the previous section, by:
\[
\mathcal{L}^{(m)}_{(p_1, p_2, \widehat{p_3})}(x_1, x_2, \hat{x}_3) = \frac{1}{2} \left( x_1 p_1 \mathring{\otimes} \nabla \bar{m}(p_1) + x_2 p_2 \mathring{\otimes} \nabla \bar{m}(p_2) + \hat{x}_3 \hat{p}_3 \mathring{\otimes} \nabla \bar{m}(\hat{p}_3) \right).
\]

Let \( D = A + B + \hat{C} \), \( \xi_0 = (\xi^{(1)}, \xi^{(2)}, \widehat{\xi^{(3)}}) \), and \( p = (p_1, p_2, \hat{p_3}) \). Again, the hat notations here mean \( \hat{a}=a \)   when three vectors \( \xi^{(1)}, \xi^{(2)}, {\xi^{(3)}} \) in Section \ref{sec-class} exist such that the matrices \( A \), \( B \), and \( C \) are linearly independent. Otherwise, if only two vectors \( \xi^{(1)}, \xi^{(2)} \) can be found, the symbols with a hat are omitted. We now present the following algebraic lemma:

\begin{lem} \label{le-geo}
	There exists non-negative  constants $c_1, c_2$ such that if $|p-\xi_0|\leq c_1$ and $|R-D|\leq c_2$, then the linear map $\mathcal{L}^{(m)}_{\left(p\right)}$ is invertible and its inverse map $\mathcal{L}^{(-m)}_{\left(p\right)}:\mathcal{A}(m)\rightarrow \R\times\R\times \hat{\R}$ depends smoothly on parameters $p$ and satisfying $|\mathcal{L}^{(-m)}_{\left(p\right)}\left(R\right)|\in(1,4)$.
	\begin{proof}
		Notice that $\mathcal{L}^{(m)}_{\left(\xi_0\right)}{\left(x_1,x_2,\hat{x}_3\right)}=\frac{1}{2}\left(x_1A+x_2B+\hat{x}_3\widehat{C}\right)$. Thus $\mathcal{L}^{(m)}_{\left(\xi_0\right)}{\left(2,2,\hat{2}
			\right)}=D$. It follows from the linear independence of $A, B, \widehat{C} $  that $\mathcal{L}_{\left(\xi_0\right)}$ maps a basis in $\R\times\R\times \hat{\R}$ to a basis in $\mathcal{A}(m)$. Hence $\mathcal{L}^{(m)}_{\left(\xi_0\right)}$ is a invertible map and $\mathcal{L}^{(-m)}_{\left(\xi_0\right)}\left(D\right)=\left(2,2,\hat{2}\right)$. Now one can choose a constant $c_1\geq0$ such that if $|p-\xi_0|\leq c_1$, then map $\mathcal{L}^{(-m)}_{\left(p\right)}$ is well defined and depends smoothly on parameter $p$. One may choose $c_1,c_2$ small enough such that $|R-D|\leq c_2\leq \frac{1}{8}\left\|\mathcal{L}_{\left(\xi_0\right)}^{(-m)}\right\|_{o p}$ and 
		$\left\|\mathcal{L}_{\left(\xi_0\right)}^{(-m)}-\mathcal{L}_{\left(p\right)}^{(-m)}\right\|_{o p}\leq\frac{1}{8}\left\|R\right\|_{L^{\infty}}^{-1}$. Then it follows that
		
		$$
		\begin{aligned}
			\left\|(2,2,\hat{2})-\mathcal{L}_{(p)}^{(-m)}(R)\right\|_{L^{\infty}} & \leqslant\left\|\mathcal{L}_{\left(\xi_0\right)}^{(-m)}(D)-\mathcal{L}_{\left(\xi_0\right)}^{(-m)}(R)\right\|_{L^{\infty}}+\left\|\mathcal{L}_{\left(\xi_0\right)}^{(-m)}(R)-\mathcal{L}_{(p)}^{(-m)}(R)\right\|_{L^{\infty}} \\
			& \leqslant\left\|\mathcal{L}_{\left(\xi_0\right)}^{(-m)}\right\|_{o p}\|R-D\|_{L^{\infty}}+\left\|\mathcal{L}_{\left(\xi_0\right)}^{(-m)}-\mathcal{L}_{(p)}^{(-m)}\right\|_{o p}\|R\|_{L^{\infty}}\\
			& \leqslant 1 / 2 .
		\end{aligned}
		$$
		Hence $|\mathcal{L}^{-1}_{\left(p\right)}\left(R\right)|\in(1,4)$.
	\end{proof}
\end{lem}

		\section{Multi-linear Fourier multiplier analysis}
	This section contains the main technical part of this paper where we will  derive several bi-linear and tri-linear Fourier multiplier estimates. We first cite some technical lemmas from \cite{IP15} and \cite{WXC18} :
	
	Define a  class of symbols $\mathcal{S}^{\infty}$ as follows
	
	$$
	\mathcal{S}^{\infty} \stackrel{\text { def }}{=}\left\{m:   \text{$\mathbb{R}^4$ or $\mathbb{R}^6$}  \rightarrow \mathbb{C}: m \text { continuous and }\|m\|_{\mathcal{S}^{\infty}}:=\left\|\mathcal{F}^{-1}(m)\right\|_{L^1}<\infty\right\}
	$$
	Then the following Lemmas hold
	\begin{lem}[Lemma 5,2(i) \cite{IP15}]\label{symbol}
		If $m, m^{\prime} \in \mathcal{S}^{\infty}$ then $m \cdot m^{\prime} \in \mathcal{S}^{\infty}$ and
		
		$$
		\left\|m \cdot m^{\prime}\right\|_{\mathcal{S}^{\infty}} \lesssim\|m\|_{\mathcal{S}^{\infty}}\left\|m^{\prime}\right\|_{\mathcal{S}^{\infty}} .
		$$

		Moreover, if $m \in \mathcal{S}^{\infty}, A: \mathbb{R}^k \rightarrow \mathbb{R}^k, (k=4, 6)$ is a linear transformation, $v \in \mathbb{R}^k$, and $m_{A, v}(\xi, \eta):=m(A(\xi, \eta)+v)$, then
		
		$$
		\left\|m_{A, v}\right\|_{\mathcal{S}^{\infty}}=\|m\|_{\mathcal{S}^{\infty}}
		$$
		
	\end{lem}
	
	\begin{lem}[Lemma 2.4. \cite{WXC18}]\label{loc}
		For $i \ge2$, if $f: \mathbb{R}^i \rightarrow \mathbb{C}$ is a smooth function and $k_1, \cdots, k_i \in \mathbb{Z}$, then 
		
		$$
		\left\|\int_{\mathbb{R}^i} f\left(\xi_1, \cdots, \xi_i\right) \prod_{j=1}^i e^{i x_j \xi_j} \bar{\chi}_{k_j}\left(\xi_j\right) d \xi_1 \cdots d \xi_i\right\|_{L_{x_1, \cdots, x_i}^1} \lesssim \sum_{m=0}^i \sum_{j=1}^i 2^{m k_j}\left\|\partial_{\xi_j}^m f\right\|_{L^{\infty}}
		$$
	\end{lem}

	\begin{lem}\label{lem.bilin}
		Let $0<\alpha<1$. Then, the bilinear Fourier multiplier operator 
		\begin{equation}
			T[f,g] = \Lambda^{-1}\left(T_1[ f]\nabla\Delta g -T_1[\nabla g]\Delta f\right),
		\end{equation}
		applied to smooth zero-mean functions $f, g: \mathbb
		T^2 \rightarrow \mathbb R$,
		satisfies
		\begin{equation}
			\|T[f,g]\|_{N+\alpha} \lesssim \|f\|_{N+1+\delta+\alpha} \|g\|_{2+\alpha} +\|f\|_{1+\delta+\alpha} \|g\|_{N+2+\alpha}+\|f\|_{N+2+\alpha} \|g\|_{1+\delta+\alpha} +
			\|f\|_{2+\alpha} \|g\|_{N+1+\delta+\alpha} 
		\end{equation}
		with implicit constant depending only on $\alpha$ and $N$.
	\end{lem}
	
	\begin{proof} The proof is similar to lemma 2.1 of \cite{DGR24}.
		It suffices to show the estimate in the case $N=0$. Indeed, 
		% by similar arguments to the ones used in the proof of Lemma \ref{le-commu}
		with $|\gamma|\leq N$ a multi-index, one gets
		\begin{eqnarray*}
			\|\partial^\gamma T[f,g]\|_\alpha &\lesssim& \sum_{\beta \leq \gamma} \|T[\partial^\beta f, \partial^{\gamma-\beta}g]\|_\alpha \\ 
			&\lesssim& \sum_{\beta \leq \gamma} \|f\|_{|\beta| +1+\delta+ \alpha} \|g\|_{|\gamma| - |\beta|+2+\alpha}+ \|f\|_{|\beta| +2+ \alpha} \|g\|_{|\gamma| - |\beta|+1+\delta+\alpha} \\
			&\lesssim&\|f\|_{|\gamma|+1+\delta+\alpha} \|g\|_{2+\alpha} +\|f\|_{1+\delta+\alpha} \|g\|_{|\gamma|+2+\alpha}+\|f\|_{|\gamma|+2+\alpha} \|g\|_{1+\delta+\alpha} +
			\|f\|_{2+\alpha} \|g\|_{|\gamma|+1+\delta+\alpha} 
		\end{eqnarray*}
		For the case $N=0$, let $j_0 \in \mathbb N$ be fixed and decompose the operator as 
		\begin{equation*}
			T[f,g] = \underbrace{\sum_{j \in \mathbb Z} T[\Delta_j f, S_{j-j_0} g]}_{T_{HL}[f, g]} + \underbrace{\sum_{j \in \mathbb Z} T[\Delta_j f, S_{j+j_0-1}g - S_{j-j_0} g]}_{T_{HH}[f, g]} + \underbrace{\sum_{j \in \mathbb Z} T[S_{j - j_0} f , \Delta_j g] }_{T_{LH}[f, g]}
		\end{equation*}
		
		Since $T_{HL}$ and $T_{LH}$ can be estimated similarly, it suffices to treat $T_{HL}$. Note that 
		\begin{equation*}
			\supp \widehat{T[\Delta_j f, S_{j-j_0} g]} \subset B_{2^{j+1}+ 2^{j-j_0+1}} \setminus B_{2^{j-1} - 2^{j-j_0 + 1}},
		\end{equation*}
		which implies that for $j_0 \geq 4$ and $l \in \mathbb Z$, 
		\begin{equation*}
			\Delta_l T_{HL}[f,g] = \sum_{|j-l| \leq 2} \Delta_l T[\Delta_j f, S_{j-j_0}g],
		\end{equation*}
		Since $T_1$ is a homogeneous operator of order $1+\delta$, thus
		\begin{eqnarray*}
			\|\Delta_l T_{HL}[f,g]\|_0 &\lesssim& \sum_{|j-l|\leq 2} 2^{-l} \big(\|T_1\Delta_j f\|_0 \|\nabla \Delta S_{j-j_0}g\|_0 + \|\Delta\Delta_j f\|_0 \|T_1\nabla  S_{j-j_0}g\|_0 \big) \\ 
			&\lesssim& 2^{-l} \sum_{|j-l|\leq 2}\big(2^{-j\alpha} \|f\|_{1+\delta+\alpha} \sum_{m=0}^{j-j_0} 2^m\|g\|_2 + 2^{-j\alpha}\|f\|_{2+\alpha} \sum_{m=0}^{j-j_0} 2^m \|g\|_{1+\delta} \big) \\ 
			&\lesssim& 2^{-l\alpha}\left(\|f\|_{1+\delta+\alpha} \|g\|_{2}+\|f\|_{2+\alpha} \|g\|_{1+\delta}\right),
		\end{eqnarray*}
		and the H\"older estimate follows.
		
		Next, we  estimate  $T_{HH}$. It follows from  the definition that for $k \in \mathbb Z^2 \setminus \{0\}$,
		\begin{equation*}
			\begin{aligned}
				-\widehat{T[f,g]}(k) =& |k|^{-1}\sum_{j \in \mathbb Z^2} (\tilde{m}(k-j)|j|^2-\tilde{m}(j)|k-j|^2) \widehat{ f}(k-j) \widehat{\nabla g}(j),\\
				=& |k|^{-1}\sum_{j \in \mathbb Z^2} (\bar{m}(k-j)-\bar{m}(j)) \widehat{(-\Delta) f}(k-j) \widehat{\nabla(-\Delta) g}(j)\\
				=& |k|^{-1}\sum_{j \in \mathbb Z^2} (\bar{m}(k-j)-\bar{m}(j))|j|^{2-\delta} \widehat{(-\Delta) f}(k-j) \widehat{\nabla\Lambda^{\delta} g}(j).
			\end{aligned}
		\end{equation*}
		where $\bar{m}(\xi)=\frac{\tilde{m}(\xi)}{|\xi|^2}$, and one has implicitly used $\tilde{m}(\xi)=\tilde{m}(\xi)\chi_{\ge1/2}(\xi)$ for some smooth cutoff function $\chi$ since $T_1$ is applied to functions with zero-mean. And thus,
		\begin{eqnarray*}
			-T[f,g] &=& \sum_{j + k \neq 0} |k+j|^{-1}(\bar{m}(k)-\bar{m}(j))|j|^{2-\delta} \widehat{(-\Delta) f}(k) \widehat{\nabla\Lambda^{\delta} g}(j) e^{i(j+k)\cdot x} \\ 
			&=& \sum_{j+ k \neq 0} \frac{j+k}{|j+k|}\cdot \left(M_1(j,k)+M_2(j,k)\right)  \widehat{(-\Delta) f}(k) \widehat{\nabla\Lambda^{\delta} g}(j) e^{i(j+k)\cdot x} \\ 
			&=& \mathcal{R} \cdot \sum_{(j,k) \in \mathbb Z^4 \setminus \{0\}} -i \left(M_1(j,k)+M_2(j,k)\right)  \widehat{(-\Delta) f}(k) \widehat{\nabla\Lambda^{\delta} g}(j) e^{i(j+k)\cdot x} \\ 
			&=:& \mathcal{R} \cdot P_1[f,g]+\mathcal{R} \cdot P_2[f,g].
		\end{eqnarray*}
		where 
		$$M_1(j,k)=\frac{j+k}{|j+k|^2}\cdot|j|^{2-\delta}\left(\bar{m}(k)-\bar{m}(j)\right)\left(\sum_{m,n\in \mathbb{Z},|j+k|\ge\frac{1}{4}|k|}\psi_n(j)\psi_m(k)\right)$$
		$$M_2(j,k)=|j|^{2-\delta}\int_{0}^{1}\nabla \bar{m}(k-\sigma(j+k))d\sigma  \left(\sum_{m,n\in \mathbb{Z},|j+k|\le\frac{1}{2}|k|}\psi_n(j)\psi_m(k)\right)$$
		and  $1=\sum_{n\in \mathbb{Z}}\psi_n(j)=\sum_{m\in \mathbb{Z}}\psi_m(k)$ is the standard homogeneous Littlewood-Paley decomposition of the frequency space.

		Since
		\begin{equation*}
			\supp \widehat{T[\Delta_j f, S_{j+j_0-1}g - S_{j-j_0}g]} \subset B_{2^{j+1} + 2^{j+j_0}},
		\end{equation*}
		one has that, for sufficiently large $j_0$,
		\begin{equation*}
			-\Delta_l T_{HH}[f,g] = \sum_{t=1}^{2}\sum_{j \geq l - j_0 -1} \mathcal{R}\cdot \Delta_l P_t[\Delta_j f, S_{j+j_0-1}g - S_{j-j_0}g],
		\end{equation*}
		and, therefore, 
		\begin{equation*}
			\|\Delta_l T_{HH}[f,g]\|_0 \lesssim \sum_{t=1}^{2}\sum_{j \geq l-j_0-1} \|P_t[\Delta_j f, S_{j+j_0-1}g - S_{j-j_0}g]\|_0.
		\end{equation*}
		Let $\bar \chi_j$ be as given in the beginning of section \ref{sec.bha}. Then, 
		\begin{equation*}
			P_t[\Delta_j f, S_{j+j_0-1}g - S_{j-j_0}g](x) = \sum_{v = j-j_0+1}^{j+j_0-1} \sum_{(l,k) \in \Z^4\setminus\{0\}} \bar \chi_j(l) \bar \chi_v(k) \tilde{M}_t(l,k)\widehat{(-\Delta)\Delta_j f}(l) \widehat{\nabla
				\Lambda^{\delta}\Delta_v g}(k) e^{i(l+k)\cdot x},
		\end{equation*}
		where for $t=1,2$
		\begin{equation*}
			\tilde{M}_t(\eta, \xi) = -i M_t(\eta, \xi)
		\end{equation*}
		is the multiplier in the definition of $P_t$. It follows that 
		\begin{equation*}
			P_t[\Delta_j f, S_{j+j_0-1}g - S_{j-j_0}g](x) = \sum_{v=j-j_0+1}^{j+j_0-1} \int_{\mathbb R^2 \times \mathbb R^2} K^t_{j,v}(x-y_1, x-y_2) {(-\Delta)\Delta_{j}f}(y_1) {\nabla\Lambda^{\delta}\Delta_{v}g}(y_2) dy_1 dy_2,
		\end{equation*}
		where $(-\Delta)\Delta_{j}f$ and $\nabla\Lambda^{\delta}\Delta_{v}g$ are identified with their periodic extensions and by the definition, $\tilde{M}_t(\eta, \xi)$, $\tilde{M}_t(\eta, \xi)$ are clearly smooth fuctions, and $K_{j,v}^{t}$ is given by
		\begin{eqnarray*}
			K^t_{j,v}(x,y) &=& \frac{1}{(2\pi)^4} \int_{\mathbb R^2 \times \mathbb R^2}\tilde{M}_t(\eta, \xi) \bar \chi_j(\eta) \bar \chi_v(\xi) e^{i(\eta \cdot x + \xi \cdot y)} d\eta d\xi \\ 
			&=&2^{4j} \frac{1}{(2\pi)^4} \int_{\mathbb R^2 \times \mathbb R^2} \tilde{M}_t(2^j\eta, 2^j\xi) \bar \chi_0(\eta) \bar \chi_0 (2^{j-v} \xi) e^{i2^j(\eta \cdot x + \xi \cdot y)} d\eta d\xi \\ 
			&=:& 2^{4j} K^t_{0, j-v} (2^jx, 2^jy). 
		\end{eqnarray*}
		It follows from  the construction of $M_t(\eta,\xi)$ that
		\begin{align*}
			&\tilde{M}_1(2^j\eta, 2^j\xi)\bar\chi_0(\eta)\bar\chi_0(2^{j-v}\xi)=\\
			&-i\frac{\eta+\xi}{|\eta+\xi|^2}|\xi|^{2-\delta}\left(\bar{m}(\eta)-\bar{m}(\xi)\right)\left(\sum_{m,n\in \mathbb{Z},|\eta+\xi|\ge\frac{1}{4}|\eta|}\psi_0(2^{j-n}\eta)\psi_0(2^{j-m}\xi)\right)\bar\chi_0(\eta)\bar\chi_0(2^{j-v}\xi),\\
		\end{align*}
		\begin{align*}
			&\tilde{M}_2(2^j\eta, 2^j\xi)\bar\chi_0(\eta)\bar\chi_0(2^{j-v}\xi)=\\
			&-i|\xi|^{2-\delta}\int_{0}^{1}\nabla \bar{m}(\eta-\sigma(\eta+\xi))d\sigma \left(\sum_{m,n\in \mathbb{Z},|\eta+\xi|\le\frac{1}{2}|\eta|}\psi_0(2^{j-n}\eta)\psi_0(2^{j-m}\xi)\right)\bar\chi_0(\eta)\bar\chi_0(2^{j-v}\xi),\\
		\end{align*}

		here the summations above are actually finite sums taken on $|m-j|\le 10j_0, |n-j|\le 10j_0$. And it is clear that from the construction above, both $\tilde{M}_1(2^j\eta, 2^j\xi)\bar\chi_0(\eta)\bar\chi_0(2^{j-v}\xi)$ and $\tilde{M}_2(2^j\eta, 2^j\xi)\bar\chi_0(\eta)\bar\chi_0(2^{j-v}\xi)$ are smooth functions and  compactly supported.   Thus, the kernels $K^t_{0, j-v}$ are in $L^1(\mathbb R^2 \times \mathbb R^2)$, and it follows that 
		\begin{equation*}
			\|P_t[\Delta_j, S_{j+j_0-1}g - S_{j-j_0}g]\|_0 \lesssim \sum_{v=j-j_0+1}^{j+j_0-1} \|K^t_{0, j-v}\|_{L^1(\mathbb R^2 \times \mathbb R^2)} \|(-\Delta)\Delta_j f\|_0\|\nabla\Lambda^{\delta}\Delta_v g\|_0 \lesssim 2^{-2j\alpha} \|f\|_{2+\alpha}\|g\|_{1+\delta+\alpha}.
		\end{equation*}
		Consequently, 
		\begin{equation*}
			\|\Delta_l T_{HH}[f,g]\|_0 \lesssim \sum_{j \geq l - j_0 + 1} 2^{-2j\alpha} \|f\|_{1+\alpha} \|g\|_{a+\alpha} \lesssim 2^{-l\alpha}\|f\|_{2+\alpha} \|g\|_{1+\delta+\alpha},
		\end{equation*}
		and the conclusion follows.
	\end{proof}
	
	Similarly, we consider the bilinear Fourier multiplier operators
	\begin{equation}\label{bi-lin}
		\begin{aligned}
			&S[f,g]=\Lambda^{-1}\left(\left(T_1\nabla f\right)^{T}\cdot\nabla^{\perp}g-\left(\nabla f\right)^{T}\cdot T_1\nabla^{\perp}g\right)\\
			&S'[w,v]=\Lambda^{-1}\left(T_1[w]^{\perp}\left(\nabla^{\perp}\cdot v\right)+T_1[v]^{\perp}\left(\nabla^{\perp}\cdot w\right)\right)
		\end{aligned}
	\end{equation}
	defined for smooth divergence-free zero-mean vector fields  $f, w, v: \mathbb
	T^2 \rightarrow \mathbb R^2$ and smooth zero-mean function  $g: \mathbb
	T^2 \rightarrow \mathbb R$. Direct computations yield
	\begin{equation}\label{S-com}
		\begin{aligned}
			&\left(T_1\nabla f\right)^{T}\cdot\nabla^{\perp}g-\left(\nabla f\right)^{T}\cdot T_1\nabla^{\perp}g=\left(\begin{array}{ll}
				\partial_1 T_1 f_1 & \partial_1 T_1 f_2 \\
				\partial_2 T_1 f_1 & \partial_2 T_1 f_2
			\end{array}\right)\binom{-\partial_2 g}{\partial_1 g}-\left(\begin{array}{ll}
				\partial_1 f_1 & \partial_1 f_2 \\
				\partial_2 f_1 & \partial_2 f_2
			\end{array}\right)\binom{-\partial_2 T_1 g}{\partial_1 T_1 g}\\
			&=\binom{-\partial_2g\partial_1T_1f_1+\pa_1g\pa_1T_1f_2+\pa_1f_1\pa_2T_1g-\pa_1f_2\pa_1T_1g}{-\partial_2g\partial_2T_1f_1+\pa_1g\pa_2T_1f_2+\pa_2f_1\pa_2T_1g-\pa_2f_2\pa_1T_1g}\\
			&=\binom{\pa_1\left(g\pa_1T_1f_2-T_1g\pa_1f_2\right)+\pa_2\left(T_1g\pa_1f_1-g\pa_1T_1f_1\right)+g\pa_1\pa_2T_1f_1-T_1g\pa_1\pa_2f_1+T_1g\pa_1\pa_1f_2-g\pa_1\pa_1T_1f_2}{\pa_1\left(g\pa_2T_1f_2-T_1g\pa_2f_2\right)+\pa_2\left(T_1g\pa_2f_1-g\pa_2T_1f_1\right)+g\pa_2\pa_2T_1f_1-T_1g\pa_2\pa_2f_1+T_1g\pa_1\pa_2f_2-g\pa_1\pa_2T_1f_2}\\
			&=\binom{\pa_1\left(g\pa_1T_1f_2-T_1g\pa_1f_2\right)+\pa_2\left(T_1g\pa_1f_1-g\pa_1T_1f_1\right)+\Lambda S_{1,2}[f_1,g]-\Lambda S_{1,1}[f_2,g]}{\pa_1\left(g\pa_2T_1f_2-T_1g\pa_2f_2\right)+\pa_2\left(T_1g\pa_2f_1-g\pa_2T_1f_1\right)+\Lambda S_{2,2}[f_1,g]-\Lambda S_{1,2}[f_2,g]}
		\end{aligned}
	\end{equation}
	\begin{equation}\label{S'-com}
		\begin{aligned}
			&T_1[w]^{\perp}\left(\nabla^{\perp}\cdot v\right)+T_1[v]^{\perp}\left(\nabla^{\perp}\cdot w\right)\\
			&=\binom{\pa_1\left(v^1T_1w^1-w^2T_1v^2\right)+\pa_2\left(v^1T_1w^2+w^1T_1v^2\right)+\Lambda S_{1}[v^1,w^1]+\Lambda S_{1}[v^2,w^2]}{\pa_1\left(v^2T_1w^1+w^2T_1v^1\right)+\pa_2\left(v^2T_1w^2-w^1T_1v^1\right)+\Lambda S_{2}[v^1,w^1]+\Lambda S_{2}[v^2,w^2]}
		\end{aligned}
	\end{equation}
	where for $(i,j)\subseteq\left\{(1,1),(2,2),(1,2)\right\}$,  $S_{i,j}$ and $S_{i}$ are defined in Lemma \ref{lem.bilin1}. By Lemma \ref{lem.bilin1} one can conclude that 
	\begin{equation}\label{est,bilin2}
		\begin{aligned}
			\|S[f,g]\|_{N+\alpha}&=\|\Lambda^{-1}\left(\left(T_1\nabla f\right)^{T}\cdot\nabla^{\perp}g-\left(\nabla f\right)^{T}\cdot T_1\nabla^{\perp}g\right)\|_{N+\alpha}\\
			&\lesssim \|f\|_{N+2+\delta+\alpha} \|g\|_{\alpha} +\|f\|_{2+\delta+\alpha} \|g\|_{N+\alpha}+\|f\|_{N+1+\alpha} \|g\|_{1+\delta+\alpha} +
			\|f\|_{1+\alpha} \|g\|_{N+1+\delta+\alpha} 
		\end{aligned}
	\end{equation}
	\begin{equation}\label{est,bilin2'}
		\begin{aligned}
			\|S'[w,v]\|_{N+\alpha}&=\|\Lambda^{-1}\left(T_1[w]^{\perp}\left(\nabla^{\perp}\cdot v\right)+T_1[v]^{\perp}\left(\nabla^{\perp}\cdot w\right)\right)\|_{N+\alpha}\\
			&\lesssim \|w\|_{N+1+\delta+\alpha} \|v\|_{\alpha} +\|w\|_{1+\delta+\alpha} \|v\|_{N+\alpha}+\|w\|_{N+\alpha} \|v\|_{1+\delta+\alpha} +
			\|w\|_{\alpha} \|v\|_{N+1+\delta+\alpha} 
		\end{aligned}	
	\end{equation}
	
	\begin{lem}\label{lem.bilin1}
		Let $0<\alpha<1$. For  $(i,j)\subseteq\left\{(1,1),(2,2),(1,2)\right\}$,  define the bilinear Fourier multiplier operators 
		\begin{equation}\label{def-bi}
			\begin{aligned}
				&S_{i}[f,g] = \Lambda^{-1}\left(\left(\partial_iT_1f\right)g-\left(\partial_i f\right)T_1g\right)\\
				&S_{i,j}[f,g] =S_i[\pa_j f,g]= \Lambda^{-1}\left(\left(\partial_i\partial_jT_1f\right)g-\left(\partial_i\partial_jf\right)T_1g\right),
			\end{aligned}
		\end{equation}
		for smooth zero-mean functions $f, g: \mathbb
		T^2 \rightarrow \mathbb R$, then it holds that
		\begin{equation}
			\begin{aligned}
				&\|S_{i}[f,g]\|_{N+\alpha} \lesssim \|f\|_{N+1+\delta+\alpha} \|g\|_{\alpha} +\|f\|_{1+\delta+\alpha} \|g\|_{N+\alpha}+\|f\|_{N+\alpha} \|g\|_{1+\delta+\alpha} +
				\|f\|_{\alpha} \|g\|_{N+1+\delta+\alpha} \\
				&\|S_{i,j}[f,g]\|_{N+\alpha} \lesssim \|f\|_{N+2+\delta+\alpha} \|g\|_{\alpha} +\|f\|_{2+\delta+\alpha} \|g\|_{N+\alpha}+\|f\|_{N+1+\alpha} \|g\|_{1+\delta+\alpha} +
				\|f\|_{1+\alpha} \|g\|_{N+1+\delta+\alpha} 
			\end{aligned}
		\end{equation}
		with implicit constants depending only on $\alpha$ and $N$.
	\end{lem}
	\begin{proof}
		The proof is almost the same as that for Lemma \ref{lem.bilin} and thus omitted.
	\end{proof}
	
	Now, we consider the following three trilinear Fourier multiplier operators:
	\begin{equation}\label{def-trilin}
		\begin{aligned}
			&S_0[u,\phi,\psi]=u\cdot\nabla T[\phi,\psi]-T[u\cdot\nabla\phi,\psi]-T[\phi,u\cdot\nabla\psi]\\
			&S_{i}[u,\phi,\psi]=u\cdot\nabla S_i[ \phi,\psi]-S_i[u\cdot\nabla \phi, \psi]-S_i[ \phi,u\cdot\nabla \psi]\\
			&S_{i,j}[u,\phi,\psi]=S_{i}[u,\pa_j\phi,\psi]=u\cdot\nabla S_i[\pa_j \phi,\psi]-S_i[u\cdot\nabla\pa_j \phi, \psi]-S_i[\pa_j \phi,u\cdot\nabla \psi]
		\end{aligned}
	\end{equation}
	where the bilinear Fourier multiplier operators $T[f,g], S_i[f,g]$ are defined in Lemmas \ref{lem.bilin} and \ref{lem.bilin1}. Then  the following trilinear estimates hold:
	\begin{lem}\label{lem.trilin}
		Let $0<\alpha<1$. Then the  trilinear Fourier multiplier operators defined in \eqref{def-trilin} for smooth vector filed $u$ and zero-mean functions $\phi,\psi$ satisfy the following estimates
		\begin{equation}\label{tri1}
			\begin{aligned}
				\|S_{0}[u,\phi,\psi]\|_{N+\alpha} &\lesssim\sum_{\left(N_1,N_2,N_3\right)\in\mathcal{N}} \|u\|_{N_1+1+\alpha}\left(\|\phi\|_{N_2+\alpha} \|\psi\|_{N_3+3+\delta+\alpha} +\|\phi\|_{N_2+2+\delta+\alpha} \|\psi\|_{N_3+1+\alpha}\right) \\
				&+\|u\|_{N_1+3+\delta+\alpha}\left(\|\phi\|_{N_2+\alpha} \|\psi\|_{N_3+1+\alpha}\right)
			\end{aligned}
		\end{equation}
		\begin{equation}\label{tri2}
			\begin{aligned}
				\|S_{i}[u,\phi,\psi]\|_{N+\alpha} &\lesssim\sum_{\left(N_1,N_2,N_3\right)\in\mathcal{N}} \|u\|_{N_1+1+\alpha}\left(\|\phi\|_{N_2+1+\delta+\alpha} \|\psi\|_{N_3+\alpha} +\|\phi\|_{N_2+\alpha} \|\psi\|_{N_3+1+\delta+\alpha}\right) \\
				&+\|u\|_{N_1+2+\delta+\alpha}\left(\|\phi\|_{N_2+\alpha} \|\psi\|_{N_3+\alpha}\right)
			\end{aligned}
		\end{equation}
		\begin{equation}\label{tri3}
			\begin{aligned}
				\|S_{i,j}[u,\phi,\psi]\|_{N+\alpha} &\lesssim\sum_{\left(N_1,N_2,N_3\right)\in\mathcal{N}} \|u\|_{N_1+1+\alpha}\left(\|\phi\|_{N_2+2+\delta+\alpha} \|\psi\|_{N_3+\alpha} +\|\phi\|_{N_2+1+\alpha} \|\psi\|_{N_3+1+\delta+\alpha}\right) \\
				&+\|u\|_{N_1+2+\delta+\alpha}\left(\|\phi\|_{N_2+1+\alpha} \|\psi\|_{N_3+\alpha}\right)
			\end{aligned}
		\end{equation}
		with implicit constants depending only on $\alpha$, here $\mathcal{N}=\left\{\left(N,0,0\right),\left(0,N,0\right),\left(0,0,N\right)\right\}$.
	\end{lem}
	\begin{proof}
		We only give a proof for estimate \eqref{tri2} since then estimate \eqref{tri3} follows directly  from \eqref{tri2} and the proof for estimate \eqref{tri1} is almost the same as that for  \eqref{tri2}. Also, we only prove the case $N=0$ of \eqref{tri2} since the case $N\ge1$ follows easily from the $N=0$ case and Newton-Leibniz's Law.
		
			Due to \eqref{def-bi},  $S_i[f, g]=\Lambda^{-1}\left(\left(\partial_i T_1 f\right) g-\left(\partial_i f\right) T_1 g\right)$. Thus one can expand bilinear Fourier multiplier operator $S_i[f,g]$ in the  frequency space as 
		\begin{equation}\label{expand-bi}
			S_i[f,g](x)=\sum_{\substack{\xi+\eta,\xi, \eta\in\mathbb{Z}^2\backslash \left\{0\right\}}}M_i(\xi,\eta)\hat{f}(\xi)\hat{g}(\eta)e^{i(\xi+\eta)\cdot x}
		\end{equation}
		where 
		\begin{equation}\label{def-M}
			M_i(x,y)= \frac{ix_i\left(\tilde{m}(x)-\tilde{m}(y)\right)}{|x+y|},
		\end{equation}
		\eqref{def-trilin} yield that 
		\begin{equation}
			\begin{aligned}
				S_i[u, \phi, \psi]&=u \cdot \nabla S_i[\phi, \psi]-S_i[u \cdot \nabla \phi, \psi]-S_i[\phi, u \cdot \nabla \psi]\\
				&=u^{s} \pa_s  S_i[\phi, \psi]-S_i[u^{s}\pa_s \phi, \psi]-S_i[\phi, u^{s}\pa_s \psi]
			\end{aligned}
		\end{equation}
		where  the summation sign for $s=1,2$ has been omitted. Then one can similarly expand the trilinear Fourier multiplier $S_i[u,\phi,\psi]$ in the frequency space as
		\begin{equation}
			S_i[u,\phi,\psi](x)=\sum_{\xi,\eta,\zeta\in \mathcal{F}_0}M_{s,i}\left(\xi,\eta,\zeta\right)\hat{u}^{s}(\xi)\hat{\phi}(\eta)\hat{\psi}(\zeta)e^{i(\xi+\eta+\zeta)\cdot x}
		\end{equation}
		where 
		\begin{equation}
			\begin{aligned}
				M_{s,i}(\xi,\eta,\zeta)&=i(\eta+\zeta)_s M_i(\eta,\zeta)-i\eta_sM_i(\xi+\eta,\zeta)-i\zeta_sM_i(\eta,\xi+\zeta)\\
				&=i\eta_{s}\left(M_i(\eta,\zeta)-M_i(\xi+\eta,\zeta)\right)+i\zeta_{s}\left(M_i(\eta,\zeta)-M_i(\eta,\xi+\zeta)\right)
			\end{aligned}
		\end{equation}
		and $\mathcal{F}_0=\left\{(\xi,\eta,\zeta)\in\mathbb{Z}^6:\eta, \zeta, \eta+\zeta, \xi+\zeta, \xi+\eta, \xi+\eta+\zeta\in \mathbb{Z}^2\backslash\left\{0\right\}\right\}$. Finally, we have the following decomposition
		\begin{equation}
			\begin{aligned}
				&S_i[u,\phi,\psi](x)=\sum_{k_1, k_2, k_3\in\mathbb{Z}}S_i[\Delta_{k_1}u,\Delta_{k_2}\phi,\Delta_{k_3}\psi](x)\\
				&=\big(\sum_{k_1, k_2, k_3\in \mathcal{Z}_1}+\sum_{k_1, k_2, k_3\in \mathcal{Z}_2}+\sum_{k_1, k_2, k_3\in \mathcal{Z}_3}+\sum_{k_1, k_2, k_3\in \mathcal{Z}_4}+\sum_{k_1, k_2, k_3\in \mathcal{Z}_5}+\sum_{k_1, k_2, k_3\in \mathcal{Z}_6}\big)S_i[\Delta_{k_1}u,\Delta_{k_2}\phi,\Delta_{k_3}\psi](x)\\
				&=:S_i^{LMH}[u,\phi,\psi](x)+S_i^{LHM}[u,\phi,\psi](x)+S_i^{LHH}[u,\phi,\psi](x)+S_i^{LLH}[u,\phi,\psi](x)+S_i^{LHL}[u,\phi,\psi](x)\\
				&+S_i^{HHH}[u,\phi,\psi](x)
			\end{aligned}
		\end{equation}
		where the summations are actually taken on $k_1, k_2, k_3\ge-1$ and 
		$$\mathcal{Z}_1=\left\{k_1\le k_2-1000, k_2\le k_3-1000\right\},\quad\mathcal{Z}_2=\left\{k_1\le k_3-1000, k_3\le k_2-1000\right\},$$
		$$\quad\mathcal{Z}_3=\left\{k_1\le k_2-1000, k_1\le k_3-1000,  k_3-999 \le k_2 \le k_3+999\right\},\quad\mathcal{Z}_4=\left\{k_1\le k_3-1000, k_1\ge k_2-999\right\},$$
		$$\mathcal{Z}_5=\left\{k_1\le k_2-1000, k_1\ge k_3-999\right\}, \quad\mathcal{Z}_6=\left\{k_1\ge k_3-999, k_1\ge k_2-999\right\}.$$
		\textbf{Step 1.} Estimate for $S_i^{LMH}[u,\phi,\psi]$: in this case, we write 
		$$M_{s,i}(\xi,\eta,\zeta)=i\xi|\zeta|^{1+\delta}\underbrace{\big(-\frac{\eta_s}{|\zeta|}|\zeta|^{-\delta}\int_{0}^{1}\nabla_{x}M_i(\sigma\xi+\eta,\zeta)d\sigma\big)}_{M_{s,i}^1(\xi,\eta,\zeta)}+i\xi|\zeta|^{1+\delta}\underbrace{\big(-\frac{\zeta_{s}}{|\zeta|}|\zeta|^{-\delta}\int_{0}^{1}\nabla_{y}M_i(\eta,\sigma\xi+\zeta)d\sigma\big)}_{M_{s,i}^2(\xi,\eta,\zeta)}$$
		Denote also that
		$$M_{s,i}^{k_1, k_2, k_3, 1}(\xi,\eta,\zeta)=M_{s,i}^1(\xi,\eta,\zeta)\bar{\chi}_{k_1}(\xi)\bar{\chi}_{k_2}(\eta)\bar{\chi}_{k_3}(\zeta)$$ $$M_{s,i}^{k_1, k_2, k_3, 2}(\xi,\eta,\zeta)=M_{s,i}^2(\xi,\eta,\zeta)\bar{\chi}_{k_1}(\xi)\bar{\chi}_{k_2}(\eta)\bar{\chi}_{k_3}(\zeta)$$
		where the cutoff function $\bar{\chi}_j(\cdot)$ is defined in the beginning of section B. Then for $k\ge-1$ we have 
		\begin{equation}\notag
			\begin{aligned}
				&\Delta_{k}S_i^{LMH}[u,\phi,\psi](x)=\sum_{\substack{k_1, k_2, k_3\in \mathcal{Z}_1\\|k-k_3|\le10}}\Delta_{k}S_i[\Delta_{k_1}u,\Delta_{k_2}\phi,\Delta_{k_3}\psi](x)\\
				&=\sum_{\substack{k_1, k_2, k_3\in \mathcal{Z}_1\\|k-k_3|\le10}}\Delta_{k}\sum_{\xi,\eta,\zeta\in \mathcal{F}_0}\left(M_{s,i}^{k_1, k_2, k_3, 1}(\xi,\eta,\zeta)+M_{s,i}^{k_1, k_2, k_3, 2}(\xi,\eta,\zeta)\right)\widehat{\nabla\Delta_{k_1} u^{s}}(\xi)\widehat{\Delta_{k_2}\phi}(\eta)\widehat{\Lambda^{1+\delta}\Delta_{k_3}\psi}(\zeta)e^{i(\xi+\eta+\zeta)\cdot x}\\
				&=\sum_{\substack{k_1, k_2, k_3\in \mathcal{Z}_1\\|k-k_3|\le10}}\Delta_{k}\sum_{t=1}^{2} \int_{\mathbb{R}^2\times\mathbb{R}^2\times\mathbb{R}^2}K_{s,i}^{k_1, k_2, k_3, t}(x-y_1,x-y_2,x-y_3){\nabla\Delta_{k_1} u^{s}}(y_1){\Delta_{k_2}\phi}(y_2){\Lambda^{1+\delta}\Delta_{k_3}\psi}(y_3)dy_1dy_2dy_3
			\end{aligned}
		\end{equation}
		where $u, \phi, \psi $ are identified with their periodic extensions and for $t=1,2$
		$$K_{s,i}^{k_1, k_2, k_3, t}(x,y,z)=\mathcal{F}^{-1}\left(M_{s,i}^{k_1, k_2, k_3, t}\right)=\int_{\mathbb{R}^2\times\mathbb{R}^2\times\mathbb{R}^2}M_{s,i}^{k_1, k_2, k_3, t}(\xi,\eta,\zeta)e^{i\left(\xi\cdot x+\eta\cdot y+\zeta\cdot z\right)}d\xi d\eta d\zeta$$
		Then using the definitions of $\mathcal{Z}_1, M_{s,i}^{k_1, k_2, k_3, t}$ and $M_i(x,y) $, one can easily verify that for $a=(a_1,a_2), b=(b_1,b_2), c=(c_1.c_2)$, with non-negative integers  $a_i, b_i, c_i\ge0, i=1,2$, it holds that 
		$$\|\pa_{\xi}^{a}\pa_{\eta}^{b}\pa_{\zeta}^{c}M_{s,i}^{k_1, k_2, k_3, t}\|_{L^{\infty}}\lesssim 2^{-|a|k_1}2^{-|b|k_2}2^{-|c|k_3}.$$
		Then combing this estimate with Lemma  \ref{loc} applying to $f=M_{s,i}^{k_1, k_2, k_3, t}$, one can get
		$$\|K_{s,i}^{k_1, k_2, k_3, t}\|_{L^{1}_{x,y,z}}\lesssim1.$$
		Hence 
		\begin{equation}\notag
			\begin{aligned}
				\|\Delta_{k}S_i^{LMH}[u,\phi,\psi]\|_0&\lesssim\sum_{\substack{k_1, k_2, k_3\in \mathcal{Z}_1\\|k-k_3|\le10}}\|\nabla\Delta_{k_1}u^s\|_0\|\Delta_{k_2}\phi\|_0\|\Lambda^{1+\delta}\Delta_{k_3}\psi\|_0\\
				&\lesssim\sum_{\substack{k_1, k_2, k_3\in \mathcal{Z}_1\\|k-k_3|\le10}}2^{-(k_1+k_2+k_3)\alpha}\|u\|_{1+\alpha}\|\phi\|_{\alpha}\|\psi\|_{1+\delta+\alpha}\\
				&\lesssim 2^{-k\alpha}\|u\|_{1+\alpha}\|\phi\|_{\alpha}\|\psi\|_{1+\delta+\alpha},
			\end{aligned}
		\end{equation}
		which yields
		$$\|S_i^{LMH}[u,\phi,\psi]\|_{\alpha}\lesssim\|u\|_{1+\alpha}\|\phi\|_{\alpha}\|\psi\|_{1+\delta+\alpha}.$$
		\textbf{Step 2.} Estimate for $S_i^{LHM}[u,\phi,\psi]$: this case is almost the same as step 1 and one can get 
		$$\|S_i^{LHM}[u,\phi,\psi]\|_{\alpha}\lesssim\|u\|_{1+\alpha}\|\phi\|_{1+\delta+\alpha}\|\psi\|_{\alpha}$$
		\textbf{Step 3.} Estimate for $S_i^{LHH}[u,\phi,\psi]$:
		in this case, for $m\ge-1$ we write 
		\begin{equation}\notag
			\begin{aligned}
				&\Delta_{m}S_i^{LHH}[u,\phi,\psi](x)=\sum_{\substack{k_1, k_2, k_3\in \mathcal{Z}_3\\k_3\ge m-1100}}\Delta_{m}S_i[\Delta_{k_1}u,\Delta_{k_2}\phi,\Delta_{k_3}\psi](x)\\
				&=\sum_{k,\ell\ge-1}\sum_{\substack{k_1, k_2, k_3\in \mathcal{Z}_3\\k_3\ge m-1100}}\Delta_{m}\sum_{\xi,\eta,\zeta\in \mathcal{F}_0}M_{s,i,k,\ell}^{k_1, k_2, k_3}(\xi,\eta,\zeta)\widehat{\Delta_{k_1} u^{s}}(\xi)\widehat{\Delta_{k_2}\phi}(\eta)\widehat{\Delta_{k_3}\psi}(\zeta)e^{i(\xi+\eta+\zeta)\cdot x}
			\end{aligned}
		\end{equation}
		where 
		$$M_{s, i,k,\ell}^{k_1, k_2, k_3}(\xi, \eta, \zeta)=M_{s, i}(\xi, \eta, \zeta) \bar{\chi}_{k_1}(\xi) \bar{\chi}_{k_2}(\eta) \bar{\chi}_{k_3}(\zeta)\bar{\chi}_{k}(\eta+\zeta) \bar{\chi}_{\ell}(\xi+\eta+\zeta)$$
		$$M_{s, i,k}^{k_1, k_2, k_3}(\xi, \eta, \zeta)=M_{s, i}(\xi, \eta, \zeta) \bar{\chi}_{k_1}(\xi) \bar{\chi}_{k_2}(\eta) \bar{\chi}_{k_3}(\zeta)\bar{\chi}_{k}(\eta+\zeta) $$
		Recall the definition of $\mathcal{Z}^3$. In this case,  $k_1\le\min \left\{k_2, k_3\right\}-1000$ and $|k_2-k_3|\le999$. Then we consider three sub-cases according to the following decomposition:
		$$\Delta_{m}S_{i}^{LHH}[u,\phi,\psi](x)=\Delta_{m}S_{i,1}^{LHH}[u,\phi,\psi](x)+\Delta_{m}S_{i,2}^{LHH}[u,\phi,\psi](x)+\Delta_{m}S_{i,3}^{LHH}[u,\phi,\psi](x)$$
		where
		\begin{equation}\notag
			\begin{aligned}
				&\Delta_{m}S_{i,1}^{LHH}[u,\phi,\psi](x)\\
				&=\sum_{\substack{k_1, k_2, k_3\in \mathcal{Z}_3\\k_3\ge m-1100}}\sum_{\substack{k,\ell\ge-1\\k\ge\min\left\{k_2, k_3\right\}-499}}\Delta_{m}\sum_{\xi,\eta,\zeta\in \mathcal{F}_0}M_{s,i,k,\ell}^{k_1, k_2, k_3}(\xi,\eta,\zeta)\widehat{\Delta_{k_1} u^{s}}(\xi)\widehat{\Delta_{k_2}\phi}(\eta)\widehat{\Delta_{k_3}\psi}(\zeta)e^{i(\xi+\eta+\zeta)\cdot x},
			\end{aligned}\\
		\end{equation}
		
		\begin{equation}\notag
			\begin{aligned}
				&\Delta_{m}S_{i,2}^{LHH}[u,\phi,\psi](x)\\
				&=\sum_{\substack{k_1, k_2, k_3\in \mathcal{Z}_3\\k_3\ge m-1100}}\sum_{\substack{k,\ell\ge-1\\k_1+250\le k\le\min\left\{k_2, k_3\right\}-500}}\Delta_{m}\sum_{\xi,\eta,\zeta\in \mathcal{F}_0}M_{s,i,k,\ell}^{k_1, k_2, k_3}(\xi,\eta,\zeta)\widehat{\Delta_{k_1} u^{s}}(\xi)\widehat{\Delta_{k_2}\phi}(\eta)\widehat{\Delta_{k_3}\psi}(\zeta)e^{i(\xi+\eta+\zeta)\cdot x},
			\end{aligned}
		\end{equation}
		\begin{equation}\notag
			\begin{aligned}
				&\Delta_{m}S_{i,3}^{LHH}[u,\phi,\psi](x)\\
				&=\sum_{\substack{k_1, k_2, k_3\in \mathcal{Z}_3\\k_3\ge m-1100}}\sum_{\substack{k,\ell\ge-1\\k\le k_1 +249}}\Delta_{m}\sum_{\xi,\eta,\zeta\in \mathcal{F}_0}M_{s,i,k,\ell}^{k_1, k_2, k_3}(\xi,\eta,\zeta)\widehat{\Delta_{k_1} u^{s}}(\xi)\widehat{\Delta_{k_2}\phi}(\eta)\widehat{\Delta_{k_3}\psi}(\zeta)e^{i(\xi+\eta+\zeta)\cdot x}
			\end{aligned}
		\end{equation}
		\textbf{Step 3.1}. Estimate for $\Delta_{m}S_{i,1}^{LHH}[u,\phi,\psi]$, in this case since $\min\left\{k_2, k_3\right\}-499\le k\le \max\left\{k_2, k_3\right\}+10$ and $|k_2-k_3|\le999$, then we have
		$$M_{s, i}(\xi, \eta, \zeta)=i \xi|\zeta|^{1+\delta} \underbrace{\left(-\frac{\eta_s}{|\zeta|}|\zeta|^{-\delta} \int_0^1 \nabla_x M_i(\sigma \xi+\eta, \zeta) d \sigma-\frac{ \zeta_s}{|\zeta|}|\zeta|^{-\delta} \int_0^1 \nabla_y M_i(\eta, \sigma \xi+\zeta) d \sigma\right)}_{\tilde{M}_{s, i}(\xi, \eta, \zeta)},$$
		$$\tilde{M}_{s, i, k}^{k_1, k_2, k_3}(\xi, \eta, \zeta)=\tilde{M}_{s, i}(\xi, \eta, \zeta) \bar{\chi}_{k_1}(\xi) \bar{\chi}_{k_2}(\eta) \bar{\chi}_{k_3}(\zeta) \bar{\chi}_k(\eta+\zeta),$$
		\begin{equation}\notag
			\begin{aligned}
				&\Delta_{m}S_{i,1}^{LHH}[u,\phi,\psi](x)\\
				&=\sum_{\substack{k_1, k_2, k_3\in \mathcal{Z}_3\\k_3\ge m-1100}}\sum_{\substack{k\ge-1\\k\ge\min\left\{k_2, k_3\right\}-499}}\Delta_{m}\sum_{\xi,\eta,\zeta\in \mathcal{F}_0}M_{s,i,k}^{k_1, k_2, k_3}(\xi,\eta,\zeta)\widehat{\Delta_{k_1} u^{s}}(\xi)\widehat{\Delta_{k_2}\phi}(\eta)\widehat{\Delta_{k_3}\psi}(\zeta)e^{i(\xi+\eta+\zeta)\cdot x}.
			\end{aligned}\\
		\end{equation}
		Then as in step 1, one can get
		
		\begin{equation}\notag
			\begin{aligned}
				&\Delta_{m}S_{i,1}^{LHH}[u,\phi,\psi](x)\\
				&=\sum_{\substack{k_1, k_2, k_3\in \mathcal{Z}_3\\k_3\ge m-1100}}\sum_{\substack{k\ge-1\\k\ge\min\left\{k_2, k_3\right\}-499}}\Delta_{m}\sum_{\xi,\eta,\zeta\in \mathcal{F}_0}\tilde{M}_{s,i, k}^{k_1, k_2, k_3}(\xi,\eta,\zeta)\widehat{\nabla\Delta_{k_1} u^{s}}(\xi)\widehat{\Delta_{k_2}\phi}(\eta)\widehat{\Lambda^{1+\delta}\Delta_{k_3}\psi}(\zeta)e^{i(\xi+\eta+\zeta)\cdot x}\\
				&=\sum_{\substack{k_1, k_2, k_3\in \mathcal{Z}_3\\k_3\ge m-1100\\k\ge\min\left\{k_2, k_3\right\}-499}}\Delta_{m} \int_{\mathbb{R}^2\times\mathbb{R}^2\times\mathbb{R}^2}\tilde{K}_{s,i,k}^{k_1, k_2, k_3}(x-y_1,x-y_2,x-y_3){\nabla\Delta_{k_1} u^{s}}(y_1){\Delta_{k_2}\phi}(y_2){\Lambda^{1+\delta}\Delta_{k_3}\psi}(y_3)dy_1dy_2dy_3
			\end{aligned}
		\end{equation}
		where $u, \phi, \psi $ are identified with their periodic extensions and 
		$$\tilde{K}_{s,i,k}^{k_1, k_2, k_3}(x,y,z)=\mathcal{F}^{-1}\left(\tilde{M}_{s,i,k}^{k_1, k_2, k_3}\right)=\int_{\mathbb{R}^2\times\mathbb{R}^2\times\mathbb{R}^2}\tilde{M}_{s,i,k}^{k_1, k_2, k_3}(\xi,\eta,\zeta)e^{i\left(\xi\cdot x+\eta\cdot y+\zeta\cdot z\right)}d\xi d\eta d\zeta$$
		Then using the definition of $\mathcal{Z}_3, \tilde{M}_{s,i,k}^{k_1, k_2, k_3}, M_i(x,y) $ and $k\ge \min \left\{k_2,k_3\right\}-499, |k_2-k_3|\le999$, one can easily verify that for $a=(a_1,a_2), b=(b_1,b_2), c=(c_1.c_2)$, with non-negative integers  $a_i, b_i, c_i\ge0, i=1,2$, it holds that
		$$\|\pa_{\xi}^{a}\pa_{\eta}^{b}\pa_{\zeta}^{c}\tilde{M}_{s,i,k}^{k_1, k_2, k_3}\|_{L^{\infty}}\lesssim 2^{-|a|k_1}2^{-|b|k_2}2^{-|c|k_3}$$
		Then combining this estimate with Lemma  \ref{loc} applying to $f=\tilde{M}_{s,i,k}^{k_1, k_2, k_3}$ yields
		$$\|\tilde{K}_{s,i,k}^{k_1, k_2, k_3}\|_{L^{1}_{x,y,z}}\lesssim1$$
		Hence 
		\begin{equation}\notag
			\begin{aligned}
				\|\Delta_{m}S_{i,1}^{LHH}[u,\phi,\psi]\|_0&\lesssim\sum_{\substack{k_1, k_2, k_3\in \mathcal{Z}_1\\k_3\ge m-1100\\k\ge\min\left\{k_2, k_3\right\}-499}}\|\nabla\Delta_{k_1}u^s\|_0\|\Delta_{k_2}\phi\|_0\|\Lambda^{1+\delta}\Delta_{k_3}\psi\|_0\\
				&\lesssim\sum_{\substack{k_1, k_2, k_3\in \mathcal{Z}_1\\k_3\ge m-1100\\k\ge\min\left\{k_2, k_3\right\}-499}}2^{-(k_1+\frac{k_2}{2}+k_3+\frac{k}{2})\alpha}\|u\|_{1+\alpha}\|\phi\|_{\alpha}\|\psi\|_{1+\delta+\alpha}\\
				&\lesssim 2^{-m\alpha}\|u\|_{1+\alpha}\|\phi\|_{\alpha}\|\psi\|_{1+\delta+\alpha},
			\end{aligned}
		\end{equation}
		which leads to
		$$\|S_{i,1}^{LHH}[u,\phi,\psi]\|_\alpha\lesssim\|u\|_{1+\alpha}\|\phi\|_{\alpha}\|\psi\|_{1+\delta+\alpha}$$
		
		\textbf{Step 3.2}. Estimate for $S_{i,3}^{LHH}[u,\phi,\psi]$.  In this case, $ k\le k_1+249$ and $|k_2-k_3|\le999$ hold, then we can write
		\begin{equation}\notag
			\begin{aligned}
				&M_{s,i}\left(\xi,\eta,\zeta\right)\\
				=&|\xi||\eta|^{1+\delta}\underbrace{\frac{i(\eta+\zeta)_{s}}{|\xi|}\frac{M_i(\eta,\zeta)}{|\eta|^{1+\delta}}}_{M^3_{s,i}(\xi,\eta,\zeta)} +|\zeta|^{1+\delta}i\xi\cdot\underbrace{\frac{\big(\int_{0}^{1}\left(\nabla_{x}m_0(\tau\xi+\eta,\zeta)-\nabla_{y}m_0(\eta,\tau\xi+\zeta)\right)d\tau\big)i \zeta_s\eta_i}{|\zeta|^{1+\delta}}\frac{\xi+\eta+\zeta}{|\xi+\eta+\zeta|}}_{M_{s,i}^4(\xi,\eta,\zeta)}\\
				&+|\xi||\zeta|^{1+\delta}\underbrace{\frac{\xi+\eta+\zeta}{|\xi+\eta+\zeta|}\frac{(\eta+\zeta)_s}{|\xi|}\frac{(\eta+\xi)_i  m_0(\xi+\eta,\zeta)}{|\zeta|^{1+\delta}}}_{M^5_{s,i}(\xi,\eta,\zeta)} +i\xi_i|\zeta|^{1+\delta}\underbrace{\frac{\xi+\eta+\zeta}{|\xi+\eta+\zeta|}\frac{i\zeta_s m_0(\eta+\xi,\zeta)}{|\zeta|^{1+\delta}}}_{M^6_{s}(\xi,\eta,\zeta)}
			\end{aligned}
		\end{equation}
		where 
		$$m_0(x,y)=\int_{0}^{1}\nabla \tilde{m}(\sigma(x+y)-y)d\sigma$$
		Set
		$$M^{k_1,k_3,k,3}_{s}(\xi,\eta,\zeta)=\frac{i(\eta+\zeta)_s}{|\xi|}\frac{\eta+\zeta}{|\eta+\zeta|}\bar{\chi}_{k_1}(\xi)\tilde{\chi}_{k}(\eta+\zeta)\tilde{\chi}_{k_3}(\zeta),\quad \tilde{M}^{k_1,k_3,k,3}_{s}(\xi,\eta,\zeta)=\frac{i\eta_s}{|\xi|}\frac{\eta}{|\eta|}\bar{\chi}_{k_1}(\xi)\tilde{\chi}_{k}(\eta)\bar{\chi}_{k_3}(\zeta),$$
		$${M}^{k_2,k_3,k,3}_{i}(\eta,\zeta)=\frac{i\eta_i\int_{0}^{1}\nabla \tilde{m}(\sigma(\eta+\zeta)-\zeta)d\sigma}{|\eta|^{1+\delta}}\bar{\chi}_{k_2}(\eta)\bar{\chi}_{k_3}(\zeta)\bar{\chi}_{k}(\eta+\zeta),$$
		$$\tilde{M}^{k_2,k_3,k,3}_{i}(\eta,\zeta)=\frac{i(\eta-\zeta)_i\int_{0}^{1}\nabla \tilde{m}(\sigma\eta-\zeta)d\sigma}{|\eta-\zeta|^{1+\delta}}\bar{\chi}_{k_2}(\eta-\zeta)\bar{\chi}_{k_3}(\zeta)\bar{\chi}_{k}(\eta),$$
		$$M^{k_2,k_3,\ell,4}(\xi,\eta,\zeta)=\frac{\xi+\eta+\zeta}{|\xi+\eta+\zeta|}\bar{\chi}_{k_2}(\eta)\tilde{\chi}_{k_3}(\zeta)\bar{\chi}_{\ell}(\xi+\eta+\zeta),\quad \tilde{M}^{k_2,k_3,\ell,4}(\xi,\eta,\zeta)=\frac{\xi}{|\xi|}\bar{\chi}_{k_2}(\eta)\tilde{\chi}_{k_3}(\zeta)\bar{\chi}_{\ell}(\xi),$$
		$${M}^{k_1,k_3,k,4}_{s,i}(\xi,\eta,\zeta)=\frac{\big(\int_{0}^{1}\left(\nabla_{x}m_0(\tau\xi+\eta,\zeta)-\nabla_{y}m_0(\eta,\tau\xi+\zeta)\right)d\tau\big) i\zeta_s\eta_i}{|\zeta|^{1+\delta}}\bar{\chi}_{k_1}(\xi)\bar{\chi}_{k_3}(\zeta)\bar{\chi}_{k}(\eta+\zeta),$$
		$$\tilde{M}^{k_1,k_3,k,4}_{s,i}(\xi,\eta,\zeta)=\frac{\big(\int_{0}^{1}\left(\nabla_{x}m_0(\tau\xi+\eta-\zeta,\zeta)-\nabla_{y}m_0(\eta-\zeta,\tau\xi+\zeta)\right)d\tau\big)i \zeta_s(\eta-\zeta)_i}{|\zeta|^{1+\delta}}\bar{\chi}_{k_1}(\xi)\bar{\chi}_{k_3}(\zeta)\bar{\chi}_{k}(\eta),$$
		$${M}^{k_1,k_3,k,5}_{s,i}(\xi,\eta,\zeta)=\frac{(\eta+\zeta)_s}{|\xi|} \frac{(\eta+\xi)_i m_0(\xi+\eta, \zeta)}{|\zeta|^{1+\delta}}\bar{\chi}_{k_1}(\xi)\bar{\chi}_{k_3}(\zeta)\bar{\chi}_{k}(\eta+\zeta),$$
		$$\tilde{M}^{k_1,k_3,k,5}_{s,i}(\xi,\eta,\zeta)=\frac{\eta_s}{|\xi|} \frac{(\eta+\xi-\zeta)_i m_0(\xi+\eta-\zeta, \zeta)}{|\zeta|^{1+\delta}}\bar{\chi}_{k_1}(\xi)\bar{\chi}_{k_3}(\zeta)\bar{\chi}_{k}(\eta),$$
		$${M}^{k_1,k_3,k,6}_{s}(\xi,\eta,\zeta)=\frac{i\zeta_s m_0(\eta+\xi, \zeta)}{|\zeta|^{1+\delta}}\bar{\chi}_{k_1}(\xi)\bar{\chi}_{k_3}(\zeta)\bar{\chi}_{k}(\eta+\zeta),$$
		$$\tilde{M}^{k_1,k_3,k,6}_{s}(\xi,\eta,\zeta)=\frac{i\zeta_s m_0(\eta-\zeta+\xi, \zeta)}{|\zeta|^{1+\delta}}\bar{\chi}_{k_1}(\xi)\bar{\chi}_{k_3}(\zeta)\bar{\chi}_{k}(\eta),$$
		$$M^{k_1,k_2,k_3,3}_{s,i,k}(\xi,\eta,\zeta)=\left(M^{k_1,k_3,k,3}_{s}M^{k_2,k_3,k,3}_{i}\right)(\xi,\eta,\zeta)=M^3_{s,i}(\xi,\eta,\zeta)\bar{\chi}_{k_1}(\xi)\bar{\chi}_{k_2}(\eta)\bar{\chi}_{k_3}(\zeta)\bar{\chi}_{k}(\eta+\zeta),$$
		$$M^{k_1,k_2,k_3,4}_{s,i,k,\ell}(\xi,\eta,\zeta)=\left(M^{k_2,k_3,\ell,4}M^{k_1,k_3,k,4}_{s,i}\right)(\xi,\eta,\zeta)=M^4_{s,i}(\xi,\eta,\zeta)\bar{\chi}_{k_1}(\xi)\bar{\chi}_{k_2}(\eta)\bar{\chi}_{k_3}(\zeta)\bar{\chi}_{k}(\eta+\zeta)\bar{\chi}_{\ell}(\xi+\eta+\zeta),$$
		$$M^{k_1,k_2,k_3,5}_{s,i,k,\ell}(\xi,\eta,\zeta)=\left(M^{k_2,k_3,\ell,4}M^{k_1,k_3,k,5}_{s,i}\right)(\xi,\eta,\zeta)=M^5_{s,i}(\xi,\eta,\zeta)\bar{\chi}_{k_1}(\xi)\bar{\chi}_{k_2}(\eta)\bar{\chi}_{k_3}(\zeta)\bar{\chi}_{k}(\eta+\zeta)\bar{\chi}_{\ell}(\xi+\eta+\zeta),$$
		$$M^{k_1,k_2,k_3,6}_{s,k,\ell}(\xi,\eta,\zeta)=\left(M^{k_2,k_3,\ell,4}M^{k_1,k_3,k,6}_{s}\right)(\xi,\eta,\zeta)=M^6_{s}(\xi,\eta,\zeta)\bar{\chi}_{k_1}(\xi)\bar{\chi}_{k_2}(\eta)\bar{\chi}_{k_3}(\zeta)\bar{\chi}_{k}(\eta+\zeta)\bar{\chi}_{\ell}(\xi+\eta+\zeta),$$
		where  $\tilde{\chi}$ is chosen such that $\tilde{\chi}\bar{\chi}=\bar{\chi}$. Then in this case one can check as before to get 
		$$\begin{aligned}
			& \left\|\partial_{\xi}^a \partial_\eta^b \partial_\zeta^c \tilde{M}_{s}^{k_1, k_3, k, 3}\right\|_{L^{\infty}} \lesssim 2^{-|a| k_1} 2^{-|b| k} 2^{-|c| k_3},\quad  \left\| \partial_\eta^b \partial_\zeta^c \tilde{M}_{i}^{k_2,k_3, k, 3}\right\|_{L^{\infty}} \lesssim  2^{-|b| k} 2^{-|c| k_3},\\
			& \left\|\partial_{\xi}^a \partial_\eta^b \partial_\zeta^c \tilde{M}^{k_2, k_3, \ell, 4}\right\|_{L^{\infty}} \lesssim 2^{-|a| \ell} 2^{-|b| k_2} 2^{-|c| k_3},\quad\left\|\partial_{\xi}^a \partial_\eta^b \partial_\zeta^c \tilde{M}_{s,i}^{k_1, k_3, k, 4}\right\|_{L^{\infty}} \lesssim 2^{-|a| k_1} 2^{-|b| k} 2^{-|c| k_3},\\
			& \left\|\partial_{\xi}^a \partial_\eta^b \partial_\zeta^c \tilde{M}_{s,i}^{k_1, k_3, k, 5}\right\|_{L^{\infty}} \lesssim 2^{-|a|k_1} 2^{-|b| k} 2^{-|c| k_3},\quad\left\|\partial_{\xi}^a \partial_\eta^b \partial_\zeta^c \tilde{M}_{s}^{k_1, k_3, k, 6}\right\|_{L^{\infty}} \lesssim 2^{-|a| k_1} 2^{-|b| k} 2^{-|c| k_3}.\\
		\end{aligned}$$
		Consequently, one can use Lemmas \ref{symbol} and \ref{loc} to get
		$$\|\mathcal{F}^{-1}({M}_{s}^{k_1, k_3, k, 3})\|_{L_1}=\|\mathcal{F}^{-1}(\tilde{M}_{s}^{k_1, k_3, k, 3})\|_{L_1}\lesssim1,\quad\|\mathcal{F}^{-1}({M}_{i}^{k_2, k_3, k, 3})\|_{L_1}=\|\mathcal{F}^{-1}(\tilde{M}_{i}^{k_2, k_3, k, 3})\|_{L_1}\lesssim1,$$
		$$\|\mathcal{F}^{-1}({M}^{k_2, k_3, \ell, 4})\|_{L_1}=\|\mathcal{F}^{-1}(\tilde{M}^{k_2, k_3, \ell, 4})\|_{L_1}\lesssim1,\quad\|\mathcal{F}^{-1}({M}_{s,i}^{k_1, k_3, k, 3})\|_{L_1}=\|\mathcal{F}^{-1}(\tilde{M}_{s,i}^{k_1, k_3, k, 3})\|_{L_1}\lesssim1,$$
		$$\|\mathcal{F}^{-1}({M}_{s,i}^{k_1, k_3, k, 5})\|_{L_1}=\|\mathcal{F}^{-1}(\tilde{M}_{s,i}^{k_1, k_3, k, 5})\|_{L_1}\lesssim1,\quad\|\mathcal{F}^{-1}({M}_{s}^{k_1, k_3, k, 6})\|_{L_1}=\|\mathcal{F}^{-1}(\tilde{M}_{s}^{k_1, k_3, k, 6})\|_{L_1}\lesssim1.$$
		This and  Lemma \ref{symbol} imply
		$$\|\mathcal{F}^{-1}({M}_{s,i,k}^{k_1, k_2, k_3, 3})\|_{L_1}=\|\mathcal{F}^{-1}({M}_{s}^{k_1, k_3, k, 3}{M}_{i}^{ k_2, k_3,k, 3})\|_{L_1}\lesssim\|\mathcal{F}^{-1}({M}_{s}^{k_1, k_3, k, 3})\|_{L_1}\|\mathcal{F}^{-1}({M}_{i}^{ k_2,k_3, k, 3})\|_{L_1}\lesssim1$$
		$$\|\mathcal{F}^{-1}({M}_{s,i,k,\ell}^{k_1, k_2, k_3, 4})\|_{L_1}=\|\mathcal{F}^{-1}({M}^{k_2, k_3, \ell, 4}{M}_{s,i}^{ k_1,k_3, k, 4})\|_{L_1}\lesssim\|\mathcal{F}^{-1}({M}^{k_2, k_3, \ell, 4})\|_{L_1}\|\mathcal{F}^{-1}({M}_{s,i}^{ k_1,k_3, k, 4})\|_{L_1}\lesssim1$$
		$$\|\mathcal{F}^{-1}({M}_{s,i,k,\ell}^{k_1, k_2, k_3, 5})\|_{L_1}=\|\mathcal{F}^{-1}({M}^{k_2, k_3, \ell, 4}{M}_{s,i}^{ k_1,k_3, k, 5})\|_{L_1}\lesssim\|\mathcal{F}^{-1}({M}^{k_2, k_3, \ell, 4})\|_{L_1}\|\mathcal{F}^{-1}({M}_{s,i}^{ k_1,k_3, k, 5})\|_{L_1}\lesssim1$$
		$$\|\mathcal{F}^{-1}({M}_{s,k,\ell}^{k_1, k_2, k_3, 6})\|_{L_1}=\|\mathcal{F}^{-1}({M}^{k_2, k_3, \ell, 4}{M}_{s}^{ k_1,k_3, k, 6})\|_{L_1}\lesssim\|\mathcal{F}^{-1}({M}^{k_2, k_3, \ell, 4})\|_{L_1}\|\mathcal{F}^{-1}({M}_{s,i}^{ k_1,k_3, k, 6})\|_{L_1}\lesssim1$$
		Note that
		\begin{equation}\notag
			\begin{aligned}
				&\Delta_{m}S_{i,3}^{LHH}[u,\phi,\psi](x)\\
				&=\sum_{\substack{k_1, k_2, k_3\in \mathcal{Z}_3\\k_3\ge m-1100}}\sum_{\substack{k,\ell\ge-1\\k\le k_1 +249}}\Delta_{m}\sum_{\xi,\eta,\zeta\in \mathcal{F}_0}M_{s,i,k,\ell}^{k_1, k_2, k_3}(\xi,\eta,\zeta)\widehat{\Delta_{k_1} u^{s}}(\xi)\widehat{\Delta_{k_2}\phi}(\eta)\widehat{\Delta_{k_3}\psi}(\zeta)e^{i(\xi+\eta+\zeta)\cdot x}\\
				&=\sum_{\substack{k_1, k_2, k_3\in \mathcal{Z}_3\\k_3\ge m-1100\\-1\le k\le k_1+249}}\Delta_{m}\sum_{\xi,\eta,\zeta\in \mathcal{F}_0}M_{s,i,k}^{k_1, k_2, k_3}(\xi,\eta,\zeta)\widehat{\Lambda\Delta_{k_1} u^{s}}(\xi)\widehat{\Lambda^{1+\delta}\Delta_{k_2}\phi}(\eta)\widehat{\Delta_{k_3}\psi}(\zeta)e^{i(\xi+\eta+\zeta)\cdot x}\\
				&+\sum_{\substack{k_1, k_2, k_3\in \mathcal{Z}_3\\k_3\ge m-1100}}\sum_{\substack{k,\ell\ge-1\\k\le k_1 +249}}\Delta_{m}\sum_{\xi,\eta,\zeta\in \mathcal{F}_0}M_{s,i,k,\ell}^{k_1, k_2, k_3,4}(\xi,\eta,\zeta)\cdot\widehat{\nabla\Delta_{k_1} u^{s}}(\xi)\widehat{\Delta_{k_2}\phi}(\eta)\widehat{\Lambda^{1+\delta}\Delta_{k_3}\psi}(\zeta)e^{i(\xi+\eta+\zeta)\cdot x}\\
				&+\sum_{\substack{k_1, k_2, k_3\in \mathcal{Z}_3\\k_3\ge m-1100}}\sum_{\substack{k,\ell\ge-1\\k\le k_1 +249}}\Delta_{m}\sum_{\xi,\eta,\zeta\in \mathcal{F}_0}M_{s,i,k,\ell}^{k_1, k_2, k_3,5}(\xi,\eta,\zeta)\cdot\widehat{\Lambda\Delta_{k_1} u^{s}}(\xi)\widehat{\Delta_{k_2}\phi}(\eta)\widehat{\Lambda^{1+\delta}\Delta_{k_3}\psi}(\zeta)e^{i(\xi+\eta+\zeta)\cdot x}\\
				&+\sum_{\substack{k_1, k_2, k_3\in \mathcal{Z}_3\\k_3\ge m-1100}}\sum_{\substack{k,\ell\ge-1\\k\le k_1 +249}}\Delta_{m}\sum_{\xi,\eta,\zeta\in \mathcal{F}_0}M_{s,k,\ell}^{k_1, k_2, k_3,6}(\xi,\eta,\zeta)\widehat{\pa_i\Delta_{k_1} u^{s}}(\xi)\widehat{\Delta_{k_2}\phi}(\eta)\widehat{\Lambda^{1+\delta}\Delta_{k_3}\psi}(\zeta)e^{i(\xi+\eta+\zeta)\cdot x}\\
			\end{aligned}
		\end{equation}
		Then one can argue as in the case of step 3.1 to get (where one may use the facts  that $2^{-k_2\alpha}\lesssim2^{-\frac{k_2+k}{2}\alpha}$ and  $2^{-k_2\alpha}\lesssim2^{-\frac{k_2+k+\ell}{3}\alpha}$ to control the sums for $k,\ell$: $\sum_{\substack{-1\le k\le k_1 +249}}$, $\sum_{\substack{k,\ell\ge-1\\k\le k_1 +249}}$) :
		$$\left\|S_{i, 3}^{L H H}[u, \phi, \psi]\right\|_\alpha \lesssim\|u\|_{1+\alpha}\|\phi\|_\alpha\|\psi\|_{1+\delta+\alpha}+\|u\|_{1+\alpha}\|\phi\|_{1+\delta+\alpha}\|\psi\|_{\alpha}$$
		
		\textbf{Step 3.3}. Estimate for $S_{i,2}^{LHH}[u,\phi,\psi]$.  In this case, $k_1+250\le k\le \min\left\{k_2, k_3\right\}-500$ and $|k_2-k_3|\le999$ hold,  and 
		$$M_{s, i}(\xi, \eta, \zeta)=i \xi|\zeta|^{1+\delta} \left(- \int_0^1|\zeta|^{-1-\delta}\big( \eta_s\nabla_x M_i(\sigma \xi+\eta, \zeta)+\zeta_s\nabla_y M_i(\eta, \sigma\xi+\zeta)\big) d \sigma\right)$$
		Due to $M_i(x,y)=\frac{ix_i(\tilde{m}(x)-\tilde{m}(y))}{|x+y|}$, one can get 
		$$\nabla_x M_i(x,y)=\frac{ie_i(\tilde{m}(x)-\tilde{m}(y))}{|x+y|}+\frac{ix_i\nabla \tilde{m}(x)}{|x+y|}-\frac{ix_i(x+y)(\tilde{m}(x)-\tilde{m}(y))}{|x+y|^3},$$
		$$\nabla_y M_i(x,y)=-\frac{ix_i\nabla \tilde{m}(y)}{|x+y|}-\frac{ix_i(x+y)(\tilde{m}(x)-\tilde{m}(y))}{|x+y|^3},$$
		where $e_1=(1,0)$ and  $e_2=(0,1)$. Then these formulas together with some calculations yield
		\begin{equation}\notag
			\begin{aligned}
				&|\zeta|^{-1-\delta}\big(\eta_s\nabla_x M_i(\sigma \xi+\eta, \zeta)+\zeta_s\nabla_y M_i(\eta, \sigma\xi+\zeta)\big)\\
				&=ie_i\frac{\sigma\xi+\eta+\zeta}{|\sigma\xi+\eta+\zeta|}\cdot\frac{\eta_s\int_{0}^{1}\nabla\tilde{m}(\tau(\sigma\xi+\eta+\zeta)-\zeta)d\tau}{|\zeta|^{1+\delta}}+\frac{i\sigma\xi_i\eta_s\nabla\tilde{m}(\sigma\xi+\eta)}{|\sigma\xi+\eta+\zeta||\zeta|^{1+\delta}}\\
				&-\frac{i\sigma\xi_i(\sigma\xi+\eta+\zeta)(\sigma\xi+\eta+\zeta)}{|\sigma\xi+\eta+\zeta|^3}\cdot\frac{\eta_s\int_{0}^{1}\nabla\tilde{m}(\tau(\sigma\xi+\eta+\zeta)-\zeta)d\tau}{|\zeta|^{1+\delta}}\\
				&+\frac{i(2\sigma\xi+\eta+\zeta)\cdot\eta_i\int_{0}^{1}\nabla m_s(\tau(2\sigma\xi+\eta+\zeta)-\sigma\xi-\zeta)d\tau}{|\sigma\xi+\eta+\zeta||\zeta|^{1+\delta}}-\frac{i\sigma\xi_s\eta_i(\nabla\tilde{m}(\sigma\xi+\eta)-\nabla\tilde{m}(\sigma\xi+\zeta))}{|\sigma\xi+\eta+\zeta||\zeta|^{1+\delta}}\\
				&-\frac{i(\sigma\xi+\eta+\zeta)(\eta+\zeta)_s(\sigma\xi+\eta+\zeta)}{|\sigma\xi+\eta+\zeta|^3}\cdot\frac{\eta_i\int_{0}^{1}\nabla\tilde{m}(\tau(\sigma\xi+\eta+\zeta)-\zeta)d\tau}{|\zeta|^{1+\delta}}\\
				&+\frac{i\sigma\xi^{T}\big(\int_{0}^{1}\int_{0}^{1}\eta_i\zeta_s\nabla^2\tilde{m}(\mu(2\tau\sigma\xi+\eta+\zeta)-\tau\sigma\xi-\zeta)(2\tau\sigma\xi+\eta+\zeta)d\tau d\mu\big)(\sigma\xi+\eta+\zeta)}{|\sigma\xi+\eta+\zeta|^3|\zeta|^{1+\delta}}
			\end{aligned}
		\end{equation}
		where $m_s(x)=x_s\nabla\tilde{m}(x)$. Then it follows from Lemmas \ref{symbol}, \ref{loc} and $k_1+250\le k\le \min\left\{k_2, k_3\right\}-500$, together with a similar argument used in step 3.2 and with some tedious calculations that
		$$\|\mathcal{F}^{-1}({M}_{s, i, k,\ell}^{k_1, k_2, k_3})\|_{L_1}\lesssim1$$
		Hence, one can use similar arguments as in step 1 and step 3.1 to obtain
		$$\|S_{i,2}^{LHH}[u,\phi,\psi]\|_\alpha\lesssim\|u\|_{1+\alpha}\|\phi\|_{\alpha}\|\psi\|_{1+\delta+\alpha}$$
		
		\textbf{Step 4}. Estimate for $S^{LLH}_{i}[u,\phi,\zeta]$. In this case,  $k_2\le k_3-1$ and one may assume $k_2\le k_3-100$ since the case that $k_3-99\le k_2\le k_3-1$ and $k_1\le k_3-1000$ can be treated in the same way as in step 3. Note that
		\begin{equation}\notag
			\begin{aligned}
				M_{s,i}(\xi,\eta,\zeta)&=|\xi||\zeta|^{1+\delta}\underbrace{\frac{i\eta_s}{|\xi|}\frac{M_i(\eta,\zeta)}{|\zeta|^{1+\delta}}}_{M^7_{s,i}(\xi,\eta,\zeta)}-|\xi||\zeta|^{1+\delta}\underbrace{\frac{i\eta_s}{|\xi|}\frac{M_i(\xi+\eta,\zeta)}{|\zeta|^{1+\delta}}}_{M^8_{s,i}(\xi,\eta,\zeta)}\\
				&+i\xi|\zeta|^{1+\delta}\underbrace{\big(-\frac{\zeta_s}{|\zeta|}|\zeta|^{-\delta}\int_{0}^{1}\nabla_{y}M_i(\eta,\sigma\xi+\zeta)d\sigma\big)}_{M_{s,i}^9(\xi,\eta,\zeta)}.
			\end{aligned}
		\end{equation} 
		For $m\ge-1$, we have the following decomposition
		\begin{equation}\notag
			\begin{aligned}
				&\Delta_{m}S_i^{LLH}[u,\phi,\psi](x)=\sum_{\substack{k_1, k_2, k_3\in \mathcal{Z}_4\\|m-k_3|\le10}}\Delta_{m}S_i[\Delta_{k_1}u,\Delta_{k_2}\phi,\Delta_{k_3}\psi](x)\\
				&=\sum_{\substack{k_1, k_2, k_3\in \mathcal{Z}_4\\|m-k_3|\le10}}\Delta_{m}\sum_{\xi,\eta,\zeta\in \mathcal{F}_0}M_{s,i}^{k_1, k_2, k_3, 7}(\xi,\eta,\zeta)\widehat{\Lambda\Delta_{k_1}u^{s}}(\xi)\widehat{\Delta_{k_2}\phi}(\eta)\widehat{\Lambda^{1+\delta}\Delta_{k_3}\psi}(\zeta)e^{i(\xi+\eta+\zeta)\cdot x}\\
				&-\sum_{\substack{k_1, k_2, k_3\in \mathcal{Z}_4\\|m-k_3|\le10\\k\ge-1}}\Delta_{m}\sum_{\xi,\eta,\zeta\in \mathcal{F}_0}M_{s,i,k}^{k_1, k_2, k_3, 8}(\xi,\eta,\zeta)\widehat{\Lambda\Delta_{k_1}u^{s}}(\xi)\widehat{\Delta_{k_2}\phi}(\eta)\widehat{\Lambda^{1+\delta}\Delta_{k_3}\psi}(\zeta)e^{i(\xi+\eta+\zeta)\cdot x}\\
				&+\sum_{\substack{k_1, k_2, k_3\in \mathcal{Z}_4\\|m-k_3|\le10}}\Delta_{m}\sum_{\xi,\eta,\zeta\in \mathcal{F}_0}M_{s,i}^{k_1, k_2, k_3, 9}(\xi,\eta,\zeta)\widehat{\nabla\Delta_{k_1}u^{s}}(\xi)\widehat{\Delta_{k_2}\phi}(\eta)\widehat{\Lambda^{1+\delta}\Delta_{k_3}\psi}(\zeta)e^{i(\xi+\eta+\zeta)\cdot x}\\
			\end{aligned}
		\end{equation}
		with
		$$M_{s,i}^{k_1, k_2, k_3, 7}(\xi,\eta,\zeta)=M_{s,i}^{7}(\xi,\eta,\zeta)\bar{\chi}_{k_1}(\xi)\bar{\chi}_{k_2}(\eta)\bar{\chi}_{k_3}(\zeta),$$
		$$M_{s,i,k}^{k_1, k_2, k_3, 8}(\xi,\eta,\zeta)=M_{s,i}^{8}(\xi,\eta,\zeta)\bar{\chi}_{k_1}(\xi)\bar{\chi}_{k_2}(\eta)\bar{\chi}_{k_3}(\zeta)\bar{\chi}_{k}(\xi+\eta),$$
		$$M_{s,i}^{k_1, k_2, k_3, 9}(\xi,\eta,\zeta)=M_{s,i}^{9}(\xi,\eta,\zeta)\bar{\chi}_{k_1}(\xi)\bar{\chi}_{k_2}(\eta)\bar{\chi}_{k_3}(\zeta).$$
		Then arguing  as before, we can use Lemma \ref{loc} to get 
		$$\|\mathcal{F}^{-1}(M^{k_1,k_2,k_3,7}_{s,i})\|_{L_1}\lesssim1,$$
		$$\|\mathcal{F}^{-1}(M^{k_1,k_2,k_3,9}_{s,i})\|_{L_1}\lesssim1.$$
		To treat $M^{k_1,k_2,k_3,8}_{s,i,k}$, one may recall that $M_i(x,y)=\frac{ix_i(\tilde{m}(x)-\tilde{m}(y))}{|x+y|}$, so that
		$$M^{8}_{s,i}(\xi,\eta,\zeta)=\underbrace{\frac{i\eta_s}{|\xi|}\frac{i(\xi+\eta)_i}{|\xi+\eta+\zeta|}\frac{\tilde{m}(\xi+\eta)}{|\zeta|^{1+\delta}}}_{M^{8}_{s,i,1}(\xi,\eta,\zeta)}-\underbrace{\frac{i\eta_s}{|\xi|}\frac{i(\xi+\eta)_i}{|\xi+\eta+\zeta|}\frac{\tilde{m}(\zeta)}{|\zeta|^{,1+\delta}}}_{M^{8}_{s,i,2}(\xi,\eta,\zeta)}$$
		\begin{equation*}
			\begin{aligned}
				M_{s,i,k,1}^{k_1, k_2, k_3, 8}(\xi,\eta,\zeta)&=M_{s,i,1}^{8}(\xi,\eta,\zeta)\bar{\chi}_{k_1}(\xi)\bar{\chi}_{k_2}(\eta)\bar{\chi}_{k_3}(\zeta)\bar{\chi}_{k}(\xi+\eta),\\
				&=M_{s,i,k,1,1}^{k_1, k_2, k_3, 8}(\xi,\eta,\zeta)\cdot M_{k,1,2}^{k_1, k_2, k_3, 8}(\xi,\eta,\zeta),
			\end{aligned}
		\end{equation*}
		$$M_{s,i,k,1,1}^{k_1, k_2, k_3, 8}(\xi,\eta,\zeta)=\frac{i\eta_s}{|\xi|}\frac{i(\xi+\eta)_i}{|\xi+\eta+\zeta|}\bar{\chi}_{k_1}(\xi)\tilde{\chi}_{k_2}(\eta)\tilde{\chi}_{k_3}(\zeta),$$
		$$M_{k,1,2}^{k_1, k_2, k_3, 8}(\xi,\eta,\zeta)=\frac{\tilde{m}(\xi+\eta)}{|\zeta|^{1+\delta}}\bar{\chi}_{k_2}(\eta)\bar{\chi}_{k_3}(\zeta)\bar{\chi}_{k}(\xi+\eta),$$
		$$\tilde{M}_{k,1,2}^{k_1, k_2, k_3, 8}(\xi,\eta,\zeta)=\frac{\tilde{m}(\xi)}{|\zeta|^{1+\delta}}\bar{\chi}_{k_2}(\eta)\bar{\chi}_{k_3}(\zeta)\bar{\chi}_{k}(\xi),$$
		$$M_{s,i,k,2}^{k_1, k_2, k_3, 8}(\xi,\eta,\zeta)=M_{s,i,2}^{8}(\xi,\eta,\zeta)\bar{\chi}_{k_1}(\xi)\bar{\chi}_{k_2}(\eta)\bar{\chi}_{k_3}(\zeta)\bar{\chi}_{k}(\xi+\eta).$$
		Then one can argue as before using Lemmas \ref{symbol} and \ref{loc} to get 
		$$\begin{aligned}
			\left\|\mathcal{F}^{-1}\left(M_{s, i,k,1}^{k_1, k_2, k_3, 8}\right)\right\|_{L_1} &\lesssim\left\|\mathcal{F}^{-1}\left(M_{s, i,k,1,1}^{k_1, k_2, k_3, 8}\right)\right\|_{L_1}\left\|\mathcal{F}^{-1}\left(M_{k,1,2}^{k_1, k_2, k_3, 8}\right)\right\|_{L_1}  \\
			&=\left\|\mathcal{F}^{-1}\left(M_{s, i,k,1,1}^{k_1, k_2, k_3, 8}\right)\right\|_{L_1}\left\|\mathcal{F}^{-1}\left(\tilde{M}_{k,1,2}^{k_1, k_2, k_3, 8}\right)\right\|_{L_1}  \\
			&\lesssim1,\\
			\left\|\mathcal{F}^{-1}\left(M_{s, i,k,2}^{k_1, k_2, k_3, 8}\right)\right\|_{L_1} &\lesssim 1,\\
			\left\|\mathcal{F}^{-1}\left(M_{s, i,k}^{k_1, k_2, k_3, 8}\right)\right\|_{L_1} 
			&\lesssim \left\|\mathcal{F}^{-1}\left(M_{s, i,k,1}^{k_1, k_2, k_3, 8}\right)\right\|_{L_1}+\left\|\mathcal{F}^{-1}\left(M_{s, i,k,2}^{k_1, k_2, k_3, 8}\right)\right\|_{L_1}\\
			&\lesssim1.
		\end{aligned}$$
		Combing these and arguing as before lead to
		$$\left\|S_{i}^{L L H}[u, \phi, \psi]\right\|_\alpha \lesssim\|u\|_{1+\alpha}\|\phi\|_\alpha\|\psi\|_{1+\delta+\alpha}$$
		\textbf{Step 5}. Estimate for $S^{LHL}_{i}[u,\phi,\zeta]$. This case can be treated in the same way as in step 4 to get
		$$\left\|S_{i}^{L H L}[u, \phi, \psi]\right\|_\alpha \lesssim\|u\|_{1+\alpha}\|\phi\|_{1+\delta+\alpha}\|\psi\|_{\alpha}$$
		\textbf{Step 6}. Estimate for $S^{HHH}_{i}[u,\phi,\zeta]$.
		In this case we use the following decompositions
		\begin{equation}\notag
			\begin{aligned}
				M_{s.i}(\xi,\eta,\zeta)&=|\xi|^{2+\delta}\frac{i\zeta_s}{|\xi|}\frac{M_i(\eta,\zeta)-M_i(\eta,\xi+\zeta)}{|\xi|^{1+\delta}}+|\xi|^{2+\delta}\frac{i\eta_s}{|\xi|}\frac{M_i(\eta,\zeta)-M_i(\xi+\eta,\zeta)}{|\xi|^{1+\delta}}\\
				&=|\xi|^{2+\delta}\underbrace{\frac{i(\eta+\zeta)_s}{|\xi|}\frac{M_i(\eta,\zeta)}{|\xi|^{1+\delta}}}_{M_{s,i}^{10}(\xi,\eta,\zeta)}-|\xi|^{2+\delta}\underbrace{\frac{i\eta_s}{|\xi|}\frac{M_i(\xi+\eta,\zeta)}{|\xi|^{1+\delta}}}_{M_{s,i}^{11}(\xi,\eta,\zeta)}-|\xi|^{2+\delta}\underbrace{\frac{i\zeta_s}{|\xi|}\frac{M_i(\eta,\xi+\zeta)}{|\xi|^{1+\delta}}}_{M_{s,i}^{12}(\xi,\eta,\zeta)},
			\end{aligned}
		\end{equation}
		$$M_{s,i,k}^{k_1,k_2,k_3,10}(\xi,\eta,\zeta)=M_{s,i}^{10}(\xi,\eta,\zeta)\bar{\chi}_{k_1}(\xi)\bar{\chi}_{k_2}(\eta)\bar{\chi}_{k_3}(\zeta)\bar{\chi}_{k}(\eta+\zeta)$$
		$$M_{s,i,k,\ell}^{k_1,k_2,k_3,11}(\xi,\eta,\zeta)=M_{s,i}^{11}(\xi,\eta,\zeta)\bar{\chi}_{k_1}(\xi)\bar{\chi}_{k_2}(\eta)\bar{\chi}_{k_3}(\zeta)\bar{\chi}_{k}(\xi+\eta)\bar{\chi}_{\ell}(\xi+\eta+\zeta)$$
		$$M_{s,i,k,\ell}^{k_1,k_2,k_3,12}(\xi,\eta,\zeta)=M_{s,i}^{12}(\xi,\eta,\zeta)\bar{\chi}_{k_1}(\xi)\bar{\chi}_{k_2}(\eta)\bar{\chi}_{k_3}(\zeta)\bar{\chi}_{k}(\xi+\zeta)\bar{\chi}_{\ell}(\xi+\eta+\zeta)$$
		and for $m\ge-1$, 
		\begin{equation}\notag
			\begin{aligned}
				&\Delta_{m}S_i^{HHH}[u,\phi,\psi](x)=\sum_{\substack{k_1, k_2, k_3\in \mathcal{Z}_6\\k_1\ge m-1100}}\Delta_{m}S_i[\Delta_{k_1}u,\Delta_{k_2}\phi,\Delta_{k_3}\psi](x)\\
				&=\underbrace{\sum_{\substack{k_1, k_2, k_3\in \mathcal{Z}_6\\k_1\ge m-1100\\k\ge-1}}\Delta_{m}\sum_{\xi,\eta,\zeta\in \mathcal{F}_0}M_{s,i,k}^{k_1, k_2, k_3, 10}(\xi,\eta,\zeta)\widehat{\Lambda^{2+\delta}\Delta_{k_1}u^{s}}(\xi)\widehat{\Delta_{k_2}\phi}(\eta)\widehat{\Delta_{k_3}\psi}(\zeta)e^{i(\xi+\eta+\zeta)\cdot x}}_{\Delta_{m}S_i^{HHH,1}[u,\phi,\psi](x)}\\
				&\underbrace{-\sum_{\substack{k_1, k_2, k_3\in \mathcal{Z}_6\\k_1\ge m-1100\\k,\ell\ge-1}}\Delta_{m}\sum_{\xi,\eta,\zeta\in \mathcal{F}_0}M_{s,i,k,\ell}^{k_1, k_2, k_3, 11}(\xi,\eta,\zeta)\widehat{\Lambda^{2+\delta}\Delta_{k_1}u^{s}}(\xi)\widehat{\Delta_{k_2}\phi}(\eta)\widehat{\Delta_{k_3}\psi}(\zeta)e^{i(\xi+\eta+\zeta)\cdot x}}_{\Delta_{m}S_i^{HHH,2}[u,\phi,\psi](x)}\\
				&\underbrace{-\sum_{\substack{k_1, k_2, k_3\in \mathcal{Z}_6\\k_1\ge m-1100\\k,\ell\ge-1}}\Delta_{m}\sum_{\xi,\eta,\zeta\in \mathcal{F}_0}M_{s,i,k,\ell}^{k_1, k_2, k_3, 12}(\xi,\eta,\zeta)\widehat{\Lambda^{2+\delta}\Delta_{k_1}u^{s}}(\xi)\widehat{\Delta_{k_2}\phi}(\eta)\widehat{\Delta_{k_3}\psi}(\zeta)e^{i(\xi+\eta+\zeta)\cdot x}}_{\Delta_{m}S_i^{HHH,3}[u,\phi,\psi](x)}\\
			\end{aligned}
		\end{equation}
		\textbf{Step 6.1}. Estimate for $\Delta_{m}S_i^{HHH,1}[u,\phi,\psi](x)$. We decompose further
		\begin{equation*}
			\begin{aligned}
				&\Delta_{m}S_i^{HHH,1}[u,\phi,\psi](x)\\
				&=\sum_{\substack{k_1, k_2, k_3\in \mathcal{Z}_6\\k_1\ge m-1100\\k\ge\max\left\{k_2,k_3\right\}-499}}\Delta_{m}\sum_{\xi,\eta,\zeta\in \mathcal{F}_0}M_{s,i,k}^{k_1, k_2, k_3, 10}(\xi,\eta,\zeta)\widehat{\Lambda^{2+\delta}\Delta_{k_1}u^{s}}(\xi)\widehat{\Delta_{k_2}\phi}(\eta)\widehat{\Delta_{k_3}\psi}(\zeta)e^{i(\xi+\eta+\zeta)\cdot x}\\
				&+\sum_{\substack{k_1, k_2, k_3\in \mathcal{Z}_6\\k_1\ge m-1100\\k\le k_2-500}}\Delta_{m}\sum_{\xi,\eta,\zeta\in \mathcal{F}_0}M_{s,i,k}^{k_1, k_2, k_3, 10}(\xi,\eta,\zeta)\widehat{\Lambda^{2+\delta}\Delta_{k_1}u^{s}}(\xi)\widehat{\Delta_{k_2}\phi}(\eta)\widehat{\Delta_{k_3}\psi}(\zeta)e^{i(\xi+\eta+\zeta)\cdot x}\\
				&+\sum_{\substack{k_1, k_2, k_3\in \mathcal{Z}_6\\k_1\ge m-1100\\k\le k_3-500}}\Delta_{m}\sum_{\xi,\eta,\zeta\in \mathcal{F}_0}M_{s,i,k}^{k_1, k_2, k_3, 10}(\xi,\eta,\zeta)\widehat{\Lambda^{2+\delta}\Delta_{k_1}u^{s}}(\xi)\widehat{\Delta_{k_2}\phi}(\eta)\widehat{\Delta_{k_3}\psi}(\zeta)e^{i(\xi+\eta+\zeta)\cdot x}
			\end{aligned}
		\end{equation*}
		Consider the first case when $k\ge\max\left\{k_2,k_3\right\}-499$. One can argue easily as before to get 
		$$\|\mathcal{F}^{-1}\left(M_{s,i,k}^{k_1, k_2, k_3, 10}\right)\|_{L_1}\lesssim1.$$
		While for the second case when $k\le k_2-500$, the condition  $|k_2-k_3|\le10$ holds. One can decompose further,
		$$M_i(\eta,\zeta)=\frac{i\eta_i(\tilde{m}(\eta)-\tilde{m}(\zeta))}{|\zeta+\eta|}=-i\eta_i\frac{\eta+\zeta}{|\eta+\zeta|}\int_{0}^{1}\nabla\tilde{m}\left(\sigma(\eta+\zeta)-\eta\right)d\sigma,$$ 
		\begin{equation*}
			\begin{aligned}
				M_{s,i,k}^{k_1,k_2,k_3,10}(\xi,\eta,\zeta)&=\frac{i(\eta+\zeta)_s}{|\xi|} \frac{M_i(\eta, \zeta)}{|\xi|^{1+\delta}}\bar{\chi}_{k_1}(\xi)\bar{\chi}_{k_2}(\eta)\bar{\chi}_{k_3}(\zeta)\bar{\chi}_{k}(\eta+\zeta),\\
				&=M_{k}^{k_2,10}(\eta,\zeta)\cdot \bar{M}_{s,i,k}^{k_1,k_2,k_3,10}(\xi,\eta,\zeta)
			\end{aligned}
		\end{equation*}
	where
		$$M_{k}^{k_2,10}(\eta,\zeta)=\frac{\eta+\zeta}{|\eta+\zeta|}\tilde{\chi}_{k_2}(\eta)\tilde{\chi}_{k}(\eta+\zeta),$$
		$$\tilde{M}_{k}^{k_2,10}(\eta,\zeta)=\frac{\zeta}{|\zeta|}\tilde{\chi}_{k_2}(\eta)\tilde{\chi}_{k}(\zeta),$$
		$$\bar{M}_{s,i,k}^{k_1,k_2,k_3,10}(\xi,\eta,\zeta)=\frac{(\eta+\zeta)_s}{|\xi|} \frac{\eta_i\int_{0}^{1}\nabla\tilde{m}\left(\sigma(\eta+\zeta)-\eta\right)d\sigma}{|\xi|^{1+\delta}}\bar{\chi}_{k_1}(\xi)\bar{\chi}_{k_2}(\eta)\bar{\chi}_{k_3}(\zeta)\bar{\chi}_{k}(\eta+\zeta),$$
		$$\tilde{M}_{s,i,k}^{k_1,k_2,k_3,10}(\xi,\eta,\zeta)=\frac{\zeta_s}{|\xi|} \frac{\eta_i\int_{0}^{1}\nabla\tilde{m}\left(\sigma\zeta-\eta\right)d\sigma}{|\xi|^{1+\delta}}\bar{\chi}_{k_1}(\xi)\bar{\chi}_{k_2}(\eta)\bar{\chi}_{k_3}(\zeta-\eta)\bar{\chi}_{k}(\zeta).$$
		Then one  can argue as before using Lemmas \ref{symbol} and \ref{loc} to get
		\begin{equation*}
			\begin{aligned}
				\left\|\mathcal{F}^{-1}\left(M_{s, i, k}^{k_1, k_2, k_3, 10}\right)\right\|_{L_1} &\lesssim \left\|\mathcal{F}^{-1}\left(M_{k}^{ k_2, 10}\right)\right\|_{L_1}\left\|\mathcal{F}^{-1}\left(\bar{M}_{s, i, k}^{k_1, k_2, k_3, 10}\right)\right\|_{L_1}\\
				&\lesssim \left\|\mathcal{F}^{-1}\left(\tilde{M}_{k}^{ k_2, 10}\right)\right\|_{L_1}\left\|\mathcal{F}^{-1}\left(\tilde{M}_{s, i, k}^{k_1, k_2, k_3, 10}\right)\right\|_{L_1}\\
				&\lesssim1.
			\end{aligned}
		\end{equation*}
		At last, the third case when $k\le k_3-500$ is similar to the second case and one can also get
		$$\left\|\mathcal{F}^{-1}\left(M_{s, i, k}^{k_1, k_2, k_3, 10}\right)\right\|_{L_1} \lesssim 1$$
		Hence, combining three cases above leads to
		$$\|\Delta_m S_i^{H H H, 1}[u, \phi, \psi]\|_0\lesssim 2^{-m\alpha}\|u\|_{2+\delta+\alpha}\|\phi\|_{\alpha}\|\psi\|_{\alpha}$$
		\textbf{Step 6.2}. Estimate for $\Delta_m S_i^{H H H,2}[u, \phi, \psi](x)$. We will use the following decomposition:
		\begin{equation*}
			\begin{aligned}
				&\Delta_m S_i^{H H H,2}[u, \phi, \psi](x)\\
				&=\sum_{\substack{k_1, k_2, k_3 \in \mathcal{Z}_6 \\ k_1 \geqslant m-1100 \\ k, \ell \geqslant-1}} \Delta_m \sum_{\xi, \eta, \zeta \in \mathcal{F}_0} M_{s, i, k, \ell}^{k_1, k_2, k_3, 11}(\xi, \eta, \zeta) \widehat{\Lambda^{2+\delta } \Delta_{k_1}u^s}(\xi) \widehat{\Delta_{k_2} \phi}(\eta) \widehat{\Delta_{k_3} \psi}(\zeta) e^{i(\xi+\eta+\zeta) \cdot x}\\
				&=\sum_{\substack{k_1, k_2, k_3 \in \mathcal{Z}_6 \\ k_1 \geqslant m-1100 \\k,\ell\ge-1\\-1\le\ell\le k-500}} \Delta_m \sum_{\xi, \eta, \zeta \in \mathcal{F}_0} M_{s, i,k, \ell}^{k_1, k_2, k_3, 11}(\xi, \eta, \zeta) \widehat{\Lambda^{2+\delta } \Delta_{k_1}u^s}(\xi) \widehat{\Delta_{k_2} \phi}(\eta) \widehat{\Delta_{k_3} \psi}(\zeta) e^{i(\xi+\eta+\zeta) \cdot x}\\
				&+\sum_{\substack{k_1, k_2, k_3 \in \mathcal{Z}_6 \\ k_1 \geqslant m-1100 \\ k, \ell \geqslant-1\\\ell\ge k-499}} \Delta_m \sum_{\xi, \eta, \zeta \in \mathcal{F}_0} M_{s, i, k, \ell}^{k_1, k_2, k_3, 11}(\xi, \eta, \zeta) \widehat{\Lambda^{2+\delta } \Delta_{k_1}u^s}(\xi) \widehat{\Delta_{k_2} \phi}(\eta) \widehat{\Delta_{k_3} \psi}(\zeta) e^{i(\xi+\eta+\zeta) \cdot x}
			\end{aligned}.
		\end{equation*}
		If $\ell\le k-500$, one has 
		$$M_i(\xi+\eta,\zeta)=\frac{i(\xi+\eta)_i\left(\tilde{m}(\xi+\eta)-\tilde{m}(\zeta)\right)}{|\xi+\eta+\zeta|}=-i(\xi+\eta)_i\frac{\xi+\eta+\zeta}{|\xi+\eta+\zeta|}\cdot\int_{0}^{1}\nabla\tilde{m}\left(\sigma(\xi+\eta+\zeta)-\xi-\eta\right)d\sigma,$$
		\begin{equation*}
			\begin{aligned}
				M_{s, i,k,  \ell}^{k_1, k_2, k_3, 11}(\xi, \eta, \zeta)&=M_{s, i}^{11}(\xi, \eta, \zeta) \bar{\chi}_{k_1}(\xi) \bar{\chi}_{k_2}(\eta) \bar{\chi}_{k_3}(\zeta) \bar{\chi}_k(\xi+\eta) \bar{\chi}_{\ell}(\xi+\eta+\zeta)\\
				&=\frac{i \eta_s}{|\xi|} \frac{M_i(\xi+\eta, \zeta)}{|\xi|^{1+\delta}} \bar{\chi}_{k_1}(\xi) \bar{\chi}_{k_2}(\eta) \bar{\chi}_{k_3}(\zeta) \bar{\chi}_k(\xi+\eta) \bar{\chi}_{\ell}(\xi+\eta+\zeta)\\
				&=M_{s}^{k_1, k_2, k_3, 11}(\xi, \eta, \zeta)\cdot M_{ i}^{k_1, k, \ell, 11}(\xi, \eta, \zeta),
			\end{aligned}
		\end{equation*}
		$$M_{s}^{k_1, k_2, k_3, 11}(\xi, \eta, \zeta)=\frac{\eta_s}{|\xi|}\tilde{\chi}_{k_1}(\xi)\bar{\chi}_{k_2}(\eta)\bar{\chi}_{k_3}(\zeta),$$
		$$M_{  i}^{k_1, k,\ell, 11}(\xi, \eta, \zeta)=\frac{\xi+\eta+\zeta}{|\xi+\eta+\zeta|}\cdot\frac{(\xi+\eta)_i\int_{0}^{1}\nabla\tilde{m}\left(\sigma(\xi+\eta+\zeta)-\xi-\eta\right)d\sigma}{|\xi|^{1+\delta}}\bar{\chi}_{k_1}(\xi)\bar{\chi}_{k}(\xi+\eta)\bar{\chi}_{\ell}(\xi+\eta+\zeta),$$
		$$\tilde{M}_{i}^{k_1, k,\ell, 11}(\xi, \eta, \zeta)=\frac{\zeta}{|\zeta|}\cdot\frac{\eta_i\int_{0}^{1}\nabla\tilde{m}\left(\sigma \zeta-\eta\right)d\sigma}{|\xi|^{1+\delta}}\bar{\chi}_{k_1}(\xi)\bar{\chi}_{k}(\eta)\bar{\chi}_{\ell}(\zeta).$$
		Then one can argue as before using Lemmas \ref{symbol} and \ref{loc} to get 
		$$\begin{aligned}
			\left\|\mathcal{F}^{-1}\left(M_{s, i,k, \ell}^{k_1, k_2, k_3, 11}\right)\right\|_{L_1} & \lesssim\left\|\mathcal{F}^{-1}\left(M_{s}^{k_1, k_2, k_3, 11}\right)\right\|_{L_1}\left\|\mathcal{F}^{-1}\left(M_{i}^{k_1, k,\ell, 11}\right)\right\|_{L_1} \\
			& \lesssim\left\|\mathcal{F}^{-1}\left(M_{s}^{k_1, k_2, k_3, 11}\right)\right\|_{L_1}\left\|\mathcal{F}^{-1}\left(\tilde{M}_{i}^{k_1, k,\ell, 11}\right)\right\|_{L_1} \\
			& \lesssim 1.
		\end{aligned}$$
		If $\ell\ge k-499$, one has
		\begin{equation*}
			\begin{aligned}
				M_{s, i, k, \ell}^{k_1, k_2, k_3, 11}(\xi, \eta, \zeta)&=M_{s, i}^{11}(\xi, \eta, \zeta) \bar{\chi}_{k_1}(\xi) \bar{\chi}_{k_2}(\eta) \bar{\chi}_{k_3}(\zeta)\bar{\chi}_{k}(\xi+\eta)  \bar{\chi}_{\ell}(\xi+\eta+\zeta)\\
				&=\frac{i \eta_s}{|\xi|} \frac{M_i(\xi+\eta, \zeta)}{|\xi|^{1+\delta}}\bar{\chi}_{k_1}(\xi) \bar{\chi}_{k_2}(\eta) \bar{\chi}_{k_3}(\zeta)\bar{\chi}_{k}(\xi+\eta)  \bar{\chi}_{\ell}(\xi+\eta+\zeta)\\
				&=-\frac{\eta_s(\xi+\eta)_i}{|\xi|^2}\frac{\tilde{m}(\xi+\eta)-\tilde{m}(\zeta)}{|\xi+\eta+\zeta||\xi|^{\delta}}\bar{\chi}_{k_1}(\xi) \bar{\chi}_{k_2}(\eta) \bar{\chi}_{k_3}(\zeta)\bar{\chi}_{k}(\xi+\eta)  \bar{\chi}_{\ell}(\xi+\eta+\zeta)\\
				&=M_{ i,k,\ell}^{k_1, 11}(\xi, \eta, \zeta)\cdot M_{s, \ell}^{k_1,k_2, k_3, 11}(\xi, \eta, \zeta),
			\end{aligned}
		\end{equation*}
		$$M_{ i,k,\ell}^{k_1, 11}(\xi, \eta, \zeta)=-\frac{(\xi+\eta)_i}{|\xi+\eta+\zeta|} \tilde{\chi}_{k_1}(\xi) \tilde{\chi}_{k}(\xi+\eta) \bar{\chi}_{\ell}(\xi+\zeta+\zeta),$$
		$$\tilde{M}_{ i,k,\ell}^{k_1, 11}(\xi, \eta, \zeta)=-\frac{\eta_i}{|\zeta|} \tilde{\chi}_{k_1}(\xi) \tilde{\chi}_{k}(\eta) \bar{\chi}_{\ell}(\zeta),$$
		$$M_{s,k}^{k_1,k_2, k_3, 11}(\xi, \eta, \zeta)=\frac{\left(\tilde{m}(\xi+\eta)-\tilde{m}(\zeta)\right)\eta_s}{|\xi|^{2+\delta}}\bar{\chi}_{k_1}(\xi) \bar{\chi}_{k_3}(\zeta)\bar{\chi}_{k}(\xi+\eta)\bar{\chi}_{k}(\eta),$$
		$$\tilde{M}_{s,k}^{k_1,k_2, k_3, 11}(\xi, \eta, \zeta)=\frac{\left(\tilde{m}(\xi)-\tilde{m}(\zeta)\right)\eta_s}{|\xi-\eta|^{2+\delta}}\bar{\chi}_{k_1}(\xi-\eta) \bar{\chi}_{k}(\xi)  \bar{\chi}_{k_3}(\zeta)\bar{\chi}_k(\eta).$$
		Again, one can argue as before using Lemmas \ref{symbol} and \ref{loc} to get 
		$$\begin{aligned}
			\left\|\mathcal{F}^{-1}\left(M_{s, i,k, \ell}^{k_1, k_2, k_3, 11}\right)\right\|_{L_1} & \lesssim\left\|\mathcal{F}^{-1}\left(M_{i,k,\ell}^{k_1, 11}\right)\right\|_{L_1}\left\|\mathcal{F}^{-1}\left(M_{s,k}^{k_1,k_2, k_3, 11}\right)\right\|_{L_1} \\
			& \lesssim\left\|\mathcal{F}^{-1}\left(\tilde{M}_{ i,k,\ell}^{k_1,  11}\right)\right\|_{L_1}\left\|\mathcal{F}^{-1}\left(\tilde{M}_{ s,k}^{k_1,k_2, k_3, 11}\right)\right\|_{L_1} \\
			& \lesssim 1.
		\end{aligned}$$
		 Combining  two cases above yields
		$$\left\|\Delta_m S_i^{H H H, 2}[u, \phi, \psi]\right\|_0 \lesssim 2^{-m \alpha}\|u\|_{2+\delta+\alpha}\|\phi\|_\alpha\|\psi\|_\alpha$$
		\textbf { Step 6.3}. Estimate for  $\Delta_m S_i^{H H H, 3}[u, \phi, \psi](x)$. This case can be treated in the same way  as in step 6.2 to get 
		$$\left\|\Delta_m S_i^{H H H, 2}[u, \phi, \psi]\right\|_0 \lesssim 2^{-m \alpha}\|u\|_{2+\delta+\alpha}\|\phi\|_\alpha\|\psi\|_\alpha$$
		Then it follows from steps 6.1-6.3 that
		$$\left\|S_i^{H H H}[u, \phi, \psi]\right\|_\alpha \lesssim \|u\|_{2+\delta+\alpha}\|\phi\|_\alpha\|\psi\|_\alpha$$
		Finally, step 1-step 6 imply that
		$$\left\| S_i[u, \phi, \psi]\right\|_\alpha \lesssim \|u\|_{1+\alpha}\|\phi\|_{1+\delta+\alpha}\|\psi\|_\alpha+\|u\|_{1+\alpha}\|\phi\|_\alpha\|\psi\|_{1+\delta+\alpha} +\|u\|_{2+\delta+\alpha}\|\phi\|_\alpha\|\psi\|_\alpha$$
	\end{proof}

	\section{The Newton steps}
	\label{sec-Newton}
	\subsection{Spatial mollifications of the velocity filed and the stress error}\label{sec-spa} As  in standard convex integration schemes, the construction starts with a mollification of the ASM-Reynolds system, whose purpose is to deal with the problem of loss of derivatives. The setup of the mollification function $\zeta$ in this section is the same as  in \cite{DGR24}.
	
	Define the mollification length scale as
	\begin{equation*}
		\ell_q = (\lambda_q \lambda_{q+1})^{-1/2},
	\end{equation*}
	and let $\zeta:\mathbb R^2 \rightarrow \mathbb R$ be a smooth function such that its Fourier transform 
	\begin{equation*}
		\hat \zeta(\xi) = \int_{\mathbb R^2} \zeta(x) e^{-ix\cdot \xi} dx
	\end{equation*}
	satisfies $\hat \zeta(\xi) = 1$ when $|\xi| \leq 1$ and $\hat \zeta (\xi) = 0$ when $|\xi| \geq 2$.  Assume  further that $\hat \zeta$ (and, thus, $\zeta$) is symmetric (this implies that $\zeta$ has first order vanishing moment). Given a periodic function $f:\mathbb T^2 \rightarrow \mathbb R$, set
	\begin{eqnarray*}
		P_{\lesssim \ell_q^{-1}} f(x) = \sum_{k \in \mathbb Z^2} \hat f(k) \hat \zeta (k \ell_q) e^{i k \cdot x}.
	\end{eqnarray*}
	Equivalently, 
	\begin{equation*}
		P_{\lesssim \ell_q^{-1}} f(x) = \int_{\mathbb R^2} f(x-y) \zeta_{\ell_q}(y) dy,
	\end{equation*}
	where $f$ is identified with its periodic extension, and 
	\begin{equation*}
		\zeta_{\ell_q}(x) = \ell_q^{-2} \zeta(\ell_q^{-1}x).
	\end{equation*}
	Set
	\begin{equation*}
		\bar v_q = P_{\lesssim \ell_q^{-1}} v_q,
	\end{equation*}
	\begin{equation*}
		\bar u_q= T[\bar v_q] = P_{\lesssim \ell_q^{-1}} u_q , 
	\end{equation*}
	\begin{equation*}
		R_{q, 0} = P_{\lesssim \ell_q^{-1}} R_q .
	\end{equation*} 
	We first collect some relevant estimates in the following lemma. We also fix a mollification multiplier \(\tilde{P}_{\lesssim \ell_q^{-1}}\), defined similarly as \(P_{\lesssim \ell_q^{-1}}\), such that \(\tilde{P}_{\lesssim \ell_q^{-1}} P_{\lesssim \ell_q^{-1}} f = P_{\lesssim \ell_q^{-1}} f\).
	 For simplicity of presentation,  denote
	\begin{equation*}
		\bar D_t = \partial_t + \bar u_q \cdot \nabla
	\end{equation*}
as the material derivative associated to the mollified velocity field $\bar u_q$. Note that $\zeta$ satisfies all of the requirements in Proposition \ref{prop-molli}. Then  the following mollification estimates hold:
	
	\begin{lem} \label{smoli_estim}
		For $N\ge0, r\in\left\{0,1\right\}$,  the following estimates 
		\begin{equation} \label{smoli_1}
			\|\bar v_q\|_N \lesssim \delta_q^{\frac12}\MMSS, \quad N\geq1
		\end{equation}
		\begin{equation} \label{smoli_1'}
			\|\bar u_q\|_N \lesssim \delta_q^{\frac12} \lambda_q^{1+\delta}\MMSS, \quad N\geq1
		\end{equation}
		\begin{equation} \label{smoli_1''}
			\|\bar D_t\left( \nabla^{\perp}\cdot\bar v_q\right)\|_{N} \lesssim \delta_q\lambda_{q}^{3+\delta} \mathcal{M}\left(N, L_R-2, \lambda_q, \ell_q^{-1}\right), \quad N\geq0
		\end{equation}
		\begin{equation} \label{smoli_1'''}
			\|\bar D_t\left( \nabla^{\perp}\bar v_q\right)\|_{N}+\|\bar D_t\left( \nabla\bar v_q\right)\|_{N} \lesssim \delta_q\lambda_{q}^{3+\delta+2\alpha} \mathcal{M}\left(N, L_R-3, \lambda_q, \ell_q^{-1}\right), \quad N\geq0
		\end{equation}
		\begin{equation} \label{smoli_1''''}
			\|\bar D_t\left( \nabla^{\perp}\bar u_q\right)\|_{N}+\|\bar D_t\left( \nabla\bar u_q\right)\|_{N} \lesssim \delta_q\lambda_{q}^{4+2\delta+2\alpha} \mathcal{M}\left(N, L_R-4, \lambda_q, \ell_q^{-1}\right), \quad N\geq0
		\end{equation}
		\begin{equation} \label{smoli_2}
			\|\bar D_t^rR_{q,0}\|_N \lesssim \delta_{q+1}\lambda_{q+1}^{1+\delta} \lambda_q^{-2\alpha}\left(\de_q^{1/2}\lambda_q^{2+\delta}\right)^r\MMR.
		\end{equation}
		hold with implicit constants depending on $M$ and $N$ . 
	\end{lem}
	\begin{proof}
		The estimates for \eqref{smoli_1}, \eqref{smoli_1'} and \eqref{smoli_2} follow from the same arguments in [\cite{GR23},lemma 3.1].  Now we  prove \eqref{smoli_1''}.
		
		Apply $-\nabla^{\perp}\cdot$ to both sides of  \eqref{ASMR-q} to get 
		\begin{equation}\label{ASE-q'}
			\partial_t \theta_q + u_q\cdot\nabla\theta_q=-\nabla^{\perp}\cdot\div R_q
		\end{equation}
		where  $\theta_q=-\nabla^{\perp}\cdot v_q$	and $\bar{\theta}_q=-\nabla^{\perp}\cdot \bar{v}_q$. Thus it holds that 
		$$\bar{D}_t\bar\theta_q+\div \left(P_{\lesssim \ell_{q}^{-1}}\left(u_q\theta_q\right)-\bar{u}_q\bar{\theta}_q\right)=-\nabla^{\perp}\cdot\div R_{q,0}$$
		Due to this, Proposition  \ref{CET_comm}, the  inductive assumption \eqref{induct-u} and \eqref{induct-u'}, and \eqref{smoli_2} for the case $N\le L_R-2$, one can get
		\begin{equation*}
			\begin{aligned}
				\|\bar{D}_t\bar\theta_q\|_{N}&\lesssim\|P_{\lesssim \ell_{q}^{-1}}\left(u_q\theta_q\right)-\bar{u}_q\bar{\theta}_q\|_{N+1}+\|R_{q,0}\|_{N+2}\\
				&\lesssim \ell_{q}(\|u_q\|_{N+1}\|\theta_q\|_1+\|\theta_q\|_{N+1}\|u_q\|_1)+\delta_{q+1} \lambda_{q+1}^{1+\delta} \lambda_q^{-2 \alpha} \mathcal{M}\left(N+2, L_{R}, \lambda_q, \ell_q^{-1}\right)\\
				&\lesssim \delta_{q}\ell_{q} \lambda_q^{4+\delta+N} +\delta_{q+1} \lambda_{q+1}^{1+\delta} \lambda_q^{2+N-2 \alpha} \\
				&\lesssim \delta_{q}  \lambda_q^{3+\delta+N},
			\end{aligned}
		\end{equation*}
	where one has taken $\beta$  such that  $\beta\ge \frac{1+\delta}{2}$, which ensures $\delta_{q+1}\la_{q+1}^{1+\delta}\le\delta_q\lambda_{q}^{1+\delta}$. If $N\ge L_R -1$, one has
		\begin{equation*}
			\begin{aligned}
				\|\bar{D}_t\bar\theta_q\|_{N}&\lesssim\|P_{\lesssim \ell_{q}^{-1}}\left(u_q\theta_q\right)-\bar{u}_q\bar{\theta}_q\|_{N+1}+\|R_{q,0}\|_{N+2}\\
				&\lesssim \ell_{q}^{-N+L_R-2}\left(\|u_q\|_{L_R-2}\|\theta_q\|_1+\|\theta_q\|_{L_R -2}\|u_q\|_1\right)+\delta_{q+1} \lambda_{q+1}^{1+\delta} \lambda_q^{L_R -2 \alpha}\ell_{q}^{-N+L_R -2}\\
				&\lesssim \delta_{q}\lambda_{q}^{1+\delta}  \lambda_q^{L_R}\ell_{q}^{-N+L_R-2} +\delta_{q+1} \lambda_{q+1}^{1+\delta} \lambda_q^{L_R -2 \alpha}\ell_{q}^{-N+L_R -2} \\
				&\lesssim \delta_{q}\lambda_{q}^{1+\delta}  \lambda_q^{L_R}\ell_{q}^{-N+L_R-2}
			\end{aligned}
		\end{equation*}
		Also, to prove \eqref{smoli_1'''} and \eqref{smoli_1''''}, noting that $\bar{v}_q=-\Delta^{-1}\nabla^{\perp}\bar{\theta}_q$,  we have that
		\begin{equation}\notag
			\begin{aligned}
				&\bar D_t\left( \nabla^{\perp}\bar v_q\right)=-\bar D_t\left( \nabla^{\perp}\Delta^{-1}\nabla^{\perp}\bar \theta_q\right)=-\nabla^{\perp}\Delta^{-1}\nabla^{\perp}\bar D_t \bar \theta_q-[\bar{u}_q\cdot\nabla,\nabla^{\perp}\Delta^{-1}\nabla^{\perp}]\bar{\theta}_q,\\
				&\bar D_t\left( \nabla\bar v_q\right)=-\bar D_t\left( \nabla\Delta^{-1}\nabla^{\perp}\bar \theta_q\right)=-\nabla\Delta^{-1}\nabla^{\perp}\bar D_t \bar \theta_q-[\bar{u}_q\cdot\nabla,\nabla\Delta^{-1}\nabla^{\perp}]\bar{\theta}_q,\\
				&\bar D_t\left( \nabla^{\perp}\bar u_q\right)=-\bar D_t\left( T_1\nabla^{\perp}\Delta^{-1}\nabla^{\perp}\bar \theta_q\right)=-T_1\nabla^{\perp}\Delta^{-1}\nabla^{\perp}\bar D_t \bar \theta_q-[\bar{u}_q\cdot\nabla,T_1\nabla^{\perp}\Delta^{-1}\nabla^{\perp}]\bar{\theta}_q,\\
				&\bar D_t\left( \nabla\bar u_q\right)=-\bar D_t\left( T_1\nabla\Delta^{-1}\nabla^{\perp}\bar \theta_q\right)=-T_1\nabla\Delta^{-1}\nabla^{\perp}\bar D_t \bar \theta_q-[\bar{u}_q\cdot\nabla,T_1\nabla\Delta^{-1}\nabla^{\perp}]\bar{\theta}_q.
			\end{aligned}
		\end{equation}
		Then \eqref{smoli_1'''} and \eqref{smoli_1''''} follow from \eqref{smoli_1}, \eqref{smoli_1'}, \eqref{smoli_1''} and Proposition \ref{prop-commu}.
	\end{proof}

	\subsection{Set up of the Newton step }\label{sec-prep}
	To set up the construction of the iterative Newton perturbations, we need some standard toolbox, which includes flow map estimates and  several  time-dependent functions. This subsection is similar to that of \cite{GR23}.
	
		First define backwards flow $\Phi_t : \T^2 \times \R \to \T$ starting at time $t \in \mathbb R$ as 
	\begin{equation} \label{Flow_t}
		\begin{cases}
			\partial_s \Phi_t(x,s) + \bar u_q(x,s) \cdot \nabla \Phi_t(x,s) = 0 \\ 
			\Phi_t \big|_{s = t}(x) = x,
		\end{cases}
	\end{equation}
	and the corresponding Lagrangian flow $X_t$ by 
	\begin{equation} \label{Lagr_t}
		\begin{cases}
			\frac{d}{ds}X_t(\alpha, s) = \bar u_q(X_t(\alpha, s), s) \\ 
			X_t(\alpha, t) = \alpha.
		\end{cases}
	\end{equation}
	Then  the following standard estimates hold.
	
	\begin{lem} \label{Flow_estim}\cite{GR23}
		Let $N\ge0, r\in\left\{0,1\right\}$, $\tau \leq  \|\bar u_q\|_1^{-1}$. Then it holds for any $|s - t| < \tau$ that
		\begin{equation} \label{Flow_estim_1}
			\|\bar D_t^r\nabla \Phi_t (\cdot, s)\|_N \lesssim \left(\delta_q^{1/2}\lambda_q^{2+\delta}\right)^r \MMR,
		\end{equation}
		\begin{equation} \label{Flow_estim_2}
			\|\nabla X_t(\cdot, s)\|_{N}  \lesssim \MMSa,
		\end{equation}
		where the implicit constants depend on $M$ and $N$.
	\end{lem}
	
To introduce some standard temporal functions, we first define the standard time scale
	\[
	\tau_q = \delta_q^{-\frac12} \lambda_q^{-(2+\delta)} \lambda_{q+1}^{-\alpha},
	\]
 which ensures
	\[
	\|\bar u_q\|_{1+\alpha} \tau_q \lesssim \bigg(\frac{\lambda_q}{\lambda_{q+1}}\bigg)^\alpha \lesssim 1, \quad \|\bar u_q\|_1 \tau_q \leq C \lambda_{q+1}^{-\alpha}\le1,
	\]
	where \(C > 0\) depends on \(M\) and $a_0$ has be chosen sufficiently large such that $C \lambda_{q+1}^{-\alpha}\le1$, and thus Lemma \ref{Flow_estim} holds with \(\tau\) replaced by \(\tau_q\).

Denote \( t_k = k \tau_q \) , \( k \in \mathbb{Z} \). Then  two family of standard temporal cut-off functions $\chi_k$ and $\tilde{\chi}_k$ can be defined as follows:
\begin{itemize}
	\item The squared $\chi_k$ is a partition of unity:
	\[
	\sum_{k \in \mathbb{Z}} \chi_k^2(t) = 1;
	\]
	\item \(\supp \chi_k \subset (t_k - \frac{2}{3} \tau_q, t_k + \frac{2}{3} \tau_q)\) and thus
	\[
	\supp \chi_{k-1} \cap \supp \chi_{k+1} = \varnothing, \quad \forall k \in \mathbb{Z};
	\]
	\item For \( N \geq 0 \) and \( k \in \mathbb{Z} \),
	\[
	|\partial_t^N \chi_k| \lesssim_{N} \tau_q^{-N},
	\]
\end{itemize}
\begin{itemize}
	\item \(\supp \tilde{\chi}_k \subset (t_k - \tau_q, t_k + \tau_q)\) and \(\tilde{\chi}_k = 1\) on \((t_k - \frac{2}{3} \tau_q, t_k + \frac{2}{3} \tau_q)\), which imply
	\[
	\chi_k \tilde{\chi}_k = \chi_k, \quad \forall k \in \mathbb{Z};
	\]
	\item For any \( N \geq 0 \) and \( k \in \mathbb{Z} \),
	\[
	|\partial_t^N \tilde{\chi}_k| \lesssim_{N} \tau_q^{-N},
	\]
\end{itemize}

Denote the number of Newton steps to be
\[
\Gamma := \bigg\lceil \frac{1 + \frac{\delta}{3}}{1 + \frac{2\delta}{3} - \beta} \bigg\rceil.
\]
Then  the following lemma holds:
\begin{lem} \label{osc_prof}(\cite{GR23})
	Let \( F \subset \mathbb{Z}^2 \) be the set defined in section \ref{sec-alge}, and \( \Gamma \in \mathbb{N} \). For any \( \xi \in F \), there exist \( 2\Gamma \) smooth, 1-periodic functions \( g_{\xi, e, n}, g_{\xi, o, n}: \mathbb{R} \to \mathbb{R} \), with \( n \in \{1, 2, \dots, \Gamma\} \), such that:
	\[
	\int_0^1 g_{\xi, p, n}^2 = 1, \quad \forall \xi \in F, \ p \in \{e, o\}, \text{ and } n \in \{1, 2, \dots, \Gamma\};
	\]
	and
	\[
	\supp g_{\xi, p, n} \cap \supp g_{\eta, q, m} = \varnothing,
	\]
	whenever \( (\xi, p, n) \neq (\eta, q, m) \in F \times \{e, o\} \times \{1, 2, \dots, \Gamma\}. \)
\end{lem}

Let \( n \in \{0, 1, \dots, \Gamma - 1\} \). After \( n \) Newton steps, one get the following ASM-Reynolds system:
\begin{equation} \label{steps}
	\begin{cases}
		\partial_t v_{q, n} + u_{q, n} \cdot \nabla v_{q, n} - (\nabla v_{q,n})^{T}\cdot u_{q,n} + \nabla p_{q,n} = \nabla \cdot R_{q, n} + \nabla \cdot S_{q, n} + \nabla \cdot P_{q+1, n}, \\
		u_{q, n} = T_1[v_{q, n}], \\
		\nabla \cdot v_{q, n} = 0,
	\end{cases}
\end{equation}
where:
\begin{itemize}
	\item \( v_{q, n} \) is the velocity field at step \( n \), starting from \( v_{q, 0} = v_q \),
	\item \( p_{q, n} \) is the pressure, initialized as \( p_{q, 0} = p_q \),
	\item \( R_{q, n} \) represents the gluing error after the \( n \)-th perturbation (\( R_{q, 0} \) is the initial mollified stress),
	\item \( S_{q, n} \) is an error term to be eliminated by oscillatory Nash perturbations, with \( S_{q, 0} = 0 \),
	\item \( P_{q+1, n} \) is a residual error, with \( P_{q+1, 0} = R_q - R_{q, 0} \).
\end{itemize}
When \( n = 0 \), system \eqref{steps}  is exactly  the ASM-Reynolds system \eqref{ASMR-q}.

	\subsection{Construction of the \texorpdfstring{$(n+1)^{\text{th}}$}{thth}  Newton step} \label{n+1-New}
In this subsection, we construct the \texorpdfstring{$(n+1)^{\text{th}}$}{thth} Newton step. To set up,   we first define the amplitude functions
\begin{equation}\label{def.a.coeff}
	a_{\xi, k, n} = \delta_{q+1, n}^{\frac{1}{2}} \lambda_{q+1} \chi_k \gamma_{\xi, k, n},
\end{equation}
where  for \( n \in \{0,1,\dots,\Gamma-1\} \), \( k \in \mathbb{Z} \), and  \( \xi \in F \),  \( \gamma_{\xi, k, n} \)  are given by Lemma \ref{le-geo}:

	$$\left(\gamma^2_{\xi^{(1)},k,n}, \gamma^2_{\xi^{(2)},k,n}, \hat{\gamma}^2_{\widehat{\xi^{(3)}},k,n}\right)
  = \mathcal{L}^{(-m)}_{\left( \nabla \Phi_k^{T} \xi^{(1)}, \nabla \Phi_k^{T} \xi^{(2)}, \nabla \Phi_k^{T} \widehat{\xi^{(3)}} \right)} \left( D - \frac{R_{q,n}}{\delta_{q+1,n} \lambda_{q+1}^{1+\delta}} \right), $$
	$$D = \xi^{(1)} \mathring{\otimes} \nabla \bar{m}\left( \xi^{(1)} \right) + \xi^{(2)} \mathring{\otimes} \nabla \bar{m}\left( \xi^{(2)} \right) + \widehat{\xi^{(3)}} \mathring{\otimes} \nabla \bar{m}\left( \widehat{\xi^{(3)}} \right),$$
  here $\Phi_k$ is the backward flow of $\bar{u}_q$ starting from $t_k$:
$$\left\{\begin{array}{l}
	\partial_t {\Phi}_k+{\bar{u}}_{q} \cdot \nabla {\Phi}_k=0 \\
	\left.{\Phi}_k\right|_{t=t_k}=x
\end{array}\right.$$

The amplitude parameters above \( \delta_{q+1, n} \) are defined as:
\begin{equation*}
	\delta_{q+1, n} := \delta_{q+1} \left( \frac{\lambda_q}{\lambda_{q+1}} \right)^{n\left( 1 + \frac{2\delta}{3} - \beta \right)}.
\end{equation*}

We now verify that \( \gamma_{\xi, k, n} \) is well-defined. It follows from Lemma \ref{Flow_estim} and Proposition \ref{NewIter} that
\[
\|R_{q, n}\|_0 \leq \delta_{q+1, n} \lambda_{q+1}^{1 + \delta} \lambda_q^{-\alpha} \Rightarrow \left\|\frac{R_{q, n}}{\delta_{q+1, n} \lambda_{q+1}^{1 + \delta}}\right\|_0 \leq C \lambda_q^{-\alpha} \leq c_2.
\]
Furthermore,
\[
\left\| \left( \nabla\Phi^{T} \xi^{(1)}, \nabla\Phi^{T} \xi^{(2)}, \nabla\Phi^{T} \widehat{\xi^{(3)}} \right) - \left( \xi^{(1)}, \xi^{(2)}, \widehat{\xi^{(3)}} \right) \right\|_0 \leq \tau_q \|\bar u_q\|_1 \leq C \lambda_{q+1}^{-\alpha} \leq c_1,
\]
provided that \( a_0 \) is sufficiently large, here \( C \) is a constant depending only on \( M \), and \( c_1, c_2 \) come from Lemma \ref{le-geo}. Therefore, \( \gamma_{\xi, k, n} \) is well-defined.

Let \( \mathcal{N}_{\tau}(A) \)  be  the neighborhood  around the set \( A \) of size \( \tau \) and define:
\[
\mathbb{Z}_{q, n} := \big\{ k \in \mathbb{Z} \mid k \tau_q \in \mathcal{N}_{\tau_q}(\supp_t R_{q, n}) \big\}.
\]
Then 
\[
\sum_{k \in \mathbb{Z}_{q, n}} \chi_k^2(t) = 1, \quad \text{for} \quad t \in \supp_t R_{q, n}.
\]
Using Lemma \ref{le-geo}, we obtain the following decomposition:
\begin{equation}\label{decomp}
\div \sum_{k \in \mathbb{Z}_{q, n}} \sum_{\xi \in F} \frac{1}{4} \lambda_{q+1}^{\delta-1} a_{\xi, k, n}^2 \nabla \Phi^{T}_k \xi \mathring{\otimes} \nabla \bar{m} \left( \nabla \Phi^{T}_k \xi \right) = - \div R_{q, n}.
\end{equation}
For the simplicity of presentation, one can introduce
\[
A_{\xi, k, n} := \frac{1}{4} \lambda_{q+1}^{\delta - 1} a_{\xi, k, n}^2 \nabla \Phi^{T}_k \xi \mathring{\otimes} \nabla \bar{m} \left( \nabla \Phi^{T}_k \xi \right).
\]

The temporal oscillation parameters are defined as
\begin{equation*}
	\mu_{q+1} = \delta_{q+1}^{\frac12} \lambda_q^{1+\frac{\delta}{3}} \lambda_{q+1}^{1+\frac{2\delta}{3}} \lambda_{q+1}^{4\alpha},
\end{equation*}
such that
\begin{equation*}
\frac{\tau_q^{-1}}{\mu_{q + 1}}= \left( \frac{\lambda_{q}}{\lambda_{q+1}} \right)^{1 + \frac{2\delta}{3} - \beta} \lambda_{q+1}^{-3\alpha} \ll 1,
\end{equation*}
due to $\beta < 1 + \frac{2\delta}{3}$. Now, define $f_{\xi, k, n+1}:\mathbb{R} \to \mathbb{R}$ by
\begin{equation}\label{def-temporal-f}
	f_{\xi, k, n+1} := 1 - g^2_{\xi, k, n+1},
\end{equation}
where
\begin{equation*}
	g_{\xi, k, n+1} = 
	\begin{cases}
		g_{\xi, e, n+1} & \text{if } k \text{ is even}, \\
		g_{\xi, o, n+1} & \text{if } k \text{ is odd}.
	\end{cases}
\end{equation*}
The primitive of $f_{\xi, k, n+1}$ is given by
\begin{equation*}
	f^{[1]}_{\xi, k, n+1}(t) = \int_0^t f_{\xi, k, n+1}(s) \, ds,
\end{equation*}
which is a well-defined 1-periodic function, since $g_{\xi, e, n+1}$ and $g_{\xi, o, n+1}$ are normalized in $L^2$. The set up of these temporal functions and oscillation parameters here are similar to that of \cite{GR23}.

Next, we solve the transport  problem:
\begin{equation}\label{LocalNewt}
	\begin{cases}
		\partial_t \psi_{k, n + 1} + \bar{u}_q \cdot \nabla \psi_{k, n + 1} = \sum_{\xi \in F} f_{\xi, k, n+1}(\mu_{q+1} t)\Delta^{-1}\nabla^{\perp} \cdot \div A_{\xi, k, n}(x, t), \\
		\psi_{k, n+1} \big|_{t=t_k}(x) = \frac{1}{\mu_{q+1}} \sum_{\xi \in F} f^{[1]}_{\xi, k, n+1}(\mu_{q+1} t_k) \Delta^{-1} \nabla^{\perp} \cdot \div A_{\xi, k, n}(x, t_k).
	\end{cases}
\end{equation}
It follows from this that
\begin{equation*}
	\Delta^{-1} \nabla^{\perp} \cdot \nabla^{\perp} \left( \partial_t \psi_{k, n + 1} + \bar{u}_q \cdot \nabla \psi_{k, n + 1} \right)
	= \Delta^{-1} \nabla^{\perp} \cdot \sum_{\xi \in F} f_{\xi, k, n+1}(\mu_{q+1} t)\div A_{\xi, k, n}(x, t),
\end{equation*}
which implies the existence of a pressure term $p_{k, n+1}$ such that
\begin{equation*}
	\nabla^{\perp} \left( \partial_t \psi_{k, n + 1} + \bar{u}_q \cdot \nabla \psi_{k, n + 1} \right)
	= \sum_{\xi \in F} f_{\xi, k, n+1}(\mu_{q+1} t) \div A_{\xi, k, n}(x, t) + \nabla p_{k, n+1}.
\end{equation*}

Finally, we define the $(n+1)^{\text{th}}$ Newton perturbation by the superposition of $\tilde{\chi}_k\nabla^{\perp} \psi_{k, n+1}$:
\begin{equation}\label{wt}
	w^{(t)}_{q+1, n+1}(x, t) = \sum_{k \in \mathbb{Z}_{q,n}} \tilde{\chi}_k(t) \nabla^{\perp} \psi_{k, n+1}(x, t).
\end{equation}

	Then using the perturbation $w_{q+1, n+1}^{(t)}$ above, one  can construct the \texorpdfstring{$(n+1)^{\text{th}}$}{thth} Newton step as follows:  
	
	First, direct computations give
	\begin{align*}
		&\partial_t w^{(t)}_{q+1,n+1} + \bar{u}_q \cdot \nabla w^{(t)}_{q+1,n+1} + T_1[w^{(t)}_{q+1,n+1}] \cdot \nabla \bar{v}_q \\
		&= \sum_{k \in \mathbb{Z}_{q,n}} \tilde{\chi}_k(t) \left( \partial_t \nabla^{\perp} \psi_{k,n+1} + \bar{u}_q \cdot \nabla \nabla^{\perp} \psi_{k,n+1} + \nabla^{\perp} T_1[\psi_{k,n+1}] \cdot \nabla \bar{v}_q \right) \\
		&\quad + \sum_{k \in \mathbb{Z}_{q,n}} \partial_t \tilde{\chi}_k \nabla^{\perp} \psi_{k,n+1}, \\
		&= \sum_{k \in \mathbb{Z}_{q,n}} \tilde{\chi}_k(t) \left( \nabla^{\perp} \left( \partial_t \psi_{k,n+1} + \bar{u}_q \cdot \nabla \psi_{k,n+1} \right) - \div \left( \psi_{k,n+1} \nabla^{\perp} \bar{u}_q + T_1[\psi_{k,n+1}] \nabla^{\perp} \bar{v}_q \right) \right) \\
		&\quad + \sum_{k \in \mathbb{Z}_{q,n}} \partial_t \tilde{\chi}_k \nabla^{\perp} \psi_{k,n+1}, \\
		&= \sum_{k \in \mathbb{Z}_{q,n}} \sum_{\xi \in F} \tilde{\chi}_k(t) f_{\xi,k,n+1}(\mu_{q+1} t) \div A_{\xi,k,n}(x,t) + \sum_{k \in \mathbb{Z}_{q,n}} \tilde{\chi}_k \nabla p_{k,n+1} \\
		&\quad - \sum_{k \in \mathbb{Z}_{q,n}} \tilde{\chi}_k(t) \div \left( \psi_{k,n+1} \nabla^{\perp} \bar{u}_q + T_1[\psi_{k,n+1}] \nabla^{\perp} \bar{v}_q \right) \\
		&\quad + \sum_{k \in \mathbb{Z}_{q,n}} \partial_t \tilde{\chi}_k \nabla^{\perp} \psi_{k,n+1}.
	\end{align*}
	
	Since $\tilde{\chi}_k A_{\xi,k,n} = A_{\xi,k,n}$ for all $k \in \mathbb{Z}$ (due to $\supp A_{\xi,k,n} \subset \supp a_{\xi,k,n} \subset \supp \chi_k \times \mathbb{T}^2$), and using \eqref{decomp} and \eqref{def-temporal-f}, one can obtain
	\begin{align*}
		\sum_{k \in \mathbb{Z}_{q,n}} \sum_{\xi \in F} \tilde{\chi}_k(t) f_{\xi,k,n+1}(\mu_{q+1} t) \div A_{\xi,k,n}
		&= \sum_{k \in \mathbb{Z}_{q,n}} \sum_{\xi \in F} \div A_{\xi,k,n} - \sum_{k \in \mathbb{Z}_{q,n}} \sum_{\xi \in F} g_{\xi,k,n+1}^2 \div A_{\xi,k,n} \\
		&= - \div R_{q,n} - \sum_{k \in \mathbb{Z}_{q,n}} \sum_{\xi \in F} g_{\xi,k,n+1}^2 \div A_{\xi,k,n}.
	\end{align*}
	
Collecting all the arguments above, we can conclude that the $(n+1)^{\text{th}}$ step ASM-Reynolds system \eqref{steps} is satisfied with 
	\begin{equation} \label{Newvelo}
		v_{q, n + 1} = v_{q, n} + w_{q+1, n + 1}^{(t)} = v_q + \sum_{m =1}^{n+1} w_{q+1, m}^{(t)},
	\end{equation}
	\begin{equation} \label{Newpressure}
		p_{q, n + 1} = p_{q, n} + \sum_{k \in \mathbb Z_{q,n}} \tilde \chi_k p_{k,n+1} - \bar{u}_q \cdot w^{(t)}_{q+1, n+1},
	\end{equation}
	\begin{equation} \label{NewStr}
		\begin{aligned}
			R_{q, n + 1} &= \div^{-1} \sum_{k \in \mathbb Z_{q,n}} \partial_t \tilde \chi_k \nabla^{\perp}\psi_{k, n+1}
			- \div^{-1} \sum_{k \in \mathbb Z_{q,n}} \tilde \chi_k (t) \div \underbrace{\left( \psi_{k, n + 1} \nabla^{\perp} \bar u_q + T_1[\psi_{k, n + 1}] \nabla^{\perp} \bar{v}_q \right)}_{=:\tilde{S}[\bar{v}_q, \psi_{k, n + 1}]}, \\
			&\quad + \div^{-1} \Lambda \sum_{k \in \mathbb{Z}_{q, n}} \tilde{\chi}_k (t) \underbrace{ \Lambda^{-1} \left( \left( \nabla \bar{u}_q \right)^T \cdot \nabla^{\perp} \psi_{k, n + 1} - \left( \nabla \bar{v}_q \right)^T \cdot T_1[\nabla^{\perp} \psi_{k, n + 1}] \right)}_{=:S[\bar{v}_q, \psi_{k, n + 1}]}, 
		\end{aligned}
	\end{equation}
	\begin{equation} \label{NewNash}
		S_{q, n + 1} = S_{q, n} - \sum_{k \in \mathbb Z_{q,n}} \sum_{\xi \in F} g_{\xi, k, n+1}^2 A_{\xi, k, n},
	\end{equation}
	\begin{eqnarray} \label{SmalStr}
		P_{q+1, n+1} &=& P_{q+1, n} + \div^{-1} \left( T_1[w_{q+1, n+1}^{(t)}]^{\perp} \nabla^{\perp} \cdot w_{q+1, n+1}^{(t)} \right) \notag \\
		&& + \div^{-1} \sum_{m = 1}^{n} \left( T_1[w_{q+1, n+1}^{(t)}]^{\perp} \nabla^{\perp} \cdot w_{q+1, m}^{(t)} + T_1[w_{q+1, m}^{(t)}]^{\perp} \nabla^{\perp} \cdot w_{q+1, n+1}^{(t)} \right) \notag \\
		&& + \div^{-1} \left( (u_q-\bar{u}_q)^{\perp} \nabla^{\perp} \cdot w_{q+1, n+1}^{(t)} + T_1[w_{q+1, n+1}^{(t)}]^{\perp} \nabla^{\perp} \cdot (v_q - \bar v_q) \right),
	\end{eqnarray}
 where the anti-divergence operator $\div^{-1}=\mathcal{R}(m)$  is given in Appendix \ref{sec-class} and will be used throughout this paper.

	\subsection{Estimates for the gluing error and the total Newton perturbation}
	In this section, we give estimates for the gluing error $R_{q,n+1}$ and the total Newton perturbation.
	
	We first establish  the following  estimates for the amplitude functions  $a_{\xi,k,n}$, $A_{\xi,k,n}$ and the Newton perturbations $\psi_{k, n + 1}$:
	\begin{lem} \label{a_estim}
		For $\forall N \ge0, r\in\left\{0,1\right\}$, it holds that 
		\begin{equation} \label{a_estim_1}
			\|\bar D_t^r a_{\xi, k, n}\|_N \lesssim \delta_{q+1, n}^{1/2} \la_{q+1}\mu_{q + 1}^r\MMSt,
		\end{equation}
		\begin{equation} \label{a_estim_2}
			\|\bar D_t^r A_{\xi, k, n} \|_N  \lesssim \delta_{q+1, n}  \lambda_{q+1}^{1+\delta}
			\mu_{q + 1}^r\MMSt,
		\end{equation}
	 where the implicit constants depend on $n$, $\Gamma$, $M$, $\alpha$, and $N$.
	\end{lem}
	\begin{proof}
		The proof is basically the same as that for Lemma 3.5 and corollary 3.6 of \cite{GR23} and one will need to use Lemmas \ref{smoli_estim}, \ref{Flow_estim} and the assumed estimates from Proposition \ref{NewIter} below to modify the proof.
	\end{proof}
	
	\begin{lem} \label{psi_estim}
		For $ N\ge0, r\in\left\{0,1\right\}$, we have the following estimates on $\supp \tilde \chi_k$:
		\begin{equation} \label{psi_estim_1}
			\|\bar D_t^r \psi_{k, n+1}\|_{N+\alpha} \lesssim \frac{\delta_{q+1, n}\lambda_{q+1}^{1+\delta}  \ell_q^{-\alpha}}{\mu_{q+1}}\left(u_{q+1}\right)^r\MMSta
		\end{equation}
		where all the implicit constants depend on $n$, $\Gamma$, $M$, $\alpha$, and $N$.
	\end{lem}
	\begin{proof}
		It follows from the definition of $\psi_{k, n + 1}$ that
		$$
		\left\{\begin{array}{l}
			\bar{D}_t \psi_{k, n+1}=\sum_{\xi \in F} f_{\xi, k, n+1}\left(\mu_{q+1} t\right) \Delta^{-1} \nabla^{\perp} \cdot  \operatorname{div} A_{\xi, k, n}(x, t) \\
			\left.\psi_{k, n+1}\right|_{t=t_k}(x)=\frac{1}{\mu_{q+1}} \sum_{\xi \in F} f_{\xi, k, n+1}^{[1]}\left(\mu_{q+1} t_k\right) \Delta^{-1} \nabla^{\perp} \cdot  \operatorname{div} A_{\xi, k, n}\left(x, t_k\right)
		\end{array}\right.$$
		Note that $\Delta^{-1} \nabla^{\perp} \cdot  \operatorname{div}$ is a $0$-order Calder\'on-Zygmund operator thus the case $r=1$  in \eqref{psi_estim_1} follows from \eqref{a_estim_2}. And by Proposition \ref{transport_estim},  \ref{a_estim_2} and Lemma \ref{smoli_estim},  one can get
		\begin{align*}
			\|\psi_{k, n + 1}\|_{N+\alpha}&\lesssim\frac{1}{\mu_{q+1}}\|A_{\xi,k,n}\|_{N+\alpha}+\|A_{\xi,k,n}\|_{N+\alpha}\int_{t_k}^{t}f_{\xi, k, n+1}\left(\mu_{q+1} \tau\right)d\tau\\
			&+\frac{1}{\mu_{q+1}}\|A_{\xi,k,n}\|_{1+\alpha}|t-t_k|\|\bar{u}_q\|_{N+\alpha}+|t-t_k|\|\bar{u}_q\|_{N+\alpha}\|A_{\xi,k,n}\|_{1+\alpha}\int_{t_k}^{t}f_{\xi, k, n+1}\left(\mu_{q+1} \tau\right)d\tau\\
			&\lesssim\frac{1}{\mu_{q+1}}\|A_{\xi,k,n}\|_{N+\alpha}++\frac{1}{\mu_{q+1}}\|A_{\xi,k,n}\|_{1+\alpha}|t-t_k|\|\bar{u}_q\|_{N+\alpha}\\&\lesssim  \frac{\delta_{q+1, n}\lambda_{q+1}^{1+\delta}\ell_{q}^{-\alpha}}{\mu_{q + 1}} \mathcal{M}\left(N, L_t, \lambda_q, \ell_q^{-1}\right)
		\end{align*}
	\end{proof}
	Then  the following key inductive proposition of the Newton step can be proved:
	\begin{prop} \label{NewIter}
		For $N\ge0, r\in\left\{0,1\right\}$,  assume that $R_{q, n}$ satisfies  
		\begin{equation} \label{NewIter_1}
			\|\bar D_t^r R_{q, n}\|_{N} \leq \delta_{q+1, n}\lambda_{q+1}^{1+\delta} \lambda_q^{ - \alpha}\mu_{q + 1}^r\MMSt, 
		\end{equation} 
		with implicit constants depending on $n$, $\Gamma$, $M$, $\alpha$ and $N$. In addition, suppose that 
		\begin{eqnarray} \label{NewIter_2}
			\supp_t R_{q,n} &\subset& [-2 +(\de_q^{\frac12}\la_q^{2+\delta})^{-1} - 2n\tau_q, -1 -(\de_q^{\frac12}\la_q^{2+\delta})^{-1} + 2n\tau_q] \\ 
			&& \cup [1 + (\de_q^{\frac12}\la_q^{2+\delta})^{-1} - 2n \tau_q, 2 - (\de_q^{\frac12}\la_q^{2+\delta})^{-1} + 2n \tau_q]. \nonumber
		\end{eqnarray}
		Then, the new stress $R_{q, n+1}$ satisfies \eqref{NewIter_1}-\eqref{NewIter_2} with $n$ replaced by $n+1$. 
	\end{prop} 
	\begin{proof} 
		
		Note that   Lemma \ref{smoli_estim} implies that the assumptions of Proposition \ref{NewIter} are satisfied at \(n = 0\) due to $\delta_q^{\frac12}\la_q^{2+\delta}\ll\mu_{q + 1}$. Also, by the definitions of \(\tilde{\chi}_k\) and \(\mathbb{Z}_{q,n}\), one has
		\[
		\supp_t R_{q,n+1} \subset \overline{\mathcal{N}_{2\tau_q}(\supp_t R_{q,n})},
		\]
		thus the temporal support assumption \eqref{NewIter_2} is satisfied at  level $n+1$.

		Now we  prove \eqref{NewIter_1} at  level $n+1$.	 First note that the set $\{\supp \tilde \chi_k\}$ is locally finite and $\div^{-1} \nabla^\perp$,  $\div^{-1}\div, \div^{-1}\Lambda$ are of Calder\'on-Zygmund type of order $0$ and $T_1$ is homogeneous of order $1+\delta$, one can conclude from Lemmas \ref{smoli_estim},  \ref{psi_estim},  \ref{NewStr} and \eqref{est,bilin2}  that
		\begin{equation*}
			\begin{aligned}
				&\|R_{q, n+1}\|_N \lesssim \|R_{q, n+1}\|_{N+\alpha}\\
				&\lesssim \tau_q^{-1} \sup_{k \in \mathbb Z_{q,n}} \|\psi_{k, n+1}\|_{N+\alpha}+\sup_{k \in \mathbb Z_{q,n}} \|\psi_{k, n+1}\nabla^{\perp}\bar{u}_q+T_1[\psi_{k, n + 1}]\nabla^{\perp}\bar{v}_q\|_{N+\alpha}+\sup_{k \in \mathbb Z_{q,n}}\|S[\bar{v}_q,\psi_{k, n + 1}]\|_{N+\alpha}\\
				&\lesssim \tau_q^{-1} \frac{\delta_{q+1, n} \lambda_{q+1}^{1+\delta} \ell_q^{-\alpha}}{\mu_{q+1}} \mathcal{M}\left(N, L_t, \lambda_q, \ell_q^{-1}\right)+\|\psi_{k, n+1}\|_{N+\alpha}\|\bar{v}_q\|_{2+\delta+\alpha}+\|\psi_{k, n+1}\|_{\alpha}\|\bar{v}_q\|_{N+2+\delta+\alpha}\\
				&+\|\psi_{k, n+1}\|_{N+1+\delta+\alpha}\|\bar{v}_q\|_{1+\alpha}+\|\psi_{k, n+1}\|_{1+\delta+\alpha}\|\bar{v}_q\|_{N+1+\alpha}\\
				&\lesssim \tau_q^{-1} \frac{\delta_{q+1, n} \lambda_{q+1}^{1+\delta} \ell_q^{-\alpha}}{\mu_{q+1}} \mathcal{M}\left(N, L_t, \lambda_q, \ell_q^{-1}\right)\\
				&\leq (C\lambda_{q+1}^{-1}){(\lambda_{q+1}\ell_q)}^{-\alpha}  \delta_{q+1, n}\left(\frac{\lambda_{q}}{\lambda_{q+1}}\right)^{1+\frac{2\delta}{3}-\beta} \lambda_{q+1}^{1+\delta-\alpha}  \mathcal{M}\left(N, L_t, \lambda_q, \ell_q^{-1}\right)\\
				&\leq   \delta_{q+1, n+1} \lambda_{q+1}^{1+\delta}\lambda_{q}^{-\alpha}  \mathcal{M}\left(N, L_t, \lambda_q, \ell_q^{-1}\right)
			\end{aligned}
		\end{equation*}
		where  constant $C>0$ is independent of $a > a_0$ and $q$. Also,  $a_0$ has been chosen sufficiently large such that 
		\begin{equation*}
			C \lambda_{q+1}^{-\alpha} \leq 1, \ \ \forall \alpha > 0
		\end{equation*}
		For the material derivative, one can use Proposition \ref{prop-commu} to get
		\begin{eqnarray*}
			\|\bar D_t R_{q, n+1}\|_{N+\alpha} &\lesssim& \sup_{k \in \mathbb Z_{q,n}} \big( \|\bar D_t(\partial_t \tilde \chi_k \psi_{k, n+1})\|_{N+\alpha} + \|[\bar u_q \cdot \nabla, \div^{-1} \nabla^\perp] \partial_t \tilde \chi_k \psi_{k, n+1}\|_{N+\alpha} \big) \\ 
			&&+\sup_{k \in \mathbb Z_{q,n}} \big( \|\bar D_t( \tilde \chi_k \tilde{S}\left[\bar{v}_q, \psi_{k, n+1}\right])\|_{N+\alpha} + \|[\bar u_q \cdot \nabla, \div^{-1} \div]  (\tilde \chi_k\tilde{S}\left[\bar{v}_q, \psi_{k, n+1}\right] \|_{N+\alpha} \big)\\
			&&+\sup_{k \in \mathbb Z_{q,n}} \big( \|\bar D_t( \tilde \chi_k S\left[\bar{v}_q, \psi_{k, n+1}\right])\|_{N+\alpha} + \|[\bar u_q \cdot \nabla, \div^{-1} \Lambda]  \left(\tilde \chi_k S\left[\bar{v}_q, \psi_{k, n+1}\right]\right) \|_{N+\alpha} \big)\\
			&\lesssim & \sup_{k \in \mathbb Z_{q,n}} \big( \tau_q^{-2} \|\psi_{k, n+1}\|_{N+\alpha} + \tau_q^{-1} \|\bar D_t \psi_{k, n+1}\|_{N+\alpha} \\ 
			&& + \tau_q^{-1} \|\bar u_q\|_{1+\alpha}\|\psi_{k, n+1}\|_{N+\alpha} + \tau_q^{-1} \|\bar u_q\|_{N+1+\alpha} \|\psi_{k,n+1}\|_\alpha \big)\\
			&&+\sup_{k \in \mathbb Z_{q,n}} \big( \tau_q^{-1} \|\tilde{S}\left[\bar{v}_q, \psi_{k, n+1}\right]\|_{N+\alpha} + \|\bar D_t \tilde{S}\left[\bar{v}_q, \psi_{k, n+1}\right]\|_{N+\alpha} \\ 
			&& +  \|\bar u_q\|_{1+\alpha}\|\tilde{S}\left[\bar{v}_q, \psi_{k, n+1}\right]\|_{N+\alpha} +  \|\bar u_q\|_{N+1+\alpha} \|\tilde{S}\left[\bar{v}_q, \psi_{k, n+1}\right]\|_\alpha \big)\\
			&&+\sup_{k \in \mathbb Z_{q,n}} \big( \tau_q^{-1} \|S\left[\bar{v}_q, \psi_{k, n+1}\right]\|_{N+\alpha} +  \|\bar D_t S\left[\bar{v}_q, \psi_{k, n+1}\right]\|_{N+\alpha} \\ 
			&& + \|\bar u_q\|_{1+\alpha}\|S\left[\bar{v}_q, \psi_{k, n+1}\right]\|_{N+\alpha} +  \|\bar u_q\|_{N+1+\alpha} \|S\left[\bar{v}_q, \psi_{k, n+1}\right]\|_\alpha \big)\\
		\end{eqnarray*}
		Recalling the notations in \eqref{NewStr},  \eqref{S-com}, \eqref{def-trilin} and  Lemma \ref{lem.bilin1} , one can get
		\begin{equation}\notag
			\begin{aligned}
				&\bar D_t \tilde{S}\left[\bar{v}_q, \psi_{k, n+1}\right]\\&=\bar{D}_t\psi_{k, n + 1}\nabla^{\perp}\bar{u}_q+\psi_{k, n + 1}\bar{D}_t\nabla^{\perp}\bar{u}_q+T_1[\psi_{k, n + 1}]\bar{D}_t\nabla^{\perp}\bar{v}_q+T_1[\bar{D}_t\psi_{k, n + 1}]\nabla^{\perp}\bar{v}_q+\left([\bar{u}_q\cdot\nabla,T_1]\psi_{k, n + 1}\right)\nabla^{\perp}\bar{v}_{q},
			\end{aligned}
		\end{equation}
		\begin{equation}\notag
			\begin{aligned}
				&\bar D_t S\left[\bar{v}_q, \psi_{k, n+1}\right]\\
				&=\bar{D}_t\binom{\partial_1\Lambda^{-1}\left(\psi_{k, n+1} \partial_1 T_1 \bar{v}_{q,2}-T_1 \psi_{k, n+1} \partial_1 \bar{v}_{q,2}\right)+\partial_2\Lambda^{-1}\left(T_1 \psi_{k, n+1} \partial_1 \bar{v}_{q,1}-\psi_{k, n+1} \partial_1 T_1 \bar{v}_{q,1}\right)}{\partial_1\Lambda^{-1}\left(\psi_{k, n+1} \partial_2 T_1 \bar{v}_{q,2}-T_1 \psi_{k, n+1} \partial_2 \bar{v}_{q,2}\right)+\partial_2\Lambda^{-1}\left(T_1 \psi_{k, n+1} \partial_2 \bar{v}_{q,1}-\psi_{k, n+1} \partial_2 T_1 \bar{v}_{q,1}\right)}\\
				&+\bar{D}_t\binom{ S_{1}\left[\pa_2\bar{v}_{q,1}, \psi_{k, n+1}\right]- S_{1}\left[\pa_1\bar{v}_{q,2}, \psi_{k, n+1}\right]}{ S_{2}\left[\pa_2\bar{v}_{q,1}, \psi_{k, n+1}\right]- S_{1}\left[\pa_2\bar{v}_{q,2}, \psi_{k, n+1}\right]}\\
				&=\underbrace{\binom{\partial_1\Lambda^{-1}\bar{D}_t\left(\psi_{k, n+1} \partial_1 \bar{u}_{q,2}-T_1 \psi_{k, n+1} \partial_1 \bar{v}_{q,2}\right)+\partial_2\Lambda^{-1}\bar{D}_t\left(T_1 \psi_{k, n+1} \partial_1 \bar{v}_{q,1}-\psi_{k, n+1} \partial_1  \bar{u}_{q,1}\right)}{\partial_1\Lambda^{-1}\bar{D}_t\left(\psi_{k, n+1} \partial_2  \bar{u}_{q,2}-T_1 \psi_{k, n+1} \partial_2 \bar{v}_{q,2}\right)+\partial_2\Lambda^{-1}\bar{D}_t\left(T_1 \psi_{k, n+1} \partial_2 \bar{v}_{q,1}-\psi_{k, n+1} \partial_2  \bar{u}_{q,1}\right)}}_{E_1}\\
				&+\underbrace{\binom{[\bar{u}_q\cdot\nabla,\partial_1\Lambda^{-1}]\left(\psi_{k, n+1} \partial_1 \bar{u}_{q,2}-T_1 \psi_{k, n+1} \partial_1 \bar{v}_{q,2}\right)+[\bar{u}_q\cdot\nabla,\partial_2\Lambda^{-1}]\left(T_1 \psi_{k, n+1} \partial_1 \bar{v}_{q,1}-\psi_{k, n+1} \partial_1  \bar{u}_{q,1}\right)}{[\bar{u}_q\cdot\nabla,\partial_1\Lambda^{-1}]\left(\psi_{k, n+1} \partial_2  \bar{u}_{q,2}-T_1 \psi_{k, n+1} \partial_2 \bar{v}_{q,2}\right)+[\bar{u}_q\cdot\nabla,\partial_2\Lambda^{-1}]\left(T_1 \psi_{k, n+1} \partial_2 \bar{v}_{q,1}-\psi_{k, n+1} \partial_2  \bar{u}_{q,1}\right)}}_{E_2}\\
				&+\underbrace{\binom{ S_{1}\left[\bar{D}_t\pa_2\bar{v}_{q,1}, \psi_{k, n+1}\right]+S_{1}\left[\pa_2\bar{v}_{q,1}, \bar{D}_t\psi_{k, n+1}\right]- S_{1}\left[\bar{D}_t\pa_1\bar{v}_{q,2}, \psi_{k, n+1}\right]-S_{1}\left[\pa_1\bar{v}_{q,2}, \bar{D}_t\psi_{k, n+1}\right]}{ S_{2}\left[\bar{D}_t\pa_2\bar{v}_{q,1}, \psi_{k, n+1}\right]+S_{2}\left[\pa_2\bar{v}_{q,1},\bar{D}_t \psi_{k, n+1}\right]- S_{1}\left[\bar{D}_t\pa_2\bar{v}_{q,2}, \psi_{k, n+1}\right]-S_{1}\left[\pa_2\bar{v}_{q,2}, \bar{D}_t\psi_{k, n+1}\right]}}_{E_3}\\
				&+\underbrace{\binom{ S_{1,2}\left[\bar{u}_q,\bar{v}_{q,1}, \psi_{k, n+1}\right]- S_{1,1}\left[\bar{u}_q,\bar{v}_{q,2}, \psi_{k, n+1}\right]}{ S_{2,2}\left[\bar{u}_q,\bar{v}_{q,1}, \psi_{k, n+1}\right]- S_{1,2}\left[\bar{u}_q,\bar{v}_{q,2}, \psi_{k, n+1}\right]}}_{E_4}.\\
			\end{aligned}
		\end{equation}
		It thus follows from Lemmas \ref{smoli_estim},  \ref{psi_estim},  \ref{lem.bilin1}, \ref{lem.trilin} and Proposition \ref{prop-commu} together with some tedious but standard calculations and  a sufficiently large $a_0>0$ that
		\begin{equation*}
			\begin{aligned}
				\|\bar D_t R_{q, n+1}\|_{N} &\leq C \mu_{q + 1} \delta_{q+1, n} \bigg(\frac{\lambda_q}{\lambda_{q+1}}\bigg)^{1+\frac{2\delta}{3} - \beta}\lambda_{q+1}^{1+\delta} (\lambda_{q+1} \ell_q)^{-\alpha} \lambda_{q+1}^{-2\alpha} \MMSt\\
				&\leq  \mu_{q + 1} \delta_{q+1, n+1}\lambda_{q+1}^{1+\delta} \lambda_{q}^{-\alpha} \MMSt.
			\end{aligned}
		\end{equation*}
		And hence \eqref{NewIter_1} holds at level $n+1$.
	\end{proof}

	Finally, we collect the estimates for the total Newton perturbations: 
	\begin{equation*}
		w_{q+1}^{(t)} = \sum_{n = 1}^\Gamma w_{q+1, n}^{(t)}\,,\qquad \tilde{w}_{q+1}^{(t)} = T_1[w_{q+1}^{(t)}]\,.
	\end{equation*}
	
	\begin{lem} \label{w_t_estim}
		For $ N \ge0, r\in\left\{0,1\right\}$, it holds that
		\begin{equation} \label{w_t_estim_0}
			\|\bar D_t^r \tilde{P}_{\lesssim\ell_{q}^{-1}}w_{q+1}^{(t)}\|_N \lesssim \frac{\delta_{q+1}\lambda_{q+1 }^{1+\delta} \lambda_q \ell_q^{-\alpha}}{\mu_{q + 1}}\mu_{q + 1}^r\MMStb,
		\end{equation}
		\begin{equation} \label{w_t_estim_1}
			\|\bar D_t^r \tilde{P}_{\lesssim\ell_{q}^{-1}}\tilde{w}_{q+1}^{(t)}\|_N \lesssim \frac{\delta_{q+1}\lambda_{q+1 }^{1+\delta} \lambda_q^{2+\delta} \ell_q^{-\alpha}}{\mu_{q + 1}}\mu_{q + 1}^r\MMStc,
		\end{equation}
		\begin{equation} \label{w_t_estim_2}
			\|\bar D_t \nabla^\perp\cdot \tilde{P}_{\lesssim\ell_{q}^{-1}} w_{q+1}^{(t)}\|_N \lesssim \delta_{q+1}\lambda_{q}^{2}\lambda_{q+1}^{1+\delta}\ell_{q}^{-\alpha}\MMStc,
		\end{equation}
		
		with implicit constants depending on $\Gamma$, $M$, $\alpha$ and $N$. 
		In addition, the temporal support satisfies 
		\begin{eqnarray}\label{temporal-qplus1}
			\supp_t w_{q+1}^{(t)} &\subset& [-2 + (\delta_q^{\frac12} \lambda_q^{2+\delta})^{-1} - 2\Gamma \tau_q, -1 - (\delta_q^{\frac12} \lambda_q^{2+\delta})^{-1} + 2\Gamma \tau_q] \\ 
			&& \cup [1 + (\delta_q^{\frac12} \lambda_q^{2+\delta})^{-1} - 2 \Gamma \tau_q, 2 - (\delta_q^{\frac12} \lambda_q^{2+\delta})^{-1} + 2 \Gamma \tau_q]. \nonumber
		\end{eqnarray}
	\end{lem}
	\begin{proof}
	The estimates \eqref{w_t_estim_0} and \eqref{w_t_estim_1} for the case $r = 0$ follow directly from $\delta_{q+1, n} \leq \delta_{q+1}$ for all $n$, the definition of $w^{(t)}_{q+1, n+1}$ in \eqref{wt}, and Lemma \ref{psi_estim}. For the case $r = 1$, we  note that
	\[
	\bar{D}_t \tilde{P}_{\lesssim \ell_q^{-1}} w_{q+1}^{(t)} = \tilde{P}_{\lesssim \ell_q^{-1}} \bar{D}_t w_{q+1}^{(t)} + [\bar{u}_q \cdot \nabla, \tilde{P}_{\lesssim \ell_q^{-1}}] w_{q+1}^{(t)},
	\]
	\[
	\bar{D}_t^r \tilde{P}_{\lesssim \ell_q^{-1}} \tilde{w}_{q+1}^{(t)} = \tilde{P}_{\lesssim \ell_q^{-1}} \bar{D}_t^r \tilde{w}_{q+1}^{(t)} + [\bar{u}_q \cdot \nabla, \tilde{P}_{\lesssim \ell_q^{-1}}] \tilde{w}_{q+1}^{(t)}.
	\]
	The commutator terms here satisfy the required estimates by using estimate \eqref{smoli_1'} and Proposition \ref{CET_comm} along with  \eqref{w_t_estim_0} and \eqref{w_t_estim_1} for the case $r=0$.
	
	Now, we verify \eqref{w_t_estim_0} for $r = 1$. By definition, 
	\[
	w^{(t)}_{q+1} = \sum_{n=0}^{\Gamma-1} \sum_{k \in \mathbb{Z}_{q,n}} \tilde{\chi}_k(t) \nabla^{\perp} \psi_{k, n + 1}(x, t) =: \nabla^{\perp} \psi^{(t)}_{q+1},
	\]
	which implies
	\begin{equation*}
		\begin{aligned}
			\bar{D}_t w^{(t)}_{q+1} &= \sum_{n=0}^{\Gamma-1} \sum_{k \in \mathbb{Z}_{q,n}} \left(\partial_t \tilde{\chi}_k(t) \nabla^{\perp} \psi_{k, n + 1}(x,t) + \tilde{\chi}_k(t) \bar{D}_t \nabla^{\perp} \psi_{k, n + 1}(x,t)\right) \\
			&= \sum_{n=0}^{\Gamma-1} \sum_{k \in \mathbb{Z}_{q,n}} \left(\partial_t \tilde{\chi}_k(t) \nabla^{\perp} \psi_{k, n + 1}(x,t) + \tilde{\chi}_k(t) \nabla^{\perp} \bar{D}_t \psi_{k, n + 1}(x,t) - \tilde{\chi}_k(t) \nabla^{\perp} \bar{u}_q \nabla \psi_{k, n + 1}\right).
		\end{aligned}
	\end{equation*}
	Thus, the desired estimates follow from estimate \eqref{smoli_1'} and Lemma \ref{psi_estim}.

		To prove \eqref{w_t_estim_1} for $r=1$, one notes that
		$$\bar{D}_t \tilde{w}^{(t)}_{q+1}=\bar{D}_t T_1 w^{(t)}_{q+1}=T_1\bar{D}_t w^{(t)}_{q+1}+[\bar{u}_q\cdot\nabla,T_1]w^{(t)}_{q+1}$$
		Then the desired estimate follows from \eqref{w_t_estim_0} and Proposition \ref{prop-commu}.
		
		It remains to prove \eqref{w_t_estim_2}. We first write
		$$\nabla^{\perp}\cdot w^{(t)}_{q+1}=\sum_{n=0}^{\Gamma-1}\sum_{k\in\mathbb{Z}_{q,n}}\tilde{\chi}_k(t)\Delta\psi_{k, n + 1}(x,t)$$
		hence
		\begin{equation}\notag
			\begin{aligned}
				&\bar{D}_t\nabla^{\perp}\cdot w^{(t)}_{q+1}\\&=\sum_{n=0}^{\Gamma-1}\sum_{k\in\mathbb{Z}_{q,n}}\left(\pa_t\tilde{\chi}_k(t)\Delta\psi_{k, n + 1}(x,t)+\tilde{\chi}_k(t)\Delta\bar{D}_t\psi_{k, n + 1}(x,t)+\tilde{\chi}_k(t)[\bar{u}\cdot\nabla,\Delta]\psi_{k, n + 1}(x,t)\right)\\
			\end{aligned}
		\end{equation}
		and note that
		$$[\bar{u}\cdot\nabla,\Delta]\psi_{k, n + 1}(x,t)=-\Delta\bar{u}_q\cdot\nabla\psi_{k,n+1}-\sum_{i=1}^{2}\pa_i\bar{u}_q\cdot\nabla\pa_i\psi_{k,n+1}$$
		Then using Lemma \ref{psi_estim} and estimate \eqref{smoli_1'}, one can get
		$$\|\bar{D}_t\nabla^{\perp}\cdot w^{(t)}_{q+1}\|_N\lesssim\delta_{q+1}\lambda_{q+1}^{1+\delta}\ell_{q}^{-\alpha}\lambda_{q}^2\MMStb$$
		Note that
		\begin{equation}\notag
			\begin{aligned}
				\bar{D}_t\nabla^{\perp}\cdot \tilde{P}_{\lesssim\ell_{q}^{-1}}{w}^{(t)}_{q+1}&=\tilde{P}_{\lesssim\ell_{q}^{-1}}\bar{D}_t\nabla^{\perp}\cdot w^{(t)}_{q+1}+[\bar{u}_q\cdot\nabla,\tilde{P}_{\lesssim\ell_{q}^{-1}}]\nabla^{\perp}\cdot w^{(t)}_{q+1}\\
				&=\tilde{P}_{\lesssim\ell_{q}^{-1}}\bar{D}_t\nabla^{\perp}\cdot w^{(t)}_{q+1}+\tilde{P}_{\lesssim\ell_{q}^{-1}}\bar{u}_q\cdot\nabla\tilde{P}_{\lesssim\ell_{q}^{-1}}\nabla^{\perp}\cdot w^{(t)}_{q+1}-{\tilde{P}}_{\lesssim\ell_{q}^{-1}}\left(\bar{u}_q\cdot\nabla\nabla^{\perp}\cdot w^{(t)}_{q+1}\right)
			\end{aligned}
		\end{equation}
		Then the estimate \eqref{w_t_estim_2} follows from the estimate of $\bar{D}_t\nabla^{\perp}\cdot w^{(t)}_{q+1}$ above, \eqref{smoli_1'}, \eqref{w_t_estim_0} and Proposition \ref{CET_comm}.
	\end{proof}

Similar to Section \ref{sec-spa}, we will need the following mollified velocity fields in the Nash step.
	\begin{equation*}
		\tilde u_{q, \Gamma} := \bar u_q + \tilde{P}_{\lesssim \ell_q^{-1}} \tilde{w}_{q+1}^{(t)} = \tilde{P}_{\lesssim \ell_q^{-1}}({P}_{\lesssim \ell_q^{-1}}u_q + \tilde{w}_{q+1}^{(t)}) .
	\end{equation*}
	\begin{equation*}
		\tilde{v}_{q, \Gamma} := \bar{v} _q + \tilde{P}_{\lesssim \ell_q^{-1}}  w
		_{q+1}^{(t)} = \tilde{P}_{\lesssim \ell_q^{-1}}({P}_{\lesssim \ell_q^{-1}}v_q + w_{q+1}^{(t)}) .
	\end{equation*}

	\begin{cor} \label{Gamma_velo_estim}
		For $ N \ge1, r\in\left\{0,1\right\}$,  the following estimates hold: 
		\begin{equation}\notag
			\|\tilde v_{q, \Gamma} \|_N  \lesssim \delta_{q}^{\frac12}\MMStb
		\end{equation}
		\begin{equation}\notag
			\|\tilde u_{q, \Gamma} \|_N \lesssim \delta_{q}^{\frac12}\lambda_{q}^{1+\delta}\MMStb
		\end{equation}
		 with implicit constants depending on $\Gamma$, $M$, $\alpha$, and $N$.
	\end{cor}
	\begin{proof}
		This is  a simple corollary of Lemma \ref{w_t_estim}.
	\end{proof}
	\subsection{The perturbed flow and stability estimate} \label{Pert.Flow.Sec} 
	In this section, we present the  standard stability estimate and estimates for the perturbed flow. This subsection is very similar to that of \cite{DGR24} and \cite{GR23}.
	
	Let $\tilde \Phi_t$ be the backward flow generated by $\tilde u_{q, \Gamma}$:
	\begin{equation} \label{Flow_t_gam}
		\begin{cases}
			\partial_s \tilde \Phi_t(x,s) + \tilde u_{q, \Gamma}(x,s) \cdot \nabla \tilde \Phi_t (x,s) = 0, \\ 
			\tilde \Phi_t(x, t) = x.
		\end{cases}
	\end{equation}
	The corresponding Lagraingian flow $\tilde X_t$ is given by
	\begin{equation} \label{Lagr_t_gam}
		\begin{cases}
			\frac{d}{ds} \tilde X_t(\alpha, s) = \tilde u_{q, \Gamma}(\tilde X_t(\alpha, s),s), \\ 
			\tilde X_t(\alpha, t) = \alpha.
		\end{cases}
	\end{equation}
	
Consistent with previous notation, we define $\tilde{D}_{t, \Gamma}$ as the material derivative:
\begin{equation*}
	\tilde{D}_{t, \Gamma} = \partial_t + \tilde{u}_{q, \Gamma} \cdot \nabla.
\end{equation*}

\begin{cor}(\cite{GR23}) \label{Flow_gam_estim}
	For $N \geq 0$ and $r \in \{0, 1\}$, let $\tilde{\Phi}_t$ and $\tilde{X}_t$ be defined by \eqref{Flow_t_gam} and \eqref{Lagr_t_gam}, respectively, for $t \in \mathbb{R}$. For any $|s - t| < \tau \leq \|\tilde{u}_{q, \Gamma}\|_1^{-1}$,  it holds that:
	\begin{equation} \label{Flow_gam_estim_1}
		\|\bar{D}_t^r \nabla \tilde{\Phi}_t(\cdot, s)\|_N \lesssim \tau_q^{-r} \mathcal{M}(N, L_t - 3, \lambda_q, \ell_q^{-1}),
	\end{equation}
	\begin{equation} \label{Flow_gam_estim_2}
		\|\nabla \tilde{X}_t(\cdot, s)\|_N \lesssim \mathcal{M}(N, L_t - 3, \lambda_q, \ell_q^{-1}),
	\end{equation}
	with implicit constants depending on $\Gamma$, $M$, $\alpha$, and $N$.
\end{cor}

\begin{rem} \label{remark_flow_bd}
It is easy to see that the condition 	$\tau_q \|\tilde{u}_{q, \Gamma}\|_1 \lesssim \lambda_{q+1}^{-\alpha}$ implies Corollary \ref{Flow_gam_estim}  for $\tau = \tau_q$. Also, Proposition \ref{transport_estim} yields that	$\|\I - \nabla \tilde{\Phi}\|_0 \lesssim \lambda_{q+1}^{-\alpha}$  for $\tau = \tau_q$, which implies that $\|\nabla \tilde{\Phi}\|_0$ is bounded independently of the parameters in the construction.
\end{rem}

\begin{lem}(Stability estimate, \cite{GR23}) \label{flow_stabil}
	For $N \geq 0$ and $r \in \{0, 1\}$, let $\tilde{\Phi}_t$ and $\Phi_t$ be the backward flows of $\tilde{u}_{q, \Gamma}$ and $\bar{u}_q$, respectively, as defined in \eqref{Flow_t_gam} and \eqref{Flow_t}. If $|s - t| < \tau \leq (\|\tilde{u}_{q, \Gamma}\|_1 + \|\bar{u}_q\|_1)^{-1}$, then
	\begin{equation*}
		\|\nabla \Phi_t(\cdot, s) - \nabla \tilde{\Phi}_t(\cdot, s)\|_N + \|(\nabla \Phi_t(\cdot, s))^{-1} - (\nabla \tilde{\Phi}_t(\cdot, s))^{-1}\|_N \lesssim \tau \frac{\delta_{q+1} \lambda_{q+1}^{1+\delta} \lambda_q^{3+\delta} \ell_q^{-\alpha}}{\mu_{q+1}} \mathcal{M}(N, L_t - 3, \lambda_q, \ell_q^{-1})
	\end{equation*}
 holds	with implicit constants depending on $\Gamma$, $M$, $\alpha$, and $N$.
\end{lem}

	\section{The Nash step}
	\label{sec-Nash}
	\subsection{Temporal mollification of the stress error }
	\label{sec-molli}
Since the estimate of transport error involves estimates on the second material derivative of the old stress error, we need to mollify the stress along the backward flow. This subsection is standard and similar to that of  \cite{GR23}.
	
	Let $\widetilde{X}_t$ be the Lagrangian flow defined through \eqref{Lagr_t_gam}, and let $\rho$ be a standard temporal mollifier. We fix the material mollification scale as
	\[
	\ell_{t,q} = \delta_{q+1}^{-\frac{1}{2}}  \lambda_{q+1}^{-2-\delta}.
	\]
Then it holds that
	\[
	\ell_{t,q} < \mu_{q+1}^{-1} < \tau_q, \quad \text{for sufficiently small} \ \alpha > 0.
	\]
Define the regularized stresses as
	\[
	\bar{R}_{q,n} = \int_{-\ell_{t,q}}^{\ell_{t,q}} R_{q,n}(\widetilde{X}_t(x,t+s), t+s) \rho_{\ell_{t,q}}(s) \, ds.
	\]
Then  the following lemma holds

	\begin{lem} \label{le-bar-R}(\cite{GR23})
		For the regularized stress $\bar R_{q,n}$ and $\forall N \ge0$,  $r\in\left\{0,1,2\right\}$, 
		\begin{equation} \label{bar_R_1}
			\|\tilde D_{t,\Gamma}^r \bar R_{q, n}\|_N \lesssim \delta_{q+1,n}\lambda_{q+1}^{1+\delta}  \lambda_q^{- \alpha}\MMSA\MMStb.
		\end{equation}
		
		holds with the implicit constants depending on $\Gamma$, $M$, $\alpha$ and $N$.
	\end{lem}
	
	\subsection{The Nash perturbation and the  ASM-Reynolds system at the \texorpdfstring{$(q+1)^{\text{th}}$}{thth} stage}
	\label{sec-construct-Nash}
	In this section, we define the \texorpdfstring{$(q+1)^{\text{th}}$}{thth} Nash perturbation and derive the  ASM-Reynolds system at the \texorpdfstring{$(q+1)^{\text{th}}$}{thth} stage.
	
	First, define the amplitude functions (similar to  \eqref{def.a.coeff}) 
	\begin{equation}\label{def-bar-a}
		\bar a_{\xi, k, n}=\lambda_{q+1}\delta_{q+1,n}^{\frac12}\chi_k\bar{\gamma}_{\xi,k,n},
	\end{equation} 
	
	where $\gamma_{\xi, k, n}$ is given by Section \ref{sec-alge} as
	
	$$
	\begin{gathered}
		\left(\bar{\gamma}_{\xi^{(1)}, k, n}^2, \bar{\gamma}_{\xi^{(2)}, k, n}^2, {\hat{\bar\gamma}}_{\bar{\xi}^{(3)}, k, n}^2\right)=\mathcal{L}_{\left(\nabla \tilde{\Phi}^T \xi^{(1)}, \nabla \tilde{\Phi}^T \xi^{(2)}, \nabla \tilde{\Phi}^T \widehat{\xi^{(3)}}\right)}^{(-m)}\left(D-\frac{\bar{R}_{q, n}}{\delta_{q+1, n}}\right) \\
		D=\xi^{(1)} \mathring{\otimes} \nabla \bar{m}\left(\xi^{(1)}\right)+\xi^{(2)} \mathring{\otimes} \nabla \bar{m}\left(\xi^{(2)}\right)+\widehat{\xi^{(3)}} \mathring{\otimes} \nabla \bar{m}\left(\widehat{\xi^{(3)}}\right)
	\end{gathered},
	$$
	and here $\widetilde \Phi_k$ is the backward flow of $\tilde u_{q,\Gamma}$ starting from time $t=t_k$:
	\begin{equation}\notag
		\begin{cases}
			\partial_t\widetilde \Phi_k+\tilde u_{q,\Gamma}\cdot\nabla \widetilde \Phi_k=0,\\
			\widetilde \Phi_k|_{t=t_k}=x.
		\end{cases}
	\end{equation}
	To verify that $\bar{\gamma}_{\xi, k, n}$ is well-defined, one has by Lemma \ref{bar_R_1} and Remark \ref{remark_flow_bd} that
	
	$$
	\begin{gathered}
		\left\|\bar{R}_{q, n}\right\|_0 \leqslant \delta_{q+1, n} \lambda_q^{-\alpha} \Rightarrow\left\|\frac{\bar{R}_{q, n}}{\delta_{q+1, n}\la_{q+1}^{1+\delta}}\right\|_0 \leqslant C \lambda_q^{-\alpha} \leqslant c_2 \\
		\left\|\left(\nabla \tilde{\Phi}^T \xi^{(1)}, \nabla \tilde{\Phi}^T \xi^{(2)}, \nabla \tilde{\Phi}^T \widehat{\xi^{(3)}}\right)-\left(\xi^{(1)}, \xi^{(2)}, \widehat{\xi^{(3)}}\right)\right\|_0 \leqslant \tau_q\left\|\bar{u}_q\right\|_1 \leqslant C \lambda_{q+1}^{-\alpha} \leqslant c_1
	\end{gathered}
	$$
	
	provided that $a_0$ is large enough, here $C$ is a constant depending only on $M$ and constants $c_1$ and $ c_2$ are from Lemma \ref{le-geo}. Thus $\bar{\gamma}_{f, k, n}$ is well-defined.
	
	Now we define the Nash perturbation as
	\begin{equation} \label{def-W-Nash}
		w_{q+1}^{(p)} = \sum_{n=0}^{\Gamma-1} \sum_{k\in \mathbb{Z}_{q,n}} \sum_{\xi \in F} g_{\xi, k, n+1}(\mu_{q+1} \cdot) \Delta^{-1} \nabla^{\perp} P_{\approx \lambda_{q+1}} \left( \bar{a}_{\xi, k, n} \cos\left(\lambda_{q+1} \widetilde{\Phi}_k \cdot \xi\right) \right).
	\end{equation}
	
The multiplier \( P_{\approx \lambda_{q+1}} \) here is defined as follows (similarly to \cite{DGR24}): let \( A \subset A^{(1)} \subset A^{(2)} \subset A^{(3)} \subset \mathbb{R}^2 \) be four annuli centered at the origin with increasing radii, such that \( 10\xi, \frac{\xi}{10} \in A \) for all \( \xi \in F \). Define two smooth bump functions \( \chi, \tilde{\chi}:\mathbb{R}^2 \rightarrow \mathbb{R} \), where \( \chi(x) = 1 \) for \( x \in A \), and \( \supp \chi \subset A^{(1)} \), while \( \tilde{\chi}(x) = 1 \) for \( x \in A^{(2)} \) and \( \supp \tilde{\chi} \subset A^{(3)} \). For any function \( f: \mathbb{T}^2 \rightarrow \mathbb{R} \), set

\begin{equation} \label{Ptilde}
	P_{\approx \lambda_{q+1}} f := \sum_{k \in \mathbb{Z}^2} \chi(\lambda_{q+1}^{-1} k) \hat{f}(k) e^{i k \cdot x}, \quad
	\tilde{P}_{\approx \lambda_{q+1}} f := \sum_{k \in \mathbb{Z}^2} \tilde{\chi}(\lambda_{q+1}^{-1} k) \hat{f}(k) e^{i k \cdot x}.
\end{equation}

Then it holds that
\[
\tilde{P}_{\approx \lambda_{q+1}} P_{\approx \lambda_{q+1}} f = P_{\approx \lambda_{q+1}} f.
\]

	Note that the terms in the sum of (\ref{def-W-Nash}) have pair-wise disjoint temporal supports.  Indeed, the condition $\supp \bar{a}_{\xi,k,n}\subset\chi_k$ implies that  
	\begin{equation*}
		\supp_t \bar a_{\xi,k,n} \cap \supp_t \bar a_{\eta, j, m} \neq \emptyset \Rightarrow |k-j|\le1
	\end{equation*}
	Then if  $j$ and $k$ have distinct parities, 
	$\supp_t g_{\xi,k,n+1}\cap \supp_t g_{\eta, j, m+1}=\emptyset$ holds.
	Otherwise, if $j=k$, Lemma \ref{osc_prof} implies
	\[\supp_t g_{\xi,k,n+1}\cap \supp_t g_{\eta, j, m+1}\neq\emptyset\Rightarrow (\xi, n)=(\eta, m)\]

	Then one can define new velocity fields $v_{q+1}$ and $u_{q+1}$:
	\begin{equation}\label{v-q1}
		\begin{split}
			v_{q+1}& =v_q+w_{q+1}^{(t)} +w_{q+1}^{(p)},\\
			u_{q+1}&=u_q+T_1[w_{q+1}^{(t)}] +T_1[w_{q+1}^{(p)}].
		\end{split} 
	\end{equation}
Hence the  ASM-Reynolds system at the $(q+1)$-th stage is satisfied with 
	\begin{equation}\notag
		\partial_tv_{q+1}+u_{q+1}\cdot\nabla v_{q+1}-(\nabla v_{q+1})^{T}\cdot u_{q+1}+\nabla p_{q+1}=\div R_{q+1},
	\end{equation}
	with  the pressure given by
	$$p_{q+1}=p_{q,\Gamma}-\tilde{u}_{q,\Gamma}\cdot w^{(p)}_{q+1}$$
	and the new stress $R_{q+1}$ decomposed as
	\begin{equation}\label{R-q1-decom}
		R_{q+1}=R_{q+1,L}+R_{q+1,O}+R_{q+1,R},
	\end{equation}
	where the linear error $R_{q+1,L}$, oscillation error $R_{q+1, O}$ and residual error $R_{q+1,R}$ are determined respectively by 
	\begin{equation}\notag
		\begin{split}
			R_{q+1, L}&:=\div^{-1}\left( \tilde D_{t,\Gamma} w_{q+1}^{(p)}+T_1[w_{q+1}^{(p)}]\cdot \nabla \tilde v_{q,\Gamma}-\left(\nabla \tilde{v}_{q,\Gamma}\right)^{T}T_1[w^{(p)}_{q+1}]+\left(\nabla \tilde{u}_{q,\Gamma}\right)^{T}\cdot w^{(p)}_{q+1}\right),\\
					&\div R_{q+1, O}= \div S_{q,\Gamma}+T_1[w^{(p)}_{q+1}]\cdot\nabla w^{(p)}_{q+1}-\left(\nabla w^{(p)}_{q+1}\right)^{T}\cdot T_1[w^{(p)}_{q+1}],\\
			R_{q+1,R}&:=R_{q,\Gamma}+P_{q+1,\Gamma}+\div^{-1} \left(T_1[w_{q+1}^{(p)}]\cdot\nabla(v_{q,\Ga}-\tilde v_{q,\Ga})-\left(\nabla v_{q,\Ga}-\nabla\tilde v_{q,\Ga}\right)^{T}\cdot T_1[ w_{q+1}^{(p)}]\right)\\
			&+\div^{-1} \left(\left(u_{q,\Gamma}-\tilde{u}_{q,\Gamma}\right)\cdot\nabla w^{(p)}_{q+1}-\left(\nabla w_{q+1}^{(p)}\right)^{T} \cdot\left(u_{q,\Gamma}-\tilde{u}_{q,\Gamma}\right)\right).
		\end{split}
	\end{equation}

Note here that \( w_{q+1}^{(t)} \), \( T_1[w_{q+1}^{(t)}] \), \( R_{q,\Gamma} \), and \( S_{q,\Gamma} \) have temporal supports contained in 
\[
[-2+(\delta_q^{\frac12}\lambda_q^{2+\delta})^{-1}-2\Gamma\tau_q, -1-(\delta_q^{\frac12}\lambda_q^{2+\delta})^{-1}+2\Gamma\tau_q] \cup [1+(\delta_q^{\frac12}\lambda_q^{2+\delta})^{-1}-2\Gamma\tau_q, 2-(\delta_q^{\frac12}\lambda_q^{2+\delta})^{-1}+2\Gamma\tau_q].
\]
This implies that \( w_{q+1}^{(p)} \), \( T_1[w_{q+1}^{(p)}] \), and thus \( v_{q+1} \), \( u_{q+1} \), and \( R_{q+1} \) also have temporal supports contained in the above set by  \eqref{def-W-Nash}. Moreover, for any \( \alpha>0 \), there exists a large enough constant \( a_0>0 \), depending on \( \alpha \), \( b \), and \( \beta \), such that
\[
(\delta_{q+1}^{\frac12}\lambda_{q+1}^{2+\delta})^{-1}+2\Gamma (\delta_q^{\frac12}\lambda_q^{2+\delta})^{-1}\lambda_{q+1}^{-\alpha} < (\delta_q^{\frac12}\lambda_q^{2+\delta})^{-1}.
\]
This implies that \eqref{induct-sup} holds with \( q \) replaced by \( q+1 \).

	\subsection{Estimates on the Nash perturbation}
	\label{sec-error-Nash}
	We start with  collecting estimates on the amplitude functions of the Nash perturbation by the same argument as for Lemma \ref{a_estim}.
	
	\begin{lem}\label{le-amp-reg}
		For $  N \ge0, r\in\left\{0,1\right\}$, $\bar a_{\xi, k, n}$ defined in (\ref{def-bar-a})  the following estimate holds:
		\begin{equation} \label{bar_a_estim_1}
			\|\tilde D_{t,\Gamma}^r \bar a_{\xi, k, n}\|_N  \lesssim \delta_{q+1, n}^{\frac12} \la_{q+1} \mu_{q + 1}^r\MMStc
		\end{equation}
		
		where the implicit constants depend on $\Gamma$, $M$, $\alpha$, and $N$.
	\end{lem}
Next, one can estimate the perturbations \( w^{(p)}_{q+1} \) and \( w^{(t)}_{q+1} \) and fix the constants \( M_0 \) and \( M \) in Proposition \ref{prop.main}.

	\begin{lem}\label{le-est-Nash}
		For $ \forall N \in \{0,1,..., L_v\}$,	there exists a constant $M_0>0$ depending only on $\beta$ and $L$ such that
		\begin{equation}\label{est-u-q-p}
			\|w_{q+1}^{(p)}\|_{N}\leq \frac12M_0\delta_{q+1}^{\frac12}\lambda_{q+1}^{N},\quad\|T_1[w_{q+1}^{(p)}]\|_{N}\leq \frac12M_0\delta_{q+1}^{\frac12}\lambda_{q+1}^{1+\delta}\lambda_{q+1}^{N}
		\end{equation}
		and, thus, 
		\begin{equation}\label{est-u-q-total}
			\|w_{q+1}^{(p)}+w_{q+1}^{(t)}\|_{N}\leq M_0\delta_{q+1}^{\frac12}\lambda_{q+1}^{N},\quad\|T_1[w_{q+1}^{(p)}]+T_1[w_{q+1}^{(t)}]\|_{N}\leq M_0\delta_{q+1}^{\frac12}\lambda_{q+1}^{1+\delta}\lambda_{q+1}^{N}
		\end{equation}
	\end{lem}
	
	\begin{proof}
		
		Since $T_1$ is $(1+\delta)$-order operator, it suffices to estimate $w_{q+1}^{(p)}$ and $w_{q+1}^{(t)}$.
		
		First recalling (\ref{def-W-Nash}), $w_{q+1}^{(p)}$ is localized at frequencies $\approx \lambda_{q+1}$, thus there exist  constants $C_1$ and  $C_2$  depending only on $\beta$ and $L$ such that
		\begin{equation*}
			\|w_{q+1}^{(p)}\|_N \leq C_1\lambda_{q+1}^{N-1}\sup_{\xi, k, n} |g_{\xi, k, n+1}| \|\bar a_{\xi, k, n}\|_0 \leq C_1C_2\delta_{q+1}^{1/2} \lambda_{q+1}^{N}\leq \frac12M_0\delta_{q+1}^{1/2} \lambda_{q+1}^{N},
		\end{equation*}
		where, $M_0 = 2 C_1C_2$.
		
		Next, by the  Lemma \ref{w_t_estim} and $\mu_{q+1}=\delta_{q+1}^{\frac{1}{2}} \lambda_q^{1+\frac{\delta}{3}} \lambda_{q+1}^{1+\frac{2 \delta}{3}} \lambda_{q+1}^{4 \alpha}$, one can obtain 
		\begin{equation}\notag
			\|w_{q+1}^{(t)}\|_N \lesssim \frac{\delta_{q+1}\lambda_{q+1}^{1+\delta} \lambda_q \ell_q^{-N-\alpha}}{\mu_{q+1}}\lesssim \delta_{q+1}^{\frac12}\left(\frac{\lambda_{q+1}}{\lambda_{q}}\right)^{\frac{\delta}{3}}(\lambda_q\lambda_{q+1})^{\frac12N}
			\leq \frac12M_0\delta_{q+1}^{\frac12}\lambda_{q+1}^N
		\end{equation}
		for a sufficiently large choice of $a_0$ and one has used the fact that $\delta\leq0$.  
	\end{proof}
	
	Finally, we will show that the estimates \eqref{induct-u} and \eqref{induct-u'} hold at level \( q+1 \) and verify the validity of \eqref{eq.prop.main}.
	
	\begin{lem}\label{le-est-u-0}
		For $N \in \{1,..., L_v\}$, the following estimates hold:
		\begin{equation}\label{est-u-N}
			\|v_{q+1}\|_N\leq M\delta_{q+1}^{\frac12}\lambda_{q+1}^{N},\quad \|u_{q+1}\|_N\leq M\delta_{q+1}^{\frac12}\lambda_{q+1}^{1+\delta}\lambda_{q+1}^{N},
		\end{equation}
		\begin{equation}\label{est-u-incre}
			\|v_{q+1}-v_q\|_0 + \lambda_{q+1}^{-1} \|(v_{q+1} - v_q)\|_1
			\leq 2M\delta_{q+1}^{\frac12}.
		\end{equation}
	\end{lem}
	\begin{proof}
	
	 (\ref{est-u-incre}) follows directly from (\ref{est-u-q-total}). Also,	if $a_0$ is chosen sufficiently large, (\ref{est-u-q-total}) implies that
		$$
		\|v_{q+1}\|_N  \leq \|v_{q}\|_N   + \|w_{q+1}^{(t)} + w_{q+1}^{(p)}\|_N \leq M\delta_q^{\frac12}\lambda_q^{N}+M_0\delta_{q+1}^{\frac12}\lambda_{q+1}^{N} 	\leq M\delta_{q+1}^{\frac12}\lambda_{q+1}^{N},
		$$
		$$
		\|u_{q+1}\|_N  \leq \|u_{q}\|_N   + \|T_1[w_{q+1}^{(t)}] + T_1[w_{q+1}^{(p)}]\|_N \leq M\delta_q^{\frac12}\lambda_q^{1+\delta}\lambda_q^{N}+M_0\delta_{q+1}^{\frac12}\lambda_{q+1}^{1+\delta}\lambda_{q+1}^{N} 	\leq M\delta_{q+1}^{\frac12}\lambda_{q+1}^{1+\delta}\lambda_{q+1}^{N},
		$$
	
	\end{proof}

	\subsection{Estimates for the new Reynolds-stress \texorpdfstring{$R_{q+1}$}{rqpo}}
	\subsubsection{Estimates for the linear error \texorpdfstring{$R_{q+1, L}$}{rqpol}}
	Recall 
	\[
	\begin{split}
		R_{q+1, L}:=\underbrace{\operatorname{div}^{-1}\left(\tilde{D}_{t, \Gamma} w_{q+1}^{(p)}\right)}_{\text{transport error}}+\underbrace{\div^{-1}\left(T_1\left[w_{q+1}^{(p)}\right] \cdot \nabla \tilde{v}_{q, \Gamma}-\left(\nabla \tilde{v}_{q, \Gamma}\right)^T\cdot T_1\left[w_{q+1}^{(p)}\right]+\left(\nabla \tilde{u}_{q, \Gamma}\right)^T\cdot w_{q+1}^{(p)}\right)}_{\text{Nash error}}
	\end{split}
	\]
	To set up, we first give some  estimates on material derivatives. 
	
	\begin{lem}\label{le-mat-d}
		$\text { For } \forall N \ge0, r\in\left\{0,1\right\}$,  the following estimates hold:
		\begin{equation}\label{est-mat-psi}
			\|\bar D_{t}^r  \psi_{q+1}^{(t)}\|_N\lesssim \frac{\delta_{q+1}\lambda_{q+1}^{1+\delta}\ell_{q}^{-\alpha}}{\mu_{q + 1}}\mu_{q + 1}^{r}\MMStd,
		\end{equation}
		\begin{equation}\label{est-mat-bar-u1}
			\|\tilde D_{t,\Gamma}\nabla \tilde u_{q,\Gamma}\|_N\lesssim \delta_q\lambda_q^{4+2\delta}\MMStd,
		\end{equation}
		\begin{equation}\label{est-mat-bar-v1}
			\|\tilde D_{t,\Gamma}\nabla\left(\nabla^{\perp} \cdot \tilde{v}_{q, \Gamma}\right)\|_N\lesssim \delta_q\lambda_q^{4+\delta}\mathcal{M}\left(N, L_t-5, \lambda_q, \ell_q^{-1}\right),
		\end{equation}
		\begin{equation}\label{est-mat-bar-a1}
			\|\tilde D^2_{t,\Gamma} \bar a_{\xi, k, n}\|_N\lesssim \delta_{q+1,n}^{\frac12}\lambda_{q+1}\mu_{q + 1}\ell_{t,q}^{-1}\MMStd,
		\end{equation}
	 here implicit constants depend on $\Gamma, M, \alpha$ and $N$.
	\end{lem}
	\begin{proof}
		The estimates on $\bar D^2_{t,\Gamma} \bar a_{\xi, k, n}$ can be established similarly as for  Lemma 4.7 from \cite{GR23} and thus it suffices to prove \eqref{est-mat-psi}, \eqref{est-mat-bar-u1} and \eqref{est-mat-bar-v1}.
		
		Set  $\tilde{\theta}_{q,\Gamma}=-\nabla^{\perp} \cdot \tilde{v}_{q, \Gamma}$ and $\bar{\theta}_{q}=-\nabla^{\perp} \cdot \bar{v}_{q}$. Then
		$$\tilde D_{t,\Gamma}\nabla\left(\nabla^{\perp} \cdot \tilde{v}_{q, \Gamma}\right)=-\nabla\left(\tilde D_{t,\Gamma}\tilde{\theta}_{q,\Gamma}\right)+\nabla\tilde u_{q,\Gamma}\nabla\tilde{\theta}_{q,\Gamma}$$
		Rewrite the first term on the right hand side of above as
		$$-\nabla\left(\tilde D_{t,\Gamma}\tilde{\theta}_{q,\Gamma}\right)=-\nabla\left(\bar D_{t}\bar{\theta}_{q}\right)-\nabla\left(\bar D_{t}\nabla^{\perp}\cdot \tilde{P}_{\lesssim\ell_{q}^{-1}}{w}^{(t)}_{q+1}\right)-\nabla\left(\tilde{P}_{\lesssim\ell_{q}^{-1}}\tilde{w}^{(t)}_{q+1}\cdot\nabla\tilde{\theta}_{q,\Gamma}\right)$$
		and note that  \eqref{smoli_1''} implies
		$$\left\|\nabla\left(\bar D_{t}\bar{\theta}_{q}\right)\right\|_{N} \lesssim \delta_q \lambda_q^{4+\delta} \mathcal{M}\left(N, L_R-3, \lambda_q, \ell_q^{-1}\right).$$
		Using Corollary \ref{Gamma_velo_estim}  and Lemma \ref{w_t_estim}, one can conclude that
		\begin{equation*}
			\begin{aligned}
				\|\nabla\tilde u_{q,\Gamma}\nabla\tilde{\theta}_{q,\Gamma}\|_{N}\lesssim\|\tilde{u}_{q,\Gamma}\|_{N+1}\|\tilde{v}_{q,\Gamma}\|_2+\|\tilde{u}_{q,\Gamma}\|_{1}\|\tilde{v}_{q,\Gamma}\|_{N+2}\lesssim\delta_{q}  \lambda_q^{4+\delta} \mathcal{M}\left(N, L_t-4, \lambda_q, \ell_q^{-1}\right),
			\end{aligned}
		\end{equation*}
		\begin{equation*}
			\begin{aligned}
				\|\tilde{P}_{\lesssim\ell_{q}^{-1}}\tilde{w}^{(t)}_{q+1}\cdot\nabla\tilde{\theta}_{q,\Gamma}\|_{N+1}&\lesssim\|\tilde{P}_{\lesssim\ell_{q}^{-1}}\tilde{w}^{(t)}_{q+1}\|_{N+1}\|\tilde{v}_{q,\Gamma}\|_2+\|\tilde{P}_{\lesssim\ell_{q}^{-1}}\tilde{w}^{(t)}_{q+1}\|_{0}\|\tilde{v}_{q,\Gamma}\|_{N+3}\\
				&\lesssim\delta_{q}^{\frac12}  \lambda_q^{3} \frac{\delta_{q+1} \lambda_{q+1}^{1+\delta}\lambda_{q}^{1+\delta} \lambda_q \ell_q^{-\alpha}}{\mu_{q+1}}\mathcal{M}\left(N, L_t-5, \lambda_q, \ell_q^{-1}\right)\\
				&\lesssim\delta_{q}  \lambda_q^{4+\delta} \mathcal{M}\left(N, L_t-5, \lambda_q, \ell_q^{-1}\right),
			\end{aligned}
		\end{equation*}
		\begin{equation*}
			\begin{aligned}
				\left\|\bar{D}_t \nabla^{\perp} \cdot \tilde{P}_{\lesssim \ell_q^{-1}} w_{q+1}^{(t)}\right\|_{N+1} &\lesssim \delta_{q+1} \lambda_q^3 \lambda_{q+1}^{1+\delta} \ell_q^{-\alpha} \mathcal{M}\left(N, L_t-3, \lambda_q, \ell_q^{-1}\right)\\
				&\lesssim \delta_{q}\lambda_q^{4+\delta} \mathcal{M}\left(N, L_t-3, \lambda_q, \ell_q^{-1}\right),
			\end{aligned}
		\end{equation*}
		hence \eqref{est-mat-bar-v1} holds. 
		
	Next, note that
		$$\tilde D_{t,\Gamma}\nabla \tilde u_{q,\Gamma}=\bar{D}_t\nabla\bar{u}_q+\bar{D}_t\tilde{P}_{\lesssim \ell_q^{-1}}\tilde{w}_{q+1}+\tilde{P}_{\lesssim \ell_q^{-1}}\tilde{w}^{(t)}_{q+1}\cdot\nabla\tilde{u}_{q,\Gamma},$$
		$$\bar{D}_t\nabla\bar{u}_q=\bar{D}_t\nabla T_1[\bar{v}_q]=-\bar{D}_t T_1[\nabla\Delta^{-1}\nabla^{\perp}\bar{\theta}_q]=-T_1\nabla\Delta^{-1}\nabla^{\perp}\left(\bar{D}_t \bar{\theta}_q\right)-[\bar{u}_q\cdot\nabla,T_1\nabla\Delta^{-1}\nabla^{\perp}]\bar{\theta}_q.$$
		Then the desired estimate \eqref{est-mat-bar-u1} follows from Corollary \ref{Gamma_velo_estim} , Lemmas \ref{w_t_estim}, \ref{smoli_estim} and Proposition \ref{prop-commu}.
		
		It remains to prove \eqref{est-mat-psi}. Note that  
		$$  \psi_{q+1}^{(t)}=\sum_{n=0}^{\Gamma-1} \sum_{k \in \mathbb{Z}_{q, n}} \tilde{\chi}_k(t)  \psi_{k, n+1}(x, t),$$
		$$\bar D_{t}  \psi_{q+1}^{(t)}=\sum_{n=0}^{\Gamma-1} \sum_{k \in \mathbb{Z}_{q, n}} \left(\pa_t\tilde{\chi}_k(t)  \psi_{k, n+1}(x, t)+\tilde{\chi}_k(t) \bar{D}_t \psi_{k, n+1}(x, t)\right).$$
		Then \eqref{est-mat-psi} follows from Lemma \ref{psi_estim}.
	\end{proof}
	
	Now we estimate  the transport error.
	\begin{lem}\label{le-tran}
		It holds that
		\begin{equation}\label{est-tran1}
			\|\div^{-1} \tilde D_{t,\Gamma} w_{q+1}^{(p)}\|_N\lesssim  \delta_{q+1}\lambda_{q+1}^{1+\delta}\left(\frac{\lambda_{q}}{\lambda_{q+1}}\right)^{1+\frac{\delta}{3}}\lambda_{q+1}^{4\alpha}\lambda_{q+1}^{N}, \ \ \forall \ N\geq 0, 
		\end{equation}
		\begin{equation}\label{est-tran2}
			\|\tilde D_{t,\Gamma}\div^{-1}\tilde D_{t,\Gamma} w_{q+1}^{(p)}\|_N\lesssim  \ell_{t,q}^{-1}\delta_{q+1}\lambda_{q+1}^{1+\delta}\left(\frac{\lambda_{q}}{\lambda_{q+1}}\right)^{1+\frac{\delta}{3}}\lambda_{q+1}^{4\alpha}\lambda_{q+1}^{N}, \ \ \forall \ N\geq 0
		\end{equation}
		with implicit constants depending on $\Gamma, \alpha, M$ and $N$.
	\end{lem}
	
	\begin{proof}
		First , direct computations give
		\begin{equation}\notag
			\begin{split}
				\div^{-1} &\tilde D_{t,\Gamma} w_{q+1}^{(p)}\\
				=&\div^{-1} \tilde D_{t,\Gamma} \sum_{\xi,k,n}g_{\xi, k, n+1}\Delta^{-1}\nabla^{\perp}P_{\approx \lambda_{q+1}}\left(\bar a_{\xi, k, n}\cos(\lambda_{q+1}\widetilde \Phi_k\cdot\xi) \right)\\
				=& \underbrace{\div^{-1} \sum_{\xi,k,n}\Delta^{-1}\nabla^{\perp}P_{\approx \lambda_{q+1}}\left(\tilde D_{t,\Gamma}(g_{\xi, k, n+1} \bar a_{\xi, k, n})\cos(\lambda_{q+1}\widetilde \Phi_k\cdot\xi) \right)}_{T_1}\\
				&\qquad + \underbrace{\div^{-1} \sum_{\xi,k,n}[\tilde D_{t,\Gamma},\Delta^{-1}\nabla^{\perp} P_{\approx \lambda_{q+1}}]\left(g_{\xi, k, n+1}\bar a_{\xi, k, n}\cos(\lambda_{q+1}\widetilde \Phi_k\cdot\xi) \right)}_{T_2},
			\end{split}
		\end{equation}
		where one has used $\tilde{D}_{t,\Gamma}\widetilde \Phi_k=0$.
		 
		Note that the operator $\div^{-1} \Delta^{-1}\nabla^\perp$ is of order $-2$, one can get
		\begin{equation*}
			\|T_1\|_N \lesssim \lambda_{q+1}^{N-2}\sup_{\xi, k, n}\big( \mu_{q+1} \|\bar a_{\xi, k, n}\|_0 + \|\tilde D_{t, \Gamma} \bar a_{\xi, k, n}\|_0\big).
		\end{equation*}
		Since  $\tilde u_{q, \Gamma}$ is localized to frequencies $\lesssim \ell_q^{-1}$, thus one has
		\begin{equation*}
			T_2 = \div^{-1} \sum_{\xi,k,n} \tilde P_{\approx \lambda_{q+1}}[\tilde D_{t,\Gamma}, \Delta^{-1}\nabla^{\perp}P_{\approx \lambda_{q+1}}]\left(g_{\xi, k, n+1}\bar a_{\xi, k, n}\cos(\lambda_{q+1}\widetilde \Phi_k\cdot\xi) \right),
		\end{equation*}
		where $\tilde P_{\approx \lambda_{q+1}}$ is defined in \eqref{Ptilde}. Then we can use  Proposition \ref{prop-commu-loc} to get
		\begin{equation*}
			\|T_2\|_N \lesssim \lambda_{q+1}^{N-2} \|\nabla \tilde u_{q, \Gamma}\|_0 \sup_{\xi, k, n}\|\bar a_{\xi, k, n}\|_0.
		\end{equation*}
	which implies
		\begin{equation*}
			\begin{aligned}
				\|\div^{-1} \tilde D_{t, \Gamma} w_{q+1}^{(p)}\|_N &\lesssim \lambda_{q+1}^{N-2} \sup_{\xi, k, n}\big(\mu_{q+1} \|\bar a_{\xi, k, n}\|_0 + \|\tilde D_{t, \Gamma} \bar a_{\xi, k, n}\|_0 + \|\nabla \tilde u_{q, \Gamma}\|_0 \|\bar a_{\xi, k, n}\|_0\big)\\
				&\lesssim \delta_{q+1}\lambda_{q+1}^{1+\delta}\left(\frac{\lambda_{q}}{\lambda_{q+1}}\right)^{1+\frac{\delta}{3}}\lambda_{q+1}^{4\alpha}\lambda_{q+1}^{N},
			\end{aligned}
		\end{equation*}
	and hence \eqref{est-tran1}  holds.
		
	It remains to prove \eqref{est-tran2}. Direct computations show
		\begin{eqnarray*}
			\tilde D_{t,\Gamma}\div^{-1} \tilde D_{t,\Gamma} w_{q+1}^{(p)} &=& \underbrace{[\tilde D_{t,\Gamma},\div^{-1}\nabla^{\perp}\Delta^{-1} P_{\approx \lambda_{q+1}}]\sum_{\xi, k, n} \tilde D_{t, \Gamma}(g_{\xi, k, n+1} \bar a_{\xi, k, n}) \cos(\lambda_{q+1} \widetilde \Phi_k\cdot\xi)}_{T_{11}} \\
			&& + \underbrace{\div^{-1}\Delta^{-1}\nabla^{\perp} P_{\approx \lambda_{q+1}} \sum_{\xi, k, n} \tilde D^2_{t, \Gamma}(g_{\xi, k, n+1} \bar a_{\xi, k, n}) \cos(\lambda_{q+1} \widetilde \Phi_k\cdot\xi)}_{T_{12}} \\ 
			&& + \underbrace{[\tilde D_{t,\Gamma},\div^{-1} \tilde P_{\approx \lambda_{q+1}}] [\tilde D_{t, \Gamma}, \Delta^{-1}\nabla^{\perp}P_{\approx \lambda_{q+1}}] \sum_{\xi, k, n} g_{\xi, k, n+1} \bar a _{\xi, k, n} \cos(\lambda_{q+1} \widetilde \Phi_k\cdot\xi)}_{T_{21}} \\ 
			&& + \underbrace{\div^{-1} \tilde P_{\approx \lambda_{q+1}} \tilde D_{t, \Gamma} [\tilde D_{t, \Gamma},\Delta^{-1}\nabla^{\perp} P_{\approx \lambda_{q+1}}]\sum_{\xi, k, n} g_{\xi, k, n+1} \bar a _{\xi, k, n} \cos(\lambda_{q+1} \widetilde \Phi_k\cdot\xi)}_{T_{22}}
		\end{eqnarray*}
	By \ref{prop-commu}, the first three terms can be treated as
		\begin{equation*}
			\|T_{11}\|_N \lesssim \lambda_{q+1}^{N-2} \|\nabla \tilde u_{q, \Gamma}\|_0 \sup_{\xi, k, n} \big( \mu_{q+1} \|\bar a_{\xi, k, n}\|_0 + \|\tilde D_{t, \Gamma} \bar a_{\xi, k, n}\|_0 \big),
		\end{equation*}
		\begin{equation*}
			\|T_{12}\|_N \lesssim \lambda_{q+1}^{N-2} \sup_{\xi, k, n} \big( \mu_{q+1}^{2} \|\bar a_{\xi, k, n}\|_0 + \mu_{q+1} \|\tilde D_{t, \Gamma} \bar a_{\xi, k, n}\|_0 + \|\tilde D_{t, \Gamma}^2 \bar a_{\xi, k, n}\|_0\big),
		\end{equation*}
		\begin{equation*}
			\|T_{21}\|_N \lesssim \lambda_{q+1}^{N-2} \|\nabla \tilde u_{q, \Gamma}\|_0^2 \sup_{\xi, k, n}\|\bar a_{\xi, k, n}\|_0.
		\end{equation*}
		It remains to estimate $T_{22}$. Note that
		\begin{eqnarray*}
			T_{22}&=& \div^{-1} \tilde P_{\approx \lambda_{q+1}} \bigg[\underbrace{\left[\tilde D_{t, \Gamma}, [\tilde D_{t, \Gamma},\Delta^{-1}\nabla^{\perp} P_{\approx \lambda_{q+1}}]\right]  \sum_{\xi, k, n} g_{\xi, k, n+1} \bar a_{\xi, k, n} \cos (\lambda_{q+1} \widetilde \Phi_k\cdot\xi)}_{\tilde{T}_{22}} \\  
			&& + [\tilde u_{q, \Gamma} \cdot \nabla,\Delta^{-1}\nabla^{\perp} P_{\approx \lambda_{q+1}}]  \sum_{\xi, k, n}\tilde D_{t, \Gamma}( g_{\xi, k, n+1} \bar a_{\xi, k, n}) \cos (\lambda_{q+1} \widetilde \Phi_k\cdot\xi).
			\bigg].
		\end{eqnarray*}
		The second term above can be estimated by using Proposition \ref{prop-commu-loc} as before. For the first term, direct computations yield
		\begin{equation}\notag
			\begin{aligned}
				-\div^{-1} \tilde P_{\approx \lambda_{q+1}}\tilde{T}_{22}&=&\div^{-1} \tilde P_{\approx \lambda_{q+1}} \bigg[ 2\bar{D}_{t,\Gamma}\Delta^{-1}\nabla^{\perp} P_{\approx \lambda_{q+1}}\sum_{\xi, k, n} D_{t, \Gamma}( g_{\xi, k, n+1} \bar a_{\xi, k, n}) \cos (\lambda_{q+1} \widetilde \Phi_k\cdot\xi) \\  
				&& - \Delta^{-1}\nabla^{\perp} P_{\approx \lambda_{q+1}} \sum_{\xi, k, n}\tilde D_{t, \Gamma}^2( g_{\xi, k, n+1} \bar a_{\xi, k, n}) \cos (\lambda_{q+1} \widetilde \Phi_k\cdot\xi)\\
				&& -\tilde D_{t, \Gamma}^2 \Delta^{-1}\nabla^{\perp} P_{\approx \lambda_{q+1}} \sum_{\xi, k, n} g_{\xi, k, n+1} \bar a_{\xi, k, n} \cos (\lambda_{q+1} \widetilde \Phi_k\cdot\xi).
				\bigg]\\
				&=&\div^{-1} \tilde P_{\approx \lambda_{q+1}} \bigg[ 2\Delta^{-1}\nabla^{\perp} P_{\approx \lambda_{q+1}}\sum_{\xi, k, n} D_{t, \Gamma}^{2}( g_{\xi, k, n+1} \bar a_{\xi, k, n}) \cos (\lambda_{q+1} \widetilde \Phi_k\cdot\xi) \\ 
				&&+2[\tilde{u}_{q,\Gamma}, \Delta^{-1}\nabla^{\perp} P_{\approx \lambda_{q+1}}]\sum_{\xi, k, n} D_{t, \Gamma}( g_{\xi, k, n+1} \bar a_{\xi, k, n}) \cos (\lambda_{q+1} \widetilde \Phi_k\cdot\xi)\\
				&& - \Delta^{-1}\nabla^{\perp} P_{\approx \lambda_{q+1}} \sum_{\xi, k, n}\tilde D_{t, \Gamma}^2( g_{\xi, k, n+1} \bar a_{\xi, k, n}) \cos (\lambda_{q+1} \widetilde \Phi_k\cdot\xi)\\
				&& -\underbrace{\tilde D_{t, \Gamma}^2 \Delta^{-1}\nabla^{\perp} P_{\approx \lambda_{q+1}} \sum_{\xi, k, n} g_{\xi, k, n+1} \bar a_{\xi, k, n} \cos (\lambda_{q+1} \widetilde \Phi_k\cdot\xi)}_{\bar{T}_{22}}.
				\bigg].
			\end{aligned}
		\end{equation}
		The first three terms above can be estimated by using Proposition \ref{prop-commu-loc} as before. While the last term can be estimated by using Lemma \ref{microlem} to get 
		\begin{equation*}
			\begin{aligned}
				&\Delta^{-1}\nabla^{\perp} P_{\approx \lambda_{q+1}} \sum_{\xi, k, n} g_{\xi, k, n+1} \bar a_{\xi, k, n} \cos (\lambda_{q+1} \widetilde \Phi_k\cdot\xi)\\
				&=\sum_{\xi, k, n}\left[ -g_{\xi, k, n+1}\frac{\left((\nabla\tilde{\Phi}_k)^{T}\xi\right)^{\perp}}{\lambda_{q+1}|(\nabla\tilde{\Phi}_k)^{T}\xi|^2} \bar a_{\xi, k, n} \sin (\lambda_{q+1} \widetilde \Phi_k\cdot\xi)+g_{\xi, k, n+1} \delta\bar a_{\xi, k, n}e^{i\lambda_{q+1} \widetilde \Phi_k\cdot\xi}\right]
			\end{aligned}
		\end{equation*}
		where $\delta\bar a_{\xi, k, n}$ is lower-order term compared to $\lambda_{q+1}^{-1}\bar a_{\xi, k, n}$ and is explicitly given in Lemma \ref{microlem}, and the estimate for $\delta\bar a_{\xi, k, n}$ is given by Lemma \ref{le_est_delta_B}. Thus we have 
		\begin{equation*}
			\begin{aligned}
				&\div^{-1} \tilde P_{\approx \lambda_{q+1}}\bar{T}_{22}\\
				&=\div^{-1} \tilde P_{\approx \lambda_{q+1}}\sum_{\xi, k, n}\left[ \tilde D_{t, \Gamma}^2\left(-g_{\xi, k, n+1}\frac{\left((\nabla\tilde{\Phi}_k)^{T}\xi\right)^{\perp}}{\lambda_{q+1}|(\nabla\tilde{\Phi}_k)^{T}\xi|^2} \bar a_{\xi, k, n}\right) \sin (\lambda_{q+1} \widetilde \Phi_k\cdot\xi)+\tilde D_{t, \Gamma}^2\left(g_{\xi, k, n+1} \delta\bar a_{\xi, k, n}\right)e^{i\lambda_{q+1} \widetilde \Phi_k\cdot\xi}\right]
			\end{aligned}
		\end{equation*}
		Hence this term  satisfies the same estimates as $T_{12}$, which is in fact the largest term. Thus the proof is completed.
	\end{proof}

	We now consider the Nash error.  recall that
	\begin{align*}
		R^{\text{Nash}}=&\operatorname{div}^{-1}\left(T_1\left[w_{q+1}^{(p)}\right] \cdot \nabla \tilde{v}_{q, \Gamma}-\left(\nabla \tilde{v}_{q, \Gamma}\right)^T\cdot T_1\left[w_{q+1}^{(p)}\right]+\left(\nabla \tilde{u}_{q, \Gamma}\right)^T\cdot w_{q+1}^{(p)}\right)\\
		=&\operatorname{div}^{-1}\left(\left(\nabla^{\perp}\cdot\tilde{v}_{q,\Gamma}\right)T_1[w^{(p)}_{q+1}]^{\perp}+\left(\nabla \tilde{u}_{q, \Gamma}\right)^T \cdot w_{q+1}^{(p)}\right)
	\end{align*}
	and
	$$T_1[w^{(p)}_{q+1}]^{\perp}=
	\nabla\underbrace{\sum_{n=0}^{\Gamma-1} \sum_{k \in \mathbb{Z}_{q, n}} \sum_{\xi \in F} g_{\xi, k, n+1}\left(\mu_{q+1} \cdot\right) T_1\Delta^{-1}  P_{\approx \lambda_{q+1}}\left(\bar{a}_{\xi, k, n} \cos \left(\lambda_{q+1} \widetilde{\Phi}_k \cdot \xi\right)\right)}_{=:W^{(p)}_{q+1}}.
	$$
	Hence 
	\begin{align*}
		R^{\text{Nash}}=&\operatorname{div}^{-1}\left(\left(\nabla^{\perp}\cdot\tilde{v}_{q,\Gamma}\right)T_1[w^{(p)}_{q+1}]^{\perp}+\left(\nabla \tilde{u}_{q, \Gamma}\right)^T w_{q+1}^{(p)}\right)\\
		=&\operatorname{div}^{-1}\left(\left(\nabla^{\perp}\cdot\tilde{v}_{q,\Gamma}\right)\nabla W^{(p)}_{q+1} +\left(\nabla \tilde{u}_{q, \Gamma}\right)^T w_{q+1}^{(p)}\right)\\
		=&\operatorname{div}^{-1}\left(\nabla\left(\left(\nabla^{\perp}\cdot\tilde{v}_{q,\Gamma}\right) W^{(p)}_{q+1}\right)-\nabla\left(\nabla^{\perp}\cdot\tilde{v}_{q,\Gamma}\right) W^{(p)}_{q+1} +\left(\nabla \tilde{u}_{q, \Gamma}\right)^T w_{q+1}^{(p)}\right)
	\end{align*}
	
	The first term above can be absorbed into the pressure $p_{q+1}$, the estimates for the second and third terms are as follow,
	\begin{lem}\label{le-Nash}
		The following estimates
		\begin{equation}\label{est-Nash1}
			\left\|\div^{-1} \left(\nabla\left(\nabla^{\perp}\cdot\tilde{v}_{q,\Gamma}\right) W^{(p)}_{q+1}\right)\right\|_N\lesssim \delta_{q+1}\lambda_{q+1}^{1+\delta}\left(\frac{\lambda_{q}}{\lambda_{q+1}}\right)^{2-\beta}\lambda_{q+1}^{N}, \ \ \forall \ N\geq 0, 
		\end{equation}
		\begin{equation}\label{est-Nash2}
			\left\|\tilde D_{t,\Gamma}\div^{-1} \left(\nabla\left(\nabla^{\perp}\cdot\tilde{v}_{q,\Gamma}\right) W^{(p)}_{q+1}\right)\right\|_N\lesssim \mu_{q+1}\delta_{q+1}\lambda_{q+1}^{1+\delta}\left(\frac{\lambda_{q}}{\lambda_{q+1}}\right)^{2-\beta}\lambda_{q+1}^{N}, \ \ \forall \ N\geq 0,
		\end{equation}
		\begin{equation}\label{est-Nash3}
			\left\|\div^{-1} \left(\left(\nabla \tilde{u}_{q, \Gamma}\right)^T w_{q+1}^{(p)}\right)\right\|_N\lesssim \delta_{q+1}\lambda_{q+1}^{1+\delta}\left(\frac{\lambda_{q}}{\lambda_{q+1}}\right)^{2+\delta-\beta}\lambda_{q+1}^{N}, \ \ \forall \ N\geq 0, 
		\end{equation}
		\begin{equation}\label{est-Nash4}
			\left\|\tilde D_{t,\Gamma}\div^{-1} \left(\left(\nabla \tilde{u}_{q, \Gamma}\right)^T w_{q+1}^{(p)}\right)\right\|_N\lesssim \mu_{q+1}\delta_{q+1}\lambda_{q+1}^{1+\delta}\left(\frac{\lambda_{q}}{\lambda_{q+1}}\right)^{2+\delta-\beta}\lambda_{q+1}^{N}, \ \ \forall \ N\geq 0
		\end{equation}
		hold for implicit constants depending on $\Gamma, \alpha, M$ and $N$.
	\end{lem}
	\begin{proof}
Note that $\tilde v_{q,\Ga}$ is localized to frequencies $\lesssim \ell_q^{-1} \ll \lambda_{q+1}$. This implies
		\begin{eqnarray*}
			\left\|\div^{-1} \left(\left(\nabla \tilde{u}_{q, \Gamma}\right)^T\cdot  w_{q+1}^{(p)}\right)\right\|_N &=& \left\|\div^{-1} \tilde{P}_{\approx \lambda_{q+1}}\left(\left(\nabla \tilde{u}_{q, \Gamma}\right)^T\cdot  w_{q+1}^{(p)}\right)\right\|_N \\ 
			&\lesssim& \lambda_{q+1}^{N-2} \sup_{\xi, k, n}\|\bar a_{\xi, k, n}\|_0 \|\nabla \tilde u_{q,\Gamma}\|_0,
		\end{eqnarray*}
		and thus \eqref{est-Nash3} holds.
		
		Next, we write
		\begin{eqnarray*}
			\tilde D_{t, \Gamma} \div^{-1} \left(\left(\nabla \tilde{u}_{q, \Gamma}\right)^T\cdot w_{q+1}^{(p)}\right) &=& \underbrace{[\tilde D_{t, \Gamma},\div^{-1} \tilde P_{\approx \lambda_{q+1}}]\left(\left(\nabla \tilde{u}_{q, \Gamma}\right)^T\cdot  w_{q+1}^{(p)}\right)}_{T_1} \\ 
			&&+ \underbrace{\div^{-1} \tilde P_{\approx \lambda_{q+1}}\tilde D_{t, \Gamma} \left(\left(\nabla \tilde{u}_{q, \Gamma}\right)^T\cdot w_{q+1}^{(p)}\right)}_{T_2},
		\end{eqnarray*}
		$$
		T_2 = \div^{-1} \tilde P_{\approx \lambda_{q+1}}  \left(\left(\tilde D_{t, \Gamma} \nabla \tilde u_{q, \Gamma}\right)^{T}\cdot w_{q+1}^{(p)}\right)+\div^{-1} \tilde P_{\approx \lambda_{q+1}}  \left(\left(\nabla \tilde u_{q, \Gamma}\right)^{T}\cdot \tilde D_{t, \Gamma}w_{q+1}^{(p)}\right)\\
		$$
		This, together with  \ref{prop-commu-loc} yields
		\begin{eqnarray*}
			\bigg\|\tilde D_{t, \Gamma} \div^{-1} \left(\left(\nabla \tilde{u}_{q, \Gamma}\right)^T w_{q+1}^{(p)}\right)\bigg\|_N &\lesssim& \lambda_{q+1}^{N-1} \big( \|\tilde D_{t, \Gamma} \nabla \tilde u_{q, \Gamma}\|_0\|w^{(p)}_{q+1}\|_0  + \|\tilde D_{t, \Gamma} w_{q+1}^{(p)}\|_0\|\tilde{u}_{q,\Gamma}\|_1+ \|\nabla \tilde u_{q, \Gamma}\|_0^2\|w^{(p)}_{q+1}\|_0\big).
		\end{eqnarray*}
		Then  \eqref{est-Nash4} follows from Lemmas  \ref{est-mat-bar-u1} and \ref{le-tran}.
		
		The estimates \eqref{est-Nash1} and \eqref{est-Nash2} can be proven  similarly by using \eqref{est-mat-bar-v1} and Lemma \ref{le-tran}.
	\end{proof}

	\subsubsection{Estimates of \texorpdfstring{$R_{q+1, O}$}{m} } \label{sec.est.osc}
	Recall that
	\[\div R_{q+1, O}= \div S_{q,\Gamma}+T_1[w_{q+1}^{(p)}]^{\perp}\left(\nabla^{\perp}\cdot w _{q+1}^{(p)}\right).\]
As in  the standard convex integration scheme,  the main idea here is to use the low frequency part of the quadratic term above to cancel the stress. To be more specific, set
	\[A_{\xi, k, n}:=\frac{1}{4}\lambda^{\delta-1}_{q+1} a_{\xi, k, n}^2 \nabla \Phi^T_{k} \xi \mathring{\otimes} \nabla \bar{m}\left(\nabla \Phi^T_{k} \xi\right),\]
	\[\bar A_{\xi, k, n}:=\frac{1}{4} \lambda^{\delta-1}_{q+1} \bar{a}_{\xi, k, n}^2 \nabla \tilde{\Phi}^T_{k} \xi \mathring{\otimes} \nabla \bar{m}\left(\nabla \tilde{\Phi}^T_{k} \xi\right)\]
and
	\[w_{\xi,k,n}^{(p)}=g_{\xi, k, n+1}\Delta^{-1}P_{\approx \lambda_{q+1}}\left(\bar a_{\xi, k, n}\cos(\lambda_{q+1}\widetilde \Phi_k\cdot\xi) \right).\]
	Since $\{g_{\xi, k, n+1}\bar a_{\xi, k,n}\}_{\xi, k,n}$ have pair-wise disjoint supports, we can use the bilinear microlocal Lemma \ref{le-bilinear-odd} to compute
	\begin{equation}\notag
		\begin{split}
			&T_1[w_{q+1}^{(p)}]^{\perp}\left(\nabla^{\perp}\cdot w_{q+1}^{(p)}\right)=\sum_{\xi, k,n}T_1[\nabla w_{\xi,k,n}^{(p)}]\left((-\Delta) w_{\xi,k,n}^{(p)}\right)\\
			=& \, \div \left(\sum_{\xi, k,n}g^2_{\xi,k,n+1}\frac{1}{4} \lambda^{\delta-1}_{q+1} \bar{a}_{\xi, k, n}^2 \nabla \tilde{\Phi}^T_{k} \xi {\otimes} \nabla \bar{m}\left(\nabla \tilde{\Phi}^T_{k} \xi\right)+g^2_{\xi,k,n+1}\delta B_{\xi,k,n}\right)+\frac12\nabla\left(\sum_{\xi, k,n}T_1[w_{q+1}^{(p)}](-\Delta) w_{\xi,k,n}^{(p)}\right),\\
			=& \, \div \left(\sum_{\xi, k,n}g^2_{\xi,k,n+1}\bar A_{\xi, k, n} +g^2_{\xi,k,n+1}\delta B_{\xi,k,n}\right)+\frac12\nabla\left(\sum_{\xi, k,n}T_1[w_{q+1}^{(p)}](-\Delta) w_{\xi,k,n}^{(p)}\right)\\
			+&\frac18\la_{q+1}^{\delta-1}\nabla\left(\sum_{\xi, k,n}g^2_{\xi,k,n+1}\bar{a}^2_{\xi,k,n}\left((\nabla \tilde{\Phi}^T_{k} \xi)_1\pa_1\bar{m}(\nabla \tilde{\Phi}^T_{k} \xi)+(\nabla \tilde{\Phi}^T_{k} \xi)_2\pa_2\bar{m}(\nabla \tilde{\Phi}^T_{k} \xi)\right)\right)
		\end{split}
	\end{equation}
	where the error term $\delta B_{\xi,k,n}$  is given in the proof of Lemma~\ref{le-bilinear-odd} and the last term can be absorbed into $p_{q+1}$. Thus we have derived
	\begin{equation}\notag
		\begin{aligned}
		&T_1[w_{q+1}^{(p)}]^{\perp}\left(\nabla^{\perp}\cdot w_{q+1}^{(p)}\right)+ \div S_{q,\Gamma}=\div \sum_{\xi, k,n}g^2_{\xi,k,n+1}\left(\bar A_{\xi,k,n}-A_{\xi,k,n}\right)+\div \sum_{\xi, k,n}g^2_{\xi,k,n+1}\delta B_{\xi,k,n}\\
		&+\frac12\nabla\left(\sum_{\xi, k,n}T_1[w_{q+1}^{(p)}](-\Delta) w_{\xi,k,n}^{(p)}+\frac14\la_{q+1}^{\delta-1}g^2_{\xi,k,n+1}\bar{a}^2_{\xi,k,n}\left((\nabla \tilde{\Phi}^T_{k} \xi)_1\pa_1\bar{m}(\nabla \tilde{\Phi}^T_{k} \xi)+(\nabla \tilde{\Phi}^T_{k} \xi)_2\pa_2\bar{m}(\nabla \tilde{\Phi}^T_{k} \xi)\right)\right),
		\end{aligned}
	\end{equation}
	which yields the following explicit formula for the oscillation error:
	\begin{equation}\notag
		\begin{split}
			R_{q+1, O}:=
			% &\div^{-1}\nabla^{\perp}\cdot\Delta^{-1}[T[\theta_{q+1}^{(p)}]\cdot\nabla\theta_{q+1}^{(p)}+\nabla^{\perp}\cdot \div S_{q,\Gamma}]\\
			&\underbrace{\sum_{\xi, k,n}g^2_{\xi,k,n+1}\left(\bar A_{\xi,k,n}-A_{\xi,k,n}\right)}_{\text{flow error}}
			+\underbrace{\sum_{\xi, k,n}g^2_{\xi,k,n+1}\delta B_{\xi,k,n}}_{\text{main oscillation error}}.
		\end{split}
	\end{equation}
	Then  the following estimates hold:
	\begin{lem} \label{le_est_A}
	For	$ N \ge0, r\in\left\{0,1\right\}$, 
	\begin{equation}\notag
		\|\tilde D_t^{r} \bar A_{\xi, k, n} \|_N  \lesssim \delta_{q+1, n}\lambda_{q+1}^{1+\delta}\mu_{q + 1}^r\MMStc
	\end{equation}
   holds true	with implicit constants depending on $\Gamma$, $M$, $\alpha$, and $N$.
	
	Moreover the difference $\bar A_{\xi,k,n}-A_{\xi,k,n}$ satisfies
	\begin{equation}\label{A-diff}
		\|\bar A_{\xi,k,n}-A_{\xi,k,n}\|_0\lesssim \frac{\delta_{q+1}^2\lambda_{q+1}^{2\delta+2}\lambda_{q}^{3+\delta}\tau_q}{\mu_{q + 1}},
	\end{equation}
	\begin{equation}\notag
		\begin{split}
			\left\|\sum_{\xi,k,n}g_{\xi,k,n}^2(\bar A_{\xi,k,n}-A_{\xi,k,n})\right\|_N
			\lesssim&\ \frac{\delta_{q+1}^2\lambda_{q+1}^{2\delta+2}\lambda_{q}^{3+\delta}\tau_q}{\mu_{q + 1}}\lambda_{q+1}^N, \ \ \forall \ N\geq 0,\\
		\end{split}
	\end{equation}
	\begin{equation}\notag
		\begin{split}
			\left\|\tilde D_{t,\Gamma}\sum_{\xi,k,n}g_{\xi,k,n}^2(\bar A_{\xi,k,n}-A_{\xi,k,n})\right\|_N
			\lesssim&\ \mu_{q+1}\de_{q+1}\la_{q+1}^{1+\delta}\lambda_{q+1}^N, \ \ \forall \ N\geq 0,\\
		\end{split}
	\end{equation}
	with implicit constants depending on $\Gamma, M, N$ and $\alpha$. 
\end{lem}
\begin{proof}
	The proof is similar to that for Lemma 5.8 and Lemma 5.9 of \cite{DGR24} where one can use Corollary \ref{Flow_gam_estim} and Lemmas \ref{a_estim}, \ref{flow_stabil}, \ref{le-bar-R} and \ref{le-amp-reg} to modify that proof  and the details are omitted here.
\end{proof}

	\begin{lem} \label{le_est_delta_B}
		The following estimates for the main oscillation errors, 
		\begin{equation}\notag
				\|\tilde D_{t,\Gamma}^r \left(g^2_{\xi,k,n+1} \delta B_{\xi, k, n}\right) \|_N  \lesssim \delta_{q+1}\lambda_{q+1}^{1+\delta}\frac{\lambda_q}{\lambda_{q+1}}\mu_{q + 1}^r\lambda_{q+1}^N , \,\,\, \forall N\geq 0, r\in \left\{0,1\right\},
			\end{equation}
		
					\begin{equation}\notag
			\|\tilde D_{t,\Gamma}^{r} \left(g^2_{\xi,k,n+1} \delta \bar{a}_{\xi, k, n} \right)\|_N  \lesssim \delta_{q+1}^{\frac12}\lambda_{q+1}\frac{\lambda_q}{\lambda_{q+1}}\MMSA\lambda_{q+1}^N , \,\,\, \forall N\geq 0,  r\in \left\{0,1,2\right\}
		\end{equation}
	  hold	with implicit constants depending on $n$, $\Gamma$, $M$, and $N$. 
	\end{lem}
	\begin{proof}
	The proof of first estimate is similar to that for Lemma 5.10 in \cite{DGR24} and the second estimate is similar to the one for Proposition 4.6 in \cite{IM} where one will need Corollaries \ref{Gamma_velo_estim} and \ref{Flow_gam_estim} and Lemmas \ref{le-amp-reg} and  \ref{le-mat-d} to modify that proof, the details are omitted.
	\end{proof}
	
	\subsubsection{Estimates of \texorpdfstring{$R_{q+1, R}$}{m} }
	Recall that
	\[R_{q+1,R}=R_{q,\Gamma}+P_{q+1,\Gamma}+\div^{-1}\left(T_1[w_{q+1}^{(p)}]^{\perp}\left(\nabla^{\perp}\cdot(v_{q,\Ga}-\tilde v_{q,\Ga})\right)+\left(u_{q,\Gamma}-\tilde{u}_{q,\Gamma}\right)^{\perp}\left(\nabla^{\perp}\cdot w^{(p)}_{q+1}\right)\right),\]
	\begin{equation}\notag
		\begin{split}
			P_{q+1,\Gamma}=&\  \div^{-1} \left(T_1[w_{q+1}^{(t)}]^{\perp}\left(\nabla^{\perp}\cdot w_{q+1}^{(t)}\right) \right)+R_q-R_{q,0},\\
			&+\div^{-1} \left(T_1[w_{q+1}^{(t)}]^{\perp} \left(\nabla^{\perp}\cdot(v_{q}-\bar v_{q})\right)+(u_{q}-\bar u_{q})^{\perp}\left(\nabla^{\perp}\cdot w_{q+1}^{(t)}\right) \right).
		\end{split}
	\end{equation}
	Then we estimate these errors separately:
	\begin{lem}\label{le-est-glue}
		The final gluing error can be estimated as
		\begin{equation}\notag
			\begin{split}
				\|\tilde D_{t,\Gamma}^r R_{q,\Gamma}\|_N\lesssim\ \mu_{q + 1}^r\delta_{q+1}\lambda_{q+1}^{1+\delta}\left(\frac{\lambda_q}{\lambda_{q+1}}\right)^{1+\frac{\delta}{3}}\lambda_{q+1}^N, \ \ \forall N \geq 0, r\in\left\{0,1\right\}.
			\end{split}
		\end{equation}
		with implicit constants depending on $\Gamma, M,N$ and $\alpha$.
	\end{lem}
	\begin{proof}
	 We  consider first the case $r=0$. Recall that
		\[\delta_{q+1,\Gamma}=\delta_{q+1}\left(\frac{\lambda_q}{\lambda_{q+1}}\right)^{(1+\frac{2\delta}{3}-\beta)\Gamma}\leq \delta_{q+1}\left(\frac{\lambda_q}{\lambda_{q+1}}\right)^{1+\frac{\delta}{3}} \]
		Then  Proposition \ref{NewIter} implies 
		\[\|R_{q,\Gamma}\|_N \lesssim \delta_{q+1}\lambda_{q+1}^{1+\delta}\left(\frac{\lambda_q}{\lambda_{q+1}}\right)^{1+\frac{\delta}{3}}\lambda_{q+1}^N,\quad N\ge0.\]
		For the case $r=1$, direct computations yield
		\[\tilde D_{t,\Gamma}R_{q,\Gamma}=\bar D_t R_{q,\Gamma}+ \tilde{P}_{\lesssim\ell_{q}^{-1}}\tilde w_{q+1}^{(t)}\cdot\nabla R_{q,\Gamma}.\]
		Then Proposition \ref{NewIter} implies that the first term above satisfies
		\[\|\bar D_t R_{q,\Gamma}\|_N \lesssim \mu_{q + 1}\delta_{q+1}\lambda_{q+1}^{1+\delta}\left(\frac{\lambda_q}{\lambda_{q+1}}\right)^{1+\frac{\delta}{3}}\lambda_{q+1}^N.\]
		While   Proposition \ref{NewIter} and Lemma \ref{w_t_estim} yield
		\begin{equation}\notag
			\begin{aligned}
			\|\tilde{P}_{\lesssim\ell_{q}^{-1}}\tilde w_{q+1}^{(t)}\cdot\nabla R_{q,\Gamma}\|_N&\lesssim \|\tilde{P}_{\lesssim\ell_{q}^{-1}}\tilde w_{q+1}^{(t)}\|_N\|R_{q,\Gamma}\|_1+\|\tilde{P}_{\lesssim\ell_{q}^{-1}}\tilde w_{q+1}^{(t)}\|_0\|R_{q,\Gamma}\|_{N+1}\\ &\lesssim \frac{\delta_{q+1}\lambda_{q+1}^{1+\delta}\lambda_{q}^{3+\delta}\ell_{q}^{-\alpha}}{\mu_{q + 1}^2}\mu_{q + 1}\delta_{q+1}\lambda_{q+1}^{1+\delta}\left(\frac{\lambda_q}{\lambda_{q+1}}\right)^{1+\frac{\delta}{3}}\lambda_{q+1}^N.
			\end{aligned}
		\end{equation}
	   Then the desired estimate follows from
		\[\frac{\delta_{q+1}\lambda_{q+1}^{1+\delta}\lambda_{q}^{3+\delta}\ell_{q}^{-\alpha}}{\mu_{q + 1}^2}\leq 1.\]
	\end{proof}
	
	\begin{lem}\label{le-Newton-est}
		The following estimate holds
		
		\begin{equation}
			\|\tilde D_{t,\Gamma}^r\div^{-1} \left(T_1[w_{q+1}^{(t)}]^{\perp}\left(\nabla^{\perp}\cdot w_{q+1}^{(t)}\right) \right)\|_N\lesssim  \frac{\lambda_{q}^{3+\delta}\lambda_{q+1}^{2(1+\delta)}\delta_{q+1}^2\ell_{q}^{-2\alpha}}{\mu_{q + 1}^2}\mu_{q+1}^r\lambda_{q+1}^N, \ \ \forall \ N\geq 0, r\in\left\{0,1\right\}.
		\end{equation}
		with the implicit constants depending on $\Gamma, M$ and $N$.
	\end{lem}
	\begin{proof}
			For the case $r=0$, direct computations show that
		\begin{equation}\notag
			w_{q+1}^{(t)}=\sum_{n=0}^{\Gamma-1}w_{q+1,n+1}^{(t)}=\sum_{n=0}^{\Gamma-1}\sum_{k\in \mathbb Z_{q,n}}\tilde \chi_k(t)\nabla^{\perp}\psi_{k,n+1}=:\nabla^{\perp} \psi_{q+1}^{(t)},
		\end{equation}
		\[\psi_{q+1}^{(t)}=\sum_{n=0}^{\Gamma-1}\sum_{k\in \mathbb Z_{q,n}}\tilde \chi_k(t)\psi_{k,n+1}.\]
		By the argument for Lemma \ref{le-bilinear-odd} and the notations in Lemma \ref{lem.bilin}, we can actually get
		\begin{equation}\notag
			\begin{aligned}
				T_1[w_{q+1}^{(t)}]^{\perp} \left(\nabla^{\perp}\cdot w_{q+1}^{(t)}\right)&=T_1[\nabla \psi_{q+1}^{(t)}] \left((-\Delta) \psi_{q+1}^{(t)}\right)\\
				&=\frac12 \nabla\left(T_1[\psi_{q+1}^{(t)}] \left((-\Delta) \psi_{q+1}^{(t)}\right)\right)+\frac12 \underbrace{\left(T_1[\psi^{(t)}_{q+1}]\nabla\Delta\psi^{(t)}_{q+1}-T_1[\nabla\psi^{(t)}_{q+1}]\Delta\psi^{(t)}_{q+1}\right)}_{=\Lambda T[\psi^{(t)}_{q+1},\psi_{q+1}^{(t)}]}
			\end{aligned}
		\end{equation}
		Note that since $\div^{-1}\nabla$ and $\div^{-1}\Lambda$ are zero-order operator, while $T_1$ is $(1+\delta)$-order operator. Then we can use Lemmas \ref{lem.bilin} and \ref{psi_estim}  to conclude	\begin{equation}\notag
			\begin{split}
				&\|\div^{-1} \left(T_1[w_{q+1}^{(t)}]^{\perp} \left(\nabla^{\perp}\cdot w_{q+1}^{(t)}\right) \right)\|_N\lesssim\ \|\left(T_1[ \psi_{q+1}^{(t)}] \left(\Delta \psi_{q+1}^{(t)}\right) \right)\|_{N+\alpha}+\|T[\psi^{(t)}_{q+1},\psi^{(t)}_{q+1}]\|_{N+\alpha}\\
				\lesssim&\ \|\psi_{q+1}^{(t)}\|_{N+1+\delta+\alpha}\|\psi_{q+1}^{(t)}\|_{2+\alpha}+\|\psi_{q+1}^{(t)}\|_{1+\delta+\alpha}\|\psi_{q+1}^{(t)}\|_{N+2+\alpha}\\
				\lesssim&\ \frac{\lambda_{q}^{3+\delta}\lambda_{q+1}^{2(1+\delta)}\delta_{q+1}^2\ell_{q}^{-2\alpha}}{\mu_{q + 1}^2}\lambda_{q+1}^N
			\end{split}.
		\end{equation}
		
		To estimate the material derivative, we rewrite 
		\begin{equation*}
			\begin{aligned}
				&\tilde D_{t, \Gamma} \div^{-1}\left(T_1[w_{q+1}^{(t)}]^{\perp} \left(\nabla^{\perp}\cdot w_{q+1}^{(t)}\right) \right) \\
				&=\frac12[\bar{u}_{q}\cdot\nabla,\div^{-1}\nabla]\left(T_1[\psi_{q+1}^{(t)}] \left(\Delta \psi_{q+1}^{(t)}\right)\right)  + \frac12\div^{-1}\nabla \left( T_1[\bar D_{t} \psi_{q+1}^{(t)}]\left(\Delta \psi_{q+1}^{(t)}\right)\right)\\
				&+\frac12\div^{-1}\nabla\left([\bar{u}_{q}\cdot\nabla,T_1]\psi_{q+1}^{(t)}\left(\Delta\psi^{(t)}_{q+1}\right)+T_1[ \psi_{q+1}^{(t)}]\left(\Delta\bar D_{t} \psi_{q+1}^{(t)} \right)+T_1[\psi^{(t)}_{q+1}][\bar{u}_{q}\cdot\nabla,\Delta]\psi^{(t)}_{q+1}\right)\\
				&+\frac12\tilde{P}_{\lesssim\ell_{q}^{-1}}\tilde{w}^{(t)}_{q+1}\cdot\nabla\div^{-1}\nabla \left(T_1[\psi^{(t)}_{q+1}]\Delta \psi^{(t)}_{q+1}\right)\\
				&-\frac12\left(\div^{-1}\Lambda T[\bar{D}_t\psi^{(t)}_{q+1},\psi^{(t)}_{q+1}]+\div^{-1}\Lambda T[\psi^{(t)}_{q+1},\bar{D}_t\psi^{(t)}_{q+1}]+\div^{-1}\Lambda S_0[\bar{u}_q,\psi^{(t)}_{q+1},\psi^{(t)}_{q+1}]\right)\\
				&-\frac12[\bar{u}_q\cdot\nabla,\div^{-1}\Lambda]T[\psi^{(t)}_{q+1},\psi^{(t)}_{q+1}]-\frac12\tilde{P}_{\lesssim\ell_{q}^{-1}}\tilde{w}^{(t)}_{q+1}\cdot\nabla\div^{-1}\Lambda T[\psi^{(t)}_{q+1},\psi^{(t)}_{q+1}]
			\end{aligned}
		\end{equation*}
		where $S_0[u,\phi,\psi]=u\cdot\nabla T[\phi,\psi]-T[u\cdot\nabla\phi,\psi]-T[\phi,u\cdot\nabla\psi]$ is given by Lemma \ref{lem.trilin}.
		Then the desired estimate follows from Lemmas \ref{smoli_estim}, \ref{le-mat-d}, \ref{lem.trilin} and  Proposition \ref{prop-commu} together with  some tedious but standard calculation.
	\end{proof}
	
	\begin{lem}\label{le-molli-est}
		The spatial mollification error admits the following estimates
		\begin{equation}
			\|R_q-R_{q,0}\|_N\lesssim \frac{\delta_{q+1}\lambda_{q+1}^{1+\delta}\lambda_q}{\lambda_{q+1}}\lambda_{q+1}^N, \ \ N\in \{0,1, ... , L_R\},
		\end{equation}
		\begin{equation}
			\|\tilde D_{t,\Gamma}(R_q-R_{q,0})\|_N\lesssim \mu_{q + 1}\frac{\delta_{q+1}\lambda_{q+1}^{1+\delta}\lambda_q}{\lambda_{q+1}}\lambda_{q+1}^N, \ \ N\in \{0,1, ... , L_t\},
		\end{equation}
		with the implicit constants depending on $\Gamma, M, N$ and $\alpha$.
	\end{lem}
	\begin{proof}
	The proof is similar to that for Lemma 5.12 in \cite{DGR24}, one  needs Lemma \ref{smoli_estim} and Proposition \ref{prop-molli} to modify that proof. The details are omitted.
	\end{proof}
Finally it remains to estimate the following errors:
	\[ \div^{-1} \left(T_1[w_{q+1}^{(t)}]^{\perp} \left(\nabla^{\perp}\cdot(v_{q}-\bar v_{q})\right)+(u_{q}-\bar u_{q})^{\perp}\left(\nabla^{\perp}\cdot w_{q+1}^{(t)}\right) \right)=:\div^{-1}\Lambda S'[w^{(t)}_{q+1},v_q-\bar{v}_q]\,. \]
	\[ \div^{-1}\left(T_1[w_{q+1}^{(p)}]^{\perp}\left(\nabla^{\perp}\cdot(v_{q,\Ga}-\tilde v_{q,\Ga})\right)+\left(u_{q,\Gamma}-\tilde{u}_{q,\Gamma}\right)^{\perp}\left(\nabla^{\perp}\cdot w^{(p)}_{q+1}\right)\right)=:\div^{-1}\Lambda S'[w^{(p)}_{q+1},v_{q,\Gamma}-\tilde{v}_{q,\Gamma}]\,. \]
	where $S'$ is defined in  \eqref{bi-lin}.
	\begin{lem}
		For $r\in\left\{0,1\right\}$, $N\in \{0,1, ... , L_{[r]}\}$.	The following estimates 
		\begin{equation}\label{e1}
			\begin{split}
				\|\tilde D_{t, \Gamma}^r\div^{-1}\Lambda S'[w^{(t)}_{q+1},v_q-\bar{v}_q]\|_N
				\lesssim 
				\delta_{q+1}\lambda_{q+1}^{1+\delta}\left(\frac{\lambda_{q}}{\lambda_{q+1}}\right)^{2+\frac{2\delta}{3}-\beta}\mu_{q + 1}^r\lambda_{q+1}^{N},
			\end{split}
		\end{equation}
		\begin{equation}\label{e2}
			\begin{split}
				\|\tilde D_{t, \Gamma}^r\div^{-1}\Lambda S'[w^{(p)}_{q+1},v_{q,\Gamma}-\bar{v}_{q,\Gamma}]\|_N
				\lesssim 
				\delta_{q+1}\lambda_{q+1}^{1+\delta}\left(\frac{\lambda_{q}}{\lambda_{q+1}}\right)^{2-\beta}\mu_{q + 1}^r\lambda_{q+1}^{N}
			\end{split}
		\end{equation}
		hold with implicit constants depending on $\Gamma, M,N$ and $\alpha$.
	\end{lem}
	
	\begin{proof}
		We consider first  \eqref{e1} for the case $r=0$. For $N\le L_v-2$, Proposition \ref{prop-molli} gives 
		$$\|v_q-\bar{v}_q\|_N\lesssim \ell_{q}^2\|v_q\|_{N+2}\lesssim\delta_{q}^\frac12 \frac{\lambda_{q}}{\lambda_{q+1}}\lambda_{q}^N$$
		Then using Lemmas \ref{w_t_estim},  \ref{est,bilin2'} and noting that $\div^{-1}\Lambda$ is a zero-order operator, one has for that $N\le L_v-4$,
		\begin{equation}\notag
			\begin{aligned}
				\|\div^{-1}\Lambda S'[w^{(t)}_{q+1},v_q-\bar{v}_q]\|_{N+\alpha}&\lesssim\| S'[w^{(t)}_{q+1},v_q-\bar{v}_q]\|_{N+\alpha}\\
				&\lesssim \|w^{(t)}_{q+1}\|_{N+1+\delta+\alpha}\|v_q-\bar{v}_q\|_\alpha+\|w^{(t)}_{q+1}\|_{1+\delta+\alpha}\|v_q-\bar{v}_q\|_{N+\alpha}\\
				&+\|w^{(t)}_{q+1}\|_{N+\alpha}\|v_q-\bar{v}_q\|_{1+\delta+\alpha}+\|w^{(t)}_{q+1}\|_\alpha\|v_q-\bar{v}_q\|_{N+1+\delta+\alpha}\\
				&\lesssim \delta_{q}^\frac12 \frac{\lambda_{q}}{\lambda_{q+1}}\lambda_{q}^{\alpha}\frac{\delta_{q+1} \lambda_{q+1}^{1+\delta}\lambda_{q} \ell_q^{-\alpha}}{\mu_{q+1}}\lambda_{q+1}^{N}\\
				&\lesssim \delta_{q+1}\lambda_{q+1}^{1+\delta}\left(\frac{\lambda_{q}}{\lambda_{q+1}}\right)^{2+\frac{2\delta}{3}-\beta}\lambda_{q+1}^{N}.
			\end{aligned}
		\end{equation}
		Thus the desired estimate follows from $L_R\le L_v-4$.
		
		Next, we prove  \eqref{e1} for $r=1$. Rewrite
		\begin{equation*}
			\begin{aligned}
				&\tilde D_{t, \Gamma}\div^{-1}\Lambda S'[w^{(t)}_{q+1},v_q-\bar{v}_q]\\
				&=\underbrace{\tilde{P}_{\lesssim\ell_{q}^{-1}}\tilde{w}^{(t)}_{q+1}\cdot\nabla\div^{-1}\Lambda S'[w^{(t)}_{q+1},v_q-\bar{v}_q]}_{E_1}+\underbrace{[\bar{u}_q,\div\Lambda^{-1}]S'[w^{(t)}_{q+1},v_q-\bar{v}_q]}_{E_2}\\
				&+\underbrace{\div^{-1}\Lambda S'[\bar{D}_tw^{(t)}_{q+1},v_q-\bar{v}_q]}_{E_3}+\underbrace{\div^{-1}\Lambda S'[w^{(t)}_{q+1},\bar{D}_t\left(v_q-\bar{v}_q\right)]}_{E_4}\\
				&+\underbrace{\div^{-1}\Lambda\left(\bar{u}_q\cdot\nabla S'[w^{(t)}_{q+1},v_q-\bar{v}_q]-S'[\bar{u}_q\cdot\nabla w^{(t)}_{q+1},v_q-\bar{v}_q]-S'[w^{(t)}_{q+1},\bar{u}_q\cdot\nabla\left(v_q-\bar{v}_q\right)]\right).}_{E_5}
			\end{aligned}
		\end{equation*}
		For $N\le L_v-5$, one can use Lemma \ref{w_t_estim} to estimate $E_1$ as 
		\begin{equation*}
			\begin{aligned}
				\|E_1\|_{N+\alpha}&\lesssim \|\tilde{P}_{\lesssim\ell_{q}^{-1}}\tilde{w}^{(t)}_{q+1}\|_{N+\alpha}\|S'[w^{(t)}_{q+1},v_q-\bar{v}_q]\|_{1+\alpha}+\|\tilde{P}_{\lesssim\ell_{q}^{-1}}\tilde{w}^{(t)}_{q+1}\|_{\alpha}\|S'[w^{(t)}_{q+1},v_q-\bar{v}_q]\|_{N+1+\alpha}\\
				&\lesssim\frac{\delta_{q+1} \lambda_{q+1}^{1+\delta} \lambda_q^{2+\delta} \ell_q^{-\alpha}}{\mu_{q+1}}\delta_{q}^{\frac12}\frac{\lambda_{q}}{\lambda_{q+1}}\lambda_{q}^{\alpha}\frac{\delta_{q+1} \lambda_{q+1}^{1+\delta}\lambda_{q} \ell_q^{-\alpha}}{\mu_{q+1}}\lambda_{q}\lambda_{q+1}^{N}\\
				&\lesssim \delta_{q+1}\lambda_{q+1}^{1+\delta}\left(\frac{\lambda_{q}}{\lambda_{q+1}}\right)^{2+\frac{2\delta}{3}-\beta}\mu_{q + 1}\lambda_{q+1}^{N}
			\end{aligned}
		\end{equation*}
		For $N\le L_v-4$,  noting that $\div\Lambda^{-1}$ is zero-order operator, one can use Proposition \ref{prop-commu} to estimate $E_2$ as 
		\begin{equation*}
			\begin{aligned}
				\|E_2\|_{N+\alpha}&\lesssim \|\bar{u}_q\|_{N+1+\alpha}\|S'[w^{(t)}_{q+1},v_q-\bar{v}_q]\|_{\alpha}+\|\bar{u}_q\|_{1+\alpha}\|S'[w^{(t)}_{q+1},v_q-\bar{v}_q]\|_{N+\alpha}\\
				&\lesssim\delta_q\lambda_{q}^{2+\delta}\lambda_{q}^{\alpha}\delta_{q}^{\frac12}\frac{\lambda_{q}}{\lambda_{q+1}}\lambda_{q}^{\alpha}\frac{\delta_{q+1} \lambda_{q+1}^{1+\delta}\lambda_{q} \ell_q^{-\alpha}}{\mu_{q+1}}\lambda_{q}\lambda_{q+1}^{N}\\
				&\lesssim \delta_{q+1}\lambda_{q+1}^{1+\delta}\left(\frac{\lambda_{q}}{\lambda_{q+1}}\right)^{2+\frac{2\delta}{3}-\beta}\mu_{q + 1}\lambda_{q+1}^{N}
			\end{aligned}
		\end{equation*}
		Next, for $E_3$: for $N\le L_v-4$, one can use Lemma \ref{w_t_estim} to get  
		\begin{equation*}
			\begin{aligned}
				\|E_3\|_{N+\alpha}&\lesssim \|S'[\bar{D}_t w^{(t)}_{q+1},v_q-\bar{v}_q]\|_{\alpha}\\
				&\lesssim \|\bar{D}_t w^{(t)}_{q+1}\|_{N+1+\delta+\alpha}\|v_q-\bar{v}_q\|_\alpha+\|\bar{D}_t w^{(t)}_{q+1}\|_{1+\delta+\alpha}\|v_q-\bar{v}_q\|_{N+\alpha}\\
				&+\|\bar{D}_t w^{(t)}_{q+1}\|_{N+\alpha}\|v_q-\bar{v}_q\|_{1+\delta+\alpha}+\|\bar{D}_t w^{(t)}_{q+1}\|_\alpha\|v_q-\bar{v}_q\|_{N+1+\delta+\alpha}\\
				&\lesssim\delta_q^{\frac{1}{2}} \frac{\lambda_q}{\lambda_{q+1}} \lambda_q^\alpha \frac{\delta_{q+1} \lambda_{q+1}^{1+\delta} \lambda_q \ell_q^{-\alpha}}{\mu_{q+1}}\mu_{q + 1} \lambda_{q+1}^N\\
				&\lesssim \delta_{q+1}\lambda_{q+1}^{1+\delta}\left(\frac{\lambda_{q}}{\lambda_{q+1}}\right)^{2+\frac{2\delta}{3}-\beta}\mu_{q + 1}\lambda_{q+1}^{N}
			\end{aligned}
		\end{equation*}
		Next, we estimate $E_4$. Set $\theta_q=-\nabla^{\perp}\cdot v_q$ and  $v_q=-\nabla^{\perp}\Delta^{-1}\theta_q$. Then $D_t\theta_q=-\nabla^{\perp}\cdot\div R_q$ and one can rewrite
		\begin{equation*}
			\begin{aligned}
				\bar{D}_t\left(v_q-\bar{v}_q\right)&=D_t\left(v_q-\bar{v}_q\right)-\left(u_q-\bar{u}_q\right)\cdot\nabla\left(v_q-\bar{v}_q\right)\\
				&=-D_t\Delta^{-1}\nabla^{\perp}P_{\gtrsim \ell_{q}^{-1}}\theta_q-\left(u_q-\bar{u}_q\right)\cdot\nabla\left(v_q-\bar{v}_q\right)\\
				&=-\Delta^{-1}\nabla^{\perp}P_{\gtrsim \ell_{q}^{-1}}D_t\theta_q-[u_q\cdot\nabla,\Delta^{-1}\nabla^{\perp}P_{\gtrsim \ell_{q}^{-1}}]\theta_q-\left(u_q-\bar{u}_q\right)\cdot\nabla\left(v_q-\bar{v}_q\right)\\
				&=\Delta^{-1}\nabla^{\perp}P_{\gtrsim \ell_{q}^{-1}}\nabla^{\perp}\cdot\div R_q-[u_q\cdot\nabla,\Delta^{-1}\nabla^{\perp}P_{\gtrsim \ell_{q}^{-1}}]\theta_q-\left(u_q-\bar{u}_q\right)\cdot\nabla\left(v_q-\bar{v}_q\right)\\
				&=\Delta^{-1}\nabla^{\perp}\nabla^{\perp}\cdot\div (R_q-{R}_{q,0})-[u_q\cdot\nabla,\Delta^{-1}\nabla^{\perp}P_{\gtrsim \ell_{q}^{-1}}]\theta_q-\left(u_q-\bar{u}_q\right)\cdot\nabla\left(v_q-\bar{v}_q\right)
			\end{aligned}
		\end{equation*}
		Thus for $N\le L_R-4$, we can use the inductive assumptions \eqref{induct-R}, \eqref{induct-u} and  \eqref{induct-u'}, Proposition \ref{prop-molli} and Lemma \ref{GK_commu} to get 
		\begin{equation*}
			\begin{aligned}
				&\|\bar{D}_t\left(v_q-\bar{v}_q\right)\|_{N}\\
				&\lesssim\ell_{q}^2\|R_q\|_{N+3}+\ell_{q}^{2-N}\|u_q\|_1\|v_q\|_2+\ell_{q}^{4-	N}\|u_q\|_2\|v_q\|_3\\
				&\lesssim\ell_{q}^{-N}\left(\delta_{q+1}\lambda_{q+1}^{1+\delta}\lambda_{q}^{1-2\alpha}\frac{\lambda_{q}}{\lambda_{q+1}}+\delta_{q}\lambda_{q}^{2+\delta}\frac{\lambda_{q}}{\lambda_{q+1}}+\delta_{q}\lambda_{q}^{2+\delta}\left(\frac{\lambda_{q}}{\lambda_{q+1}}\right)^2\right)\\
				&\lesssim \delta_{q}\lambda_{q}^{2+\delta}\frac{\lambda_{q}}{\lambda_{q+1}}\ell_{q}^{-N}
			\end{aligned}
		\end{equation*}
		Then one can  conclude that for $N\le L_R-4$,
		\begin{equation*}
			\begin{aligned}
				\|E_4\|_{N+\alpha}&\lesssim \|S'[ w^{(t)}_{q+1},\bar{D}_t\left(v_q-\bar{v}_q\right)]\|_{\alpha}\\
				&\lesssim \|w^{(t)}_{q+1}\|_{N+1+\delta+\alpha}\|\bar{D}_t\left(v_q-\bar{v}_q\right)\|_\alpha+\| w^{(t)}_{q+1}\|_{1+\delta+\alpha}\|\bar{D}_t\left(v_q-\bar{v}_q\right)\|_{N+\alpha}\\
				&+\| w^{(t)}_{q+1}\|_{N+\alpha}\|\bar{D}_t\left(v_q-\bar{v}_q\right)\|_{1+\delta+\alpha}+\| w^{(t)}_{q+1}\|_\alpha\|\bar{D}_t\left(v_q-\bar{v}_q\right)\|_{N+1+\delta+\alpha}\\
				&\lesssim\delta_q\lambda_{q}^{2+\delta} \frac{\lambda_q}{\lambda_{q+1}} \lambda_q^\alpha\lambda_{q}^{1+\delta} \frac{\delta_{q+1} \lambda_{q+1}^{1+\delta} \lambda_q \ell_q^{-\alpha}}{\mu_{q+1}^2}\mu_{q + 1} \lambda_{q+1}^N\\
				&\lesssim \delta_{q+1}\lambda_{q+1}^{1+\delta}\left(\frac{\lambda_{q}}{\lambda_{q+1}}\right)^{2+\frac{2\delta}{3}-\beta}\mu_{q + 1}\lambda_{q+1}^{N}
			\end{aligned}
		\end{equation*}
		Next, we estimate $E_5$. Using notations in \eqref{S'-com} and direct computations yield
		\begin{equation}\notag
			\begin{aligned}
				&\bar{u}_q \cdot \nabla S^{\prime}\left[w_{q+1}^{(t)}, v_q-\bar{v}_q\right]-S^{\prime}\left[\bar{u}_q \cdot \nabla w_{q+1}^{(t)}, v_q-\bar{v}_q\right]-S^{\prime}\left[w_{q+1}^{(t)}, \bar{u}_q \cdot \nabla\left(v_q-\bar{v}_q\right)\right]\\
				&=\binom{[\bar{u}_q\cdot\nabla,\partial_1\Lambda^{-1}]\left((v_{q,1}-\bar{v}_{q,1}) T_1 w^{(t),1}_{q+1}- w^{(t),2}_{q+1}T_1(v_{q,2}-\bar{v}_{q,2})\right)}{[\bar{u}_q\cdot\nabla,\partial_1\Lambda^{-1}]\left((v_{q,2}-\bar{v}_{q,2}) T_1 w^{(t),1}_{q+1}+ w^{(t),2}_{q+1}T_1(v_{q,1}-\bar{v}_{q,1})\right)}\\
				&+\binom{[\bar{u}_q\cdot\nabla,\partial_2\Lambda^{-1}]\left((v_{q,1}-\bar{v}_{q,1}) T_1 w^{(t),2}_{q+1}+ w^{(t),1}_{q+1}T_1(v_{q,2}-\bar{v}_{q,2})\right)}{[\bar{u}_q\cdot\nabla,\partial_2\Lambda^{-1}]\left((v_{q,2}-\bar{v}_{q,2}) T_1 w^{(t),2}_{q+1}- w^{(t),1}_{q+1}T_1(v_{q,1}-\bar{v}_{q,1})\right)}\\
				&+\binom{S_1[\bar{u}_q, v_{q,1}-\bar{v}_{q,1}, w^{(t),1}_{q+1}]+S_1[\bar{u}_q,v_{q,2}-\bar{v}_{q,2},w^{(t),2}_{q+1}]}{S_2[\bar{u}_q,v_{q,1}-\bar{v}_{q,1},w^{(t),1}_{q+1}]+S_2[\bar{u}_q,v_{q,2}-\bar{v}_{q,2},w^{(t),2}_{q+1}]}
			\end{aligned}
		\end{equation}
		Then using the fact that $\div^{-1}\Lambda, \pa_1\Lambda^{-1}$ and $ \pa_2\Lambda^{-1}$ are zero operator, Proposition \ref{prop-commu}, and the trilinear estimates Lemma \ref{lem.trilin} , we  conclude that for $N\le L_v-5$, it holds that
		\begin{equation}\notag
			\begin{aligned}
				\|E_5\|_{N+\alpha}&\lesssim \|\bar{u}_q \cdot \nabla S^{\prime}\left[w_{q+1}^{(t)}, v_q-\bar{v}_q\right]-S^{\prime}\left[\bar{u}_q \cdot \nabla w_{q+1}^{(t)}, v_q-\bar{v}_q\right]-S^{\prime}\left[w_{q+1}^{(t)}, \bar{u}_q \cdot \nabla\left(v_q-\bar{v}_q\right)\right]\|_{N+\alpha}\\
				&\lesssim  \sum_{\left(N_1, N_2, N_3\right) \in \mathcal{N}}\|\bar{u}_q\|_{N_1+1+\alpha}\left(\|w_{q+1}^{(t)}\|_{N_2+1+\delta+\alpha}\|v_q-\bar{v}_q\|_{N_3+\alpha}+\|w^{(t)}_{q+1}\|_{N_2+\alpha}\|v_q-\bar{v}_q\|_{N_3+1+\delta+\alpha}\right) \\
				&+\|\bar{u}_q\|_{N_1+2+\delta+\alpha}\left(\|w^{(t)}_{q+1}\|_{N_2+\alpha}\|v_q-\bar{v}_q\|_{N_3+\alpha}\right)\\
				&\lesssim \de_q^{\frac12}\lambda_{q}^{2+\delta}\lambda_{q}^{\alpha}\frac{\delta_{q+1} \lambda_{q+1}^{1+\delta} \lambda_q \ell_q^{-\alpha}}{\mu_{q+1}}\delta_{q}^{\frac12}\frac{\lambda_{q}}{\lambda_{q+1}}\lambda_{q}^{\alpha}\lambda_{q}^{1+\delta}\lambda_{q+1}^{N}\\
				&\lesssim \delta_{q+1} \lambda_{q+1}^{1+\delta}\left(\frac{\lambda_q}{\lambda_{q+1}}\right)^{2+\frac{2 \delta}{3}-\beta} \mu_{q+1} \lambda_{q+1}^N
			\end{aligned}
		\end{equation}
		Hence  \eqref{e1} holds by the condition $L_t\le\min\left\{L_v-5,L_R-4\right\}$.
		
 Finally, the proof of estimate for \eqref{e2} is almost the same as that for \eqref{e1}, except that we need to apply Corollary \ref{Gamma_velo_estim} and Lemmas \ref{le-est-Nash}, \ref{le-tran} to  control   $w^{(p)}_{q+1}, v_{q,\Gamma}-\tilde{v}_{q,\Gamma}$. Here we omit the details and conclude the proof.
	\end{proof}
	
	\subsection{Estimates of the total Reynolds stress error \texorpdfstring{$R_{q+1}$}{m}  } Note that  Lemma \ref{le-est-u-0} shows that the inductive assumptions \eqref{induct-u}, \eqref{induct-u'} and \eqref{eq.prop.main}  holds at level $q+1$. Thus it only remains to shows \eqref{induct-R} holds at level $q+1$. We denote
	\begin{equation*}
		D_{t,q+1} = \pa_t + u_{q+1}\cdot \na.
	\end{equation*}
	\begin{cor}
	The following estimates hold:
			\begin{equation} \label{Spatial_estim_qplus1}
				\|R_{q+1}\|_N \leq \delta_{q+2}\lambda_{q+2}^{1+\delta} \lambda_{q+1}^{N-2\alpha}, \,\,\, \forall N \in \{0,1,...,L_R\},
			\end{equation}
			\begin{equation} \label{Material_estim_qplus1}
				\|D_{t, q+1} R_{q+1}\|_N \leq \delta_{q+2}\lambda_{q+2}^{1+\delta} \delta_{q+1}^{1/2}\lambda_{q+1}^{2+\delta} \lambda_{q+1}^{N-2\alpha}, \,\,\, \forall N \in \{0,1,...,L_t\}.
		\end{equation}
	\end{cor}
	\begin{proof}
		We can collect the estimates  for $R_{q+1,L}, R_{q+1,R}, R_{q+1,O}$ obtained before together  to get
		\begin{equation*}
			\begin{aligned}
				\|R_{q+1}\|_N &\lesssim \bigg(\delta_{q+1}\lambda_{q+1}^{1+\delta}\left(\left(\frac{\lambda_{q}}{\lambda_{q+1 }}\right)^{1+\frac{\delta}{3}}\lambda_{q+1}^{4\alpha}+\left(\frac{\lambda_{q}}{\lambda_{q+1 }}\right)^{2+\delta-\beta}+\frac{\lambda_{q}}{\lambda_{q+1 }}\right)+\frac{\delta_{q+1}^2 \lambda_{q+1}^{2 \delta+2} \lambda_q^{3+\delta} \tau_q}{\mu_{q+1}}\bigg)  \lambda_{q+1}^N.\\
				&\lesssim \bigg(\delta_{q+1}\lambda_{q+1}^{1+\delta}\left(\frac{\lambda_{q}}{\lambda_{q+1 }}\right)^{1+\frac{\delta}{3}}\lambda_{q+1}^{4\alpha}+\delta_{q+1}\lambda_{q+1}^{1+\delta}\left(\frac{\lambda_{q}}{\lambda_{q+1 }}\right)^{\beta-\frac{\delta}{3}}\bigg)  \lambda_{q+1}^N\\
				&\lesssim \delta_{q+2}\lambda_{q+2}^{1+\delta}\lambda_{q+1}^{N-2\alpha},\quad \forall N\in\left\{0,1,2,...,L_R\right\}
			\end{aligned}
		\end{equation*}
		by choosing $\alpha$ sufficiently small in terms of $b, \beta$, and here  one has  used the fact that $\mu_{q + 1}=\delta_{q+1}^\frac12 \lambda_{q+1}^{1+\frac{2\delta}{3}}\lambda_{q}^{1+\frac{\delta}{3}}\lambda_{q+1}^{4\alpha}$ and 
		 $$\beta< 1+\frac{2\delta}{3}, \quad\beta<\min\left\{\frac{\delta}{3}+\left(1+\frac{\delta}{3}\right) \frac{b+1}{2 b}, \frac{3(1+\delta)b-\delta}{3(2b-1)}\right\} $$
		It remains to  prove the material derivative estimate (\ref{Material_estim_qplus1}) corresponding to $u_{q+1}$. 
		Recall 
		\[u_{q+1}=u_q+T_1[w_{q+1}^{(t)}]+T_1[w_{q+1}^{(p)}]=\tilde u_{q,\Gamma}+(T_1[w_{q+1}^{(t)}]-\tilde{P}_{\lesssim \ell_{q}^{-1}} T_1[w_{q+1}^{(t)}])+(u_q-\bar u_q) +T_1[w_{q+1}^{(p)}]. \]
		Hence
		\begin{align*}
			\|D_{t,q+1} R_{q+1}\|_N &\lesssim \|\tilde D_{t,\Ga} R_{q+1}\|_N+ \|(u_q - \bar u_q) \cdot\na R_{q+1}\|_N\\
			&+ \|\left(T_1[w_{q+1}^{(t)}]-\tilde{P}_{\lesssim \ell_{q}^{-1}}T_1 [w_{q+1}^{(t)}]\right)\cdot\na R_{q+1}\|_N + \|T_1[w_{q+1}^{(p)}] \cdot \nabla R_{q+1}\|_{N}.
		\end{align*}
		
Then collecting the material derivative estimates for $R_{q,O}, R_{q,R}, R_{q,L}$ and noting that actually the cost of one material derivative $\tilde D_{t, \Gamma}$ is at most $\ell_{t,q}^{-1}= \delta_{q+1}^{\frac12}\lambda_{q+1}^{2+\delta}$, one can get that  for all $N \in \{0,1,..., L_t\}$,
$$
	\|\tilde D_{t, \Gamma} R_{q+1}\|_N \lesssim \delta_{q+2} \lambda_{q+2}^{1+\delta}\delta_{q+1}^{\frac12}\lambda_{q+1}^{2+\delta} \lambda_{q+1}^{N-2 \alpha}
$$
On the other hand, Proposition \ref{prop-molli} and the inductive assumption (\ref{induct-u}) imply that
\begin{equation*}
	\|\bar u_q - u_q\|_N\lesssim \ell_q^2\|u_q\|_{N+2} \lesssim \ell_q^2\delta_q^{1/2}\lambda_{q}^{3+\delta}\lambda_q^{N} \lesssim \delta_q^{1/2}\lambda_q^{1+\delta} \frac{\lambda_q}{\lambda_{q+1}}\lambda_{q+1}^N, \,\,\, \forall N \in \{0,1,...,L_v-2\}.
\end{equation*}
Thus for $N \leq L_t$, it holds that
\begin{eqnarray*}
	\|(\bar u_q - u_q) \cdot \nabla R_{q+1}\|_N &\lesssim& \|\bar u_q - u_q\|_N \|R_{q+1}\|_1 + \|\bar u_q - u_q\|_0\|R_{q+1}\|_{N+1} \\
	&\lesssim & \delta_{q+2}^{1/2}\lambda_{q+2}^{1+\delta} \de_q^{\frac12}\lambda_{q}^{2+\delta}\lambda_{q+1}^{N}\\
	&\lesssim & \delta_{q+2}^{1/2}\lambda_{q+2}^{1+\delta} \de_{q+1}^{\frac12}\lambda_{q+1}^{2+\delta}\lambda_{q+1}^{N-2\alpha}.
\end{eqnarray*}
Also one  can  use Proposition \ref{prop-molli} and Lemma \ref{w_t_estim} to get
\begin{equation*}
	\begin{aligned}
		&\|(T_1\left[w_{q+1}^{(t)}\right]-\tilde{P}_{\lesssim \ell_q^{-1}} T_1\left[w_{q+1}^{(t)}\right]) \cdot \nabla R_{q+1}\|_N \\
		&\lesssim \|T_1\left[w_{q+1}^{(t)}\right]-\tilde{P}_{\lesssim \ell_q^{-1}} T_1\left[w_{q+1}^{(t)}\right]\|_N \|R_{q+1}\|_1 + \|T_1\left[w_{q+1}^{(t)}\right]-\tilde{P}_{\lesssim \ell_q^{-1}} T_1\left[w_{q+1}^{(t)}\right]\|_0\|R_{q+1}\|_{N+1} \\
	&\lesssim  \frac{\delta_{q+1} \lambda_{q+1}^{1+\delta} \lambda_q^{2+\delta} \ell_q^{-\alpha}}{\mu_{q+1}}\frac{\lambda_{q}}{\lambda_{q+1}}\lambda_{q+1}\delta_{q+2}^{1/2}\lambda_{q+2}^{1+\delta}\lambda_{q+1}^{N-2\alpha}\\
	&\lesssim  \delta_{q+2}^{1/2}\lambda_{q+2}^{1+\delta} \de_{q+1}^{\frac12}\lambda_{q+1}^{2+\delta}\lambda_{q+1}^{N-2\alpha}.
	\end{aligned}
\end{equation*} 
Finally, we use  Lemma \ref{le-est-Nash} to get 
\begin{eqnarray*}
	\|T_1[w_{q+1}^{(p)}] \cdot \nabla R_{q+1}\|_N &\lesssim& \|T_1[w_{q+1}^{(p)}]\|_N \|R_{q+1}\|_1 + \|T_1[w_{q+1}^{(p)}]\|_0 \|R_{q+1}\|_{N+1} \\ 
	&\lesssim& \delta_{q+2}^{1/2}\lambda_{q+2}^{1+\delta} \de_{q+1}^{\frac12}\lambda_{q+1}^{2+\delta}\lambda_{q+1}^{N-2\alpha}.
\end{eqnarray*}
	\end{proof}
	
	\appendix
	\section{Standard toolbox for the convex integration}
	In this section, we collect some standard tools for the convex integration scheme. 
	We first collect estimates for solutions to the transport problem
	\begin{equation} \label{transport}
		\begin{cases}
			\partial_t f + u \cdot \nabla f = g, \\ 
			f \big | _{t = t_0} = f_0,
		\end{cases}
	\end{equation}
	
	\begin{prop}\cite[Proposition B.1]{BDLSV} \label{transport_estim}
		Assume that $|t-t_0| \|u\|_1 \leq 1$. Any solution $f$ to \eqref{transport} satisfies 
		\begin{equation*}
			\|f(\cdot, t)\|_0 \leq \|f_0\|_0 + \int_{t_0}^t \|g(\cdot, \tau)\|_0 d \tau, 
		\end{equation*}
		\begin{equation*}
			\|f(\cdot, t)\|_\alpha \leq 2 \big( \|f_0\|_\alpha + \int_{t_0}^t \|g(\cdot, \tau)\|_\alpha d\tau\big),
		\end{equation*}
		for $\alpha \in [0,1]$.
		More generally, for any $N \geq 1$ and $\alpha \in [0,1)$,
		\begin{equation*}
			[f(\cdot, t)]_{N+\alpha} \lesssim [f_0]_{N+\alpha} + |t-t_0| [u]_{N+\alpha} [f_0]_1 + \int_{t_0}^t \big( [g(\cdot, \tau)]_{N+\alpha} + (t-\tau) [u]_{N+\alpha} [g(\cdot, \tau)]_1\big) d\tau,
		\end{equation*}
		where the implicit constant depends on $N$ and $\alpha$.
		Consequently, the backwards flow $\Phi$ of $u$ starting at time $t_0$ satisfies 
		\begin{equation*}
			\|D\Phi(\cdot, t) - \I\|_0 \lesssim |t-t_0|[u]_1, 
		\end{equation*}
		\begin{equation*}
			[\Phi(\cdot, t)]_N \lesssim  |t - t_0| [u]_N, \,\,\,  \forall N \geq 2. 
		\end{equation*}
	\end{prop}

\begin{prop}\cite[Lemma 2.1]{CDLS12}\label{prop-molli}
	Let $\phi$ be a symmetric mollifier with $\int \phi=1$. For any smooth function $f$, the estimate 
	\[\|f-f*\phi_\ell\|_N\lesssim \ell^2\|f\|_{N+2}, \ \ \forall \ N\geq 0\]
	holds with implicit constant depending only on $N$.
\end{prop}
\begin{prop}\cite[Proposition A.3]{GR23} \label{CET_comm}
	Let $\phi$ be a standard mollifier. Then, for any $f, g \in C^\infty(\mathbb T^2)$ and $N \geq M \geq 0$,
	\begin{equation*}
		\|(fg)*\phi_\ell - (f*\phi_\ell)(g*\phi_\ell)\|_{N} \lesssim \ell^{2 - N + M}\big( \|f\|_{M+1}\|g\|_1 + \|f\|_{1}\|g\|_{M+1} \big),
	\end{equation*}
	where the implicit constant depends only on $N$ and $M$. 
\end{prop}
\begin{lem}\cite[Lemma A.6]{GK24}\label{GK_commu}
	\label{lem:com2} 
	For any $N\geq 0$, it holds that
	$$
	\norm{ [u_\ell\cdot \na, \Delta^{-1}\nabla^{\perp}{P}_{\ge \ell^{-1}}]H}_N \lec \ell^{2-N}\norm{\na u}_0\norm{\na H}_0\label{est.com1}.
	$$
\end{lem}

\begin{prop}\cite[Proposition A.1]{BDLSV} \label{comp_estim}
	Let $\Psi:\Omega \rightarrow \mathbb R$ and $u: \mathbb R^n \rightarrow \Omega$ be two smooth functions, with $\Omega \subset \mathbb R^N$. Then, for any $m \in \mathbb N \setminus \{0\}$, there exists a constant $C = C(m, N,n)$ such that 
	% \begin{equation*}
		%     [\Psi \circ u]_m \leq C\big([\Psi]_1 \|Du\|_{m-1} + \|D\Psi\|_{m-1} \|u\|_0^{m-1} \|u\|_m \big),
		% \end{equation*}
	\begin{equation*}
		[\Psi \circ u]_m \leq C\big([\Psi]_1 \|Du\|_{m-1} + \|D\Psi\|_{m-1} [u]_1^m \big).
	\end{equation*}
\end{prop}

\begin{prop}\cite[Proposition D.1]{BDLSV} \label{prop-commu}
	Let $\alpha\in(0,1)$, $N\geq 0$ and $u\in C^{N+\alpha}$ be a vector field. Let $T_K$ be a Calder\'on-Zygmund operator with kernel $K$. Then the estimates
	\begin{equation}\notag
		\begin{aligned}
			&\|[u\cdot\nabla, T_K]f\|_{N+\alpha}\lesssim \|u\|_{1+\alpha}\|f\|_{N+\alpha}+\|u\|_{N+1+\alpha}\|f\|_{\alpha},\\
			&\|[u\cdot\nabla, T_1]f\|_{N+\alpha}\lesssim \|u\|_{1+\alpha}\|f\|_{N+1+\delta+\alpha}+\|u\|_{N+1+\alpha}\|f\|_{1+\delta+\alpha}+\|u\|_{2+\delta+\alpha}\|f\|_{N+\alpha}+\|u\|_{N+2+\delta+\alpha}\|f\|_{\alpha},\\
			&	\|[u\cdot\nabla, T_{K}T_1]f\|_{N+\alpha}\lesssim \|u\|_{1+\alpha}\|f\|_{N+1+\delta+\alpha}+\|u\|_{N+1+\alpha}\|f\|_{1+\delta+\alpha}+\|u\|_{2+\delta+\alpha}\|f\|_{N+\alpha}+\|u\|_{N+2+\delta+\alpha}\|f\|_{\alpha}
		\end{aligned}
	\end{equation}
	hold for any $f\in C^{N+\alpha}$, with implicit constant depending on $\alpha, N, K$.
\end{prop}
\begin{proof}
the first estimate holds due to proposition D.1 in \cite{BDLSV} and the third estimate follows from the first and second estimates, the second estimate is actually a bi-linear Fourier multiplier estimate, which can be proven in the same way as the tri-linear Fourier multiplier estimate Lemma \ref{lem.trilin} and actually is simpler.
\end{proof}
\begin{prop}\cite[Lemma A.6]{BSV}\label{prop-commu-loc}
	Let $s\in \R$, $\la\geq 1$, and let $T_K$ be an order $s$ convolution operator localized at length scale $\la^{-1}$. That is, $T_K$ acts on smooth functions $f$ as 
	$$ T_K f(x) = \int_{\R^2} K(y)f(x-y)\,dy  $$
	for some kernel $K:\R^2\to\R$ that obeys
	$$ \||x|^a \na^b K(x)\|_{L^1(\R^2)} \lesssim \la^{b-a+s} $$
	for all $0\leq a , |b|\leq 1$ and some implicit constants $C=C(a,b)$. Then, for any smooth function $f:\T^2\to\R^2$ and smooth incompressible vector field $u:\T^2\to\R^2$, it holds that
	$$ \|[u\cdot\na, T_K]f\|_{0} \leq \la^s \|\na u\|_{0} \|f\|_0\,. $$
\end{prop}

	\section{Fourier analysis}\label{sec.bha}
	\subsection{Littlewood-Paley theory}
	In this section we recall some basic and standard Littlewood-Paley theory and will cite the notations and lemmas in \cite{DGR24}:
	
	Let $d \geq 2$ and $\psi:\mathbb R^d \rightarrow \mathbb R$ be a smooth, spherically-symmetric function such that $\supp \psi \subset B_{3/2}(0)$ and $\psi(x) = 1$ for all $x \in B_1(0)$. For $j \geq 0$, denote 
	\begin{equation*}
		\chi_j(\xi) = \psi\bigg( \frac{\xi}{2^j} \bigg) - \psi\bigg( \frac{\xi}{2^{j-1}} \bigg).
	\end{equation*}
Also let $\bar\chi_0:\mathbb R^d\setminus \{0\} \rightarrow \mathbb R$ be a smooth, compactly supported function satisfying $\bar \chi_0(\xi) = 1$ for all $\xi \in \supp \chi_0$. Set $\bar \chi_j(\xi) = \bar \chi_0(2^{-j}\xi)$ for $j \geq 0$ . 
	Then, for $f:\mathbb T^d \rightarrow \mathbb R$, we define the Littlewood-Paley projections 
	\begin{equation*}
		\Delta_j f = \sum_{k \in \mathbb Z^d} \chi_j(k) \hat{f}(k) e^{i k \cdot x}, \,\,\, j \geq 0,
	\end{equation*}
	and 
	\begin{equation*}
		\Delta_{-1} f = \hat f(0).
	\end{equation*}
	It will be notationally convenient to extend the definition to all $j \in \mathbb Z$ by 
	\begin{equation*}
		\Delta_{j} f = 0 \,\,\, \forall j < -1.
	\end{equation*}
	We will also use the low frequency projections: 
	\begin{equation*}
		S_j f = \sum_{i \leq j} \Delta_i f = \sum_{k \in \mathbb Z^2} \psi\bigg(\frac{k}{2^j}\bigg) \hat{f}(k) e^{ik\cdot x}.
	\end{equation*}
	It is not difficult to verify that for all $f \in C^\infty(\mathbb T^d)$, the following properties hold: 
	\begin{itemize}
		\item $f(x) = \sum_{j=-1}^\infty \Delta_j f(x)=\hat{f}(0) + \sum_{j=0}^\infty \Delta_j f(x)$;
		
		\item $\supp \widehat{\Delta_j f} \subset \mathbb Z^d \cap \big(B_{2^{j+1}}(0) \setminus B_{2^{j-1}}(0)\big)$, for all $j \geq 0$;
		
		\item $\Delta_j \Delta_k f = 0$, whenever $|j-k| > 1$. 
	\end{itemize}
	\begin{defn}
	 $T$ is said to be a Fourier multiplier operator of order $s \in \mathbb R$ if 
		\begin{equation*}
			Tf(x) = \sum_{k \in \mathbb Z^d \setminus \{0\}} m(k) \hat f(k) e^{ik\cdot x},
		\end{equation*}
		for a multiplier $m \in C^\infty (\mathbb R^d \setminus \{0\})$ which is $s$-homogeneous.
	\end{defn}
	
	\begin{lem}[Lemma B.2. \cite{DGR24}] \label{Bernstein}
		Let $f \in C^\infty(\mathbb T^d)$ and $T$ be an operator of order $s \in \mathbb R$. Then, it holds that 
		\begin{equation*}
			\|T \Delta_j f\|_0 \lesssim 2^{sj} \|\Delta_j f\|_0, 
		\end{equation*}
		with an implicit constant depending only on the operator $T$. 
	\end{lem}
	
	\begin{rem}
		Similar scaling arguments can be used to show that 
		\begin{equation*}
			\|\Delta_j f\|_N \lesssim \|f\|_N, \,\,\, \forall j \geq -1, N\geq 0;
		\end{equation*}
		\begin{equation*}
			\|S_j f\|_N \lesssim \|f\|_N, \,\,\, \forall j \geq -1, N\geq 0.
		\end{equation*}
	\end{rem}
	
	The following corollary  is an immediate consequence of Lemma \ref{Bernstein}.
	
	\begin{cor} [Corollary B.4. \cite{DGR24}]\label{Bern_cor}
		The following estimates
		\begin{equation*}
			\|\Delta_j f\|_N \lesssim 2^{Nj}\|\Delta_j f\|_0,
		\end{equation*}
		\begin{equation*}
			\|\Lambda \Delta_j f\|_0 \lesssim 2^j \|\Delta_j f\|_0 \lesssim \|\Lambda \Delta_j f\|_0,
		\end{equation*}
		hold for all $f \in C^\infty(\mathbb T^d)$.
	\end{cor}
	
	\begin{lem}[Lemma B.5. \cite{DGR24}] \label{Besov_char} Let $0<\alpha<1$. There exists a constant $C>0$, depending only on $\alpha$, such that 
		\begin{equation} \label{Besov_equiv}
			\frac{1}{C} \|f\|_\alpha \leq \sup_{j \geq -1} 2^{j\alpha} \|\Delta_j f\|_0 \leq C \|f\|_\alpha,
		\end{equation}
		for all $f \in C^\infty(\mathbb T^d)$.
	\end{lem}
	
	As a corollary, the standard estimate for $0$-homogeneous operators follows.
	
	\begin{cor}[Corollary B.6. \cite{DGR24}] \label{0-hom_estim}
		Let $0 < \alpha < 1$ and $T$ be a Fourier multiplier operator of order $0$. Then, there exists a constant $C > 0$ depending on $T$ and $\alpha$  such that 
		\begin{equation*}
			\|Tf\|_\alpha \leq C \|f\|_\alpha,
		\end{equation*}
		for all $f \in C^\infty (\mathbb T^d)$. 
	\end{cor}
	
	\begin{rem}
		Corollary \ref{0-hom_estim} can be used to deduce estimates also for multiplier operators of different orders. Particularly relevant will be the estimate 
		\begin{equation*}
			\|\Lambda f\|_\alpha \lesssim \|f\|_{1+\alpha},
		\end{equation*}
		which follows by writing $\Lambda = \mathcal{R}\cdot \nabla$, where $\mathcal{R}$ is the ($0$-order) Riesz transform. 
	\end{rem}

	\section{A bilinear microlocal lemma}
  In this subsection,	we establish now a bilinear microlocal lemma for the ASM nonlinearity which is very similar to that introduced in \cite{IM} and the notations and arguments are similar to that of \cite{DGR24}.

	We first recall the  definitions of the operator $P_{\approx \lambda}, \tilde{P}_{\approx \lambda}$. Let $F=\left\{\pm\xi^{(1)},\pm\xi^{(2)},\pm\widehat{\xi^{(3)}}\right\}$, here directions $\xi^{(1)}, \xi^{(2)},\widehat{\xi^{(3)}} $ are chosen such that the matrices $\xi^{(1)} \mathring{\otimes}\nabla\bar{m}(\xi^{(1)})$ , $\xi^{(2)} \mathring{\otimes}\nabla\bar{m}(\xi^{(2)})$ and $\widehat{\xi^{(3)}} \mathring{\otimes}\nabla\bar{m}(\widehat{\xi^{(3)}})$ are linear independent.
let \( A \subset A^{(1)} \subset A^{(2)} \subset A^{(3)} \subset \mathbb{R}^2 \) be four annuli centered at the origin with increasing radii, such that \( 10\xi, \frac{\xi}{10} \in A \) for all \( \xi \in F \). Define two smooth bump functions \( \chi, \tilde{\chi}:\mathbb{R}^2 \rightarrow \mathbb{R} \), where \( \chi(x) = 1 \) for \( x \in A \), and \( \supp \chi \subset A^{(1)} \), while \( \tilde{\chi}(x) = 1 \) for \( x \in A^{(2)} \) and \( \supp \tilde{\chi} \subset A^{(3)} \). For any function \( f: \mathbb{T}^2 \rightarrow \mathbb{R} \),  define
	\begin{equation*}
		\chi_\lambda (\xi) = \chi(\lambda^{-1} \xi),\quad\tilde{\chi}_\lambda (\xi) = \tilde{\chi}(\lambda^{-1} \xi)
	\end{equation*}
	and 
	\begin{equation}\label{def.lpproj}
		{P_{\approx \la} f}(\xi) := \sum_{\xi \in \mathbb Z^2} \chi_\lambda(\xi) \hat f(\xi) e^{i \xi \cdot x} \quad 		{\tilde{P}_{\approx \la} f}(\xi) := \sum_{\xi \in \mathbb Z^2} \tilde{\chi}_\lambda(\xi) \hat f(\xi) e^{i \xi \cdot x}.
	\end{equation}
	\begin{lem}[Bilinear Microlocal Lemma] \label{le-bilinear-odd}

		Let $\xi, \xi_1, \xi_2 \in F=\left\{\pm\xi^{(1)},\pm\xi^{(2)},\pm\widehat{\xi^{(3)}}\right\}  \subset \mathbb{Z}^2$, $a: \mathbb T^2 \rightarrow \mathbb R$ be a smooth function, and $\Phi:\mathbb T^2 \rightarrow \mathbb T$ a smooth diffeomorphism satisfying, for all $x \in \mathbb T^2$, $\nabla\Phi^{T}\xi\in A$, where $A$ is the annulus in the definition of $P_{\approx \lambda}$.  Define $\theta_{\xi}: \T^2 \to \R$ as
		\[\theta_\xi(x)=\Delta^{-1} P_{\approx \la}\left[a(x) \cos(\lambda\Phi(x)\cdot\xi)\right]\,.\]
		Then, there exists a smooth  $2$-tensor field $B_\la$ such that
		
		\begin{equation}\notag
			Q\left[\theta_1,\theta_2\right](x)= T_1[\nabla\theta_1](-\Delta) \theta_2+ T_1[\nabla
			\theta_2](-\Delta)\theta_1=\nabla\left(T_1[\theta_1](-\Delta)\theta_2\right)+\div B_{\lambda} [\theta_1,\theta_2](x) 
		\end{equation}
		where $\theta_{1}, \theta_2,\theta_{1}', \theta_2': \T^2 \to \R$ are smooth functions such that $\theta_1=P_{\approx\lambda} \theta_1' ,\theta_2=P_{\approx \lambda} \theta_2'$ and $B_\lambda$ satisfy that
		$$B_\lambda[\theta_\xi,\theta_{\xi}](x)=\frac{1}{2} \lambda^{\delta-1}a^2(x)\left(\nabla\Phi^{T}\xi\right)\otimes\nabla \bar{m}\left(\nabla\Phi^{T}\xi\right)+\delta B^{\xi}_\lambda (x)$$
		
		where $\bar{m}(\xi)=\frac{\tilde{m}(\xi)}{|\xi|^2}$ and $\delta B^\xi_\lambda(x)$ is an explicit error term.
	\end{lem}
	\begin{proof}

		we compute  
		\begin{eqnarray*}
			\widehat {Q[\theta_1, \theta_2]}(k) &=& i\sum_{j \in \mathbb Z^2}\bigg( |j|^2(k-j)\tilde{m}(k-j)\tilde{\chi}_\lambda (k-j)+|k-j|^2 j\tilde{m}(j)\tilde{\chi}_{\lambda}(j)\bigg)   \widehat \theta_1(k-j) \widehat \theta_2 (j) \\
			&=&ik\sum_{j \in \mathbb Z^2} |j|^2\tilde{m}(k-j)\tilde{\chi}_\lambda (k-j)  \widehat \theta_1(k-j) \widehat \theta_2 (j)\\
			&-&i\sum_{j \in \mathbb Z^2}j|j|^2 |k-j|^2\bigg( \bar{m}_\lambda(k-j)-\bar{m}_\lambda(-j)\bigg)   \widehat \theta_1(k-j) \widehat \theta_2 (j)
			\\
			&=&\widehat{\nabla\left(T_1[\theta_1](-\Delta)\theta_2\right)}(k)\underbrace{-\sum_{j \in \mathbb Z^2}\bigg( \bar{m}_\lambda(k-j)-\bar{m}_\lambda(-j)\bigg)    \widehat{\Delta\theta}_1(k-j)  \widehat{\nabla\Delta\theta}_2 (j)}_{\widehat {Q_1[\theta_1, \theta_2]}(k)}.
		\end{eqnarray*}
		where $\bar{m}_\lambda(\eta)=\bar{m}(\eta)\tilde{\chi}_\lambda(\eta)$. Thus it holds that
		\begin{eqnarray*}
			\widehat {Q_1[\theta_1, \theta_2]}(k) &=&  ik\cdot\sum_{j \in \mathbb Z^2}  \bigg( i\int_{0}^{1}\nabla \bar{m}_\lambda\left(\sigma(k-j)-(1-\sigma)j\right)d\sigma\bigg) \widehat {\Delta\theta}_1(k-j) \widehat {\nabla\Delta\theta}_2 (j) \\ 
			&=& ik\cdot\sum_{j \in \mathbb Z^2}K_\lambda(k-j,j) \widehat {\Delta\theta'}_1(k-j) \widehat {\nabla\Delta\theta'}_2 (j).
		\end{eqnarray*}
		where  $K_\lambda(k-j,j)=\int_{0}^{1}i\nabla \bar{m}_\lambda\left(\sigma(k-j)-(1-\sigma)j\right)d\sigma\chi_{\lambda}(k-j)\chi_{\lambda}(j)$.
		Then one can get  
		\begin{align*}
			{Q[\theta_1, \theta_2]}(x) &=  \div\sum_{j,k  \in \mathbb Z^2}K_{\lambda}(k-j,j) \widehat {\Delta\theta'}_1(k-j) \widehat {\nabla\Delta\theta'}_2 (j)e^{ik\cdot x}.\\
			&=  \div\sum_{j,k  \in \mathbb Z^2}K_{\lambda}(k,j) \widehat {\Delta\theta'}_1(k) \widehat {\nabla\Delta\theta'}_2 (j)e^{i(k+j)\cdot x}.\\
			&= \div  \int_{\R^2\times \R^2}\widehat{K}_{\lambda}(x-h_1,x-h_2)  \Delta\theta'_1(h_1) \nabla\Delta\theta'_2 (h_2)dh_1 dh_2\\
			&= \div  B_\lambda[\theta_1,\theta_2](x).
		\end{align*}
		where
		$$
		\widehat{K}_{\lambda}(x,y)=\frac{1}{(2\pi)^4}\int_{\mathbb R^2 \times \mathbb R^2}{K}_{\lambda}(\xi,\eta)e^{i(\xi\cdot x+\eta\cdot y)}d\xi d\eta
		$$
		and  $B_\lambda$ is the desired symmetric $2$-tensor field.
		
		It remains to prove that 
		$$B_\lambda[\theta_\xi,\theta_{\xi}](x)=\frac{1}{2} \lambda^{\delta-1}a^2(x)\left(\nabla\Phi^{T}\xi\right)\otimes\nabla \bar{m}\left(\nabla\Phi^{T}\xi\right)+\delta B^\xi_\lambda (x)$$
		Indeed, let $\Theta'_\xi=\Delta^{-1}\left(a(x)e^{i\lambda\Phi(x)\cdot\xi}\right)$ and $\Theta_\xi=P_{\approx \lambda}\Theta'_\xi$.   Then $\theta_\xi=\frac{\Theta_\xi+\Theta_{-\xi}}{2}$ and one has
		\begin{equation*}
			\begin{aligned}
				B_\lambda[\theta_\xi,\theta_{\xi}](x) =&\frac{1}{4}\sum_{\zeta,\eta\in\left\{\xi,-\xi\right\}}B_\lambda[\Theta_\eta,\Theta_{\zeta}](x)\\
				=&\frac{1}{4}\sum_{\zeta,\eta\in\left\{\xi,-\xi\right\}}\int_{\R^2\times \R^2}\widehat{K}_{\lambda}(h_1, h_2)  a(x-h_1)e^{i\lambda\eta\cdot\Phi(x-h_1)}\nabla\left(a(x-h_2)e^{i\lambda\zeta\cdot\Phi(x-h_2)}\right)dh_1 dh_2.\\
				=&\frac{1}{4}\sum_{\zeta,\eta\in\left\{\xi,-\xi\right\}}\int_{\R^2\times \R^2}\nabla_{h_2}\widehat{K}_{\lambda}(h_1, h_2)  a(x-h_1)e^{i\lambda\eta\cdot\Phi(x-h_1)} a(x-h_2)e^{i\lambda\zeta\cdot\Phi(x-h_2)}dh_1 dh_2.\\
				=&\frac{1}{4}\sum_{\zeta,\eta\in\left\{\xi,-\xi\right\}}\int_{\R^2\times \R^2}\widehat{iy\otimes K}_{\lambda}(h_1, h_2)  a(x-h_1)e^{i\lambda\eta\cdot\Phi(x-h_1)} a(x-h_2)e^{i\lambda\zeta\cdot\Phi(x-h_2)}dh_1 dh_2.\\
			\end{aligned}
		\end{equation*}
		The desired expansion will follow as a consequence of Taylor's formula: 
		\begin{equation*}
			\Phi(x-h) = \Phi(x) - \nabla \Phi(x)h + R_{\Phi}(x, h),
		\end{equation*}
		where the remainder is given by
		\begin{equation*}
			R_{\Phi}(x, h) = \sum_{j,k = 1}^2 h^j h^k \int_{0}^1 (1-s) \partial_j \partial_k \Phi(x-sh) ds.
		\end{equation*}
		Then we have
		\begin{equation*}
			a(x-h_1)e^{i\lambda \Phi(x-h_1)\cdot \eta} = \big(a(x) +\underbrace{(a(x-h_1)e^{i\lambda R_{\Phi}(x,h_1)\cdot\eta} - a(x))}_{=:Y^\eta_\lambda(x,h_1)} \big) e^{i\lambda \Phi(x)\cdot\eta}e^{-i \lambda \nabla \Phi(x)h_1\cdot\eta}.
		\end{equation*}
		\begin{equation*}
			a(x-h_2) e^{i\lambda \Phi (x-h_2)\cdot \zeta}= \big(a(x) +\underbrace{(a(x-h_2) e^{i\lambda R_{\Phi}(x,h_2)\cdot\zeta} - a(x)}_{=:{Y}^\zeta_\lambda(x,h_2)} \big) e^{i\lambda \Phi(x)\cdot\zeta}e^{-i \lambda \nabla \Phi(x)h_2\cdot\zeta}.
		\end{equation*}
		For convenience of notations, we also introduce the function
		\begin{equation*}
			Y_\lambda^{\eta,\zeta}(x,h_1,h_2) = a(x)({Y}_\lambda^\zeta(x,h_1) +  Y_\lambda^\eta(x,h_2)) +  Y_\lambda^\eta(x,h_1){Y}_\lambda^\zeta(x,h_2).
		\end{equation*}
		With this expansion, one can get
		\begin{eqnarray*}
			B_\lambda(x) &=& \frac{1}{4}\sum_{\zeta,\eta\in\left\{\xi,-\xi\right\}}e^{i\lambda\Phi(x)\cdot(\eta+\zeta)}\int_{\mathbb R^2 \times \mathbb R^2} a^2(x) \widehat{iy\otimes K}_{\lambda}(h_1, h_2) e^{-i\lambda (\nabla \Phi^{T}\eta)\cdot h_1} e^{-i\lambda (\nabla \Phi^{T}\zeta)\cdot h_2} dh_1 dh_2 \\ 
			&& + \frac{1}{4}\sum_{\zeta,\eta\in\left\{\xi,-\xi\right\}}\underbrace{e^{i\lambda\Phi(x)\cdot(\eta+\zeta)}\int_{\mathbb R^2 \times \mathbb R^2} Y_\lambda^{\eta,\zeta}(x,h_1,h_2)\widehat{iy\otimes K}_{\lambda}(h_1, h_2) e^{-i\lambda (\nabla \Phi^{T}\eta)\cdot h_1} e^{-i\lambda (\nabla \Phi^{T}\zeta)\cdot h_2} dh_1 dh_2}_{=:\delta B_{\lambda}^{\eta,\zeta}(x)}. \\
			&=&\frac{1}{4} e^{i\lambda\Phi(x)\cdot(\eta+\zeta)}\sum_{\zeta,\eta\in\left\{\xi,-\xi\right\}} a^2(x) \left(iy\otimes K_{\lambda}\right)(\lambda \nabla \Phi^{T}\eta, \lambda \nabla \Phi^{T}\zeta) + \delta B_\lambda^{\xi}(x).\\
			&=&\frac{1}{2}\lambda^{\delta-1} a^2(x)\left(\nabla\Phi^{T}\xi\right)\otimes\nabla \bar{m}\left(\nabla\Phi^{T}\xi\right) + \delta B_\lambda^{\xi}(x).
		\end{eqnarray*}
		where
		$$\delta B_\lambda^{\xi}(x)=\frac{1}{4}\sum_{\zeta,\eta\in\left\{\xi,-\xi\right\}}\delta B_{\lambda}^{\eta,\zeta}(x).$$
		and one has  used the fact that $\bar{m}$ is homogeneous of order $\delta-1$, $K_\lambda (v,v)=0$ and $K(v,-v)=i\nabla \bar{m} (v)$ and $\chi_\lambda=1$ on the range $\nabla\Phi^{T}\xi, v\in A$.

	\end{proof}
	
	In order to estimate the error term $\delta B_\xi$, we will need the following estimate, obtained by scaling for the physical space kernel $K_\la$. For $\bar h = (h_1,h_2) \in \R^2\times \R^2$, we have for every $m\in \N$
	\begin{equation}\label{est.K}
		\la^m \||\bar h|^m K_\la (h_1,h_2)\|_{L^1(\R^2\times \R^2)} \lesssim_m \la^{-1}\,.
	\end{equation}
	\begin{rem}
One may use the Taylor expansion for $a(x-h)e^{i\lambda\Phi(x-h)\cdot\eta}$ to higher order , but this will make it harder to analyze our algebraic lemma. However, here we are lucky enough to use the Taylor expansion up to second order without affecting  the construction of our algebraic lemma as seen below. Note that
		\begin{equation*}
	a(x-h_1)e^{i\lambda \Phi(x-h_1)\cdot \eta} = \big(a(x)-\nabla a(x)\cdot h_1 +\underbrace{(a(x-h_1)e^{i\lambda R_{\Phi}(x,h_1)\cdot\eta} - a(x)+\nabla a(x)\cdot h_1}_{=:Y^\eta_\lambda(x,h_1)} \big) e^{i\lambda \Phi(x)\cdot\eta}e^{-i \lambda \nabla \Phi(x)h_1\cdot\eta}.
\end{equation*}
\begin{equation*}
	a(x-h_2) e^{i\lambda \Phi (x-h_2)\cdot \zeta}= \big(a(x)-\nabla a(x)\cdot h_2 +\underbrace{(a(x-h_2) e^{i\lambda R_{\Phi}(x,h_2)\cdot\zeta} - a(x)+\nabla a(x)\cdot h_2}_{=:{Y}^\zeta_\lambda(x,h_2)} \big) e^{i\lambda \Phi(x)\cdot\zeta}e^{-i \lambda \nabla \Phi(x)h_2\cdot\zeta},
\end{equation*}
and we denote 
\begin{equation*}
	\begin{aligned}
		Y_\lambda^{\eta,\zeta}(x,h_1,h_2) &= \left(a(x)-\nabla a(x)\cdot h_2\right){Y}_\lambda^\zeta(x,h_1)+\left(a(x)-\nabla a(x)\cdot h_1\right){Y}_\lambda^\zeta(x,h_2)\\
		&+  Y_\lambda^\eta(x,h_1){Y}_\lambda^\zeta(x,h_2)+\left(\nabla a(x)\cdot h_1\right)\left(\nabla a(x)\cdot h_2\right).
	\end{aligned}
\end{equation*}
Then we have
\begin{eqnarray*}
	B_\lambda(x) &=& \frac{1}{4}\sum_{\zeta,\eta\in\left\{\xi,-\xi\right\}}e^{i\lambda\Phi(x)\cdot(\eta+\zeta)}\int_{\mathbb R^2 \times \mathbb R^2} a^2(x) \widehat{iy\otimes K}_{\lambda}(h_1, h_2) e^{-i\lambda (\nabla \Phi^{T}\eta)\cdot h_1} e^{-i\lambda (\nabla \Phi^{T}\zeta)\cdot h_2} dh_1 dh_2 \\ 
	&&-\frac{1}{4}\sum_{\zeta,\eta\in\left\{\xi,-\xi\right\}}e^{i\lambda\Phi(x)\cdot(\eta+\zeta)}\underbrace{\int_{\mathbb R^2 \times \mathbb R^2} a(x)\nabla a(x)\cdot(h_1+h_2) \widehat{iy\otimes K}_{\lambda}(h_1, h_2) e^{-i\lambda (\nabla \Phi^{T}\eta)\cdot h_1} e^{-i\lambda (\nabla \Phi^{T}\zeta)\cdot h_2} dh_1 dh_2}_{=0} \\ 
	&& + \frac{1}{4}\sum_{\zeta,\eta\in\left\{\xi,-\xi\right\}}\underbrace{e^{i\lambda\Phi(x)\cdot(\eta+\zeta)}\int_{\mathbb R^2 \times \mathbb R^2} Y_\lambda^{\eta,\zeta}(x,h_1,h_2)\widehat{iy\otimes K}_{\lambda}(h_1, h_2) e^{-i\lambda (\nabla \Phi^{T}\eta)\cdot h_1} e^{-i\lambda (\nabla \Phi^{T}\zeta)\cdot h_2} dh_1 dh_2}_{=:\delta B_{\lambda}^{\eta,\zeta}(x)}. \\
	&=&\frac{1}{4} \sum_{\zeta,\eta\in\left\{\xi,-\xi\right\}}e^{i\lambda\Phi(x)\cdot(\eta+\zeta)} a^2(x) \left(iy\otimes K_{\lambda}\right)(\lambda \nabla \Phi^{T}\eta, \lambda \nabla \Phi^{T}\zeta) + \delta B_\lambda^{\xi}(x).\\
	&=&\frac{1}{2}\lambda^{\delta-1} a^2(x)\left(\nabla\Phi^{T}\xi\right)\otimes\nabla \bar{m}\left(\nabla\Phi^{T}\xi\right) + \delta B_\lambda^{\xi}(x).
\end{eqnarray*}
where the second term in the first equality is actually zero by the oddness of $K_\lambda$. The point here is that due to such special cancellation, our explicit expression of $B_\lambda(x)$ is unaffected. But if one uses the Taylor expansion up to higher order, then besides the main term $\frac{1}{2}\lambda^{\delta-1} a^2(x)\left(\nabla\Phi^{T}\xi\right)\otimes\nabla \bar{m}\left(\nabla\Phi^{T}\xi\right)$ above, there will also be some lower order terms  involving higher derivatives of $a(x)$ and $K_\lambda$. This will make our algebraic lemma harder to analyze.
 
	\end{rem}

	\begin{lem}	\label{microlem}(\cite{IV})
		(Microlocal Lemma). Suppose that $T^l[\Theta](x)=\int_{\mathbb{R}^2} \Theta(x-h) K^l(h) d h$ is a convolution operator acting on functions $\Theta: \mathbb{T}^2 \rightarrow \mathbb{C}$ with a kernel $K^l: \mathbb{R}^2 \rightarrow \mathbb{C}$ in the Schwartz class. Let $\xi: \mathbb{T}^2 \rightarrow \mathbb{C}$ and $\theta: \mathbb{T}^2 \rightarrow \mathbb{C}$ be smooth functions and $\lambda \in \mathbb{Z}$ be an integer. Then for any input of the form $\Theta=e^{i \lambda \xi(x)} \theta(x)$ we have the formula
		$$
		\left.T^l[\Theta](x)=e^{i \lambda \xi(x)}\left(\theta(x) \hat{K}^l(\lambda \nabla \xi(x))+\delta[T \Theta]\right)(x)\right)
		$$
		where the error in the amplitude term has the explicit form
		$$
		\begin{aligned}
			\delta\left[T^l \Theta\right](x) & =\int_0^1 d r \frac{d}{d r} \int_{\mathbb{R}^2} e^{-i \lambda \nabla \xi(x) \cdot h} e^{i Z(r, x, h))} \theta(x-r h) K^l(h) d h \\
			Z(r, x, h) & =r \lambda \int_0^1 h^a h^b \partial_a \partial_b \xi(x-s h)(1-s) d s
		\end{aligned}
		$$
	\end{lem}
	
	\section{An sketch of proof of the Onsager conjecture for even active scalar equations.}
	In this section, we give a brief framework  of the convex integration scheme for even active scalar equations but do not provide any detailed proofs, since the proofs are almost the same as (and actually much simpler than that for odd active scalar equations.) ones given in the \cite{GR23} for the 2D Euler equation. Consider the following equations
	
	\begin{equation} \label{ie}
		\begin{cases}
			\partial_t \theta + \div(\theta u)  = 0, \\ 
			\div u = 0,\\
			u=T[\theta],
		\end{cases}
	\end{equation}
	
	The operator $T[\cdot]$ defining the drift velocity $u$ in \ref{ie} is represented in the frequency space by a multiplier
	$$
	\hat{u}(\xi)=\widehat{T}[\theta](\xi)=m(\xi) \hat{\theta}(\xi) .
	$$
	We assume that $m(\xi)$ is defined on the whole frequency space as a tempered distribution and is homogeneous of degree 0 so that $T$ is an operator of order 0 . The multiplier must satisfy $m(-\xi)=\overline{m(\xi)}$ so that the driftvelocity $u$ is real-valued whenever the scalar $\theta$ is real-valued, and  $m(\xi)$ is assumed to be smooth away from the origin. The requirement that $u$ is divergence free corresponds to the requirement that $m(\xi)$ takes values perpendicular to the frequency vector $\xi$, i.e. $\xi \cdot m(\xi)=0$ for $\xi \neq 0$.
	
	\begin{thm}\label{main-thm2}
Consider the active scalar equations \ref{main-thm2} with divergence free drift velocity, and assume that the multiplier $m(\xi)$ defining the operator $T$ satisfies the conditions imposed above. Let $\alpha<1 / 3$ and  I be an open interval. Then there exist nontrivial solutions to \ref{main-thm2} with Hölder regularity $\theta, u \in C_{t, x}^\alpha\left(\mathbb{R} \times \mathbb{T}^2\right)$ which are identically 0 outside of $I \times \mathbb{T}^2$.
	\end{thm}
	
	The Hölder exponent $1/3$ in the above result is sharp since the energy $\int|\theta|^2(t, x) d x$ is a conserved quantity and uniqueness holds for solutions with Hölder regularity above $\alpha>1 / 3$ . This result improves the previous result of \cite{IV} where nonunique solutions are constructed with Hölder regularity exponent $1/9$.

	\subsection{Parameters, the inductive assumptions, and the main proposition}
	
	\label{Section_Main_Prop}
	The proof of Theorem \ref{main-thm2} will be achieved by the iterative construction of smooth solutions $(\theta_q, u_q, R_q)$ to the scalar-stress equations
	\begin{equation} \label{AS}
		\begin{cases}
			\partial_t \theta_q + \div(\theta_q  u_q) = \div R_q, \\ 
			\div u_q = 0,\\
			u_q=T[\theta_q],
		\end{cases}
	\end{equation}
	where the  stress $R_q$ is a vector field. Here and throughout $q \in \mathbb N$ will denote the stage of the iteration. The goal is to construct this sequence so that $(\theta_q,u_q)$ converges in the required H\"older space, while $R_q$ converges to zero. In the limit, we will have, thus, recovered a weak solution to the ASE. 
	
	We start with defining frequency parameters which will quantify the approximate Fourier support of $u_q$,
	\begin{equation*}
		\lambda_q = 2\pi \lceil a^{b^q} \rceil,
	\end{equation*}
	as well as amplitude parameters 
	\begin{equation*}
		\delta_q = \lambda_q^{-2\beta}.
	\end{equation*}
	The constant $a > 1$ will be chosen to be large, $b>1$ will be close to $1$, while $0 < \beta< 1/3$ will determine the H\"older regularity of the constructed solution.
	
	Let $L \in \mathbb N \setminus \{0\}$, $M > 0$ and $0< \alpha < 1$.Then for any $ N \in \{0, 1,..., L-r\}$ and any $\forall r \in \{0, 1\}$,we assume the following inductive estimates:
	
	\begin{equation} \label{ind_est_1}
		\|\theta_q\|_0+\|u_q\|_0 \leq  M(1 - \delta_q^{1/2}), 
	\end{equation}
	\begin{equation} \label{ind_est_2}
		\|D_t^r \theta_q\|_N+\|D_t^r u_q\|_N \leq M \delta_q^{1/2} (\delta_q^{1/2}\lambda_q)^r\lambda_q^N,\quad N+r\geq1 \,\,\, 
	\end{equation}
	\begin{equation} \label{ind_est_3}
		\|D_t^r R_q\|_{N} \leq \delta_{q+1} (\delta_q^{1/2}\lambda_q)^r\lambda_q^{N - 2\alpha}, \,\,\, 
	\end{equation}
	where  $D_t$ denotes the material derivative corresponding to the velocity field $u_q$: 
	\begin{equation*}
		D_t = \partial_t + u_q \cdot \nabla. 
	\end{equation*}
	We also assume the following on the temporal support of the stress: 
	\begin{equation} \label{ind_est_4}
		\supp_t R_q \subset [-2 + (\delta_q^{1/2} \lambda_q)^{-1}, -1 - (\delta_q^{1/2} \lambda_q)^{-1}] \cup [1+(\de_q^{1/2}\la_q)^{-1}, 2-(\de_q^{1/2}\la_q)^{-1}],
	\end{equation}
 $a$ may be taken to be large enough such that  $	(\delta_0^{1/2} \lambda_0)^{-1} < \frac{1}{4}$ , so the definition of $\supp_t R_0$ makes sense.
	
	\begin{prop} \label{Main_prop2}
		Assume $L \geq 10$, $0< \beta < 1/3$ and
		\begin{equation*}
			1 < b < \frac{1 + 3\beta}{6\beta}.
		\end{equation*}
		Then there is a universal constant $M_0 > 0$ depending  on $\beta$ and $L$, and an $\alpha_0$ depending on $\beta$ and $b$, such that for any  $0< \alpha < \alpha_0$, there exists $a_0 > 1$ depending on $L$, $\beta$, $b$, $\alpha$, and $M_0$, such that for any $a > a_0$ the following holds: given a smooth solution $(\theta_q, u_q, R_q)$ to \eqref{AS} satisfying \eqref{ind_est_1} - \eqref{ind_est_4}, there exists a smooth solution $(\theta_{q+1}, u_{q+1}, R_{q+1})$ to \ref{AS} satisfying \eqref{ind_est_1} - \eqref{ind_est_4} with $q$ replaced by $q+1$ . Moreover, is holds that
		\begin{equation} \label{Main_prop_eqn}
			\|\th_{q+1} - \th_{q}\|_0 +\|u_{q+1} - u_{q}\|_0 + \frac{1}{\lambda_{q+1}} \|\th_{q+1} - \th_{q}\|_1 + \frac{1}{\lambda_{q+1}} \|u_{q+1} - u_{q}\|_1 \leq 2M_0 \delta_{q+1}^{1/2}, 
		\end{equation}
		and 
		\begin{equation} \label{Main_prop_eqn_2}
			\supp_t (\th_{q+1} - \th_q)\cup\supp_t (u_{q+1} - u_q) \subset (-2, -1) \cup (1, 2).
		\end{equation}
	\end{prop}
		\subsection{Proof of Theorem \ref{main-thm2}}
	
	We fix the choice $L \geq 4$, $\beta < 1/3$, $b$, $M_0$ , $\alpha_0$ and  $a_0$  be the constants in Proposition \ref{Main_prop2}. Then let  $M > M_0+|m(1,0)|+1$ and $\alpha < \min \{\alpha_0, 1/4\}$ be fixed. Then since $(2b-1)\beta< 1/3$, one can  take $a\geq a_0$ sufficiently large such that 
	\begin{equation}\label{a}
		\delta_0^{1/2}\leq 1/2,\quad \|f\|_1\leq \min\{\delta_1\delta_0^{-1/2}\lambda_0, \delta_0^{1/2}\lambda_0\},\quad \|f\|_2\leq \delta_1 \lambda_0^{3/2}.
	\end{equation}
	
	First, We  construct the base case $(\theta_0,u_0,R_0)$. Let $f : \R \to [0,1]$ be a smooth bump function such that $f = 1$ on $[-5/4, 5/4]$ and $\supp f \subset[-7/4,7/4]$.
	Then the base case is given by  
	\begin{equation*}
		\begin{aligned}
			&\theta_0 (x,t) =  \delta_0^{1/2}f(t)\cos ( \la_0 x_1),\quad \\
			&u_0=T[\theta_0]=  \frac{1}{2}\delta_0^{1/2}f(t)\left(m(\lambda_0,0)e^{i\lambda_0 x_1}+\overline{m(\lambda_0,0)}e^{-i\lambda_0 x_1}\right),\\
			&R_0(x,t) = \frac{\delta_0^{1/2}}{\lambda_0}f'(t)\left(\sin(\lambda_0 x_1),0\right)\\
		\end{aligned}
	\end{equation*}
	
	Due to the divergence free condition $\xi\cdot\ m(\xi)=0$, it is easy to check that $(\theta_0.u_0,R_0)$ solves the Scalar-stress system~\eqref{AS} and  the following estimates holds:
	
	\begin{equation*}
		\begin{aligned}
			&\|\theta_0\|_0 \leq M\delta_0^{1/2} \leq M(1 - \delta_0^{1/2}),\\
			&\|u_0\|_0 \leq M\delta_0^{1/2} \leq M(1 - \delta_0^{1/2}),\\
		\end{aligned}
	\end{equation*}
	Hence  the inductive assumption \eqref{ind_est_1} holds. Also, for $N\geq 0$, one has that
	\begin{equation*}
		\begin{aligned}
			&\|\theta_0\|_N \leq M\delta_0^{1/2} \lambda_0^N,\\
			&\|u_0\|_N \leq M\delta_0^{1/2} \lambda_0^N,\\
			&	\|R_0\|_N \leq  \sup_t |f'(t)| \frac{\delta_0^{1/2}}{\lambda_0} \lambda_0^N \leq\delta_1 \lambda_0^{-1/2} \lambda_0^N \leq \delta_1 \lambda_0^{-2\alpha} \lambda_0^N ,
		\end{aligned}
	\end{equation*}
	where  $\alpha\leq\frac{1}{4} $ has been used. Thus  \eqref{ind_est_2} and \eqref{ind_est_3} holds for $r=0$.
	
	For the case $r=1$, direct computations yield
	
	\begin{align*}
		&\partial_t \theta_0 + u_0 \cdot \nabla \theta_0=\delta_0^{1/2}f'(t)\cos ( \la_0 x_1),\\
		&\partial_t u_0 + u_0 \cdot \nabla u_0=\frac{1}{2}\delta_0^{1/2}f'(t)\left(m(\lambda_0,0)e^{i\lambda_0 x_1}+\overline{m(\lambda_0,0)}e^{-i\lambda_0 x_1}\right),\\
		&\partial_t R_0 + u_0 \cdot \nabla R_0=\frac{\delta_0^{1/2}}{\lambda_0}f''(t)\left(\sin(\lambda_0 x_1),0\right)\\
	\end{align*}
	In view of the condition for $a$ in \eqref{a}, it is clear that \eqref{ind_est_2} and \eqref{ind_est_3} hold for $r=1$.
	
	The temporal support condtion \eqref{ind_est_4} is satisfied since  $\supp_t R_0 \subset\supp f'\subset [-7/4, 7/4] \setminus (-5/4, 5/4)$ and $(\delta_0^{1/2} \lambda_0)^{-1} < \frac{1}{4}$.

	Therefore, $(\theta_0, u_0, R_0)$ satisfy all the inductive assumptions for the Scalar-stress equations \eqref{AS}. Now let $\{(\theta_q, u_q, R_q)\}$ be the sequence of solutions to \eqref{AS} given by the proposition.   Then, \eqref{ind_est_3} and \eqref{Main_prop_eqn} implying that 
	\begin{equation*}
		\begin{aligned}		
			&\|\theta_{q+1} - \theta_q\|_{\beta'} \lesssim \|\theta_{q+1} - \theta_q\|_0^{1-\beta'} \|\theta_{q+1} - \theta_q\|_1^{\beta'} \lesssim \delta_{q+1}^{1/2} \lambda_{q+1}^{\beta'} \lesssim \lambda_{q+1}^{\beta' - \beta},\\
			&\|u_{q+1} - u_q\|_{\beta'} \lesssim \|u_{q+1} - u_q\|_0^{1-\beta'} \|u_{q+1} - u_q\|_1^{\beta'} \lesssim \delta_{q+1}^{1/2} \lambda_{q+1}^{\beta'} \lesssim \lambda_{q+1}^{\beta' - \beta},\\
			&\|R_q\|_{\beta'} \lesssim \|R_q\|_0^{1-\beta'}\|R_q\|_1^{\beta'} \lesssim \delta_{q+1}\lambda_q^{\beta'} \lesssim \lambda_{q+1}^{\beta' - 2 \beta}.
		\end{aligned}
	\end{equation*}
	with $\beta'< \beta < 1/3$. Therefore, $\{\theta_q\}$, $\{u_q\}$, and $\{R_q\}$ are Cauchy sequences in $C_tC^{\beta'}_x$ and they converge to $(\theta,u,0)$ in $C_tC^{\beta'}_x$, and it holds that $u=T[\theta]$.
	
	In conclusion,   $\theta$ is a weak solution for the active scalar equations with $\theta \in C_tC^{\beta'}_x$, which satisfies $\supp_t u \subset [-2,2]$ by inductive assumption \eqref{Main_prop_eqn_2}, and for $t\in[-1,1]$, it holds that
	\begin{equation*}
		\theta(x,t) = \delta_0^{1/2} \cos(\lambda_0 x_1),
	\end{equation*}
	Then to prove Proposition \eqref{Main_prop2}, one can use the standard Newton-Nash scheme as used for the 2D Euler equations in \cite{GR23}. Since this process is quite similar but tedious, so we just conclude the main results here without providing  details.
    \begin{rem}
    For higher dimensional even active scalar equation, as noted in \cite{IV}, we can get the same result as long as one requires the image of the even part of the multiplier contains the following $d$ vectors 
    $$A_{(i)}=m(\xi^{(i)})+m(-\xi^{(i)}),\quad i=1, 2, ..., d$$
    such that the vectors $A_{(1)}, ..., A_{(d)}$ span $\mathbb{R}^{d}$. This condition is satisfied by the 3D IPM equation and the MG equation.
    \end{rem}
	
	\subsection*{Acknowledgements}
 This is part of my PH.D. thesis written under the supervision of Professor Zhouping XIN at the Institute of mathematical sciences, the Chinese University of Hong Kong. I would like to express my gratitude to my Ph.D. advisor, Zhouping XIN, for his guidance and valuable discussions. I also thank Han CUI for helpful discussions on the part of  multilinear Fourier multiplier estimates  related to this paper. This research is supported by Hong Kong RGC  Research Grants CUHK-14301421, CUHK-14301023, CUHK-14302819, CUHK-14300819, and the key project of NSFC Grants No.12131010 and No.11931013.

\end{document}